\documentclass[a4paper]{report}
\usepackage[T1]{fontenc}
\usepackage{lmodern}   
\usepackage{xcolor}

\usepackage{csquotes}

\usepackage{ifthen}
\usepackage{amsmath,amsthm,amssymb,amsfonts}
\usepackage{mathtools}
\usepackage{tensor}
\usepackage{tikz-cd}
\usepackage{esint}
\usepackage{mathrsfs}
\usepackage{thmtools}
\usepackage{slashed}
\usepackage{mathabx}

\usepackage{bbm}
\usepackage{microtype}
\usepackage{accents}
\usepackage{quiver}
\usepackage{enumitem}
\usepackage[toc,page]{appendix}
\usepackage{imakeidx}

\usepackage{xstring}

\setlist[enumerate,1]{label = \roman*.,
ref = \roman*}
\usepackage{emptypage}
\usepackage[a4paper,	bindingoffset=0in,		left=1.2in,		right=1.2in,	top=1in,	bottom=1in,	footskip=.25in]{geometry}
\usepackage{setspace}
\usepackage[backend=biber,style=alphabetic,sorting=ynt,sortcites]{biblatex}
\usepackage{crossreftools}
\usepackage{hyperref}
\usepackage{cleveref}
\allowdisplaybreaks[1]
\theoremstyle{plain}

\usepackage[english]{babel}

\newtheorem{theorem}{Theorem}[section]
\newtheorem*{theorem*}{Theorem}
\newtheorem{thmx}{Theorem}

\newtheorem{cor}[theorem]{Corollary}
\newtheorem{lem}[theorem]{Lemma}
\newtheorem{prop}[theorem]{Proposition}
\theoremstyle{definition}

\newtheorem{ex}[theorem]{Example}
\newtheorem{exs}[theorem]{Examples}
\newtheorem{dfn}[theorem]{Definition}

\newtheorem{rem}[theorem]{Remark}
\newtheorem{rems}[theorem]{Remarks}

\theoremstyle{remark}

\newtheoremstyle{linked}
  {}
  {}
  {\itshape}
  {}
  {\bfseries}
  {.}
  { }
{\hyperlink{InternalLink:\thislink}{\thmname{#1} \thmnumber{#2}}\thmnote{ (#3)}}
\theoremstyle{linked}

\newtheorem{innlinkthm}[theorem]{Theorem} 
 \NewDocumentEnvironment{linkthm}{m o}
   { 
   \IfNoValueTF{#2}{\def\thislink{#1}\begin{innlinkthm}\label{thm:#1}}{\def\thislink{#1}\begin{innlinkthm}[#2]\label{thm:#1}}
   }
   {\end{innlinkthm}}

 \newtheorem{innlinklem}[theorem]{Lemma} 
 \newenvironment{linklem}[1]
  {\def\thislink{#1}\innlinklem\label{thm:#1}}
  {\endinnlinkthm}

  \newtheorem{innlinkprop}[theorem]{Proposition} 
  \NewDocumentEnvironment{linkprop}{m o}
   { 
   \IfNoValueTF{#2}{\def\thislink{#1}\begin{innlinkprop}\label{thm:#1}}{\def\thislink{#1}\begin{innlinkprop}[#2]\label{thm:#1}}
   }
   {\end{innlinkprop}}
\NewDocumentEnvironment{linkproof}{m o}
  {
    \IfNoValueTF{#2}
      {
        \hypertarget{InternalLink:#1}{
          \paragraph{Proof of \Cref{thm:#1}.}
        }
      }
      {
        \hypertarget{InternalLink:#1}{}
        \hypertarget{InternalLink:#2}{
          \paragraph{Proof of \Cref{thm:#1} and \Cref{thm:#2}.}
        }
      }
  }
  {
    \hfill$\qed${\parfillskip=0pt\par}
  }

 \newcommand\linksec[3][\DefaultOpt]{\def\DefaultOpt{#2}\section[#1]{\texorpdfstring{\hyperref[sec:#3_proof]{#2}}{#2}}\label{sec:#3}}
 \newcommand\linkchap[3][\DefaultOpt]{\def\DefaultOpt{#2}\chapter[#1]{\texorpdfstring{\hyperref[chap:#3_proof]{#2}}{#2}}\label{chap:#3}}

\newcommand{\RNum}[1]{\uppercase\expandafter{\romannumeral #1\relax}}

\makeatletter
\providecommand*{\twoheadrightarrowfill@}{%
  \arrowfill@\relbar\relbar\twoheadrightarrow}
\providecommand*{\twoheadleftarrowfill@}{%
  \arrowfill@\twoheadleftarrow\relbar\relbar}
\providecommand*{\xtwoheadrightarrow}[2][]{%
  \ext@arrow 0579\twoheadrightarrowfill@{#1}{#2}}
\providecommand*{\xtwoheadleftarrow}[2][]{%
  \ext@arrow 5097\twoheadleftarrowfill@{#1}{#2}}
\newcommand\setItemnumber[1]{\setcounter{enum\romannumeral\@enumdepth}{\numexpr#1-1\relax}}
\makeatother

\newcommand\norm[1]{\left\lVert#1\right\rVert}
\newcommand\Crefitem[2]{\nameCref{#1} \hyperref[#1:#2]{\ref*{#1}.\ref*{#1:#2}}}
\newcommand\refitem[2]{\hyperref[#1:#2]{\ref*{#1}.\ref*{#1:#2}}}

\NewDocumentCommand{\MyCref}{m o}{%
  \IfNoValueTF{#2}
    {\Cref{#1}} 
    {\nameCref{#1} \hyperref[#1:#2]{\ref*{#1}.\ref*{#1:#2}}}
}

\DeclareMathOperator\supp{supp}

\DeclareMathOperator\Ad{Ad}

\DeclareMathOperator\sgn{sgn}

\DeclareMathOperator\Hom{Hom}
\DeclareMathOperator\im{Im}

\DeclareMathOperator\dom{Dom}
\DeclareMathOperator\codim{Codim}

\newcommand{\R}{\mathbb{R}}

\newcommand{\N}{\mathbb{N}}
\newcommand{\C}{\mathbb{C}}
\newcommand{\Z}{\mathbb{Z}}

\newcommand{\cA}{\mathcal{A}}

\newcommand{\cF}{\mathcal{F}}
\newcommand{\cG}{\mathcal{G}}

\newcommand{\cK}{\mathcal{K}}
\newcommand{\cL}{\mathcal{L}}
\newcommand{\cM}{\mathcal{M}}

\newcommand{\cQ}{\mathcal{Q}}

\newcommand{\cX}{\mathcal{X}}

\newcommand{\cZ}{\mathcal{Z}}
\newcommand{\Rpt}{\mathbb{R}_+^\times}

\newcommand\bb[1]{[\![ #1]\!]}

\usepackage{derivative}
\usepackage{calc}

\usepackage{accents}
\newcommand{\dbtilde}[1]{\accentset{\approx}{#1}}

\usepackage{tcolorbox}
\tcbuselibrary{theorems}
\newtcbox{\redbox}[1]{colback=red!5!white,
colframe=red!75!black,fonttitle=\bfseries,
title={#1}}

\newtcbtheorem[number within=section]{redBoxTheorem}{Theorem}%
{colback=red!5!white,
colframe=red!75!black,fonttitle=\bfseries}{thmBox}

\usepackage{prerex}
\AtEveryBibitem{\clearfield{url}}

\DeclareMathOperator{\BCH}{BCH}

\newcommand{\Id}{{\hbox{Id}}}

\newcommand{\weak}[1]{\hat{#1}}

\newcommand{\normtriple}[1]{{\left\vert\kern-0.25ex\left\vert\kern-0.25ex\left\vert #1 
    \right\vert\kern-0.25ex\right\vert\kern-0.25ex\right\vert}}
\newcommand{\Grass}{\mathrm{Grass}}

\newcommand{\WF}{\mathrm{WF}}

\newcommand{\vol}{\mathrm{vol}}

\newcommand{\omegahalf}{\Omega^\frac{1}{2}}

\newcommand{\HL}[1]{\mathrm{H\!N}_{#1}}
\newcommand{\HN}{\mathrm{H\!N}}

\newcommand{\grF}[1]{\mathfrak{g}_{#1}}
\newcommand{\GrF}[1]{\mathfrak{G}_{#1}}
\newcommand{\cGF}{\mathcal{G}}
\newcommand{\HatHLnon}{\widehat{\mathrm{H\!N}}^{\mathrm{Nonsing}}}
\newcommand{\HLnon}{\mathrm{H\!N}^{\mathrm{Nonsing}}}

\newcommand\cinv{C^\infty_{c,\mathrm{inv}}(\mathbb{G},\omegahalf)}
\newcommand\cbbG{C^\infty_{c}(\mathbb{G},\omegahalf)}

\newcommand\cinvf[1]{C^\infty_{c,\mathrm{inv},#1}(\mathbb{G},\omegahalf)}
\newcommand{\cinvdis}{C^{-\infty}_{r,s,\mathrm{inv}}(\mathbb{G},\omegahalf)}
\newcommand{\tdom}{\widetilde{\dom}}
\usepackage[symbols]{glossaries-extra}

\DeclareMathOperator{\singsupp}{singsupp}
\newcommand{\DO}{\mathrm{DO}}

\makeatletter
\ams@newcommand{\iiiiint}{\DOTSI\protect\MultiIntegral{5}}
\renewcommand{\MultiIntegral}[1]{%
\edef\ints@c{\noexpand\intop
\ifnum#1=\z@\noexpand\intdots@\else\noexpand\intkern@\fi
\ifnum#1>\tw@\noexpand\intop\noexpand\intkern@\fi
\ifnum#1>\thr@@\noexpand\intop\noexpand\intkern@\fi
\ifnum#1>4 \noexpand\intop\noexpand\intkern@\fi 
\noexpand\intop
\noexpand\ilimits@
}%
\futurelet\@let@token\ints@a
}
\makeatother
\begin{document}
\title{Microlocal maximal hypoellipticity from the geometric viewpoint: I}
\author{Omar Mohsen\thanks{{
Université Paris Cité, Sorbonne Université, CNRS, IMJ-PRG, F-75013 Paris, France, \href{mailto:omar.mohsen@imj-prg.fr}{\texttt{omar.mohsen@imj-prg.fr}}}}}
\date{}
\maketitle
\begin{abstract}
	Given some vector fields on a smooth manifold satisfying Hörmander's condition, 
	we define a bi-graded pseudo-differential calculus which contains the classical pseudo-differential calculus and a pseudo-differential calculus adapted to the sub-Riemannian structure induced by the vector fields.

	Our approach is based on geometric constructions (resolution of singularities) together with methods from operators algebras. 
	We develop this calculus in full generality, including Sobolev spaces, the wavefront set, and the principal symbol, etc.

	In particular, using this calculus, we prove that invertibility of the principal symbol implies microlocal maximal hypoellipticity.
	This allows us to resolve affirmatively the microlocal version of a conjecture of Helffer and Nourrigat.

\end{abstract}

\setcounter{tocdepth}{2} 
\tableofcontents
\addcontentsline{toc}{chapter}{Introduction}
\chapter*{Introduction}

	Let $M$ be a smooth manifold, $\Omega^\frac{1}{2}$ the bundle of half-densities on $TM$. A pseudo-differential operator $P:C^\infty(M,\omegahalf)\to C^\infty(M,\omegahalf)$ is called hypoelliptic if for any distribution $w\in C^{-\infty}(M,\omegahalf)$ on $M$, 
		\begin{equation*}\begin{aligned}
			\mathrm{singsupp}(P(w))=\mathrm{\singsupp}(w).
		\end{aligned}\end{equation*}
	Equivalently, if $P(w)$ is smooth on an open subset $U\subseteq M$, then $w$ is also smooth on $U$.
	Similarly, $P$ is called microlocally hypoelliptic on a cone (a subset that is closed under the $\Rpt$-dilation) $\Gamma\subseteq T^*M\backslash 0$ if
		\begin{equation*}\begin{aligned}
			\WF(P(w))\cap \Gamma=\WF(w)\cap \Gamma,\quad \forall w\in C^{-\infty}(M,\omegahalf).
		\end{aligned}\end{equation*}
	If $P$ is microlocally hypoelliptic on $T^*M\backslash 0$, then it is hypoelliptic.
	A well-known sufficient criterion for microlocal hypoellipticity is that the classical principal symbol of $P$ vanishes nowhere on $\Gamma$.

	In our previous work  \cite{MohsenMaxHypo}, we generalised the main regularity theorem for elliptic operators to a sub-Riemannian setting. Thus giving a general sufficient criterion for hypoellipticity of differential operators, and answering the global and local versions of the Helffer-Nourrigat conjecture.
	In this article, we generalise our theorem in two simultaneous directions: First, we microlocalize the theorem and thus solving the microlocal version of Helffer-Nourrigat conjecture.
	Second, we extend our theorem to what we call multi-parameter sub-Riemannian geometry.

	Let us start with the context for our theorem. For simplicity, we will describe the context of our theorem in the $2$-parameter setting.
	Our article treats the more general setting of $\nu$-parameters where $\nu\in \N$ is arbitrary.

	Let	$X_1,\cdots,X_n$ vector fields satisfying Hörmander's condition, that is for any $x\in M$, the tangent space $T_xM$ is linearly spanned
	by $X_1(x),\cdots,X_n(x)$ and $[X_i,X_j](x)$ and higher iterated Lie brackets evaluated at $x$.
	Let $a_1,\cdots,a_n\in \N$ be any natural numbers thought of as weights associated to each vector field $X_i$.

	Let $s\in \R$, $t\in \N$ such that $a_i|t$ for all $i$.
	We define Sobolev spaces $H^{(s,t)}(M)$ adapted to the vector fields $X_1,\cdots,X_n$ and their weights $a_1,\cdots,a_n$ as follows: For any $w\in C^{-\infty}(M,\omegahalf)$,
		\begin{equation*}\begin{aligned}
			w\in H^{(s,t)}(M)\iff Q(X_1,\cdots,X_n)u\in L^2_{\mathrm{loc}}(M) \forall Q,	
		\end{aligned}\end{equation*}
	where $Q(x_1,\cdots,x_n)$ is any noncommutative formal polynomial of weighted degree $\leq t$ where each variable $x_i$ is given weight $a_i$, and whose coefficients are classical pseudo-differential operators of order $\leq s$.
	
	In this article we extend these Sobolev spaces to $s,t\in \R$ by an interpolation type procedure. We prove that these Sobolev spaces satisfy the following Sobolev lemma:
		\begin{equation}\label{intro:Sobolev_lemma}\begin{aligned}
			C^\infty(M,\omegahalf)=\bigcap_{s\in \R}H^{(s,t)}_{\mathrm{loc}}(M)=\bigcap_{t\in \R}H^{(s,t)}_{\mathrm{loc}}(M),\quad \forall s,t\in \R.
		\end{aligned}\end{equation}
		These Sobolev spaces generalise both the classical local Sobolev spaces $H^s_{\mathrm{classical}}(M)$ and the Sobolev spaces we defined in \cite{MohsenMaxHypo}, $H^t_{\mathrm{sub-Riemannian}}(M)$, by the relations 
			\begin{equation*}\begin{aligned}
				H^s_{\mathrm{classical}}(M)=H^{(s,0)}_{\mathrm{loc}}(M),\quad H^t_{\mathrm{sub-Riemannian}}(M)=H^{(0,t)}_{\mathrm{loc}}(M),\quad \forall s,t\in \R.
			\end{aligned}\end{equation*}
		
		 We also define a bi-graded pseudo-differential calculus $\Psi^{(k,l)}(M)$, where $k,l\in \C$.
	This calculus contains both the classical pseudo-differential calculus and a pseudo-differential calculus adapted to the sub-Riemannian structure defined by the vector fields $X_1,\cdots,X_n$ and their weights $a_1,\cdots,a_n$.
	More precisely, we have the following properties: 
		\begin{enumerate}
		\item  If $k\in \C,l\in \N$, then for any $Q(x_1,\cdots,x_n)$ noncommutative formal polynomial of weighted degree $\leq l$ whose coefficients are classical pseudo-differential operators of order $\leq k$, one has $Q(X_1,\cdots,X_n)\in \Psi^{(k,l)}(M)$. 
		In particular, if $P$ is a classical pseudo-differential operator of order $k\in \C$, then $P\in \Psi^{(k,0)}(M)$, and for any $i$, $X_i\in \Psi^{(0,a_i)}(M)$. 
		\item If $P$ is an operator in the calculus we defined in \cite{MohsenMaxHypo} of order $k\in \C$, then $P\in \Psi^{(0,k)}(M)$. 
				For example, for any $s\in \N$ such that $2a_i|s$ for all $i$, and the generalised Hörmander sum of squares 
			\begin{equation*}\begin{aligned}
				\Delta_{\mathrm{SubRiem}}=\sum_{i=1}^n  (X_i^{*}X_i)^\frac{s}{2a_i}\in \Psi^{(0,s)}(M)		
			\end{aligned}\end{equation*}
		admits a parametrix in $\Psi^{(0,-s)}(M)$.
	\end{enumerate}
	So, this calculus allows one to mix operators from Riemannian geometry like elliptic operators and operators in sub-Riemannian geometry like the generalised Hörmander sum of squares and its parametrix.
	This mixture of calculi together with the fact that our definition is quite geometric (see \Cref{chap:bi-graded_tangent_groupoid}) 
	and that our calculus has a principal symbol (which we will define shortly) makes this calculus useful in transporting questions from sub-Riemannian geometry to 
	classical pseudo-differential operators and vice versa, as well as apply homotopy type arguments between the two.
	For example, homotopies of the form $\Delta^t\Delta_{\mathrm{SubRiem}}^{1-t}\in \Psi^{(2t,s(1-t))}(M)$.
	The fact that we define appropriate Sobolev spaces on which such operators are Fredholm and that Fredholmness in general can be detected using our principal symbol is among the principal strengths of our calculus.
	This will be exploited in forthcoming articles.
	
	Operators in our calculus act continuously between the Sobolev spaces defined above:
		\begin{equation}\label{eqn:qksdfjmkqsdjfmqsd}\begin{aligned}
			P\in \Psi^{(k,l)}(M)\implies P\left(H^{(s+\Re(k),t+\Re(l))}_{\mathrm{loc}}(M)\right) \subseteq H^{(s,t)}_{\mathrm{loc}}(M),\quad \forall k,l\in \C,s,t\in \R.
		\end{aligned}\end{equation}
	
	Motivated by microlocal hypoellipticity, we define a refinement of the wave-front set. Let $s,t\in \R$, $w\in C^{-\infty}(M,\omegahalf)$.
	The set $\WF_{(s,t)}(w)$ is the cone in $T^*M\backslash 0$ defined as follows:
		\begin{equation*}\begin{aligned}
			(\xi,x)\notin \WF_{(s,t)}(w)\iff \exists w'\in C^{-\infty}(M,\omegahalf), w'\in H^{(s,t)}_{\mathrm{loc}}(M)\text{ and }(\xi,x)\notin \WF(w-w').		
		\end{aligned}\end{equation*}
	In other words, $(\xi,x)\notin \WF_{(s,t)}(M)$ if $w$ belongs to $H^{(s,t)}_{\mathrm{loc}}(M)$ microlocally around $(\xi,x)$.
	The set $\WF_{s,t}(w)$ is a closed cone in $T^*M\backslash 0$.
	These wave-front sets recover the classical wave-front set by the relation
		\begin{equation}\label{eqn:lqskdfkqsdlmfkqsdmlf}\begin{aligned}
			\WF(w)=\overline{\bigcup_{s,t\in \R}\WF_{(s,t)}(w)}.
		\end{aligned}\end{equation}
		Furthermore, as expected, $\WF_{(s,t)}(w)=\emptyset$ if and only if $w\in H^{(s,t)}_{\mathrm{loc}}(M)$.
		We can microlocalize \eqref{eqn:qksdfjmkqsdjfmqsd} as follows:
		\begin{equation*}\begin{aligned}
			\WF_{(s,t)}(P(w))\subseteq \WF_{(s+\Re(k),t+\Re(l))}(w),\quad \forall P\in \Psi^{(k,l)}(M), k,l\in \C,s,t\in \R.		
		\end{aligned}\end{equation*}
		We now come to the main definition of interest to us
		\begin{dfn}\label{intro:dfn:micro-Local_maxilam_hypo}
			One says that $P\in \Psi^{(k,l)}_{\mathrm{loc}}(M)$ is microlocally maximally hypoelliptic of order $(\Re(k),\Re(l))$ on a cone $\Gamma\subseteq T^*M\backslash 0$ if
				\begin{equation}\label{intro:eqn:dfn:micro-Local_maxilam_hypo}\begin{aligned}
								\WF_{(s,t)}(P(w))\cap \Gamma=\WF_{(s+\Re(k),t+\Re(l))}(w)\cap \Gamma,\quad \forall w\in  C^{-\infty}(M,\omegahalf), s,t\in \R.		
				\end{aligned}\end{equation}
			One says that $P$ is maximally hypoelliptic of order $(\Re(k),\Re(l))$ if it is microlocally maximally hypoelliptic on $T^*M\backslash 0$.
			By \eqref{eqn:lqskdfkqsdlmfkqsdmlf}, microlocal maximal hypoellipticity on $\Gamma$ implies microlocal hypoellipticity on $\Gamma$.
		\end{dfn}
	
		The main strength in our calculus lies in the fact that we define a principal symbol which detects microlocal maximal hypoellipticity.
	So, let us now define our principal symbol. 
	Since the principal symbol is a local object, we can suppose for simplicity that in Hörmander's condition, we only need iterated commutators of depth $\leq N\in \N$.
	This is always satisfied locally or if $M$ is compact.	
	To simplify the exposition, we will also restrict our attention in the introduction to operators $P=Q(X_1,\cdots,X_n)\in \Psi^{(k,l)}(M)$ where $Q$ is noncommutative formal polynomial of weighted degree $\leq l\in \N$ 
	whose coefficients are classical pseudo-differential operators of order $\leq k\in \C$ (when $k=0$, then these are exactly the operators which appear in the microlocal Helffer-Nourrigat conjecture).

	Let $\mathfrak{g}$ be the free nilpotent Lie algebra of depth $N$ with $n$ generators $\tilde{X}_1,\cdots,\tilde{X}_{n}\in \mathfrak{g}$, and $G$ be the simply connected nilpotent Lie group integrating $\mathfrak{g}$.
	Let $\pi$ be an irreducible unitary representation of $G$ on a Hilbert space $L^2(\pi)$, $C^\infty(\pi)\subseteq L^2(\pi)$ the subspace of smooth vectors, $\odif{\pi}:\mathfrak{g}\to \mathrm{End}(C^\infty(\pi))$ the differential of $\pi$.
	We define the principal symbol by 
		\begin{equation}\label{into:eqn_symbol}\begin{aligned}
			\sigma^{(k,l)}(P,\pi,\xi,x)=Q_{\xi,x,\mathrm{max}}(\odif{\pi}(\tilde{X}_{1})\cdots\odif{\pi}(\tilde{X}_{n})):C^\infty(\pi)\to C^\infty(\pi),
		\end{aligned}\end{equation}
	 where $Q_{\xi,x,\mathrm{max}}$ is the homogeneous polynomial constructed from $Q$ where we removed all lower order terms and replaced all coefficients $L$ with $\sigma^k_{\mathrm{classical}}(L,\xi,x)\in \C$, the classical principal symbol of $L$ at $(\xi,x)$.
	 There is the obvious issue: Is this well-defined? 
	A priori, \eqref{into:eqn_symbol} depends on the presentation of $P$ as a polynomial in $X_1,\cdots,X_n$. Different presentations of $P$ might give different values for $\sigma^{(k,l)}(P,\pi,\xi,x)$.
 	In \Cref{sec:Tangent_groupoid_representation}, we define for each $(\xi,x)\in T^*M\backslash 0$, a set  $\HatHLnon_{(\xi,x)}$ of irreducible unitary representations of $G$ which we call the Helffer-Nourrigat cone at $(\xi,x)$.
	 We can now state the main theorem of this article.
	 \begin{thmx}\label{intro:main_thm}
		Let \begin{itemize}
			\item     $M$ be a smooth manifold.
			\item $X_1,\cdots,X_n$ vector fields satisfying Hörmander's condition.
			\item  $k,l\in \C$, and $P\in \Psi^{(k,l)}(M)$.
		\end{itemize}
		Then, for any $(\xi,x)\in T^*M\backslash 0$ and $\pi\in \HatHLnon_{(\xi,x)}$, $\sigma^{(k,l)}(P,\pi,\xi,x):C^\infty(\pi)\to C^\infty(\pi)$ is well-defined.
		Let 
			\begin{equation*}\begin{aligned}
					\Gamma=\left\{(\xi,x)\in T^*M\backslash 0:\forall \pi\in \HatHLnon_{(\xi,x)},\ \sigma^{(k,l)}(P,\pi,\xi,x):C^\infty(\pi)\to C^\infty(\pi) \text{ is injective}\right\} 
			\end{aligned}\end{equation*}
		The following hold:
				\begin{enumerate}
				\item\label{intro:main_thm:inj} The cone $\Gamma$ is open, and $\Gamma\subseteq \{(\xi,x)\in T^*M\backslash 0:(-\xi,\xi;x,x)\in \WF(P)\}$, where $\WF(P)$ is the wave-front set of the Schwartz kernel of $P$.
				\item\label{intro:main_thm:micro_local_max_hypo}
				The operator $P$ is microlocally maximally hypoelliptic on $\Gamma$ of order $(\Re(k),\Re(l))$, i.e., 
					 \begin{equation}\label{eqn:intro_sdqjofjlkqsdjf}\begin{aligned}
						 				\WF_{(s,t)}(P(w))\cap \Gamma=	 \WF_{(s+\Re(k),t+\Re(l))}(w)\cap \Gamma.
					 \end{aligned}\end{equation}
				 Furthermore, if $\Gamma'$ is any cone such that \eqref{eqn:intro_sdqjofjlkqsdjf} holds (with $\Gamma$ replaced by $\Gamma'$) for some $s,t\in \R$, then $\Gamma'\subseteq \Gamma$.
				 So, $\Gamma$ is the largest cone on which $P$ is microlocally maximally hypoelliptic.
				\item\label{intro:main_thm:parametrix}
				For any closed cone $\Gamma'\subseteq \Gamma$. 
				There exists a continuous $P':C^{-\infty}(M,\omegahalf)\to C^{-\infty}(M,\omegahalf)$ such that for any $t,s\in \R$ and $w\in C^{-\infty}(M,\omegahalf)$,
					\begin{equation}\label{eqn:intro_parametrix_WF}\begin{aligned}
							\WF_{(s,t)}(P'(w))\subseteq \WF_{(s-\Re(k),t-\Re(l))}(w),\quad \WF(P'(P(w))-w)\cap \Gamma'=\emptyset.
					\end{aligned}\end{equation}
					In particular, by \eqref{eqn:lqskdfkqsdlmfkqsdmlf}, $\WF(P'(w))\subseteq \WF(w)$. So, $P'(C^\infty(M,\omegahalf))\subseteq C^\infty(M,\omegahalf)$.
				\item\label{intro:main_thm:inequa}For any $P''\in \Psi^{(k,l)}(M)$ which is compactly supported, 
				there exists $\chi\in \Psi^0_{\mathrm{classical}}(M)$ a classical pseudo-differential operator of order $0$, $C>0$ and $R\in C^{\infty}_c(M\times M,\omegahalf)$ a smoothing operator
				such that $\chi$ is elliptic on $\Gamma$ and
						 \begin{equation}\label{eqn:intro_inequ_microlocal_main_thm}\begin{aligned}
								 \norm{\chi(P''( w))}_{L^2(M)}\leq C\left(\norm{P(w)}_{L^2(M)}+\norm{R(w)}_{L^2(M)}\right),\quad \forall w\in C^\infty_c(M,\omegahalf).
						 \end{aligned}\end{equation}
					 If $\Gamma=T^*M\backslash 0$, then one can suppose that $\chi=\mathrm{Id}$.

					\item\label{intro:main_thm:Fredholm} If $M$ is compact and $\Gamma=T^*M\backslash 0$, then for any $s,t\in \R$, 
						$P:H^{(s+\Re(k),t+\Re(l))}(M)\to H^{(s,t)}(M)$ is left-invertible modulo compact operators.
				\item 
				If $(\xi,x)\in \Gamma$, $\pi\in \HatHLnon_{(\xi,x)}$, then the operators 
					 \begin{equation*}\begin{aligned}
						& \sigma^{(k,l)}(P,\pi,x):C^{-\infty}(\pi)\to C^{-\infty}(\pi)&&
						\\ &\sigma^{(k,l)}(P,\pi,x):C^\infty(\pi)\to C^\infty(\pi)&&
					 \end{aligned}\end{equation*}
				admit a continuous left-inverse. Here, $C^{-\infty}(\pi)$ is the topological dual of $C^\infty(\pi)$, which contains $L^2(\pi)$, and is called the space of distribution vectors of $\pi$. The map $ \sigma^{(k,l)}(P,\pi,x):C^{-\infty}(\pi)\to C^{-\infty}(\pi)$ is defined using duality.
			\end{enumerate}
	\end{thmx}
	\begin{rems}Here are some remarks on \Cref{intro:main_thm}.
	\begin{enumerate}
		\item If $P$ is a classical pseudo-differential operator of order $k\in \C$, seen as element of $\Psi^{(k,0)}(M)$, then $\sigma^{(k,0)}(P,\pi,\xi,x)=\sigma^{k}_{\mathrm{classical}}(P,\xi,x)$ is the classical principal symbol. 
	In this case, Theorem \ref{intro:main_thm} is the microlocalized main regularity theorem for elliptic operators. 
	If $k=0$ and $\Gamma=T^*M\backslash 0$, then \Cref{intro:main_thm} is the main theorem of \cite{MohsenMaxHypo}. 
	It is actually quite stronger than the main theorem of \cite{MohsenMaxHypo} because we couldn't obtain \eqref{eqn:intro_parametrix_WF} in \cite{MohsenMaxHypo}.

	\item I don't know if the parametrix $P'$ belongs to $\Psi^{(-k,-l)}(M)$. This question will be investigated in a forthcoming article.
	\item All the results of this article work equally well if one adds vector bundle coefficients everywhere, i.e., one considers two complex vector bundles $E,F$, and operators $P:C^\infty(M,E\otimes \omegahalf)\to C^\infty(M,F\otimes \omegahalf)$.
		In this case, the principal symbol $\sigma^{(k,l)}(P,\pi,\xi,x)$ is an operator from $C^\infty(\pi)\otimes E_x$ to $C^\infty(\pi)\otimes F_x$. The theory extends in a rather straightforward way. 
		To simplify the notation above all, we decided to omit vector bundles.
	\item \MyCref{intro:main_thm}[parametrix] is optimal in the sense that we can't take $\Gamma'=\Gamma$ unless $\Gamma=T^*M\backslash 0$ in which case $\Gamma$ is closed. 
	\end{enumerate}
	\end{rems}
	\begin{cor}
		The microlocal version of the Helffer-Nourrigat conjecture \cite[Conjcture 2.3]{HelfferNourrigatBook} is true.
	\end{cor}
	\begin{proof}
		Their conjecture is that \eqref{eqn:intro_inequ_microlocal_main_thm} holds when $k=0$ and $P$ and $P''$ are noncommutative formal polynomial of weighted degree $\leq l\in \N$ 
	whose coefficients are classical pseudo-differential operators of order $0$.
	\end{proof}

	We prove many properties of the Sobolev spaces, principal symbol, wave-front set in this calculus
	see \MyCref{chap:chapter_bi-garded_pseudo_diff} for the statements in full generality.
	Many of these properties generalise well-known properties of the classical pseudo-differential calculus.
	But some aren't as expected. In fact, they have no counterpart in classical calculus, for example \MyCref{thm:principal_symbol_characterizations}, see the preceding paragraph.
	We also refer the reader to \MyCref{sec:example_pseudo-diff} where we compute the principal symbol in an explicit example, and where \MyCref{thm:principal_symbol_characterizations} has subtle and interesting consequences of diagonalizability of certain Schrodinger-type operators.

	There has been some work in recent years by different authors on multi-parameter pseudo-differential calculus, see for instance Street's books \cite{StreetBook2,StreetMulti} and the references there.
	Street's work considers \eqref{eqn:intro_inequ_microlocal_main_thm} on different Sobolev-type spaces.	
	In particular, the following is proved:
		\begin{theorem}[{\cite[Proposition 8.2.12  and Corollary 6.9.2]{StreetBook2}}]
			Inequality \eqref{eqn:intro_inequ_microlocal_main_thm} holds with the $L^2$-norm if and only if it holds with the $L^2$-norm replaced by the $L^p$-norm for all $p\in ]1,+\infty[$.
		\end{theorem}
	Street also defines a multi-parameter pseudo-differential calculus. 
	The calculus is defined in terms of estimates on the kernels which is a more analytic approach.
	We expect that our calculus is contained in Street's calculus and that this is not too hard to prove.
	We are quite certain that the inclusion is strict.

	The main distinction between our approach and previous ones in the literature lies in the systematic use of groupoids, together with substantial input from techniques from $C^*$-algebras.
	We now outline the main steps of the proof of \Cref{intro:main_thm}, and at least explain why our proof is long.
	The argument is divided into three main parts.

\begin{enumerate}
	\item \textbf{Construction of the groupoid.}
	The first step consists in constructing the space $\mathbb{G}$ in \Cref{chap:bi-graded_tangent_groupoid} and \Cref{chap:bi-graded_tangent_groupoid_proof}.
	This space is a generalization of the groupoid introduced in \cite{MohsenTangentCone} (which is itself a generalization of Connes' tangent groupoid \cite{ConnesBook}, and the groupoid of van Erp and Yuncken \cite{ErikBobTangentGrp} and Choi and Ponge \cite{ChoiPonge}).
	The groupoid $\mathbb{G}$ provides a resolution of the singularities arising from variations in the rank of the distribution generated by the vector fields $X_1,\ldots,X_n$ and their iterated brackets.
	The novelty in this construction compared to that of \cite{MohsenTangentCone} is that this groupoid also takes into account the variation of the rank at the microlocal level.
	This part of the proof is purely geometric and does not involve pseudo-differential operators.

	This step presents several technical difficulties.
	One major problem is that the space $\mathbb{G}$ is not a smooth manifold but rather a differential space in the sense of Sikorski \cite{SikorskiDifferentialSpacesBook} and Aronszajn \cite{AronszajnAMS}.
	So, it makes sense to talk about smooth functions on $\mathbb{G}$, but smooth functions generally don't have a Taylor series expansion (this is the main reason why I don't know if $P'$ belongs to $\Psi^{(-k,-l)}(M)$ or not).
	For this reason, we introduce a special class of smooth functions, which we call \emph{invariant smooth functions}.
	These functions behave like smooth functions on a smooth manifold.

	All of these difficulties together with the intricate definition of $\mathbb{G}$ lead to the length of \Cref{chap:bi-graded_tangent_groupoid} and \Cref{chap:bi-graded_tangent_groupoid_proof}. 
	We refer the reader for example to \MyCref{sec:tangent_groupoid_exs} where the groupoid is constructed in the quite simple setting where $M=\R^2$, and $X_1=\frac{\partial }{\partial x}$ and $X_2=x^n\frac{\partial}{\partial y}$ for some $n\in \N$.
	In this setting one clearly sees how subtle the definition of the groupoid is.

	The idea behind the definition of the groupoid $\mathbb{G}$ is quite simple. One embeds $M\times \Rpt$ inside a large enough Grassmannian manifold which takes into account all the different 
	singularities which one needs to resolve, and then one takes the closure of the image of this embedding. This idea is classical in algebraic geometry, and is usually called Nash's resolution.

	\item \textbf{Construction of the pseudo-differential calculus.}
	The second step consists in constructing the calculus $\Psi^{(k,l)}(M)$.
	This construction is inspired by the work of Debord and Skandalis \cite{DebordSkandalis1,DebordSkandalis2} and van Erp and Yuncken \cite{ErikBobCalculus}, where an analogous approach was developed for the classical pseudo-differential calculus using Connes' tangent groupoid \cite{ConnesBook}.
	This part of the argument is rather straightforward generalization of arguments from Debord, Skandalis, van Erp, and Yuncken.
	The proofs get slightly longer due to the multi-parameter structure of the calculus but in my view, there are no new conceptual difficulties.

	This part is mostly in \Cref{sec:bi_graded_pseudo_diff_proof}, and \Cref{sec:principal_symbol_proof}.
	\item \textbf{Use of $C^*$-algebras of Type 1 techniques.}
	The third and final step is to prove \Cref{intro:main_thm}.
	This part relies heavily on arguments from $C^*$-algebra theory.
	The main difficulty arises from fact that $\Gamma$ in \Cref{intro:main_thm} is defined using injectivity on $C^\infty(\pi)$.
	If one instead assumes injectivity on the space $C^{-\infty}(\pi)$, this step becomes simpler.
	Passing from injectivity on $C^{\infty}(\pi)$ to injectivity on $C^{-\infty}(\pi)$ is in some sense a hypoellipticity statement.
	One needs to show that if a vector in $C^{-\infty}(\pi)$ is in the kernel of the symbol map, then the vector is actually smooth.
	In fact, if one uses Kirillov's orbit method to construct the representations of $G$, then the symbol maps are pseudo-differential operators on some Euclidean space $\R^k$, and the above statement on the kernel consisting of smooth vectors is a hypoellipticity statement, see \MyCref{sec:example_pseudo-diff} where this is explicitly carried out in an example.

	This part of the proof relies on the structure theory of $C^*$-algebras Type~\RNum{1}, and more precisely on an abstract maximal hypoellipticity theorem that we proved in \cite{MohsenAbstractMaxHypo}.
	This abstract result shows that in $C^*$-algebras of Type I, one can bootstrap injectivity from smooth vectors to distributional vectors.
\end{enumerate}
	Our previous work \cite{MohsenMaxHypo} follows a similar approach. The main differences are that the groupoid here is more complicated which leads to some technical difficulties in the first two Parts of the proof. 
	The third part of the proof is new and differs from the arguments we used in \cite{MohsenMaxHypo}.

	For the reader looking to see only the general structure of the argument, in \cite[Section 2]{MohsenAbstractMaxHypo}, we added a proof of the theorem of Helffer and Nourrigat \cite{HelfferRockland} (their resolution of the Rockland conjecture) which is a special case of 
	\Cref{intro:main_thm}.
	\cite[Section 2]{MohsenAbstractMaxHypo} is considerably shorter and simpler.
	The main simplification is that Step 1 becomes trivial, as the space $\mathbb{G}$ is replaced by a simply connected nilpotent Lie group.
	Consequently, the second and third steps are substantially easier to follow, even though the general structure of the argument remains essentially the same.

	In forthcoming articles, we will present further applications of our calculus. 

		\paragraph{Organization of the article:}

	We tried to make this article as self-contained as we possibly could.
			We don't assume any familiarity with our previous work \cite{MohsenMaxHypo,MohsenTangentCone}. 
		In fact, in contrast to our previous work \cite{MohsenMaxHypo}, in this article, we don't rely on or assume any familiarity with the results in \cite{RotschildStein,HelfferRockland,ChrGelGloPol,HebischRockland}.
		Technically, we don't even assume Hörmander's sum of squares theorem \cite{Hormander:SoS} (we could but that wouldn't simplify our proof).
		We can circumvent all the results in \cite{RotschildStein,HelfferRockland,ChrGelGloPol,HebischRockland} thanks to the abstract maximal hypoellipticity theorem mentioned above.

		We use techniques from groupoids and $C^*$-algebras. These techniques were developed by the non-commutative geometry community for the last 30 years.
		We have included a chapter in the beginning of our article which introduces the necessary tools from groupoids and $C^*$-algebras.
		This chapter is mainly based on Renault's book \cite{RenaultBook} and Connes's book \cite{ConnesBook}.

		Finally, to make the article easier to read, we divided the proofs in \MyCref{chap:bi-graded_tangent_groupoid} and \MyCref{chap:chapter_bi-garded_pseudo_diff} into separate \MyCref{chap:bi-graded_tangent_groupoid_proof} and \MyCref{chap:chapter_bi-garded_pseudo_diff_proof}.
		This way, the reader can more easily find all the results in this article.
		Also, we tried to make \MyCref{chap:chapter_bi-garded_pseudo_diff_proof} as independent as possible from \MyCref{chap:bi-graded_tangent_groupoid_proof}.
		This way, the reader interested only in the pseudo-differential calculus can read \MyCref{chap:chapter_bi-garded_pseudo_diff} and \MyCref{chap:chapter_bi-garded_pseudo_diff_proof} without going through \MyCref{chap:bi-graded_tangent_groupoid_proof}.

		We also added an appendix on Kirillov's orbit method for nilpotent Lie groups, see \MyCref{chap:Orbit_method}, and on the abstract maximal hypoellipticity result from \cite{MohsenAbstractMaxHypo}, see \MyCref{thm:asbtract}.
	\paragraph*{Acknowledgement:}
		This research was supported in part by the ANR project OpART (ANR-23-CE40-0016).

		I would like to thank the following people for their help and support during this work: 
		Iakovos Androulidakis, Claire Debord, Nigel Higson, Bernard Helffer, Francis Nier, Pierre Pansu, Georges Skandalis, Stéphane Vassout, Robert Yuncken.
			\chapter{Quasi-Lie groupoids}\label{chap:quasi_Lie_groupoid}
\paragraph{Conventions and notations:}
	\begin{enumerate}
		\item   We use $\N=\{1,2,\cdots\}$, $\Z_+=\N\cup \{0\}$, $\R_+=\{x\in \R:x\geq 0\}$, $\Rpt=\{x\in \R:x>0\}$.
		\item   If $n,m\in \Z$ with $n\leq m$, then $\bb{n,m}:=[n,m]\cap \Z$. 
			\item   Functions, vector spaces and vector bundles  are complex unless otherwise stated.
			\item   Hilbert spaces are always separable. Inner products $\langle \xi,\eta\rangle$ are linear in $\eta$ and antilinear in $\xi$.
			
	\end{enumerate}

\section{Differential structures}\label{sec:sub-Cartesian}
		In this chapter, we recall the notion of a differential structure space due to Sikorski \cite{SikorskiDifferentialSpacesBook} and Aronszajn \cite{AronszajnAMS}, see \cite{SniatyckiDifferentialGeometryBook,AronszajnSikorskiSpaces} for a modern treatment of differential spaces.
		\begin{dfn}\label{dfn:sub-Cartesian}
			A differential structure on a topological space $M$ is a non-empty subset $C^\infty(M)$ of the space of continuous complex-valued functions $C(M)$ such that:
			\begin{enumerate}
				\item\label{dfn:sub-Cartesian:1}   For every $x,y\in M$ with $x\neq y$, there exists $f\in C^\infty(M)$ such that $f(x)\neq f(y)$.
				\item\label{dfn:sub-Cartesian:2}   If $f_1,\cdots,f_n\in C^\infty(M)$ and $F\in C^\infty(\C^n)$, then $F(f_1,\dots,f_n)\in C^\infty(M)$.
				\item\label{dfn:sub-Cartesian:3}   If $f:M\to \C$ is a function such that for every $x\in M$, there exists a neighborhood $U$ of $x$ and $g\in C^\infty(M)$ such that $f_{|U}=g_{|U}$, then $f\in C^\infty(M)$.
			\end{enumerate}
		\end{dfn}
		A differential space is a Hausdorff locally compact second countable\footnote{Some authors only include a subset of these conditions in the definition of a differential space.} equipped with a differential structure.
		A differential space is metrizable by Urysohn's metrization theorem.
		We denote by $C^\infty_c(M)$ the subspace of $C^\infty(M)$ of functions with compact support.
		By Stone–Weierstrass theorem, $C^\infty_c(M)$ is a dense $*$-subalgebra of $C_0(M)$.
		\begin{prop}[{\cite[Theorem 2.2.4]{SniatyckiDifferentialGeometryBook}}]\label{prop:parition_of_unity}
			Let $(U_i)_{i\in I}$ be an open cover of a differential space $M$. There exists a countable family $f_n:M\to [0,1]$ of smooth functions such that:
			\begin{itemize}
				\item   For all $n\in \N$, there exists an $i\in I$, such that $\supp(f_n)\subseteq U_i$.
				\item   The sum $\sum_{n\in \N}f_n$ is locally finite and equal to $1$.
			\end{itemize}
		\end{prop}
	\paragraph{Generating set:}
		Let $M$ be a topological space, $\cF\subseteq C(M)$ a subset which satisfies \MyCref{dfn:sub-Cartesian}[1].
		We define the differential structure generated by $\cF$ to be the set of functions $f:M\to \C$ such that for every $x\in M$, there exists a neighborhood $U$ of $x$, $f_1,\cdots,f_n\in \cF$, $F\in C^\infty(\C^n)$ such that $f_{|U}=F(f_1,\cdots,f_n)_{|U}$.
	\paragraph{Subspaces:}
		If $A\subseteq M$  is a locally closed subset of a differential space $M$, then $C^\infty(A)$ denotes the space of functions $f:A\to \C$ such that every $x\in A$,
		there exists an open neighborhood $U$ of $x$ and $g\in C^\infty(M)$ such that $f_{|U\cap A}=g_{|U\cap A}$.
		By \Cref{prop:parition_of_unity}, if $A$ is closed, then $f\in C^\infty(A)$ if and only if there exists $g\in C^\infty(M)$ such that $f=g_{|A}$.
		The condition $A$ is locally closed is needed to ensure that $A$ is locally compact.
		The reader can check that $A$ is a differential space.
		From now on, any locally closed subset of a differential space is equipped with this differential structure.

		\begin{dfn}\label{dfn:smooth_map_sub_Cart}
			\begin{enumerate}\label{enum:smooth_map_sub_Cart}
				\item 	A map $f:M\to M'$ between differential spaces is smooth if for any $g\in C^\infty(M')$, $g\circ f\in C^\infty(M)$.
				\item 	A diffeomorphism is a smooth homeomorphism whose inverse is smooth.
				\item\label{dfn:smooth_map_sub_Cart:smooth} 	A (closed, open) smooth embedding is a (closed, open) topological embedding which is a diffeomorphism onto its image.
				\item	A smooth map $f:M\to M'$ is called a local diffeomorphism at $x\in M$, if there exists an open neighborhood $U\subseteq M$ of $x$ such that $f_{|U}:U\to M'$ is an open smooth embedding.
			\end{enumerate}
		\end{dfn}
		By \MyCref{dfn:sub-Cartesian}[1], a smooth map is automatically continuous.
		\begin{rem}
			Since a subspace $A$ of a locally compact space is locally compact if and only if $A$ is locally closed, it follows that the image of a topological embedding between differential spaces is locally closed.
			Hence, the image inherits a differential structure. This justifies \Crefitem{dfn:smooth_map_sub_Cart}{smooth}.
		\end{rem}
		A differential space is a smooth manifold if for some $n\in \N$, $M$ is locally diffeomorphic to $\R^n$.
		\begin{theorem}\label{prop:Whitney_full}
			Let $M$ be a differential space. The following conditions are equivalent: 
			\begin{enumerate}
				\item\label{qsljdqdslkqjskldj1} There exists a closed smooth embedding $M\hookrightarrow \R^{n}$ for some $n\in \N$.
				\item\label{qsljdqdslkqjskldj2} There exists a finite number $f_1,\cdots,f_n\in C^\infty(M)$ such that for any $f\in C^\infty(M)$, there exists $F\in C^\infty(\C^n)$ such that $f=F(f_1,\cdots,f_n)$. 
				\item\label{prop:Whitney_full:3} There exists $n\in \N$ such that for any $x\in M$, one can find an open neighborhood $U\subseteq M$ of $x$, $f_1,\cdots,f_n\in C^\infty(U)$ such that for all $f\in C^\infty(M)$, there exists $F\in C^\infty(\C^n)$ such that $f_{|U}=F(f_1,\cdots,f_n)$.
			\end{enumerate}
			We say that $M$ is a finite-dimensional differential space if it satisfies the above conditions.
			We say that $M$ is locally finite-dimensional\footnote{Sometimes called sub-Cartesian space in the literature.} if every $x\in M$ has an open neighborhood which is finite dimensional.
		\end{theorem}
		\begin{proof}
			The implications $\ref{qsljdqdslkqjskldj1}\Rightarrow \ref{qsljdqdslkqjskldj2}\Rightarrow \ref{prop:Whitney_full:3}$ are trivial.
			Our proof of $\ref{prop:Whitney_full:3}\Rightarrow \ref{qsljdqdslkqjskldj1}$ differs from Whitney's proof.
			Classical proof of Whitney's theorem is based on Sard's theorem.
			We base ours on Lebesgue's covering dimension.
			In fact, our proof is a variant of the proof of the topological Whitney embedding theorem, see \cite[Theorem 50.5]{MunkresBook}.
			Fix $N\in \N$ for which the third condition is satisfied.
			By replacing the $f_1,\cdots,f_N$ with their real and imaginary parts, we can suppose (after doubling $N$) that $f_1,\cdots,f_N$ are real-valued.
			We fix a proper smooth map $\rho:M\to \R$.
			For example, if $(f_n)_{n\in \N}$ is as in \Cref{prop:parition_of_unity} associated to the trivial open cover, then $\rho=\sum_{n}n f_n$ is proper smooth.
			\begin{lem}\label{prop:Whitney}
				For any $x\in M$, there exists an open neighborhood $U$ of $x$, and a smooth embedding $U\to \R^{N}$.
			\end{lem}
			\begin{proof}
				Let $U$ and $f_1,\cdots,f_N\in C^\infty(U)$ as in Condition $3$.
				We define
				\begin{equation*}\begin{aligned}
						\phi:U\to \R^N,\quad \phi(y)=(f_1(y),\cdots,f_N(y)).
					\end{aligned}\end{equation*}
				The map $\phi$ is a smooth embedding when restricted to any compact neighborhood $V\subseteq U$ of $x$.
			\end{proof}
			\Cref{prop:Whitney} together with \cite[Theorem 3.1.14 and Corollary 3.1.20]{EngelkingDimension} imply that $M$ has Lebesgue covering dimension $\leq N$.
			By \Cref{prop:Whitney} and  \cite[Proposition 3.2.2.b]{EngelkingDimension}, we can find a locally finite countable open cover $(U_n)_{n\in \N}$ of $M$ and smooth embeddings $f_n:U_n\to \R^{N}$ such that any $N+1$ intersection
			$U_{n_1}\cap \cdots \cap U_{n_{N+1}}$ is empty.
			To be more precise, one starts with any open cover such that its elements admit smooth embeddings into $\R^{N}$, then by paracompactness one replaces the cover by a locally finite refinement,
			then by Lindelöf one replaces the cover by a countable subcover, then one applies \cite[Proposition 3.2.2.b]{EngelkingDimension}.
			By modifying $f_n$, we can suppose that $\mathrm{Im}(f_n)\subseteq \R^{N}\backslash 0$.
			Let
			\begin{equation*}\begin{aligned}
					\tilde{f}_n:U_n\to \R^{N+1}\backslash 0,\quad \tilde{f}_n(x)=\left(f_n(x),\norm{f_n(x)}^2\right).
				\end{aligned}\end{equation*}
			Notice that $\tilde{f}_n$ has the property that
			\begin{equation}\label{eqn:injective_proof_whitney}\begin{aligned}
					\forall x,y\in U_n,\lambda\in \R_+, \quad   \tilde{f}_n(x)=\lambda \tilde{f}_n(y)\implies x=y
				\end{aligned}\end{equation}

			Since $M$ is normal, we can find an open cover $(V_n)_{n\in \N}$ such that $\overline{V_n}\subseteq U_n$.
			For each $n\in \N$, by \Cref{prop:parition_of_unity}, there exists $\xi_n:M\to [0,1]$ a smooth function such that $\xi_n=1$ on $V_n$ and $\supp(\xi_n)\subseteq U_n$.
			Let
			$g_n=\xi_n\tilde{f}_n:M\to \R^{N+1}$
			extended outside $U_n$ by $0$ which is smooth by \MyCref{dfn:sub-Cartesian}[3].
			By \eqref{eqn:injective_proof_whitney}, 
			\begin{equation}\label{eqn:injective_proof_whitney_2}\begin{aligned}
				\forall x,y\in M,n\in \N, \quad   g_n(x)=g_n(y) \text{ and }  g_n(x)\neq 0\implies x=y
			\end{aligned}\end{equation}
			Since $(U_n)_{n\in \N}$ is a locally finite cover, it follows that the sums
			\begin{equation*}\begin{aligned}
					h_l:=\sum_{n=1}^{+\infty} n^l g_n:M\to \R^{N+1},\quad l\in \Z_+
				\end{aligned}\end{equation*}
			are locally finite, and thus smooth.
			We can now define
			\begin{equation*}\begin{aligned}
					\phi:M\to \R^{2N^2+2N+1},\quad \phi(x)=(h_0(x),\cdots,h_{2N-1}(x),\rho(x)).
				\end{aligned}\end{equation*}
			We will show that $\phi$ is a closed smooth embedding.
			For every $x\in M$, let $S_x=\{n\in \N:x\in U_n\}$.
			By our assumptions on $(U_n)_{n\in \N}$, $|S_x|\leq N$.
			Let
				\begin{equation*}\begin{aligned}
						W_x=\bigcap_{n\in S_x}U_n \cap \bigcap_{n\notin S_x}\supp(g_n)^c.
					\end{aligned}\end{equation*}
			The set $W_x$ is open because $(U_n)_{n\in \N}$ is a locally finite cover.
			To see this, let $U\subseteq M$ be an open set such that $ |\{n\in \N:U\cap U_n\neq \emptyset\}|<\infty$.
			Then, for any $n\in \N$ such that $U\cap U_n=\emptyset$, we have $U\cap \supp(g_n)^c=U$.
			Hence,
			\begin{equation*}\begin{aligned}
					U\cap W_x=      \bigcap_{n\in S_x}(U\cap U_n)  \cap\bigcap_{\{n\in S_x^c:U\cap U_n\neq \emptyset\}}  (U\cap \supp(g_n)^c)
				\end{aligned}\end{equation*}
			This is a finite intersection.
			Hence, we have proved that $W_x$ is locally open.
			So $W_x$ is open.
			By the construction of $W_x$, $x\in W_x$ and $g_{n|W_x}=0$ for all $n\notin S_x$.
			
			We can now show that $\phi$ is injective. Let $x,y\in M$ with $\phi(x)=\phi(y)$.
			We have
			\begin{equation}\label{eqn:WS}\begin{aligned}
					n\notin (S_x\cup S_y)\implies g_{n|W_x\cup W_y}=0.
				\end{aligned}\end{equation}
			If $s_1,\cdots,s_{k}$ are the elements of $S_x\cup S_y$, then
			using the identity
			\begin{equation*}\begin{aligned}
					\begin{pmatrix}
						h_0(z) \\h_1(z)\\\vdots \\ h_{k-1}(z)
					\end{pmatrix}=\begin{pmatrix}
						              1             & \cdots & 1               \\
						              s_1           & \cdots & s_{k}         \\
						              \vdots        & \ddots & \vdots          \\
						              s_{1}^{k-1} & \cdots & s_{k}^{k-1}
					              \end{pmatrix}\begin{pmatrix}
						                           g_{s_1}(z) \\g_{s_2}(z)\\
						                           \vdots
						                           \\
						                           g_{s_{k}}(z)
					                           \end{pmatrix}\quad \forall z\in W_x\cup W_y,
				\end{aligned}\end{equation*}
			we can find an invertible linear map $L$ such that
			\begin{equation}\label{eqn:WS2}\begin{aligned}
					L\circ \phi=(g_{s_1},\cdots,g_{s_{k}},h_{k},\cdots,h_{2N-1},\rho) \quad \text{on } W_x\cup W_y.
				\end{aligned}\end{equation}
			Since $x,y\in W_x\cup W_y$, by \eqref{eqn:WS2}, we deduce that
			$g_n(x)=g_n(y)$ for all $n\in S_x\cup S_y$.
			Since $x$ belongs to $V_n\subseteq U_n$ for some $n\in \N$, it follows that $g_n(x)=g_n(y)$ and $g_n(x)\neq 0$.
			By \eqref{eqn:injective_proof_whitney_2}, $x=y$.

			Since $\rho$ is proper, it follows that $\phi$ is a closed topological embedding.
			It remains to show that $\phi^{-1}$ is smooth.
			On $V_n\cap W_x$, $g_n=\tilde{f_n}$, so by \eqref{eqn:WS2}, the inverse of $\phi$ on  $V_n\cap W_x$ is given by $f_n^{-1}\circ \pi\circ L$ where $\pi$ is the projection onto the coordinate with $f_n$.
			This is smooth because $f_n$ is a smooth embedding.
		\end{proof}
	\textbf{From now on, all differential spaces are supposed locally finite dimensional.}
	\paragraph{Product:}
		Let $M_1$ and $M_2$ be differential spaces. 
		We equip $M_1\times M_2$ with the differential structure generated by $\{(x,y)\mapsto f(x):f\in C^\infty(M_1)\}\cup \{(x,y)\mapsto g(y):g\in C^\infty(M_2)\}$.
		If $M_1$ and $M_2$ are smooth manifolds, then $C^\infty(M_1\times M_2)$ agrees with the usual $C^\infty(M_1\times M_2)$.
	\paragraph{Disjoint union:}
		If $M_1$ and $M_2$ are differential spaces, then $M_1\sqcup M_2$ is a differential space with $C^\infty(M_1\sqcup M_2)=\{f\in C(M_1\sqcup M_2):f_{|M_1}\in C^\infty(M_1),f_{|M_2}\in C^\infty(M_2)\}$.				
	\paragraph{Tangent vectors, tangent bundle, and vector fields:}Let $M$ be a differential space.
		A tangent vector of $M$ at $x\in M$ is a $\C$-linear map $X:C^\infty(M)\to \C$ such that
		\begin{equation*}\begin{aligned}
				X(fg)=X(f)g(x)+f(x)X(g),\quad X(\overline{f})=\overline{X(f)},\quad \forall f,g\in C^\infty(M).
			\end{aligned}\end{equation*}
		The space of all tangent vectors is denoted by $T_xM$.
		By a standard trick using Taylor's formula, see \cite[Proof of Lemma 7.1.4]{HormanderBook1}, if $f_1,\cdots,f_n\in C^\infty(M)$ and $F\in C^\infty(\C^n)$, then
		\begin{equation}\label{eqn:LeibnizGeneralised}\begin{aligned}
				X(F(f_1,\cdots,f_n))=\sum_{i=1}^n\frac{\partial F}{\partial x_i}(f_1(x),\cdots,f_n(x))X(f_i).
			\end{aligned}\end{equation}
		By the same trick,
		\begin{equation}\label{eqn:TangentVectorIsLocal}\begin{aligned}
			\forall f\in C^\infty(M),\quad  x\notin \supp(f) \implies  X(f)(x)=0.
			\end{aligned}\end{equation}
		By \eqref{eqn:LeibnizGeneralised}, \eqref{eqn:TangentVectorIsLocal} and since $M$ is locally finite dimensional, $T_xM$ is  a finite-dimensional real vector space.
		We define the tangent bundle by
		\begin{equation*}\begin{aligned}
				TM:=\{(X,x):x\in M,X\in T_xM\}.
			\end{aligned}\end{equation*}
		Let $\pi:TM\to M$ be the natural projection.
		If $f\in C^\infty(M)$, then we define
		\begin{equation*}\begin{aligned}
				\odif{f}:TM\to \C,\quad \odif{f}(X,x)=X(f).
			\end{aligned}\end{equation*}
		We equip $TM$ with the weakest topology such that $\pi:TM\to M$ and $\odif{f}$ are continuous for any $f\in C^\infty(M)$.
		Let $C^\infty(TM)$ be the differential structure generated by $\{f\circ \pi,\odif{f}:f\in C^\infty(M)\}$.
		One can show that $TM$ is a locally finitely generated differential space, see \cite[Section 3.3]{SniatyckiDifferentialGeometryBook}.
		A smooth map $f:M_1\to M_2$ induces for every $x\in M_1$, a linear map
		\begin{equation*}\begin{aligned}
				\odif{f}_x:T_xM_1\to T_{f(x)}M_2,\quad \odif{f}_x(X)(g)=X(g\circ f),\quad \forall g\in C^\infty(M_2),
			\end{aligned}\end{equation*}
		and a smooth map $\odif{f}:TM_1\to TM_2$.	
		
		If $M_1,M_2$ are differential spaces, then $T(M_1\times M_2)$ is naturally diffeomorphic to $TM_1\times TM_2$.

		A vector field on $M$ is a linear map   $X:C^\infty(M)\to C^\infty(M)$ such that
		\begin{equation*}\begin{aligned}
				X(fg)=X(f)g+fX(g),\quad X(\overline{f})=\overline{X(f)},\quad \forall f,g\in C^\infty(M).
			\end{aligned}\end{equation*}
		We denote the space of vector fields by $\cX(M)$.
		Since the commutator of derivations is a derivation, $\cX(M)$ is a Lie algebra.
		By composing $X$ by the evaluation map at every point $x\in M$, one obtains a section $M\to TM$.
		This defines a bijection between vector fields and smooth sections $M\to TM$, see \cite[Proposition 3.3.5]{SniatyckiDifferentialGeometryBook}.
		\begin{ex}\label{ex:sub_cart}
			Let $M=\{(x,y)\in \R^2:xy=0\}$.
			One can check that $\dim(T_{(x,y)}M)=1$ if $(x,y)\neq 0$ and $\dim(T_{0}M)=2$.
			A vector field on $M$ has the form $X=f(x,y)x\frac{\partial}{\partial x}+g(x,y)y\frac{\partial}{\partial y}$ where $f,g\in C^\infty(M)$.
			It follows that if $X$ is a vector field, then $X$ at $(0,0)$ vanishes $0$.
			In particular, $\{X(x):X\in \cX(M)\}$ is usually a strict subspace of $T_xM$.
		\end{ex}
		\begin{dfn}[Vector bundles]\label{dfn:vector_bundles_sub-Cartesian}
			A real vector bundle on a differential space $M$ is a differential space $E$, a smooth map $\pi:E\to M$, a real finite dimensional vector space structure on $E_x:=\pi^{-1}(x)$ for each $x\in M$ such that for every $x\in M$, there exists an open neighborhood $U$ of $x$, $k\in \Z_+$, a diffeomorphism $\phi: \R^k\times U\to \pi^{-1}(U)$ such that:
			\begin{itemize}
				\item   For all $(v,x)\in \R^k\times U$, $\pi(\phi(v,x))=x$.
				\item   For all $x\in U$, $\phi(\cdot,x):\R^k\to E_x$ is a linear isomorphism.
			\end{itemize}
			Complex vector bundles are defined analogously. 
			We denote by $C^\infty(M,E)$ the space of smooth sections. We denote by $C^\infty_c(M,E)$ the subspace of compactly supported smooth section.
		\end{dfn}
		Classical constructions like direct sum, tensor product, pullback extend immediately to vector bundles over differential spaces.
		The space $TM$ is generally not a real vector bundle because $\dim(T_xM)$ generally depends on $x$.
		One can show that if $\dim(T_xM)$ is locally constant, then $TM$ is a real vector bundle, see \cite[Proposition 3.3.15]{SniatyckiDifferentialGeometryBook}.
		\begin{prop}\label{rem:finitely_generated_section}
			If $E\to M$ is a real vector bundle over a finite dimensional differential space $M$ such that $x\in M\mapsto \dim(E_x)\in \Z_+$ is bounded, then $C^\infty(M,E)$ is a finitely generated $C^\infty(M,\R)$-module.
		\end{prop}
		\begin{proof}
			The total space of $E^*$ satisfies \MyCref{prop:Whitney_full}[3].
			So, by \Cref{prop:Whitney_full}, there exists a closed smooth embedding $\phi:E^*\to \R^n$ for $n$ big enough.
			The differential of $\phi$ restricted to $M$ gives a smooth bundle embedding of $E^*$ into the trivial vector bundle $M\times \R^n$.
			By taking the adjoint, one obtains generators of $C^\infty(M,E)$.
		\end{proof}
			\paragraph{Submersions:}
			We say that a smooth map $f:M_1\to M_2$ between differential spaces is a submersion at $x\in M_1$ if there exists a smooth map $g:M_1\to \R^n$ for some $n\in \Z_+$
			such that 
				\begin{equation*}\begin{aligned}
					M_1\to M_2\times \R^n,\quad y\mapsto (f(y),g(y))		
				\end{aligned}\end{equation*}
			is a local diffeomorphism at $x\in M_1$, i.e., $f$ is locally a projection whose fiber is an open subset of a Euclidean space.
			The map $f$ is called a submersion if it is a submersion at every $x\in M_1$.
		
			If $f$ is a submersion, $x\in M_1$, then $\odif{f}_{x}:T_xM_1\to T_{f(x)}M_2$ is surjective, and for every $y\in M_2$, $f^{-1}(y)$ is a smooth manifold whose tangent space at $x\in f^{-1}(y)$ is equal to $\ker(\odif{f}_x)$.
		Moreover, $\ker(\odif{f})$, equipped with the subspace topology and differential structure from $TM_1$, is a real vector bundle over $M_1$.
		The composition of two submersions is again a submersion.
		\begin{rem}\label{rem:flow_vector_field}
			If $X\in \cX(M)$ is a vector field, then its flow at time $t$ starting from $x\in M$ is denoted by $e^{tX}\cdot x$. 
			In general the flow might not be well-defined for any $t\neq 0$.
			But if $f:M\to M'$ is a submersion, $X\in C^\infty(M,\ker(\odif{f}))$, then the flow exists in a neighborhood of $(x,0)$ in $M\times \R$.
		\end{rem}

	\begin{prop}\label{prop:composition_submersion}
		Let $f:M_1\to M_2$ and $g:M_2\to M_3$ be smooth maps between differential spaces.
		If $f$ and $g\circ f$ are submersions, and $f$ is surjective, then $g$ is a submersion.
	\end{prop}
	\begin{proof}
		Let $x_0\in M_2$. By hypothesis, there exists $n,m\in \Z_+$, $h:M_2\times \R^{n}\to \R^{m}$ a smooth map such that $h(x_0,0)=0$ and the map
			\begin{equation*}\begin{aligned}
					\phi:M_2\times \R^{n}\to M_3\times \R^{m},\quad \phi(x,y)=(g(x),h(x,y))
			\end{aligned}\end{equation*}
		is a local diffeomorphism at $(x_0,0)$. So, its differential at $(x_0,0)$ is a linear isomorphism which restricts to a linear isomorphism $L:\ker(\odif{g}_{x_0})\times \R^n\to \R^m$.
		Let $\pi:M_2\times \R^n\to M_2$ be the natural projection, $\psi$ the locally defined map
			\begin{equation*}\begin{aligned}
				\psi:M_3\times \ker(\odif{g}_{x_0})\times \R^n\to M_3\times \R^m,\quad \psi(x,a,b)=\phi\Big( \pi\left(\phi^{-1}(x,L(a,0))\right),b\Big)		
			\end{aligned}\end{equation*}
		The map $\psi$ is of the form $\psi(x,a,b)=(x,\kappa(x,a,b))$ for some smooth map $\kappa$. Furthermore, $\odif{\kappa}_{(g(x_0),0,0)}(0,a,b)=L(a,b)$.
		The parameterized inverse mapping theorem (where $M_3$ is seen as a parameter space) implies that $\psi$ is a local diffeomorphism at $(g(x_0),0,0)$.
		So, the map 
			\begin{equation*}\begin{aligned}
				M_3\times \ker(\odif{g}_{x_0})\to M_2,\quad (x,a)\mapsto\pi(\phi^{-1}(x,L(a,0)))
			\end{aligned}\end{equation*}
		is a local diffeomorphism at $(g(x_0),0)$ whose inverse is a map of the form $x\mapsto (g(x),\varphi(x))$ for some smooth map $\varphi:M_2\to \ker(\odif{g}_{x_0})$. The result follows.
	\end{proof}
	By taking $M_3=\{\mathrm{pt}\}$, \Cref{prop:composition_submersion} implies that, if $M$ is a differential space such that for some $n\in \Z_+$, $M\times \R^n$ is a smooth manifold, then $M$ is a smooth manifold.
	The following theorem is a variant of Bourbaki's quotient theorem \cite[Section 5.9.5]{BourbakiDiffVariety}.
	\paragraph{Quotient space:}
		\begin{theorem}\label{thm:quotient_differential_space}
			Let $M$ be a differential space, $R\subseteq M\times M$ an equivalence relation, $\pi:M\to M/R$ the quotient map, $C^\infty(M/R)=\{f:M/R\to \C:f\circ \pi\in C^\infty(M/R)\}$. 
			Suppose the following holds:
			\begin{enumerate}
				\item The set $R$ is a closed subset of $M\times M$.
				\item The natural projection $(x,y)\in R\mapsto x\in M$ is a submersion.
			\end{enumerate}
			Then, $M/R$ equipped with the quotient topology is Hausdorff locally compact and second countable, and $C^\infty(M/R)$ is a locally finite dimensional differential structure on $M/R$, and $\pi$ is a submersion.
		\end{theorem}
		\begin{proof}
			Let $U\subseteq M$ be an open subset, $x_0\in \pi^{-1}(\pi(U))$. 
			So, $(x_0,y_0)\in R$ for some $y_0\in U$.
			Since $(x,y)\in R\mapsto x\in M$ is a submersion and a submersion admits local sections, there exists a locally defined smooth map $s:M\to M$ such that $s(x_0)=y_0$ and $(x,s(x))\in R$ for all $x\in \dom(s)$.
			Hence, $s^{-1}(U)$ is a neighborhood of $x_0$ which is contained in $\pi^{-1}(\pi(U))$.
			So, $\pi$ is open.
			Hence, $M/R$ is Hausdorff because $R$ is closed. It also follows that $M/R$ is second countable and locally compact. 

			We now show that $C^\infty(M/R)$ is a differential structure on $M/R$.
			If $U\subseteq M$ is an open set, $f:\pi(U)\to \C$ is a function, then since $f\circ \pi\circ s=f\circ \pi$ for any section $s$ as above, we deduce that 
				\begin{equation}\label{eqn:qsjdofjmoqksd}\begin{aligned}
					f\circ\pi \in C^\infty(U) \iff f\circ\pi \in C^\infty(\pi^{-1}(\pi(U))).	
				\end{aligned}\end{equation}
			
			It is easy to see that \MyCref{dfn:sub-Cartesian}[2] and \MyCref{dfn:sub-Cartesian}[3] hold.
			So we need to prove \MyCref{dfn:sub-Cartesian}[1]. 
			Let $x_0\in M$ be fixed.
			Since $(x,y)\in R\mapsto x\in M$ is a submersion at $(x_0,x_0)$, there exists $g:R\to \R^n$ a smooth map such that $g(x_0,x_0)=0$ and the map
				\begin{equation*}\begin{aligned}
					\psi:R\to M\times \R^n,\quad (x,y)\mapsto (x,g(x,y))
				\end{aligned}\end{equation*}
				is a local diffeomorphism at $(x_0,x_0)$.
				Since $R$ is closed, we can extend $g$ to a smooth map on $M\times M$.
				Let $g'(y)=g(x_0,y)$.
				We claim that the map 
				\begin{equation*}\begin{aligned}
					\phi:R\to M\times \R^n,\quad (x,y)\mapsto (x,g'(y))
				\end{aligned}\end{equation*}
				is a local diffeomorphism at $(x_0,x_0)$. To see this, notice that the map $\phi\circ \psi^{-1}:M\times \R^n\to M\times \R^n$ is of the form $(x,t)\mapsto (x,k(x,t))$ for some smooth map $k$ which satisfies $k(x_0,t)=t$.
				The parameterized inverse mapping theorem (with $M$ as a parameter space) implies that $\phi\circ \psi^{-1}$ is a local diffeomorphism at $(x_0,0)$. 
				Hence, $\phi$ is a local diffeomorphism at $(x_0,x_0)$. 
			
			Let $V$ be an open neighborhood of $x_0$ such that $\phi$ is an open embedding on $(V\times V)\cap R$, $U$ the open neighborhood of $x_0$ defined by
				\begin{equation*}\begin{aligned}
					U=\{y\in V:(y,0)\in \phi((V\times V)\cap R)\}.
				\end{aligned}\end{equation*}
			The set $\pi(V\cap g^{\prime -1}(0))$ is an open neighborhood of $\pi(x_0)$ because 
				$\pi(V\cap g^{\prime -1}(0))=\pi(U)$.
			The map $\pi$ is injective on $V\cap g^{\prime -1}(0)$ because $\phi$ is injective on $(V\times V)\cap R$.
			Let $h:\pi(V\cap g^{\prime -1}(0))\to M$ be the inverse of $\pi$ on $V\cap g^{\prime -1}(0)$.
			On $U$, $h\circ \pi$ is simply defined by $\phi(y,h\circ \pi(y))=(y,0)$.
			So, the map $h\circ \pi$ is smooth on $U$.
			By \eqref{eqn:qsjdofjmoqksd}, $h\circ \pi\in C^{\infty}(\pi^{-1}(\pi(U)))$.
			Now, suppose that $x_1\in M$ such that $(x_0,x_1)\notin R$.
			We replace $U$ by a neighborhood of $x_0$ small enough so that $x_1\notin \pi^{-1}(\pi(U))$ which is possible because $\pi$ is open and $R$ is closed.
			Since $h$ is injective by replacing $U$ with a smaller relatively compact neighborhood, we can suppose that $h$ is a homeomorphism onto its image.
			So, $h(U)$ is locally closed, i.e., $h(U)=A\cap B$, where $A\subseteq M$ is open and $B\subseteq M$ is closed. 
			Let $g\in C^\infty_c(M)$ such that $g(x_0)=1$ and $\supp(g)\subseteq A$.
			So, $\supp(g)\cap h(U)=\supp(g)\cap B$ is compact.
			Since $h$ is a homeomorphism onto its image, $g\circ h$ is compactly supported.
			We extend $g\circ h$ to $M/R$ by $0$ outside $\pi(U)$.
			By \MyCref{dfn:sub-Cartesian}[3], $g\circ h\circ \pi \in C^\infty(M)$.
			So, $g\circ h\in C^\infty(M/R)$.
			Notice that $g\circ h(\pi(x_0))=1$ and $g\circ h(\pi(x_1))=0$.
			Thus, $C^\infty(M/R)$ satisfies \MyCref{dfn:sub-Cartesian}[1].

			We now show that $\pi$ is a submersion at $x_0$. 
			We claim that the map 
				\begin{equation*}\begin{aligned}
					M\to M/R\times \R^n,\quad x\mapsto(\pi(x),g'(x))
				\end{aligned}\end{equation*}
			 is a local diffeomorphism at $x_0$. In fact a local inverse is the map 
				 \begin{equation*}\begin{aligned}
					 M/R\times \R^n\to M,\quad (z,t)\mapsto y,\quad \text{where }(h(z),y)=\phi^{-1}(h(z),t)
				 \end{aligned}\end{equation*}	 
			 Finally, since $M$ is locally finite dimensional, $M/R$ is locally finite dimensional.
		\end{proof}
	\paragraph{Fréchet topology:}
		Let $M$ be a differential space, $K\subseteq M$ a compact subset. 
		Since $K$ is compact and $M$ is locally finite dimensional, $K$ is finite dimensional.
		Let $f_1,\cdots,f_n\in C^\infty(K)$ such that for every $f\in C^\infty(K)$, $f=F(f_1,\cdots,f_n)$ for some $F\in C^\infty(\C^n)$.
		We can define the semi-norm
		\begin{equation}\label{eqn:norm_phi_K_sub_cart}\begin{aligned}
				\norm{f}_{C^m,K}=\inf\{\norm{F}_{C^m(\C^n)}:f=F(f_1,\cdots,f_n)\},\quad m\in \Z_.
			\end{aligned}\end{equation}
		It is clear that $\norm{\cdot}_{C^m,K}$ up to equivalence doesn't depend on the choice of $f_1,\cdots,f_n$.	
		We equip $C^\infty(M)$ with the topology generated by $\norm{\cdot}_{C^m,K}$ as $m\in \Z_+$, $K\subseteq M$ vary.
		By a standard diagonalization argument, $C^\infty(M)$ is complete, hence a Fréchet space.
		We equip $C^\infty_c(M)$ with the strongest topology that makes the maps $C^\infty(K)\to C^\infty_c(M)$ continuous for any compact subset $K$ of $M$.
		This makes $C^\infty_c(M)$ an LF space (limit of Fréchet spaces).
		If $E\to M$ is a vector bundle, then we define similarly topologies on $C^\infty(M,E)$ and $C^\infty_c(M,E)$.

	\paragraph{Densities:}
		If $E\to M$ is a real vector bundle on a differential space $M$ and $k\in \C$, then $\Omega^k(E)$ denotes the vector bundle of $k$-densities of $E$.
		Its fiber over $x\in M$ is the $1$-dimensional complex vector space of all functions $\phi:\Lambda^{\mathrm{rank}(E_x)}E_x\backslash 0\to \C$ such that $\phi(\lambda v)=|\lambda|^{k}\phi(v)$ for all $v \in \Lambda^{\mathrm{rank}(E_x)}E\backslash 0$ and $\lambda\in \R^\times$.
		The following holds: 
		\begin{itemize}
			\item    If 		 $\phi\in C^\infty(M,\Omega^k(E))$, then $|\phi|\in C^\infty(M,\Omega^{\Re(k)}(E)) $.
			\item  	If $\phi\in C^\infty(M,\Omega^k(E)),\psi\in C^\infty(M,\Omega^l(E))$, then $\phi\psi\in C^\infty(M,\Omega^{k+l}(E))$.
			\item If $k\in \R$, then we say that $\phi\in C^\infty(M,\Omega^k(E))$ is positive if for all $x\in M$, $\im(\phi_x)\subseteq \R_+$.
		\end{itemize}
		The proof of the following classical proposition is left to the reader.
		\begin{prop}\label{prop:density_natural_isomorphisms}
			\begin{enumerate}
				\item   If $E\to M$ is a real vector bundle, $k,l\in \C$, then there are natural isomorphisms
					\begin{equation}\label{eqn:densities_natural_iso}\begin{aligned}
							\Omega^{k+l}(E)\simeq \Omega^k(E)\otimes \Omega^{l}(E),\quad \Omega^k(E^*)\simeq\Omega^k(E)^*\simeq \Omega^{-k}(E),\quad \overline{\Omega^k(E)}\simeq \Omega^{\overline{k}}(E)
						\end{aligned}\end{equation}
				\item\label{prop:density_natural_isomorphisms:2}   If $0\to E\to F\to G\to 0$ is a short exact sequence of real vector bundles, $k\in \C$, then there is a natural isomorphism
					\begin{equation}\label{eqn:densities_short_exact_iso}\begin{aligned}
							\Omega^{k}(F)\simeq \Omega^{k}(E)\otimes \Omega^{k}(G).
						\end{aligned}\end{equation}
			\end{enumerate}
		\end{prop}
		\begin{cor}\label{prop:cor_triangle_isom}
			If $X,Y,Z$ are differential spaces, $f:X\to Y$, $g:Y\to Z$ are submersions
		then for any $k\in \C$, 
			\begin{equation*}\begin{aligned}
				\Omega^k(\ker(\odif{(g\circ f)}))\simeq \Omega^k(\ker(\odif{f}))\otimes f^*(\Omega^k(\ker(\odif{g}))).
			\end{aligned}\end{equation*}
		\end{cor}
		\begin{proof}
			The map $g\circ f$ is a submersion because $f,g$ are submersions.
			Since $f,g,g\circ f$ are submersions, $\ker(\odif{f}),\ker(\odif{g}),\ker(\odif{(g\circ f)})$ are real vector bundles. 
			The corollary follows \Crefitem{prop:density_natural_isomorphisms}{2} applied to the short exact sequence 
				\begin{equation*}
					0\to \ker(\odif{f})\to \ker(\odif{(g\circ f)})\to f^*(\ker(\odif{g}))\to 0\qedhere
				\end{equation*}
		\end{proof}

	\paragraph{Integration along fibers:}
		If $M$ is a differential space, $f\in C^\infty(M\times \R^n)$ such that the natural projection $\supp(f)\to M$ is proper, then $x\mapsto \int f(x,y)dy\in C^\infty(M)$.
		This follows easily from \Cref{prop:Whitney_full} (or directly from \Cref{dfn:sub-Cartesian}).

		Let $f:M_1\to M_2$ be a submersion between differential spaces, $E$ a vector bundle on $M_2$.
		Thanks to \Cref{prop:parition_of_unity}, we can define an integration along the fibers map
		\begin{equation}\label{eqn:integration_along_fibers}\begin{aligned}
				f_*:C^\infty_c(M_1,\Omega^1(\ker(\odif{f}))\otimes f^*(E) )\to C^\infty_c(M_2,E)
			\end{aligned}\end{equation}
		The map $f_*$ is continuous with respect to the $LF$-topologies defined above.
	
		Let $X\in C^\infty(M_1,\ker(\odif{f}))$. For any $k\in \C$, we define the Lie derivative 
				\begin{equation}\label{eqn:Lie_derivative_dfn}\begin{aligned}
					\cL_X:	C^\infty(M_1,\Omega^k(\ker(\odif{f}))\otimes f^*(E))\to C^\infty(M_1,\Omega^k(\ker(\odif{f}))\otimes f^*(E))
				\end{aligned}\end{equation}
		as follows: Locally, $M_1=M_2\times \R^n$ and $f$ is the projection map.
		So, if $\vol^k:\Lambda^n\R^n\backslash 0\to \C$ denotes the canonical element which maps $e_1\wedge\cdots \wedge e_n$ to $1$, then
				\begin{equation*}\begin{aligned}
				\cL_X(g \vol^k\otimes s)=X(g)\vol^k\otimes  s+ k g \mathrm{div}(X)\vol^k\otimes  s ,\quad g\in C^\infty(M_2\times \R^n),s\in C^\infty(M_2,E),		
			\end{aligned}\end{equation*}
			where $\mathrm{div}(X)\in C^\infty(M_2\times \R^n)$ is the divergence of $X$, i.e., if $X=\sum X_i\frac{\partial}{\partial t_i}$ with $X_i\in C^\infty(M_2\times \R^n)$, then $\mathrm{div}(X)=\sum_{i=1}^n\frac{\partial X_i}{\partial t_i}$.

	\paragraph{Pullback of smooth maps:}
		We say that a diagram of smooth maps between differential spaces
		\begin{equation*}\begin{tikzcd}
			M_1\arrow[d,"\pi_1"]\arrow[r, "f"]    &   M_2 \arrow[d,"\pi_2"]\\
			\tilde{M_1}\arrow[r, "\tilde{f}"]& \tilde{M}_2
		\end{tikzcd}\end{equation*} 
		is a pullback diagram if it commutes and the map $(\pi_1,f):M_1\to \tilde{M}_1\times_{\tilde{f},\pi_2} M_2$ is a diffeomorphism.

		\begin{prop}\label{lem:pullback_local_diff}
			In a pullback diagram, if $\tilde{f}$ is a submersion, then $f$ is also a submersion and the map $\odif{\pi_{1}}:\ker(\odif{f})\to \pi_{1}^*(\ker(\odif{\tilde{f}}))$ is an isomorphism.
		   Furthermore, if $E\to \tilde{M}_2$ is a vector bundle,  then
				\begin{equation}\label{eqn:integration_pullback}\begin{aligned}
					f_{*}(h\circ \pi_{1})=\tilde{f}_*(h)\circ \pi_2,\quad \forall h\in C^\infty_c(\tilde{M_1},\tilde{f}^*(E)\otimes \Omega^1(\ker(\odif{\tilde{f}})),
				\end{aligned}\end{equation}
			where $f_*(h\circ \pi_1)$ is well-defined because $f_{|\supp(h\circ \pi_1)}:\supp(h\circ \pi_1)\to M_2$ is proper.
		\end{prop}
		\begin{proof}
			Since $\tilde{f}$ is a submersion, we can suppose that $\tilde{M}_1=\tilde{M}_2\times \R^n$ and $\tilde{f}$ is the natural projection.
			The proposition is now straightforward to prove.
		\end{proof}
\section{Quasi-Lie groupoid}\label{sec:quasi_Lie_groupoid}
	A groupoid is a small category where all morphisms are invertible. We denote by \begin{itemize}
		\item     $G$ the set of morphisms.
		\item $G^{(0)}$ the set of objects.
		\item $s,r:G\to G^{(0)}$ the maps sending a morphism to its domain and range respectively.
		\item $\iota:G\to G$ the inverse map. Generally, the inverse of $\gamma\in G$ is denoted  by $\gamma^{-1}$.
		\item $u:G^{(0)}\to G$ the identity map. Generally, we will identify $G^{(0)}$ with $u(G^{(0)})$.
		\item $G^{(2)}:=\{(\gamma',\gamma)\in G\times G:s(\gamma')=r(\gamma)\}$ the set of composable pairs.
		\item $m:G^{(2)}\to G$ the composition map. Generally, we use $\gamma'\gamma$ instead of $m(\gamma',\gamma)$.
	\end{itemize}
	We use $G\rightrightarrows G^{(0)}$ to denote a groupoid. If $A,B\subseteq G^{(0)}$ are subsets, then 
		\begin{equation*}\begin{aligned}
			G_A:=s^{-1}(A),\quad G^{B}:=r^{-1}(B),\quad G_A^B:=G_A\cap G^B.
		\end{aligned}\end{equation*}
	If $A$ or $B$ are singletons $\{x\}$ and $\{y\}$, then we use $G_x$, $G^x$ and $G_x^y$ for simplicity.
	We also define 
			\begin{equation*}\begin{aligned}
				AB:=\{\gamma\gamma':(\gamma,\gamma')\in G^{(2)}\cap (A\times B)\},\quad A^{-1}:=\{\gamma^{-1}:\gamma\in A\},\quad \forall A,B\subseteq G.
			\end{aligned}\end{equation*}

	\begin{exs}\label{ex:exs_groupoids}
		\begin{enumerate}
			\item  The trivial groupoid over a set $X$ is the unique groupoid where $G=G^{(0)}=G^{(2)}=X$, and all the structure maps are the identity.
			\item    The pair groupoid over a set $X$ is the unique category where $G^{(0)}=X$ and for every $x,y$, there exists a unique morphism from $x$ to $y$.
				So, $G=X\times X$ and the structure maps are
				\begin{equation*}\begin{aligned}
						s(y,x)=x, \quad r(y,x)=y, \quad \iota(y,x)=(x,y), \quad u(x)=(x,x), \quad m((z,y),(y,x))=(z,x).
					\end{aligned}\end{equation*}
				More generally, if 
				$R\subseteq X\times X$ is an equivalence relation, then $R\rightrightarrows X$ is a groupoid where the structure maps are the restriction of the structure maps from $X\times X\rightrightarrows X$.
			\item   If $G^{(0)}=\{\bullet\}$, then a groupoid is just a group. More generally, if $s=r$, then a groupoid is a bundle of groups over $G^{(0)}$.
			\item\label{ex:exs_groupoids:quotient}   Let $\Gamma$ be a group, $S$ a set of subgroups of $\Gamma$ which is closed under conjugation.
				We define the groupoid
				\begin{equation*}\begin{aligned}
						G=\{\gamma H\eta:\gamma,\eta\in \Gamma,H\in S\}        ,\quad G^{(0)}=S
					\end{aligned}\end{equation*}
				Notice that since $S$ is closed under conjugation,
				\begin{equation*}\begin{aligned}
						\gamma H\eta=\gamma\eta (\eta^{-1 }H\eta)=(\gamma H\gamma^{-1})\gamma\eta
					\end{aligned}\end{equation*}
				is a left coset of $\eta^{-1}H\eta\in S$ and a right coset of $\gamma H\gamma^{-1}\in S$.
				The structure maps are 
				\begin{equation*}\begin{aligned}
						s(A)=A^{-1}A,\quad r(A)=AA^{-1},\quad u(H)=H,\quad \iota(A)=A^{-1},\quad m(B,A)=BA
					\end{aligned}\end{equation*}
				Notice that if $A=\gamma H\eta$, then $s(A)=\eta^{ -1}H \eta$ is an element of $S$.
				Similar computation with $r$ and $m$.
				If $S=\{H\}$ where $H$ is a normal subgroup, then $G=\Gamma/H$ is the quotient of $\Gamma$ by $H$.
				In general, we call the groupoid $G$, \textit{the quotient of $\Gamma$ by $S$.}
				This groupoid plays a fundamental role in this article. 
				It has also been used in a different context by Bradd, Higson, Yuncken, see \cite{bradd2025liegroupoidssatakecompactification,bradd2025liegroupoidssatakecompactification2}.
		\end{enumerate}
	\end{exs}

	\begin{dfn}\label{dfn:quasi-lie}
		A \textit{quasi-Lie groupoid} is a groupoid $G\rightrightarrows G^{(0)}$ together with a Hausdorff locally compact second countable topology and locally finite dimensional differential structure on $G$ and on $G^{(0)}$ such that
		\begin{enumerate}
			\item   The maps $u$, $r$, $s$, $\iota$ and $m$ are smooth where $G^{(2)}$ is equipped with the differential structure inherited from $G\times G$.
			\item   The source map $s:G\to G^{(0)}$ is a submersion.
		\end{enumerate}
	\end{dfn}
	\begin{rem}\label{rem:quasi-lie_is_Lie}
		If $G^{(0)}$ is a smooth manifold, then $G$ is also a smooth manifold because $s$ is a submersion.
		In this case, $G\rightrightarrows G^{(0)}$ is called a Lie groupoid.
	\end{rem}
	\begin{prop}\label{lem:multi_is_sub}
		\begin{enumerate}
			\item The identity map $u:G^{(0)}\to G$ is a closed smooth embedding.
			\item The map $\iota:G\to G$ is a diffeomorphism.
			\item The maps $m:G^{(2)}\to G$ and $r:G\to G^{(0)}$ are submersions.
		\end{enumerate}
	\end{prop}
	\begin{proof}
		\begin{enumerate}
			\item The map $u$ is a smooth embedding because $s\circ u=\mathrm{Id}$. 
				It is closed because if $u(x_n)\to \gamma\in G$, then $u(x_n)u(x_n)^{-1}\to \gamma\gamma^{-1}$. 
				But since $G$ is Hausdorff, and $u(x_n)u(x_n)^{-1}=u(x_n)$, it follows that $\gamma=\gamma\gamma^{-1}\in u(G^{(0)})$.
			\item The map $\iota $ is its own inverse.
			\item     
			Since $r=s\circ \iota$, it follows that $r$ is a submersion.
			Let $(\gamma_0,\gamma_1)\in G^{(2)}$. Since $s$ is a submersion at $\gamma_1$, 
			we can find $g:G\to \R^n$ such that $\phi: x\in G\mapsto  (s(x),g(x))\in G^{(0)}\times \R^n$ is a local diffeomorphism at $\gamma_1$.
			The map 
			\begin{equation*}\begin{aligned}
					G^{(2)}\to  G\times \R^n,\quad (\gamma,\gamma')\mapsto(\gamma\gamma',g(\gamma'))
			\end{aligned}\end{equation*}
			is a local diffeomorphism at $(\gamma_0,\gamma_1)$.
			A local inverse is the map
			\begin{equation*}\begin{aligned}
					G\times \R^n\to G^{(2)},\quad (\gamma,l)\mapsto (\gamma \phi^{-1}(s(\gamma),l)^{-1},\phi^{-1}(s(\gamma),l)).
				\end{aligned}\end{equation*}
			Hence, $m$ is a submersion.\qedhere
		\end{enumerate}
	\end{proof}

	\paragraph{Morphisms of quasi-Lie groupoids:}
	Let $G\rightrightarrows G^{(0)}$ and $H\rightrightarrows H^{(0)}$ be quasi-Lie groupoids.
	A morphism of quasi-Lie groupoids is a smooth map $f:H\to G$ which is also a functor of categories, i.e., if $(\gamma,\gamma')\in H^{(2)}$, then $(f(\gamma),f(\gamma'))\in G^{(2)}$ and $f(\gamma \gamma')=f(\gamma)f(\gamma')$.

	\paragraph{Quasi-Lie subgroupoid}	If $H\subseteq G$, and the inclusion is a morphism of quasi-Lie groupoids, then we say that $H\rightrightarrows H^{(0)}$ is a quasi-Lie subgroupoid of $G\rightrightarrows G^{(0)}$. 
\paragraph{Saturated subset:}
		A locally closed subset $A$ of $G^{(0)}$ is called saturated if $G_A=G^A$.		
		In this case, $G_A\rightrightarrows A$ is a quasi-Lie subgroupoid of $G$.
\paragraph{Pullback:}If $G\rightrightarrows G^{(0)}$ is a quasi-Lie groupoid, $f:X\to G^{(0)}$ is a smooth map, then we can define the pullback groupoid $f^*(G)\rightrightarrows X$ by 
	\begin{equation*}\begin{aligned}
		f^*(G)=\left\{(y,\gamma,x)\in X\times G\times X:\gamma\in G_{f(x)}^{f(y)}\right\},	\quad	 r(y,\gamma,x)=y,\quad s(y,\gamma,x)=x\\
		u(x)=(x,f(x),x),\quad (y,\gamma,x)^{-1}=(x,\gamma^{-1},y),\quad (z,\gamma,y)\cdot (y,\eta,x)=(z,\gamma\eta,x).
	\end{aligned}\end{equation*}
\paragraph{Action of a groupoid on a space:}
	Let $X$ be a differential space.
	A smooth action of $G\rightrightarrows G^{(0)}$ on $X$ consists of smooth maps $\sigma:X\to G^{(0)}$ and $(\gamma,x)\in G\times_{s,\sigma}X\mapsto \gamma\cdot x\in X$ which satisfy 
		\begin{equation*}\begin{aligned}
			\sigma(\gamma\cdot x)=r(\gamma),\quad \eta\cdot (\gamma\cdot x)=(\eta\gamma)\cdot x,\quad \sigma(x)\cdot x=x,\quad \forall x\in X,\gamma\in G_{\sigma(x)},\eta\in G_{r(\gamma)}.		
		\end{aligned}\end{equation*}
	Given such a smooth action, one can define the quasi-Lie groupoid $G\ltimes X\rightrightarrows X$ as
		\begin{equation}\label{eqn:crossed_product_groupoid}\begin{aligned}
			G\ltimes X=\{(y,\gamma,x)\in X\times G\times X:\sigma(x)=s(\gamma),y=\gamma \cdot x \}\subseteq \sigma^*(G)
		\end{aligned}\end{equation}
	equipped with the subspace topology, sub-differential structure, and subgroupoid structure.
\paragraph{Proper subset:} 
	A subset $A\subseteq G$ is called proper if the maps $s_{|A},r_{|A}:A\to G^{(0)}$ are proper maps.
	A proper subset is automatically closed.
	If $E\to G$ is a vector bundle, then $C^\infty_{\mathrm{proper}}(G,E)$ denotes the space of properly supported sections.
	
\section{Lie algebroid}\label{sec:algebraic_constructions}

\begin{dfn}\label{dfn:Lie_algebroid}
	Let $M$ be a differential space. A Lie algebroid over $M$ is a real vector bundle $A\to M$ together with a family of $\R$-linear maps $\rho_x:A_x\to T_xM$ for $x\in M$ 
	and a Lie algebra structure on $C^\infty(M,A)$ such that 
		\begin{equation}\label{eqn:Lie_algebra}\begin{aligned}
			[X,fY]=f[X,Y]+\rho(X)(f)Y,		
		\end{aligned}\end{equation}
		where $\rho:A\to TM$ is the map induced from the family $(\rho_x)_{x\in M}$ which is called the anchor map.
\end{dfn}
By \eqref{eqn:Lie_algebra}, $\rho$ is a smooth map.
By looking at $[[X,Y],fZ]$ and using \eqref{eqn:Lie_algebra}, one deduces that 
	\begin{equation}\begin{aligned}
		\rho([X,Y])=[\rho(X),\rho(Y)],\quad \forall X,Y\in C^\infty(M,A)
	\end{aligned}\end{equation}

\begin{ex}\label{ex:Lie_algebroids}
	If $f:M\to M'$ is a submersion, then $\ker(\odif{f})$ is a Lie algebroid, where $\rho:\ker(\odif{f})\to TM$ is the inclusion.
\end{ex}

\paragraph{Lie algebroid of a quasi-Lie groupoid:} Let $G\rightrightarrows G^{(0)}$ be a quasi-Lie groupoid.
The restriction of the vector bundle $\ker(\odif{s})$ to $G^{(0)}$ is called the \textit{Lie algebroid} of $G$ and is denoted by $A(G)$.
\footnote{One can define $A(G)$ to be $\ker(\odif{r})_{|G^{(0)}}$. 
		The choice is a matter of convention. 
		It should be consistent with the choice of the Lie algebra of a Lie group as that of right or left invariant vector fields.
		Our convention identifies the Lie algebra of a Lie group with right invariant vector fields
		which implies a sign in the BCH formula, see \eqref{eqn:BCH}.}
The anchor map $\rho:A(G)\to TG^{(0)}$ is the restriction of $\odif{r}$ to $A(G^0)$. We will now define a Lie algebra structure on $C^\infty(G^{(0)},A(G))$.
For any $\gamma\in G$, let
\begin{equation*}\begin{aligned}
		R_\gamma:G_{r(\gamma)}\to G_{s(\gamma)},\ R_\gamma(\gamma')=\gamma'\gamma
		,\quad 
		L_\gamma:G^{s(\gamma)}\to G^{r(\gamma)},\ L_\gamma(\gamma')=\gamma\gamma'
	\end{aligned}\end{equation*}
The maps $R_\gamma$ and $L_\gamma$ diffeomorphisms because their inverses are $R_{\gamma^{-1}}$ and $L_{\gamma^{-1}}$.
The differential of $R_\gamma$ and $L_\gamma$ induces isomorphisms
\begin{equation}\label{eqn:left_right_iso_groupoid}\begin{aligned}
		(\odif{R}_\gamma)_{\gamma'}  :\ker(\odif{s})_{\gamma'}\to \ker(\odif{s})_{\gamma'\gamma} ,\quad   (\odif{L}_{\gamma'})_{\gamma}  :\ker(\odif{r})_\gamma\to \ker(\odif{r})_{\gamma'\gamma},\quad \forall(\gamma',\gamma)\in G^{(2)} .
	\end{aligned}\end{equation}
The map $\iota$ induces a diffeomorphism from $G^x$ to $G_x$. Hence, its differential gives an isomorphism 
	\begin{equation}\label{eqn:iota_iso_grpd}\begin{aligned}
		\ker(\odif{r})\simeq \iota^*(\ker(\odif{s})).		
	\end{aligned}\end{equation}
	The isomorphisms \eqref{eqn:left_right_iso_groupoid}  and \eqref{eqn:iota_iso_grpd} induce isomorphisms
	\begin{equation}\label{eqn:iso_kerds_kerdr}\begin{aligned}
			&\ker(\odif{s})\simeq r^*(A(G)), &&\ker(\odif{s})_\gamma\xrightarrow{(\odif{R}_{\gamma^{-1}})_{\gamma}}r^*(A(G))_{\gamma} \\
			&\ker(\odif{r})\simeq s^*(A(G)),  &&\ker(\odif{r})_\gamma\xrightarrow{\odif{\iota}_{s(\gamma)}\circ  (\odif{L}_{\gamma^{-1}})_{\gamma}}s^*(A(G))_{\gamma}.
		\end{aligned}\end{equation}

	\begin{prop}\label{prop:left_right_invariant_sections}
		Let $X\in C^\infty(G^{(0)},A(G))$. There exists a unique $X_R\in C^\infty(G,\ker(\odif{s}))$ called the right invariant vector field associated to $X$ such that:
		\begin{enumerate}
			\item\label{item:restriction:ksdjmfqkjsdmjfmjsdqmqf} The restriction of $X_R$ to $G^{(0)}$ is equal to $X$.
			\item\label{item:invariance:ksdjmfqkjsdmjfmjsdqmqf} For any $(\gamma',\gamma)\in G^{(2)}$, $\odif{R}_\gamma (X_R(\gamma'))=X_R(\gamma'\gamma)$, i.e., $\odif{R}_\gamma\circ X_R=X_R\circ R_\gamma$.
		\end{enumerate}
		Furthermore,
		\begin{equation}\label{eqn:drXR}\begin{aligned}
				\odif{r}\circ X_R= \rho(X)\circ r       .
			\end{aligned}\end{equation}
	\end{prop}
	\begin{proof}
		The vector field $X_R$ has to be defined by the formula
		$X_R(\gamma)=\odif{R}_\gamma( X(r(\gamma)))$.
		This proves the uniqueness. It is clear that $X_R$ satisfies $\ref{item:restriction:ksdjmfqkjsdmjfmjsdqmqf}$ and $\ref{item:invariance:ksdjmfqkjsdmjfmjsdqmqf}$.
		Since $r\circ R_\gamma=r$, \eqref{eqn:drXR} follows.
	\end{proof}

	\begin{prop}
		There exists a unique Lie bracket on $C^\infty(G^{(0)},A(G))$ which satisfies the following identities:
		\begin{equation*}\begin{aligned}
				[X,Y]_R      =[X_R,Y_R],\quad
				\rho([X,Y])  =[\rho(X),\rho(Y)],\quad
				[X,fY]       =f[X,Y]+\rho(X)(f)Y
			\end{aligned}\end{equation*}
		where $X,Y\in C^\infty(G^{(0)},A(G))$ and $f\in C^\infty(G^{(0)})$.
	\end{prop}
	\begin{proof}
		Uniqueness is clear from the first identity. For existence, we need the following lemma.
		\begin{lem}\label{prop:Lie_bracket}
			If $f:M\to M'$ is a smooth map, $X,Y\in \cX(M)$ and $X',Y'\in \cX(M')$ such that $\odif{f}\circ X=X'\circ f$ and $\odif{f}\circ Y=Y'\circ f$, then
			$\odif{f}\circ [X,Y]=[X',Y']\circ f$.
		\end{lem}
		\begin{proof}
			Both sides are functions $M\to TM'$. Both map an element $(Z,x)$ to an element of the form $(Z',f(x))$ for some $Z'\in T_{f(x)}M'$.
			Therefore, to prove that both sides are equal, it is enough to check that for any $g\in C^\infty(M')$, both sides are equal after composing with $\odif{g}$.
			We have
			\begin{equation*}\begin{aligned}
					\odif{g}\circ \odif{f} \circ [X,Y]=\odif{(g\circ f)}\circ [X,Y]=[X,Y](g\circ f) & =X(Y(g\circ f))-Y(X(g\circ f))               \\
					                                                              & = X(Y'(g)\circ f)-Y(X'(g)\circ f)            \\
					                                                              & =X'(Y'(g))\circ f-Y'(X'(g))\circ f           \\
					                                                              & =[X',Y'](g)\circ f   =\odif{g}\circ [X',Y']\circ f.
				\end{aligned}\end{equation*}
				The result follows.
		\end{proof}
		Let $X,Y\in C^\infty(G^{(0)},A(G))$.
		By \Cref{prop:Lie_bracket}, $[X_R,Y_R]\in C^\infty(G,\ker(\odif{s}))$.
		Let $[X,Y]\in C^\infty(G^{(0)},A(G))$ be the restriction of $[X_R,Y_R]$ on $G^{(0)}$.
		By \Cref{prop:left_right_invariant_sections} and \Cref{prop:Lie_bracket},
		$[X,Y]_R=[X_R,Y_R]$.
		By taking $\odif{r}$ on both sides and using \eqref{eqn:drXR} and \Cref{prop:Lie_bracket}, we deduce that
		$\rho([X,Y])=[\rho(X),\rho(Y)]$.
		Finally,
		it is clear that $(fY)_R=f\circ r Y_R$.
		Hence,
		\begin{equation*}\begin{aligned}
				[X,fY]_R=[X_R,f\circ r Y_R]=f\circ r [X_R,Y_R]+ X_R(f\circ r) Y_R & =f\circ r[X,Y]_R+ \rho(X)(f)\circ r Y_R  \\&=    (f[X,Y])_R + (\rho(X)(f) Y)_R .
			\end{aligned}\end{equation*}
		Therefore, $[X,fY]=f[X,Y] + \rho(X)(f) Y$.
	\end{proof}
	Let $X\in C^\infty(G^{(0)},A(G))$. The section
	\begin{equation}\label{eqn:left_invariant_dfn}\begin{aligned}
			X_L:=-\odif{\iota} \circ X_R\circ \iota\in C^\infty(G,\ker(\odif{r}))
		\end{aligned}\end{equation}
	is called \textit{the left-invariant vector field associated to $X$}.
	The choice of the sign in \eqref{eqn:left_invariant_dfn} is to ensure that \eqref{eqn:X_distribut} holds.
	We have
	\begin{equation*}\begin{aligned}
		\odif{L}_\gamma\circ X_L=X_L\circ L_\gamma,\quad [X,Y]_L=-[X_L,Y_L],\quad [X_R,Y_L]=0,\quad \forall \gamma\in G,X,Y\in C^\infty(G^{(0)},A(G)).
		\end{aligned}\end{equation*}

	The identities $\odif{R}_\gamma\circ X_R=X_R\circ R_\gamma$ and  $\odif{L}_\gamma\circ X_L=X_L\circ L_\gamma$ imply that whenever the flows are well-defined, we have 
		\begin{equation}\label{eqn:flow_Left_and_right_invariant_vector_fields}\begin{aligned}
			&m(e^{tX_R}\cdot r( \gamma), \gamma)	=e^{tX_R}\cdot \gamma,\quad 
			&m(\gamma,e^{tX_L}\cdot s(\gamma))	=e^{tX_L}\cdot \gamma.
		\end{aligned}\end{equation}
	It follows that if $X$ is compactly supported, then $X_R$ and $X_L$ are complete.
	\begin{ex}\label{ex:Right_Left_invariant_vector_field_Pair}
		Let $M$ be a smooth manifold, $G=M\times M\rightrightarrows G^{(0)}$ the pair groupoid.
		In this case, $A(G)=TM$, and our sign convention implies that the Lie bracket agrees with the Lie bracket of vector fields.
		If $X\in \cX(M)$, then $X_R,X_L\in \cX(M\times M)$ are given by 
			\begin{equation}\label{eqn:left_right_invariant_vector_fields_example}\begin{aligned}
				X_R(x,y)=(X(x),0),\quad X_L(x,y)=(0,-X(y)).		
			\end{aligned}\end{equation}
	\end{ex}
\section{Convolution algebra}\label{sec:conv_algebra}
	
		 Let $G\rightrightarrows G^{(0)}$ be a quasi-Lie groupoid. For any $k,l\in \C$, we define
			\begin{equation*}\begin{aligned}
					\Omega^{k,l}(G):=\Omega^k(\ker(\odif{s}))\otimes         \Omega^{l}(\ker(\odif{r})),\quad \Omega^k(G):=\Omega^{k,k}(G).
				\end{aligned}\end{equation*}
		By \eqref{eqn:iso_kerds_kerdr},
		\begin{equation}\label{eqn:sqdiojfioqjsdiomfjqsidmojfimojsdoqimf}\begin{aligned}
				\Omega^{k,l}(G)\simeq r^*(\Omega^{k}(A(G)))\otimes s^*(\Omega^{l}(A(G))).
			\end{aligned}\end{equation}
		To simplify the notation, we will generally use $C^\infty(G,\Omega^{k,l})$ instead of  $C^\infty(G,\Omega^{k,l}(G))$.
		
		\paragraph{Adjoint:}
	      We define the adjoint of $f\in C^\infty(G,\Omega^{k,l})$ by
		\begin{equation*}\begin{aligned}
			 f^*(\gamma):=f(\gamma^{-1})^*\in C^\infty(G,\Omega^{\overline{l},\overline{k}}).
		\end{aligned}\end{equation*}
		\paragraph{Convolution:}We now define the convolution of two functions on $G$.
		\begin{lem}\label{lem:natural_isom_multiplication}
			 There is a natural isomorphism between the real vector bundles
			\begin{equation*}\begin{aligned}
					m^*\left(\Omega^{k_1,k_4}(G)\right)\otimes \Omega^{k_2+k_3}(\ker(\odif{m})) \simeq \pi_1^*\left(\Omega^{k_1,k_2}(G)\right)\otimes \pi_2^*\left(\Omega^{k_3,k_4}(G)\right),\quad \forall  k_1,k_2,k_3,k_4\in \C,
				\end{aligned}\end{equation*}
			where $\pi_i:G^{(2)}\to G$ are the natural projections, and $m:G^{(2)}\to G$ is the multiplication map.
		\end{lem}
		\begin{proof}
			Let $l_1,l_2,l_3:G^{(2)}\to G^{(0)}$ be the maps $l_1(\gamma,\gamma')=r(\gamma)$, $l_2(\gamma,\gamma')=s(\gamma)$, $l_3(\gamma,\gamma')=s(\gamma')$.
			Clearly, we have
			\begin{equation*}\begin{aligned}
					\pi_1^*\left(\Omega^{k_1,k_2}(G)\right)\otimes \pi_2^*\left(\Omega^{k_3,k_4}(G)\right) & \simeq l_1^*\left(\Omega^{k_1} A(G)\right)\otimes l_2^*\left(\Omega^{k_2+k_3}A(G)\right)\otimes l_3^*\left(\Omega^{k_4} A(G)\right) 
					\\ &\simeq   m^*\left(\Omega^{k_1,k_4}(G)\right)\otimes \Omega^{k_2+k_3}(l_2^*\left(A(G)\right))
				\end{aligned}\end{equation*}
			It remains to show that $l_2^*(A(G))\simeq \ker(\odif{m})$. For any $(\gamma,\gamma')\in G^{(2)}$, the map 
				\begin{equation*}\begin{aligned}
					G_{l_2(\gamma,\gamma')}\to m^{-1}(\gamma\gamma'),\quad \gamma''\mapsto (\gamma \gamma^{\prime\prime -1},\gamma''\gamma')
				\end{aligned}\end{equation*}
			is a diffeomorphism whose differential at $l_2(\gamma,\gamma')$ induces the required isomorphism.
		\end{proof}
		Let $f\in C^\infty(G,\Omega^{k_1,k_2})$, $g\in C^\infty(G,\Omega^{1-k_2,k_3})$.
		By \Cref{lem:natural_isom_multiplication}, we define
			\begin{equation*}\begin{aligned}
				 f\boxtimes g(\gamma,\gamma')=f(\gamma)g(\gamma')\in C^\infty\left(G^{(2)}, m^*\left(\Omega^{k_1,k_3}(G)\right)\otimes \Omega^{1}(\ker(\odif{m})) \right).
			\end{aligned}\end{equation*}
		If $f$ or $g$ is properly supported, then $m:\supp(f\boxtimes g)\to G$ is proper. 
			Hence, we can define
		\begin{equation*}\begin{aligned}
				f\ast g:= m_*(f\boxtimes g)\in C^\infty(G,\Omega^{k_1,k_3}).
			\end{aligned}\end{equation*}
		Furthermore, if $f$ and $g$ are properly supported, then $f\boxtimes g$ is also properly supported.
		The main case of interest to us is when $k_1=k_2=k_3=\frac{1}{2}$.
		\begin{prop}\label{prop:convolution_algebra}
			The space $C^\infty_{\mathrm{proper}}(G,\omegahalf)$ is a $*$-algebra, and  
			\begin{equation}\label{eqn:supp_product_functions}\begin{aligned}
				\supp(f\ast g)\subseteq \supp(f)\supp(g),\quad \supp(f^*)=\supp(f)^{-1}		,\quad \forall f,g\in C^\infty_{\mathrm{proper}}(G,\omegahalf).
			\end{aligned}\end{equation}
		\end{prop}
	\paragraph{Action of the Lie algebroid:}
		If $X\in C^\infty(G,\ker(\odif{s}))$ or $X\in C^\infty(G,\ker(\odif{r}))$, then by \eqref{eqn:sqdiojfioqjsdiomfjqsidmojfimojsdoqimf} and \eqref{eqn:Lie_derivative_dfn}, we have a map
		\begin{equation*}\begin{aligned}
				\cL_X:C^\infty(G,\Omega^{k,l})\to C^\infty(G,\Omega^{k,l}),\quad \forall k,l\in \C.
			\end{aligned}\end{equation*}
		Hence, if $X\in C^\infty(G^{(0)},A(G))$, then we have maps
		\begin{equation*}\begin{aligned}
				\cL_{X_R},\cL_{X_L}:        C^\infty(G,\Omega^{k,l})\to C^\infty(G,\Omega^{k,l}),\quad \forall k,l\in \C.
			\end{aligned}\end{equation*}
		It is very convenient to use the notation
		\begin{equation*}\begin{aligned}
				X\ast f:=\cL_{X_R}(f),\quad f\ast X:=\cL_{X_L}(f),\quad \forall f\in C^\infty(G,\Omega^{k,l}).
			\end{aligned}\end{equation*}
		If $f,g\in C^\infty_{\mathrm{proper}}(G,\omegahalf)$, then 
		\begin{equation}\label{eqn:X_distribut}\begin{aligned}
				  X\ast (f\ast g)= (X\ast f)\ast g,\quad
				  (f\ast X)\ast g=f\ast (X\ast g),\quad (X\ast f)^*=-f^*\ast X.
			\end{aligned}\end{equation}

	\section{\texorpdfstring{$L^1$}{L1}-Banach algebra}\label{sec:compl}
	
			We fix for the rest of this section, a positive section
			$\omega\in C^\infty(G,\Omega^{\frac{1}{2},-\frac{1}{2}})$
			 which satisfies 
				\begin{equation}\label{eqn:omega_symmetry_condition}\begin{aligned}
					\omega(\gamma_1\gamma_2)=\omega(\gamma_1)\omega(\gamma_2),\ \forall(\gamma_1,\gamma_2)\in G^{(2)},\quad \omega(\gamma)=1,\ \forall\gamma\in G^{(0)}.	
				\end{aligned}\end{equation}
	
				\begin{ex}
					 If $\eta\in C^\infty(G^{(0)},\Omega^{\frac{1}{2}}(A(G)))$ is a positive section which vanishes nowhere, then $\omega(\gamma)=\eta(r(\gamma))\eta(s(\gamma))^{-1}$ satisfies \eqref{eqn:omega_symmetry_condition}.
					 Not all $\omega\in C^\infty(G,\Omega^{\frac{1}{2},-\frac{1}{2}})$ which satisfy \eqref{eqn:omega_symmetry_condition} come from a section $\eta\in C^\infty(G^{(0)},\Omega^{\frac{1}{2}}(A(G)))$.
					We will consider such sections, see \Cref{thm:L1normRiemMetric}.
				\end{ex}
	
			If $f\in C^\infty_c(G,\omegahalf)$, then $|f|\omega\in C_c(G,\Omega^{1,0})$.
			By \eqref{eqn:iso_kerds_kerdr}, we can integrate $|f|\omega$ along the fibers $s$.
			Similarly, we can integrate $|f|\omega^{-1}$ along the fibers of $r$.
			We define
			\begin{equation}\label{eqn:L1norms}\begin{aligned}
					\norm{f}_{L^1(G,\omega)}:&=\max\left(\norm{s_*(|f|\omega)}_{C^0(G^{(0)})},\norm{r_*(|f|\omega^{-1})}_{C^0(G^{(0)})}\right)
			\end{aligned}\end{equation}
			The completion of $C^\infty_c(G,\Omega^\frac{1}{2})$ by the norm $\norm{\cdot}_{L^1(G,\omega)}$ is denoted by $L^1(G,\omega)$.
			We have,
				\begin{equation*}\begin{aligned}
						\norm{f\ast g}_{L^1(G,\omega)}\leq \norm{f}_{L^1(G,\omega)}\norm{g}_{L^1(G,\omega)},\quad \norm{f^*}_{L^1(G,\omega)}=\norm{f}_{L^1(G,\omega)},\quad \forall f,g\in C^\infty_c(G,\omegahalf) .		
					\end{aligned}\end{equation*}
			In particular, $L^1(G,\omega)$ is a Banach $*$-algebra.
	
		\begin{prop}[Approximate identity]\label{prop:Approximate_identity}
			There exists $(g_n)_{n\in \N}\subseteq C^\infty_c(G,\omegahalf)$ a sequence of self-adjoint elements such that 
			\begin{enumerate}
				\item\label{prop:Approximate_identity:G0}\label{prop:Approximate_identity:1}  
				If $K\subseteq G^{(0)}$ is compact, then the sequence $\supp(g_n)\cap s^{-1}(K)$ is eventually decreasing, and its set theoretic limit is equal to $K$, i.e., $\bigcap_{n\in \N}\bigcup_{m\geq n}\supp(g_m)\cap s^{-1}(K)=K$.
				\item\label{prop:Approximate_identity:2}   For any $f\in C^\infty_c(G,\omegahalf)$, $g_n\ast f$ and $f\ast g_n$ converge to $f$ in the $C^\infty_c(G,\omegahalf)$-topology.
				\item\label{prop:Approximate_identity:proper} 	For any $f\in C^\infty_{\mathrm{proper}}(G,\omegahalf)$, $g_n\ast f$ and $f\ast g_n$ converge to $f$ in the $C^\infty(G,\omegahalf)$-topology.
				\item\label{prop:Approximate_identity:L1} 	For any $f\in L^1(G,\omega)$, $g_n\ast f$ and $f\ast g_n$ converge to $f$ in the $L^1(G,\omega)$-topology.
			\end{enumerate}
		\end{prop}
		In the continuous setting, this result is due to Renault \cite[Lemma 3.2]{RenaultRepProdCroises}.
		\begin{proof}
			To simplify the exposition, we will only deal with the case where $G^{(0)}$ is compact. 
			The general case is dealt with similarly using a partition of unity on $G^{(0)}$.
			By \Cref{rem:finitely_generated_section}, there exists a finite family $X_{1},\cdots,X_{k}\in C^\infty(G^{(0)},A({G}))$ which generates $C^\infty(G^{(0)},A({G}))$.
			The vector fields $X_{1R},\cdots,X_{kR}$ are complete, see \eqref{eqn:flow_Left_and_right_invariant_vector_fields}.
			The map
				\begin{equation*}\begin{aligned}
					\phi:	\R^{k}\times G^{(0)}\to G,\quad \phi(t_1,\cdots,t_{k},x)=e^{t_1 X_{1R}(x)+\cdots+t_{k}X_{kR}}\cdot x	
				\end{aligned}\end{equation*}
			is a submersion in a neighborhood of $\{0\}\times G^{(0)}$ because $s$ is a submersion.
			To see this, let $\psi:G\to \R^n\times G^{(0)}$ be a local diffeomorphism of the form $\gamma\mapsto (g(\gamma),s(\gamma))$.
			The map $\psi\circ \phi:\R^k\times G^{(0)}\to \R^n\times G^{(0)}$ is of the form $(t,x)\mapsto (h(t,x),x)$ for some smooth map $h$.
			The differential of $h$ in the direction of $t$ at $(0,x)$ is surjective. So, by the parameterized inverse mapping theorem, we deduce that $\psi\circ \phi$ is a submersion which implies that $\phi$ is a submersion.

			By \eqref{eqn:flow_Left_and_right_invariant_vector_fields}, $\phi(t,x)^{-1}=\phi(-t,r(\phi(t,x)))$.
			We denote by $U$ the set of $(t,x)$ such that $\phi$ is a well-defined submersion at $(t,x)$ and at $(-t,\phi(t,x))$.
			Since $s\circ \phi(t,x)=x$, by \Cref{prop:cor_triangle_isom}, we get a canonical isomorphism $\Omega^1(\ker(\odif{\phi}))\otimes \phi^*(\Omega^{1,0}(G))\simeq \C$.
			So, we have an integration along the fibers map 
				\begin{equation*}\begin{aligned}
					\phi_*:C^\infty_c(U)\to C^\infty_c(G,\Omega^{1,0}).		
				\end{aligned}\end{equation*}
			Let $(t,x)\in U$. 
			The exponential map $\gamma\in G_x\mapsto e^{t_1X_{1R}+\cdots+t_{k}X_{kR}}\cdot \gamma\in G_x$ defines a local diffeomorphism from $x$ to $\phi(t,x)$.
			Hence, its differential at $x$ is a linear isomorphism $\ker(\odif{s})_x\to \ker(\odif{s})_{\phi(t,x)}$. 
			So, it defines an element of $\Omega^{1,-1}(G)_{\phi(t,x)}$.
			The element $\omega^2(\phi(t,x))$ also defines an element of $\Omega^{1,-1}(G)_{\phi(t,x)}$.
			By taking the quotient, we obtain a smooth function $\kappa:U\to \Rpt$.
			One has 
				\begin{equation}\label{eqn:sqdjkofjmsdkqjkmfjsdf}\begin{aligned}
					\phi_*(h)^*=\phi_*(\tilde{h}\kappa)\omega^{-2},\quad \forall h\in C^\infty_c(U),		
				\end{aligned}\end{equation}
			where $\tilde{h}(t,x)=\overline{h(-t,r(\phi(t,x)))}$. Furthermore, $\tilde{\kappa}=\frac{1}{\kappa}$ and $\kappa(0,x)=1$.
			Now, let $(h_{n})_{n\in \N}\subseteq C^\infty_c(\R^{k},\R_+)$ be a sequence of functions such that $\int_{\R^{k}} h_{n}(t)dt=1$, $h_{n}(t)=h_{n}(-t)$, and the sequence $(\supp(h_{n}))_{n\in \N}$ is decreasing with intersection equal to $\{0\}$, and $\supp(h_{n})\times G^{(0)}\subseteq U$.
			We see $h_{n}$ as functions on $\R^{k}\times G^{(0)}$ which don't depend on $x$. 
			Let 
				\begin{equation*}\begin{aligned}
					g_n=\phi_*(h_n\kappa^{\frac{1}{2}})\omega^{-1}	
				\end{aligned}\end{equation*}
			Clearly, $g_n$ is real-valued. By \eqref{eqn:sqdjkofjmsdkqjkmfjsdf}, $g^*=g$. 
			\Crefitem{prop:Approximate_identity}{G0} follows from the inclusion $G^{(0)}\subseteq \supp(g_n)\subseteq \phi(\supp(h_n))$.
			Let $f\in C^\infty_c(G,\omegahalf)$. We have
				\begin{equation*}\begin{aligned}
					f(\gamma)- f\ast g_n(\gamma) &=
					f(\gamma)-\int_{G_{s(\gamma)}}f(\gamma \eta^{-1})\omega(\eta^{-1})g_n(\eta)\\
					&=f(\gamma)-\int_{\R^{k}}f(\gamma \phi(t,s(\gamma))^{-1})\omega(\phi(t,s(\gamma))^{-1})h_n(t)\kappa^{\frac{1}{2}}(t,s(\gamma))\\
					&=\int_{\R^{k}}\Big(f(\gamma)-f(\gamma \phi(t,s(\gamma))^{-1})\omega(\phi(t,s(\gamma))^{-1})\kappa^{\frac{1}{2}}(t,s(\gamma))\Big)h_n(t) 
				\end{aligned}\end{equation*}
			The map $\psi(t,\gamma)=f(\gamma)-f(\gamma \phi(t,s(\gamma))^{-1})\omega(\phi(t,s(\gamma))^{-1})\kappa^{\frac{1}{2}}(t,s(\gamma))\in C^\infty(\R^k\times G,\omegahalf)$ satisfies 
			$\lim_{t\to 0}\psi(t,\gamma)=\psi(0,\gamma)=0$, where the limit is in the topology of $C^\infty_c(G,\omegahalf)$.
			This proves \Crefitem{prop:Approximate_identity}{2}.
			\Crefitem{prop:Approximate_identity}{proper} follows by a similar argument.
			Finally, \Crefitem{prop:Approximate_identity}{L1} follows from \Crefitem{prop:Approximate_identity}{2} and the fact that $\norm{g_n}_{L^1(G,\omega)}\leq \int_{\R^k} h_n\kappa^{1/2}\leq \sup{\kappa^{1/2}_{|\supp(h_1)}}$ is uniformly bounded.
		\end{proof}

\section{Distributions on quasi-Lie groupoids}\label{sec:distributions_quasi_Lie}
	In this section, we define an algebra $C^{-\infty}_{r,s}(G,\omegahalf)$ of distributions on a quasi-Lie groupoid $G\rightrightarrows G^{(0)}$. 
	If $G\rightrightarrows G^{(0)}$ is a Lie groupoid, then $C^{-\infty}_{r,s}(G,\omegahalf)$ was introduced by Lescure, Manchon and Vassout \cite{LescureManchonVassout}.
	
		\begin{dfn}\label{dfn:distributions_convolution}
			An element $u\in C^{-\infty}_{r,s}(G,\omegahalf)$ consists of two functions $u\ast \cdot,\cdot \ast u:C^\infty_{c}(G,\omegahalf)\to C^\infty_{c}(G,\omegahalf)$ such that 
				\begin{equation*}\begin{aligned}
					(f\ast u) \ast g=f\ast (u \ast g)	,\quad \forall f,g\in 	C^\infty_{c}(G,\omegahalf).
				\end{aligned}\end{equation*}
		\end{dfn}
		\begin{prop}\label{prop:automatic_continuity_distributions_quasi_Lie_groupoid}
			Let $u\in C^{-\infty}_{r,s}(G,\omegahalf)$. The maps $u\ast \cdot$ and $\cdot \ast u$ are continuous $\C$-linear maps.
			Furthermore, the natural map $C^\infty_{\mathrm{proper}}(G,\omegahalf)\to C^{-\infty}_{r,s}(G,\omegahalf)$ is injective.
		\end{prop}
		\begin{proof}
			Let $(g_n)_{n\in \N}$ be an approximate identity as in \Cref{prop:Approximate_identity}.
			We have 
				\begin{equation*}\begin{aligned}
						u\ast f=\lim_{n\to +\infty}g_n\ast (u\ast f)=\lim_{n\to +\infty}(g_n\ast u)\ast f		
				\end{aligned}\end{equation*}
			Since $g_n\ast u\in C^{\infty}_{c}(G,\omegahalf)$, the map $f\in C^\infty_{c}(G,\omegahalf)\mapsto (g_n\ast u)\ast f\in C^\infty_{c}(G,\omegahalf)$ is continuous.
			By the definition of the LF topology on $C^\infty_c(G,\omegahalf)$, a linear map $C^\infty_c(G,\omegahalf)\to C^\infty_c(G,\omegahalf)$ is continuous if and only if its restriction to $C^\infty(K,\omegahalf)$ is continuous for every compact $K\subseteq G$.
			Now, continuity of $u\ast \cdot:C^\infty(K,\omegahalf)\to C^\infty_c(G,\omegahalf)$ follows from the uniform boundedness principle, see \cite[Theorem 2.8]{RudinFA}.
			Injectivity of the inclusion $C^\infty_{\mathrm{proper}}(G,\omegahalf)\hookrightarrow C^{-\infty}_{r,s}(G,\omegahalf)$ follows from \Crefitem{prop:Approximate_identity}{proper}.
		\end{proof}
		The space $C^{-\infty}_{r,s}(G,\omegahalf)$ is naturally a $*$-algebra with the product and adjoint defined by 
			\begin{equation*}\begin{aligned}
				(u\ast v) \ast f:=u\ast (v\ast f),\quad f\ast (u\ast v):=(f\ast u)\ast v,\quad u^*\ast f:=(f^*\ast u)^*,\quad f\ast u^*:=(u\ast f^*)^*,	
			\end{aligned}\end{equation*}
		where $f\in C^\infty_c(G,\omegahalf),u,v\in C^{-\infty}_{r,s}(G,\omegahalf)$.
		Furthermore, $C^\infty_{\mathrm{proper}}(G,\omegahalf)$ is a $*$-subalgebra of $C^{-\infty}_{r,s}(G,\omegahalf)$. 
		In fact, in \Cref{prop:ideal_proper_smooth_distributions}, we will prove that it is a two-sided ideal.
		We equip the space $C^{-\infty}_{r,s}(G,\omegahalf)$ with the weakest topology such that for any $f\in C^\infty_c(G,\omegahalf)$, the maps 
			\begin{equation*}\begin{aligned}
				C^{-\infty}_{r,s}(G,\omegahalf)\to C^\infty_c(G,\omegahalf),\quad	u\mapsto f\ast u ,\quad u\mapsto u\ast f 
			\end{aligned}\end{equation*}
		are continuous. 
		The proof of \Cref{prop:automatic_continuity_distributions_quasi_Lie_groupoid} shows that $ C^\infty_c(G,\omegahalf)$ is dense in $C^{-\infty}_{r,s}(G,\omegahalf)$.
		\begin{exs}\label{ex:distributions_quasi_Lie_groupods}
			\begin{enumerate}
				\item\label{ex:distributions_quasi_Lie_groupods:inclusion_Lie_algebroid}   By \eqref{eqn:X_distribut}, we have a linear map $C^\infty(G^{(0)},A(G))\to C^{-\infty}_{r,s}(G,\omegahalf)$.
				\item\label{ex:distributions_quasi_Lie_groupods:dirac} 	Let $f\in C^\infty(G^{(0)})$. 
					We define $\delta_{f}\in  C^{-\infty}_{r,s}(G,\omegahalf)$ by
					\begin{equation}\label{eqn:convolution_multiplication_function_base}\begin{aligned}
						\delta_{f}\ast g:=(f\circ r)\cdot g,\quad g\ast \delta_{f}:=g \cdot(f\circ s),\quad g\in C^\infty_c(G,\omegahalf),
					\end{aligned}\end{equation}
				where $\cdot$ is pointwise multiplication.
				This defines a $*$-algebra inclusion 
					\begin{equation}\label{eqn:functions_are_dist_quasi_Lie}\begin{aligned}
						C^\infty(G^{(0)})\hookrightarrow C^{-\infty}_{r,s}(G,\omegahalf).
					\end{aligned}\end{equation}
				Clearly $\delta_1$ is the unit of $C^{-\infty}_{r,s}(G,\omegahalf)$.
				If $f\circ r=f\circ s$, then $\delta_{f}$ is in the center of $C^{-\infty}_{r,s}(G,\omegahalf)$.
				If $G\rightrightarrows G^{(0)}$ is a Lie groupoid, then $\delta_f$ is the distribution equal to $f$ supported on $G^{(0)}$.
				\item\label{ex:distributions_quasi_Lie_groupods:3} If $G=M\times M\rightrightarrows M$  is the pair groupoid over a smooth manifold $M$, then an element of $C^{-\infty}_{r,s}(M\times M,\omegahalf)$ is a distribution $u$ on $M\times M$ with values in $\omegahalf(T(M\times M))$ such that
					\begin{equation*}\begin{aligned}
						x\mapsto \int_M u(x,y)f(y) \in C^\infty_c(M,\omegahalf(TM)),\quad y\mapsto \int_M u(x,y)f(x)\in C^\infty_c(M,\omegahalf(TM)),		
					\end{aligned}\end{equation*}
					for all $f\in C^\infty_c(M,\omegahalf(TM))$. Such distributions are also called semiregular \cite[p.532]{TrevesTopolgoicalVectorSPacesBoook}.
			\end{enumerate}
		\end{exs}

		\begin{prop}\label{prop:proper_support_distribution}
		There is a unique way to associate to each $u\in C^{-\infty}_{r,s}(G,\omegahalf)$ a proper subset $\supp(u)\subseteq G$ such that the following holds:
			\begin{enumerate}
				\item\label{prop:proper_support_distribution:compat} If $u\in C^\infty_{\mathrm{proper}}(G,\omegahalf)$, then $\supp(u)$ is equal to the support of $u$ as a smooth section.
				\item\label{prop:proper_support_distribution:Limit} If $(u_n)_{n\in \N}\subseteq  C^{-\infty}_{r,s}(G,\omegahalf)$ converges to $u$, then $\supp(u)\subseteq \bigcap_{n\in \N}\overline{\bigcup_{m\geq n}\supp(u_m)}$.
				\item\label{prop:proper_support_distribution:2} If $u,v\in  C^{-\infty}_{r,s}(G,\omegahalf)$, then $\supp(u\ast v)\subseteq \supp(u)\supp(v)$ and $\supp(u^*)=\supp(u)^{-1}$.
			\end{enumerate}
		\end{prop}
		\begin{proof}
			Let $(g_n)_{n\in \N}\subseteq C^\infty_c(G,\omegahalf)$ as in \Cref{prop:Approximate_identity}.
			If it is possible to define $\supp(u)$ which satisfies the above properties, then by \Crefitem{prop:Approximate_identity}{G0}, one has 
				\begin{equation}\label{eqn:support_dfn}\begin{aligned}
					\supp(u)=\bigcap_{n\in \N}\overline{\bigcup_{m\geq n}\supp(u\ast g_m)}.
				\end{aligned}\end{equation}
			This proves uniqueness. To prove existence, we define $\supp(u)$ using \eqref{eqn:support_dfn}.
			\Crefitem{prop:proper_support_distribution}{compat} is straightforward to prove.
			We now show that $s_{|\supp(u)}:\supp(u)\to G^{(0)}$ is proper.
			Let $K\subseteq G^{(0)}$ be compact and $h\in C^\infty_c(G^{(0)})$ which is equal to $1$ on $K$. 
			It suffices to show that $\supp(u\ast g_n\ast \delta_{h})$ is uniformly compactly supported.
			By \Crefitem{prop:Approximate_identity}{1}, $g_n\ast \delta_{h}$ is uniformly compactly supported. So, the result follows from the following lemma.
			\begin{lem}\label{lem:automatic_uniform_compact}
				If $T:C^\infty_c(G,\omegahalf)\to C^\infty_c(G,\omegahalf)$ is a continuous linear map, then for any $K\subseteq G$ compact, there exists $K'\subseteq G$ compact such that $\supp(Tf)\subseteq K'$ for any $f\in C^\infty_c(G,\omegahalf)$ with $\supp(f)\subseteq K$.
			\end{lem}
			\begin{proof}
				If not, then we can find a sequence $(f_n)_{n\in \N}\subseteq C^\infty_c(G,\omegahalf)$ such that $(f_n)_{n\in \N}$ are uniformly compactly supported but $(T(f_n))_{n\in \N}$ aren't uniformly compactly supported.
				By a standard diagonal argument, we can find a sequence $c_n\in \Rpt$ such that $\sum_{n}c_nf_n$ converges in $C^\infty_c(G,\omegahalf)$ but $\sum_{n\in \N}c_nT(f_n)$ isn't compactly supported.
				This contradicts the continuity of $T$.
			\end{proof}
			We now show that $\supp(u)$ is the smallest closed subset $A$ of $G$ such that $\supp(u\ast f)\subseteq A\supp(f)$ for all $f\in C^\infty_c(G,\omegahalf)$.
			For one inclusion, let $f\in  C^\infty_c(G,\omegahalf)$. Then, $u\ast g_n\ast f\to u\ast f$.
			So, 
				\begin{equation*}\begin{aligned}
					\supp(u\ast f)\subseteq \bigcap_{n\in \N}\overline{\bigcup_{m\geq n}\supp(u\ast g_n\ast f)}\subseteq \left(\bigcap_{n\in \N}\overline{\bigcup_{m\geq n}\supp(u\ast g_n)}\right)\supp(f)	,	
				\end{aligned}\end{equation*}
			where we used the fact that $\supp(f)$ is compact in the second inclusion.
			For the other inclusion, let $A$ be a closed subset of $G$ such that $\supp(u\ast f)\subseteq A\supp(f)$ for all $f\in C^\infty_c(G,\omegahalf)$.
			Then, for any $n\in \N$, $\supp(u\ast g_n)\subseteq A\supp(g_n)$.
			Hence, 
				\begin{equation*}\begin{aligned}
					\bigcap_{n\in \N}\overline{\bigcup_{m\geq n}\supp(u\ast g_n)}\subseteq \bigcap_{n\in \N}\overline{\bigcup_{m\geq n}A\supp(g_n)} \subseteq A \left(\bigcap_{n\in \N}\overline{\bigcup_{m\geq n}\supp(g_n)} \right)=A
				\end{aligned}\end{equation*}
			where we used \Crefitem{prop:Approximate_identity}{1}.
			From this new description of the support, \Crefitem{prop:proper_support_distribution}{Limit} follows easily, as well as that $\supp(u\ast v)\subseteq \supp(u)\supp(v)$.
			Here, we use the fact that $\supp(u) \supp(v)$ is closed which follows from properness of $s:\supp(v)\to G^{(0)}$.
			One can prove by similar arguments that $\supp(u)=\bigcap_{n\in \N}\overline{\bigcup_{m\geq n}\supp(g_m\ast u)}$.
			Hence, $\supp(u^*)=\supp(u)$ and $\supp(u)$ is proper.
		\end{proof}
		\begin{prop}\label{prop:ideal_proper_smooth_distributions}
			The space $C^\infty_{\mathrm{proper}}(G,\omegahalf)$ is a two-sided ideal of $C^{-\infty}_{r,s}(G,\omegahalf)$.
		\end{prop}
		\begin{proof}
			Let $u\in C^{-\infty}_{r,s}(G,\omegahalf) $ and $f\in C^\infty_{\mathrm{proper}}(G,\omegahalf)$, $h\in C^\infty_c(G^{(0)})$. Since $f$ is properly supported, $f\ast h\in C^\infty_c(G,\omegahalf)$. 
			So, $(u\ast f)\ast \delta_h\in C^\infty_c(G,\omegahalf)$.
			By taking different functions $h$, one deduces that $u\ast f\in C^\infty(G,\omegahalf)$. It is properly supported by \Cref{prop:proper_support_distribution}.
		\end{proof}

		We denote by $C^{-\infty}_{c,r,s}(G,\omegahalf)$ the subspace of $C^{-\infty}_{r,s}(G,\omegahalf)$ of compactly supported distributions.

	\paragraph{Restriction to saturated subsets:}
		Let $A\subseteq G$ be a saturated subset.	
		Since for any $x\in A$, $(G_A)_x=G_x$ and $(G_A)^x=G^x$, it follows that $\omegahalf(G_A)$ is equal to the restriction of $\omegahalf(G)$ to $G_A$.
		By taking the restriction to $G_A$, we get a $*$-algebra homomorphism 
			\begin{equation*}\begin{aligned}
				C^\infty_{\mathrm{proper}}(G,\omegahalf)\to C^\infty_{\mathrm{proper}}(G_A,\omegahalf),\quad f\mapsto f_{|G_A}.
			\end{aligned}\end{equation*}
		This map extends to a $*$-algebra homomorphism 
			\begin{equation}\label{eqn:restriction_map_saturated_distribution_closed}\begin{aligned}
				C^{-\infty}_{r,s}(G,\omegahalf)\to C^{-\infty}_{r,s}(G_A,\omegahalf),\quad u\mapsto u_{|G_A},
			\end{aligned}\end{equation}
		where $u_{|G_A}\ast f=(u\ast g)_{|G_A}$ and $f\ast u_{|G_A}=(g\ast u)_{|G_A}$, where $f\in C^\infty_c(G_A,\omegahalf)$ and $g\in C^\infty_{c}(G,\omegahalf)$ is any extension of $f$.
		The fact that $u_{|G_A}\ast f$ and $f\ast u_{|G_A}$ don't depend on the choice of $g$ follows easily using the distributions \Crefitem{ex:distributions_quasi_Lie_groupods}{dirac}, see proof of \Cref{prop:ideal_proper_smooth_distributions}.
		Notice that
			\begin{equation}\label{eqn:support_restriction_dist}\begin{aligned}
				\supp(u_{|G_A})=\supp(u)\cap G_{|A},\quad \forall u\in C^{-\infty}_{r,s}(G,\omegahalf).
			\end{aligned}\end{equation}
			
	\section{Representation theory of quasi-Lie groupoids}\label{sec:represen_quasi_lie}
		\paragraph{Conventions and notations:}
			The following will be used throughout this article.
			\begin{itemize}
				\item 	If $V,W$ are topological vector spaces, then $\mathcal{L}(V,W)$ denotes the space of continuous linear maps $V\to W$.
				We denote by $\mathcal{L}(V)$ the space of continuous endomorphisms of $V$.
				\item If $M$ is a smooth manifold, then $L^2(M)$ denotes the Hilbert space completion of the space $C^\infty_c(M,\omegahalf(TM))$ by the inner product $\langle f,g\rangle=\int_M \bar{f}g$.
			\end{itemize}
		\begin{dfn}
			Let $H$ a Hilbert space.
			A unitary representation of a quasi-Lie groupoid $G\rightrightarrows G^{(0)}$ over $H$ is a $*$-algebra homomorphism 
				$\pi:C^\infty_c(G,\omegahalf)\to \mathcal{L}(H)$ which is continuous when $C^\infty_c(G,\omegahalf)$ is equipped with the LF-topology, and $\mathcal{L}(H)$ with the strong operator topology.
			We say that $\pi$ is non-degenerate if 
				\begin{equation*}\begin{aligned}
					\pi(C^\infty_c(G,\omegahalf))H:=\left\{\sum_{i\in I}\pi(f_i)\xi_i:I \text{ a finite index set }, (f_i)_{i\in I}\subseteq C^\infty_c(G,\omegahalf),(\xi_i)_{i\in I}\subseteq  H\right\}
				\end{aligned}\end{equation*}
			 is dense in $H$.
			We will restrict our attention to non-degenerate representations.
			It is convenient to denote the space $H$ by $L^2(\pi)$.
		\end{dfn}

		\begin{ex}\label{ex:Regular_representation}
			Let $x\in G^{(0)}$. 
			We define the regular representation $\Xi_x$ of $G$ acting on $L^2(G_x)$ by 
					$\Xi_x(f)g=f\ast g$.
		\end{ex}
					
		\begin{rem}\label{rem:equivalence}
			Let $\gamma\in G$. The smooth diffeomorphism $R_\gamma:G_{r(\gamma)}\to G_{s(\gamma)}$ induces a unitary equivalence between $\Xi_{s(\gamma)}$ and $\Xi_{r(\gamma)}$.
		\end{rem}
		\paragraph{Smooth vectors:} If $u\in C^{-\infty}_{r,s}(G,\omegahalf)$ and $\xi \in L^2(\pi)$, then $\pi(u)\xi$ denotes the unique vector in $L^2(\pi)$, if it exists, such that 
			\begin{equation}\label{eqn:smooth_vector_equation}\begin{aligned}
				\langle \pi(u)\xi,\pi(f)\eta\rangle=\langle \xi,\pi(u^*\ast f)\eta\rangle,\quad \forall \eta\in L^2(\pi),f\in C^\infty_c(G,\omegahalf).
			\end{aligned}\end{equation}
		Uniqueness follows from the non-degeneracy of $\pi$. 
		Let 
			\begin{equation*}\begin{aligned}
				C^\infty_c(\pi):=\left\{\xi\in L^2(\pi):\forall u\in C^{-\infty}_{r,s}(G,\omegahalf), \ \pi(u)\xi \text{ exists}\right\}.\\
			\end{aligned}\end{equation*}
		The space $C^\infty_c(\pi)$ is equipped with the complete locally convex topology given by the family of semi-norms $\xi\mapsto \norm{\pi(u)\xi}_{L^2(\pi)}$ for $u\in C^{-\infty}_{r,s}(G,\omegahalf)$.
		The space $C^\infty_c(\pi)$ is dense in $L^2(\pi)$ because $\pi(C^\infty_c(G,\omegahalf))L^2(\pi)\subseteq  C^\infty_c(\pi)$.

		\begin{dfn}
			We denote by $C^{-\infty}(\pi)$ the space of continuous anti-linear maps $C^\infty_c(\pi)\to \C$.
			The action of $\xi \in C^{-\infty}(\pi)$ on $\eta \in C^\infty(\pi)$ is denoted by $\langle \eta,\xi\rangle$.
		\end{dfn}
		\begin{rem}
			One can also define $C^\infty(\pi):=\left\{\xi\in L^2(\pi):\forall u\in C^{-\infty}_{c,r,s}(G,\omegahalf), \ \pi(u)\xi \text{ exists}\right\}$, and $C^{-\infty}_c(\pi)$ as its continuous dual.
		\end{rem}
		We have obvious inclusions 
			\begin{equation*}\begin{aligned}
				C^\infty_c(\pi)\subseteq L^2(\pi)\subseteq C^{-\infty}(\pi).
			\end{aligned}\end{equation*}
		If $u\in C^{-\infty}_{r,s}(G,\omegahalf)$, then $\pi(u)C^\infty_c(\pi)\subseteq C^\infty_c(\pi)$.
		By duality, we define
		\begin{equation}\label{eqn:qimosdjifojqsodmjfmqsdjfmq22}\begin{aligned}
			\pi(u):C^{-\infty}(\pi)\to C^{-\infty}(\pi),	\quad \langle \eta,\pi(u)\xi\rangle:=\langle \pi(u^*)\eta,\xi\rangle,\quad \forall \eta\in C^\infty_c(\pi),\xi\in C^{-\infty}_c(\pi).
		\end{aligned}\end{equation}
		By a Riesz representation theorem argument, if $f\in C^\infty_{c}(G,\omegahalf)$, then $\pi(f)C^{-\infty}(\pi)\subseteq C^{\infty}_c(\pi)$.
		We now state the Dixmier-Malliavin theorem \cite{DixmierMalliavin}.
		The original version of Dixmier-Malliavin theorem is stated for Lie groups. It was generalised to Lie groupoids by Francis \cite{DixmierMalliavinLieGroupoids}. 
		The proof also works for quasi-Lie groupoids.
		\begin{theorem}[Dixmier-Malliavin theorem]\label{thm:Dixmier-Malliavin_convolution}
			If $f\in C^\infty_c(G,\omegahalf)$ and $U$ is an open neighborhood of $G^{(0)}$, then $f$ can be written as sum of the form $g_1\ast h_1+\cdots+g_n\ast f_n$, where
			$g_i,h_i\in C^\infty_c(G,\omegahalf)$, $\supp(g_i)\subseteq \supp(f)$ and $\supp(h_i)\subseteq U$ for all $i\in \bb{1,n}$.
		\end{theorem}
		\begin{theorem}[Dixmier-Malliavin theorem]\label{thm:Dixmier-Malliavin_representation}
					Let $\xi\in C^\infty_c(\pi)$. There exists $\xi_1,\cdots,\xi_n\in C^\infty_c(\pi)$ and  $f_1,\cdots,f_n\in C^\infty_c(G,\omegahalf)$ such that $\xi=\pi(f_1)\xi_1+\cdots+\pi(f_n)\xi_n$.
		\end{theorem}
		The following theorem in the continuous setting is due to Renault \cite{RenaultBook}. 
		We refer the reader to \cite{AS1} for a modern proof using $C^*$-modules.
		\begin{theorem}\label{prop:automatic_continuity}
			Let $\omega\in C^\infty(G,\Omega^{\frac{1}{2},-\frac{1}{2}})$ be a positive section which satisfies \eqref{eqn:omega_symmetry_condition}.
			For any representation $\pi$ of $G$, and $f\in C^\infty_c(G,\omegahalf)$, $\norm{\pi(f)}_{\cL(L^2(\pi))}\leq \norm{f}_{L^1(G,\omega)}$.
		\end{theorem}
	\section{Full and reduced \texorpdfstring{$C^*$}{Cstar}-algebras}\label{sec:amenability}
For any $f\in C^\infty_c(G,\omegahalf)$, we define
\begin{equation*}\begin{aligned}
	\norm{f}_{C^*G}&:=\sup\left\{\norm{\pi(f)}:\pi \text{ non-degenerate representation of } G\right\}\\
	\norm{f}_{C^*_rG}&:=\sup\{\norm{\Xi_x(f)}:x\in G^{(0)}\}		.
\end{aligned}\end{equation*}
\Cref{prop:automatic_continuity} implies that $\norm{f}_{C^*G}<+\infty$ for all $f\in C^\infty_c(G,\omegahalf)$.
The completion of $C^\infty_c(G,\omegahalf)$ by the norms $\norm{\cdot}_{C^*G}$ and $\norm{\cdot}_{C^*_rG}$ are denoted by $C^*G$ and $C^*_rG$ respectively.
By \Cref{prop:automatic_continuity}, for any $\omega$ satisfying \eqref{eqn:omega_symmetry_condition}, 
\begin{equation}\label{eqn:ksqjfmksqd}\begin{aligned}
	\norm{f}_{C^*_rG}\leq \norm{f}_{C^*G}\leq \norm{f}_{L^1(G,\omega)},\quad \forall f\in C^\infty_c(G,\omegahalf).		
\end{aligned}\end{equation}
So, we have natural $*$-homomorphisms 
			\begin{equation*}\begin{aligned}
				L^1(G,\omega)\xrightarrow{i} C^*G \xrightarrow{j} C^*_rG.		
			\end{aligned}\end{equation*}
The map $j$ is surjective because its image is dense, and it is a $*$-homomorphism between $C^*$-algebras. 
The map $j\circ i$ injective by \Cref{prop:Approximate_identity}. 
\begin{prop}\label{prop:fiber_dense_reduced}
	Let $A\subseteq G^{(0)}$ be a dense subset. For any $f\in C^*_rG$, 
		\begin{equation*}\begin{aligned}
			\norm{f}_{C^*_rG}=\sup\left\{\norm{\Xi_x(f)}:x\in A\right\}		
		\end{aligned}\end{equation*}
\end{prop}
\begin{proof}
	Let $B$ be the $C^*$-algebra completion of $C^\infty_c(G,\omegahalf)$ with respect to the norm given by $\sup\left\{\norm{\Xi_x(f)}:x\in A\right\}$.
	We have a surjective $*$-homomorphism $C^*_rG\to B$.
	We need to show that this map is injective. 
	\begin{lem}\label{lem:sqdklfjklsqdjlfjkmqsdmjfqjsdmlfjqsmldfqsdjmf}
		If $f\in C^*_rG$ and $g\in C^\infty_c(G,\Omega^{\frac{1}{2},0})$, then 
			\begin{equation}\label{eqn:kjqsmkodjfmlqsjdlfjkmqsdfjqsmdojf}\begin{aligned}
				G^{(0)}\to \R_+,\quad x\mapsto \norm{\Xi_x(f)g_{|G_x}}_{L^2(G_x)}
			\end{aligned}\end{equation}
		 is a continuous function.
	\end{lem}
	\begin{proof}
		If $f\in C^\infty_c(G,\omegahalf)$, then continuity of $f\ast g\in C^\infty_c(G,\Omega^{\frac{1}{2},0})$ and the dominated convergence theorem imply continuity of \eqref{eqn:kjqsmkodjfmlqsjdlfjkmqsdfjqsmdojf}.
		The result follows from density of $C^\infty_c(G,\omegahalf)$ in $C^*_r(G)$.	
	\end{proof}
	Let $f\in C^*_rG$ such that $\Xi_x(f)=0$ for all $x\in A$.
	If $g\in C^\infty_c(G,\Omega^{\frac{1}{2},0})$, then by \Cref{lem:sqdklfjklsqdjlfjkmqsdmjfqjsdmlfjqsmldfqsdjmf} and density of $A$, $\Xi_x(f) g_{|G_x}$ vanishes for all $x\in G^{(0)}$.
	So, $\Xi_x(f)$ vanishes for all $x\in G^{(0)}$ by density of elements of the form $g_{|G_x}$ in $L^2(G_x)$. 
	So, $f=0$.
\end{proof}
Let $A\subseteq G^{(0)}$ be a closed saturated subset. 
The restriction map $C^\infty_c(G,\omegahalf)\to C^\infty_c(G_A,\omegahalf)$ and the inclusion map $C^\infty_c(G_{A^c},\omegahalf)\to C^\infty_c(G,\omegahalf)$ fit into a short exact sequence
		\begin{equation}\label{eq:exact_saturated_closed}\begin{aligned}
			0\to C^*G_{A^c}\to C^*G\to C^*G_{A}\to 0.		
		\end{aligned}\end{equation}
They also extend to an inclusion and a surjection 
		\begin{equation}\label{eq:exact_saturated_closed_2}\begin{aligned}
			C^*_rG_{A^c}		\hookrightarrow C^*_rG \twoheadrightarrow C^*_rG_{A}.
		\end{aligned}\end{equation}
	Generally \eqref{eq:exact_saturated_closed_2} is \textit{not} exact in the middle.
\begin{dfn}\label{dfn:WCP}
	We say that a quasi-Lie groupoid $G\rightrightarrows G^{(0)}$ has the weak containment property (WCP) if the natural surjective map $C^*G\to C^*_rG$ is an isomorphism.
\end{dfn}

\begin{prop}\label{prop:weak_amenable_saturated}
	Let $G\rightrightarrows G^{(0)}$ be a quasi-Lie groupoid, $A\subseteq G^{(0)}$ a closed saturated subset. If $G_{A}$ and $G_{A^c}$ have the WCP, then $G$ has the WCP.
\end{prop}
\begin{proof}
	This follows immediately from \eqref{eq:exact_saturated_closed} and \eqref{eq:exact_saturated_closed_2}.
\end{proof}
The converse of \Cref{prop:weak_amenable_saturated} is false, see \cite{WillettNonAmenable}.
For this reason, we say that a quasi-Lie groupoid has the stable weak containment property (SWCP) if $G_{A}$ has the WCP for any locally compact saturated subset $A\subseteq G^{(0)}$.
It is conjectured by C. Anantharaman-Delaroche that the SWCP is equivalent to topological amenability, see \cite{anantharamandelaroche2021remarksweakcontainmentproperty}.
The following proposition is an obvious corollary of \Cref{prop:weak_amenable_saturated}.
\begin{prop}\label{prop:stable_weak_amenable_saturated}
	Let $G\rightrightarrows G^{(0)}$ be a quasi-Lie groupoid, $A\subseteq G^{(0)}$ a closed saturated subset. The groupoids $G_{A}$ and $G_{A^c}$ have the SWCP if and only if $G$ has the SWCP.
\end{prop}
The following proposition will be used in \Cref{sec:Tangent_groupoid_representation_proof}.
\begin{prop}\label{prop:Crossed_product_amenable}
	Let $G$ be an amenable Lie group which acts on a differential space $X$, $H\rightrightarrows X$ a quasi-Lie groupoid. 
	If there exists a surjective submersion  $f:G\rtimes X\to H$ which is a morphism of quasi-Lie groupoid that is equal to the identity on $X$, then
	$H\rightrightarrows X$ has the SWCP property.
\end{prop}
\begin{proof}
	A saturated subset of $X$ for the groupoid $H$ is also a saturated subset for the groupoid $G\rtimes X$ because $f$ is surjective.
	A saturated subset is just a $G$-invariant subset of $X$. So, without loss of generality we only need to show that $H$ has the WCP.
	Topological amenability implies the WCP property, see \cite[Proposition 6.1.8 and Proposition 3.3.5]{AnaRenaultGrps}. The groupoid $G\rtimes X$ is topologically amenable by \cite[Corollary 2.2.10]{AnaRenaultGrps}.
	It follows from \cite[Proposition 5.1.2]{AnaRenaultGrps} that $H\rightrightarrows X$ is also topologically amenable.
\end{proof}

\paragraph{Fiberation of $C^*$-algebras and multipliers.}
	We will use of the theory of multipliers of $C^*$-algebras. We refer the reader to \cite[Section 2.3]{WilliamsCrossedBook} or \cite{LanceBook} for more details.
If $A$ is a $C^*$-algebra, then we denote by $\cM(A)$ its multiplier $C^*$-algebra.
Recall that a multiplier of $A$ is a linear map $T:A\to A$ such that there exists a map $T^*:A\to A$ (called the adjoint) which satisfy 
	\begin{equation*}\begin{aligned}
		T(ab)=T(a)b,\quad T(a)^*b=a^*T^*(b),\quad T^*(ab)=T^*(a)b,\quad \forall a,b\in A.		
	\end{aligned}\end{equation*}
The map $T^*$ is unique. Its existence implies by the closed graph theorem that $T$ is bounded. The algebra $\cM(A)$ with the operator norm of $T$ is a $C^*$-algebra.

We will also make use of the theory of fibered $C^*$-algebras in the sense of Kasparov \cite{KasparovInvent}.
	We refer the reader to \cite{BlanchardHopf} and \cite[Appendix]{WilliamsCrossedBook} for an introduction.
	A fibration of a $C^*$-algebra $A$ over a locally compact space $X$ is a unital $*$-homomorphism $C_b(X)\to \cZ\cM(A)$, where $C_b(X)$ is the $C^*$-algebra of bounded continuous functions on $X$, $\cZ\cM(A)$ is the center of the multiplier algebra.
	The fiber of $A$ at $x\in X$ is the $C^*$-algebra $A_x:=A/I_xA$, where $I_x=\{f\in C_0(X):f(x)=0\}$, $I_xA$ is the closure of the linear span of elements of the form $fa$ where $f\in I_x$, $a\in A$.
\paragraph{Fibration of groupoids:}
	Let $G\rightrightarrows G^{(0)}$ be a quasi-Lie groupoid, $X$ a differential space, $\rho:G^{(0)}\to X$ a smooth map which satisfies $\rho\circ r=\rho\circ s$.
	In this case for each locally closed subset $A\subseteq X$, $\rho^{-1}(A)$ is a saturated subset. 
	So, we can form the quasi-Lie groupoid $G_{\rho^{-1}(A)}\rightrightarrows \rho^{-1}(A)$.
	
	The map
		\begin{equation*}\begin{aligned}
				C^\infty_c(X)\to \cL(C^\infty_c(G)),\quad f\mapsto (g\mapsto \delta_{f\circ \rho}\ast g)
		\end{aligned}\end{equation*}
	extends to a unital $*$-homomorphism 
		$C_b(X)\to \cM(C^*G)$
	So, $C^*G$ is a $C_0(X)$-$C^*$-algebra.

	\begin{prop}\label{prop:fiber_C*_algebra}
		If $G\rightrightarrows G^0$ has the SWCP, then for any $x\in X$, the fiber (as a $C_0(X)$-$C^*$-algebra) of $C^*G$ at $x$ is equal to $C^*G_{\rho^{-1}(x)}$. 
	\end{prop}
	\begin{proof}
		By \eqref{eq:exact_saturated_closed}, it suffices to show that $I_xC^*G=C^*G_{\rho^{-1}(\{x\}^c)}$. 
		The groupoid $G_{\rho^{-1}(F^c)}$ has the WCP.
		The obvious inclusion 
				$C^\infty_c(G_{\rho^{-1}(\{x\}^c)}) \hookrightarrow	C^\infty_c(G,\omegahalf)$ maps into $I_xC^*G$. By the definition, of the maximal $C^*$-algebra, this inclusion extends to a $*$-homomorphism $C^*G_{\rho^{-1}(\{x\}^c)}\to I_xC^*G$.
		It is surjective because it has dense image.
		The inclusion $C^\infty_c(G_{\rho^{-1}(\{x\}^c)}) \hookrightarrow	C^\infty_c(G,\omegahalf)$ also extends to an injective $*$-homomorphism $C^*_rG_{\rho^{-1}(\{x\}^c)}\to C^*_rG$.
		Since $G_{\rho^{-1}(\{x\}^c)}$ has the WCP, $C^*G_{\rho^{-1}(\{x\}^c)}=C^*_rG_{\rho^{-1}(\{x\}^c)}$. So, the map $C^*G_{\rho^{-1}(\{x\}^c)}\to I_xC^*G$ is injective
	\end{proof}
	\Cref{prop:fiber_C*_algebra} is false in general without the SWCP.

\section{Automorphisms}\label{sec:automorphisms}
		An automorphism of a quasi-Lie groupoid $G$ is a morphism of quasi-Lie groupoids $\alpha:G\to G$ which is a diffeomorphism. 
	The differential of $\alpha$ induces an action on $A(G)$
	\begin{equation}\label{eqn:lqjsopdokfqskdfpokqsdof}\begin{aligned}
			\alpha_*:A(G)_x\to A(G)_{\alpha(x)},\quad     \alpha_*(X)=\odif{\alpha}_x(X)
		\end{aligned}\end{equation}
	and an action on $C^\infty(G^{(0)},A(G))$ by
	\begin{equation*}\begin{aligned}
			\alpha_*:C^\infty(G^{(0)},A(G))\to C^\infty(G^{(0)},A(G)),\quad \alpha_*(X)(x)=\alpha_*(X(\alpha^{-1}(x))).
		\end{aligned}\end{equation*}
	By \eqref{eqn:lqjsopdokfqskdfpokqsdof}, for any $k,l\in \C$, we get a map 
		\begin{equation*}\begin{aligned}
			\alpha_*:\Omega^{k,l}(G)_{\gamma}\to \Omega^{k,l}(G)_{\alpha(\gamma)},\quad \gamma\in G		
		\end{aligned}\end{equation*}
	which induces an action on $C^\infty_c(G,\omegahalf)$ by 
		\begin{equation*}\begin{aligned}
			\alpha_*(f)(\gamma)=\alpha_{*}(f(\alpha^{-1}(\gamma))).
		\end{aligned}\end{equation*}
	This action extends to $C^{-\infty}_{r,s}(G,\omegahalf)$ by the formula 
		\begin{equation*}\begin{aligned}
			\alpha_*(u)\ast f=\alpha_*(u\ast (\alpha^{-1})_*(f)),\quad f\ast \alpha_*(u)=\alpha_*((\alpha^{-1})_*(f)\ast u).		
		\end{aligned}\end{equation*}
	The map $\alpha_*:C^{-\infty}_{r,s}(G,\omegahalf)\to C^{-\infty}_{r,s}(G,\omegahalf)$ is a $*$-algebra automorphism which satisfies
		\begin{equation}\label{eqn:support_automorphism}\begin{aligned}
			\supp(\alpha_*(u))=\alpha(\supp(u)),\quad \forall u\in C^{-\infty}_{r,s}(G,\omegahalf).
		\end{aligned}\end{equation}
	For any $f\in C^\infty_c(G,\omegahalf)$, 
			\begin{equation}\label{eqn:automorphis_preserves_Cnorm}\begin{aligned}
				\norm{\alpha_*(f)}_{C^*G}=\norm{f}_{C^*G},\quad \norm{\alpha_*(f)}_{C^*_rG}=\norm{f}_{C^*_rG}.
			\end{aligned}\end{equation}
	Hence, $\alpha_*$ extends to a $*$-algebra automorphisms $\alpha_*:C^*_rG\to C^*_rG$ and $\alpha_*:C^*G\to C^*G$.

\linkchap{Multi-graded sub-Riemannian geometry}{bi-graded_tangent_groupoid}
	\paragraph{Notations and conventions:}
	We fix for the rest of this article a smooth manifold $M$ without boundary and a natural number $\nu\in \N$ which denotes the number of multi-parameters.
	\begin{enumerate}
					\item 
					  If $t\in \R_+^\nu$, $k\in \C^\nu$, then 
					\begin{equation}\label{eqn:tk}\begin{aligned}
						t^k:=t_{1}^{k_1}\cdots t_{\nu}^{k_\nu}\in \C		
					\end{aligned}\end{equation}
				with the convention $0^0=1$ and $0^z=0$ for all $z\neq 0$.
				The notation \eqref{eqn:tk} is only used either in the case where $t\in \R_+^\nu$ and $k\in \Z_+^\nu$, or in the case where $t\in (\Rpt)^\nu$ and $k\in \C^\nu$.
				The key fact which relies on this convention is that if $k\in \Z_+^\nu$, then the map $t\in \R_+^\nu\mapsto t^k\in \C$ is smooth.
				\item Let $k,l\in \C^\nu$.
				We use the partial order $k\preceq  l$ if $l-k\in \Z_+^\nu$.
				This induces the relation $k\prec l$ if $k\preceq  l$ and $k\neq l$.
				The symbols $\preceq $ and $\prec $ are reserved for this partial order.
	\item   If $\mathfrak{g}$ is a nilpotent Lie algebra and $a,b\in \mathfrak{g}$, then we denote by
			\begin{equation}\label{eqn:BCH}\begin{aligned}
					\BCH(a,b):=a+b-\frac{[a,b]}{2}+\frac{1}{12}[a,[a,b]]-\frac{1}{12}[b,[a,b]]+\cdots
				\end{aligned}\end{equation}
			the Baker-Campbell-Hausdorff product of $a$ and $b$ which is a finite sum because $\mathfrak{g}$ is nilpotent.
			Often, we use the notation $a\cdot b$ instead of $\BCH(a,b)$.
			Our BCH formula differs by a sign from the classical formula, i.e., it is the BCH formula for the opposite Lie algebra.
			See \eqref{eqn:BCH_group} for the motivation behind our choice.
			The inverse of $a$ which is $-a$ is often denoted by $a^{-1}$.
		
	\item	 We recall that if $X\in \cX(M)$ and $x\in M$, then $e^X\cdot x$ denotes the flow, if it exists, of $X$ at time $1$ starting from $x$.
			So, the flow of $X$ at time $t$ is equal to $e^{tX}\cdot x$.			
			\textbf{We stress that we never use time dependent flows.}
			Throughout the article, we will have real vector spaces $V$ equipped with a linear map $\natural:V\to \cX(M)$.
			To ease the notation, we will use $e^v\cdot x$ instead of $e^{\natural(v)}\cdot x$.
			This shouldn't cause any confusion because the map $\natural$ will be clear from context. 
			If $v,w\in V$, then we write $e^ve^w\cdot x$ instead of $e^v\cdot (e^w\cdot x)$.
			We remark that in some situations $V$ will be a nilpotent Lie algebra. 
			Our sign convention in the BCH formula implies that if $\natural:V\to \cX(M)$ is a Lie algebra homomorphism, then
			\begin{equation}\label{eqn:BCH_group}\begin{aligned}
				e^{\BCH(v,w)}\cdot x=e^ve^w\cdot x.
				\end{aligned}\end{equation}
			Generally,  $\natural:V\to \cX(M)$ won't be a Lie algebra homomorphism.
			It will only preserve some Lie brackets up to some error term, see \eqref{eqn:bi_graded_lie_basis_commutator}.
			This adds an error term to \eqref{eqn:BCH_group}, which is ultimately resolved by \Cref{thm:composition_bisub}, see \eqref{eqn:product_phi_compsotion_bisub}.
		\end{enumerate}
\section{Multi-graded sub-Riemannian structure}\label{sec:weighted_sub_riem_structure}	
		\begin{dfn}\label{dfn:weighted_sub_riemannian_structure}
			A $\nu$-graded sub-Riemannian structure on $M$ of depth $N\in \N^\nu$ is a family  $(\cF^{k})_{k\in \Z_+^\nu}$  of finitely generated $C^\infty(M,\R)$-submodules of $\cX(M)$ such that:
			\begin{enumerate}
				\item For any $k,l\in \Z_+^\nu$, if $k\preceq  l$, then $\cF^{k}\subseteq \cF^{l}$.
				\item\label{eqn:kjqsdfjmljqsdmjfmqsdljfkmqlsdf} One has $\cF^{0}=0$ and 
					\begin{equation}\label{eqn:finit_N_condition}\begin{aligned}
						\cF^{(N_1,0,\cdots,0)}=\cF^{(0,N_2,0,\cdots,0)}=\cdots=\cF^{(0,\cdots,0,N_\nu)}=\cX(M).
					\end{aligned}\end{equation}
				\item For any $k,l\in \Z_+^\nu$,
					\begin{equation}\label{eqn:Liebracket_cFi}
						[\cF^{k},\cF^{l}]\subseteq \cF^{k+l}.
					\end{equation}
			\end{enumerate}
		\end{dfn}
		Notice that $\cF^{k}=\cX(M)$ if $N_i\leq k_i$ for some $i\in \bb{1,\nu}$.
		So, to define a $\nu$-graded sub-Riemannian structure of depth $N$, it is enough to define $\cF^{k}$ for $k\in \Z_+^\nu$ with $k_i<N_i$ for all $i\in \bb{1,\nu}$ and $k\neq 0$.

		\begin{rem}
			For any examples of $\nu$-graded sub-Riemannian structures to exist, $\cX(M)$ needs to be finitely generated as a $C^\infty(M,\R)$-module.
			This is always true, see \Cref{rem:finitely_generated_section}.
		\end{rem}
		\begin{exs}\label{exs:exs_weighted}
			\begin{enumerate}
				\item The theory of elliptic operators corresponds to the trivial filtration where $\nu=1$ and $N=1$.
				\item Let $X_1,\cdots,X_d$ be a family of vector fields satisfying Hörmander's condition of depth $N\in \N$, i.e., for any $x\in M$, $T_xM$ is linearly spanned by iterated commutators of $X_i$ of length $\leq N$. 
					We define a $1$-graded sub-Riemannian structure of depth $N$ by declaring $\cF^{i}$ to be the module generated by all iterated Lie brackets of length $\leq i$ using the vector fields $X_\bullet$.
				\item\label{exs:exs_weighted:Weighted_bigraded}
					Let $X_1,\cdots,X_{d}$ and $Y_1,\cdots,Y_{d'}$ be two families of vector fields such that each family satisfies Hörmander's condition of depth $N\in \N$ and $N'\in \N$ respectively.
					We define a $2$-graded sub-Riemannian structure of depth $(N,N')$ by declaring $\cF^{(i,j)}$ to be the module generated by the iterated Lie brackets of length $\leq i+j$ using the vector fields $X_\bullet$ and $Y_\bullet$ and such that the vector fields $X_\bullet$ appear $\leq i$ times and the vector fields $Y_\bullet$ appear $\leq j$ times.
			\end{enumerate}
		\end{exs}

	\paragraph{Pullback:}
		One can pull back $\nu$-graded sub-Riemannian structures along a submersion.
		Let $\phi:M'\to M$ be a smooth submersion, $\cF\subseteq \cX(M)$ a $C^\infty(M,\R)$-submodule.
		We define $\phi^*(\cF)$ to be the $C^\infty(M',\R)$-submodule of $\cX(M')$ consisting of all $X\in \cX(M')$ such that there exists $X_1,\cdots,X_n\in \cF$ and $g_1,\cdots,g_n\in C^\infty(M')$ such that
		\begin{equation*}\begin{aligned}
				\odif{\phi}(X)=\sum_{i=1}^n g_iX_i\circ \phi\in C^\infty(M',\phi^*(TM)).
			\end{aligned}\end{equation*}
			
		\begin{prop}\label{prop:pullback_weighted_subRiem}
			If $\cF^{\bullet}$ is a $\nu$-graded sub-Riemannian structure on $M$ of depth $N$, then $\phi^*(\cF^{\bullet})$ is a $\nu$-graded sub-Riemannian structure on $M'$ of depth $N$.
		\end{prop}
		\begin{proof}
			Let $k\in \Z_+^\nu$, $X_1,\cdots,X_n\in \cF^{k}$ be generators.
			We will show that $\phi^*(\cF^{k})$ is finitely generated.
			Since $\phi$ is a submersion, for each $i\in \bb{1,n}$, there exists $X'_i\in \cX(M')$ such that $\odif{\phi}(X'_i)=X_i\circ \phi$.
			By \Cref{rem:finitely_generated_section}, $C^\infty(M',\ker(\odif{\phi}))$ is a finitely generated $C^\infty(M',\R)$-module.
			Let $Y_1,\cdots,Y_m\in C^\infty(M',\ker(\odif{\phi}))$ be some generators.
			Clearly, $X'_1,\cdots,X'_n,Y_1,\cdots,Y_m$ generate $\phi^*(\cF^{k})$.
			Finally, it is enough to check \eqref{eqn:Liebracket_cFi} on generators which follows from \Cref{prop:Lie_bracket}.
		\end{proof}
		We fix for the rest of this chapter a $\nu$-graded sub-Riemannian structure $\cF^{\bullet}$ on $M$.
	\paragraph{Multi-graded basis:}
		The modules $\cF^{\bullet}$ are finitely generated.
		It is well known that linear algebra is simpler when one avoids using linear bases.
		Nevertheless, we still need to make use of generators of $\cF^{\bullet}$.
		So, to avoid complicated notation involving indices, we introduce the following notion.
		\begin{dfn}\label{dfn:graded_lie_basis}
			A $\nu$-graded basis is a pair $(V,\natural)$, where
					$V=\bigoplus_{k\in \Z_+^\nu}V^{k}$
			is a $\nu$-graded finite dimensional  real vector space, $\natural:V\to \cX(M)$ an $\R$-linear map such that:
			\begin{enumerate}
				\item\label{dfn:graded_lie_basis:1} One has $V^{0}=0$. 
				\item\label{dfn:graded_lie_basis:2} For each $k\in \Z_+^\nu$, $\natural(V^{k})\subseteq \cF^{k}$.
				\item\label{dfn:graded_lie_basis:3} For each $k\in \Z_+^\nu$, $\cF^{k}$ is equal to the $C^\infty(M,\R)$-module generated by
							$\natural\left(\bigoplus_{l\preceq  k}V^{l}\right)$.
			\end{enumerate}
		\end{dfn}
	\paragraph{Dilations on $\nu$-graded basis:}
		Let $(V,\natural)$ be a $\nu$-graded basis, $\lambda\in \R_+^\nu$. We define the dilation
		\begin{equation}\label{eqn:dilations_basic_V}\begin{aligned}
				 & \alpha_{\lambda}:V\to V, & \alpha_\lambda\left(\sum_{k\in \Z_+^\nu}v_{k}\right)=\sum_{k\in \Z_+^\nu} \lambda^{k}v_{k},\quad\text{where } v_{k}\in V^{k} \text{ and }\lambda^k \text{ is defined in }\eqref{eqn:tk}. \\
			\end{aligned}\end{equation}
			
		\begin{rem}
			Various different dilations will be defined throughout this article. All of them will be denoted by $\alpha_\lambda$ as there is no risk of confusion.
			It is crucial to notice that some are well-defined for $\lambda\in \R_+^\nu$ while others are only well-defined for $\lambda \in (\Rpt)^\nu$, see for example \Cref{sec:smooth_functions_vanishing_Fourier_Transform} and more precisely \Cref{thm:equivariance_Cn_norm} where this distinction plays an important role.
		\end{rem}
	\paragraph{Evaluation map:}
		For any $x\in M$ and $t\in \R_+^\nu\backslash 0$, we define
		\begin{equation*}\begin{aligned}
				\natural_{x,t}:V\to T_xM,\quad \natural_{x,t}(v):=\natural(\alpha_{t}(v))(x).
			\end{aligned}\end{equation*}
		The map $\natural_{x,0}$ which is more subtle to define is given in the next section.
		By \eqref{eqn:finit_N_condition}, the map $\natural_{x,t}$ is surjective for any $x\in M$ and $t\in \R_+^\nu\backslash 0$.
		Hence, $\ker(\natural_{x,t})$ defines an element of $\Grass(V)$, the Grassmannian manifold of linear subspaces of $V$ of codimension equal to $\dim(M)$.
		Since the map $\natural_{x,t}$ depends smoothly on $x$ and $t$, the following inclusion is smooth:
		\begin{equation}\label{eqn:incl_no_zeros}\begin{aligned}
				M\times \R_+^\nu\setminus \{0\}\to \Grass(V)\times M\times \R_+^\nu,\quad (x,t)\mapsto (\ker(\natural_{x,t}),x,t).
			\end{aligned}\end{equation}
		We can't extend this inclusion to $t=0$ because
			$\lim_{t\to 0}\ker(\natural_{x,t})$ exists for each fixed $x$ but in general it is not uniform in $x$.
		We will resolve this issue by a blowup procedure in the next section. 
		Essentially, we will take the closure of the image of \eqref{eqn:incl_no_zeros} and call it $\mathbb{G}^{(0)}$, our blow-up space.
\linksec{Tangent bundle in multi-graded sub-Riemannian geometry}{cotangent_cone}
	\paragraph{Summary:} In this section, we define the tangent bundle  $A(\cGF)$ in $\nu$-graded sub-Riemannian geometry which plays a fundamental role in the study of multi-parameter maximally hypoelliptic differential operators. For example:
		\begin{enumerate}
			\item   The dual of $A(\cGF)$ is naturally connected to the domain of the principal symbol of multi-parameter maximally hypoelliptic differential operators, see \eqref{eqn:Hellfer_Nourrigat_set_condition_union_cones}.
			\item   The Pontryagin characteristic classes of $A(\cGF)$ enter into our index formula in a way similar to that of Atiyah-Singer index formula but more subtle due to the noncommutativity of the principal symbol. This is the content of Part II in this series of articles.
			\item   Trivializing the density bundle of $A(\cGF)$ leads to the leading term in the asymptotics of the heat kernel and Weyl Law of multi-parameter maximally hypoelliptic differential operators. This is the content of Part III in this series of articles.
		\end{enumerate}
		The peculiar thing about $A(\cGF)$ is that its base space is not $M$ but a topological blowup of $M$ which we denote by $\cGF^{(0)}$.
		We also construct a real vector bundle $A(\mathbb{G})\to \mathbb{G}^{(0)}$ which creates a deformation between the bundles $TM\to M$ and $A(\cGF)\to \cGF^{(0)}$.
		As the notation suggests, the bundles $A(\cGF)$ and $A(\mathbb{G})$ are the Lie algebroids of quasi-Lie groupoids $\cGF\rightrightarrows \cGF^{(0)}$ and $\mathbb{G}\rightrightarrows \mathbb{G}^{(0)}$ which we construct in \Cref{sec:bi-graded_tangent_groupoid}.
		This will endow the space of smooth sections of $A(\cG)$ and $A(\mathbb{G})$ with a Lie bracket as expected of a tangent bundle.
	\paragraph{Multi-graded osculating group:}
		It is useful to introduce the notation 
			\begin{equation*}\begin{aligned}
				\cF^{\prec k}:=\sum_{l\in \Z_+^\nu,l\prec k}\cF^{l},\quad \forall k\in \Z_+^\nu	.	
			\end{aligned}\end{equation*}
		Clearly, $\cF^{\prec k}\subseteq \cF^{k}$, and $\cF^{\prec k}$ is a finitely generated $C^\infty(M,\R)$-module.
		By convention, $\cF^{\prec 0}=0$.
		
		\begin{dfn}
		The \textit{$\nu$-graded osculating Lie algebra} of $\cF^{\bullet}$ at $x$ is the vector space
		\begin{equation}\label{eqn:Lie_algebra_osculating}\begin{aligned}
				\grF{x}:=\bigoplus_{k\in \Z_+^\nu}\frac{\cF^{k}}{\cF^{\prec k}+C^\infty_x(M,\R)\cF^{k}},\text{ where }C^\infty_x(M,\R)=\{f\in C^\infty(M,\R):f(x)=0\}.
			\end{aligned}\end{equation}
		\end{dfn}
		If $X\in \cF^{k}$, then $[X]_{k,x}$ denotes its class in $\cF^{k}/(\cF^{\prec k}+C^\infty_x(M,\R)\cF^{k})$.
		If $f\in C^\infty(M,\R)$, then $[fX]_{k,x}=f(x)[X]_{k,x}$.
		This identity together with the fact that the module $\cF^{k}$ is finitely generated implies that $\cF^{k}/(\cF^{\prec k}+C^\infty_x(M,\R)\cF^{k})$ is a finite dimensional real vector space.
		Since $\cF^{k}=\cX(M)$ if $k_i\geq N_i$  for some $i\in \bb{1,\nu}$, we deduce that $\grF{x}$ is a finite dimensional vector space.
		We define a Lie bracket on $\grF{x}$ by the formula
		\begin{equation}\label{eqn:Lie_bracket_gx}\begin{aligned}
				[[X]_{k,x},[Y]_{l,x}]:=
				[[X,Y]]_{k+l,x},\quad \forall X\in \cF^{k}, Y\in \cF^{l}.
			\end{aligned}\end{equation}
		The space $\grF{x}$ is a nilpotent Lie algebra, because the quotient $\cF^{k}/(\cF^{\prec k}+C^\infty_x(M,\R)\cF^{k})$ vanishes if $k_i> N_i$  for some $i\in \bb{1,\nu}$.
		Here, we make use of our assumption $\cF^{0}=0$.
		The space $\grF{x}$ becomes a Lie group with the product of $a,b\in \grF{x}$ equal to $\BCH(a,b)$.
		Since $\grF{x}$ is a $\nu$-graded vector space, we can define $\alpha_\lambda:\grF{x}\to \grF{x}$ 
		like in \eqref{eqn:dilations_basic_V}. By \eqref{eqn:Lie_bracket_gx} and \eqref{eqn:BCH}, $\alpha_\lambda$ is a Lie algebra and a Lie group homomorphism.
	\paragraph{Tangent cones:}
		Let $x\in M$, $(V,\natural)$ a $\nu$-graded basis.
		We define the linear map
		\begin{equation}\label{eqn:natural00}\begin{aligned}
				\natural_{x,0}:V\to \grF{x},\quad \natural_{x,0}\left(\sum_{k\in \Z_+^\nu}v_{k}\right)=\sum_{k\in \Z_+^\nu}[\natural(v_{k})]_{k,x},\quad v_{k}\in V^{k}.
			\end{aligned}\end{equation}
		By \MyCref{dfn:graded_lie_basis}[3], $\natural_{x,0}$ is surjective.
		It also commutes with the dilation action.
		\begin{linklem}{conv_kernel_mathbbG0}
			Let $L\in \Grass(V)$,    $(x_n,t_n)_{n\in\N}\subseteq M\times \R_+^\nu\setminus \{0\}$ be a sequence such that 
				\begin{equation*}\begin{aligned}
					(\ker(\natural_{x_n,t_n}),x_n,t_n)\to (L,x,0) \in \Grass(V)\times M\times \R_+^\nu.
				\end{aligned}\end{equation*}
			Then, $\ker(\natural_{x,0})\subseteq L$.
		\end{linklem}
		For a linear subspace $L$ of $V$,  $ \ker(\natural_{x,0})\subseteq L$ if and only if  for some subspace $L'\subseteq \grF{x}$, $L=\natural_{x,0}^{-1}(L')$. 
		In this case, $L'=\natural_{x,0}(L)$.
		We can now define tangent cones
		\begin{dfn}\label{dfn:tanget_cone}
			A linear subspace $L\subseteq \grF{x}$ of codimension $\dim(M)$ is a \textit{tangent cone} at $x$, if there exists a sequence $(x_n,t_n)_{n\in\N}\subseteq M\times \R_+^\nu\setminus \{0\}$ such that 
			$(\ker(\natural_{x_n,t_n}),x_n,t_n)\to (\natural_{x,0}^{-1}(L),x,0)$ in $\Grass(V)\times M\times \R_+^\nu$.
			We denote the set of tangent cones at $x$ by $ \cGF^{(0)}_x$.
		\end{dfn}
		\begin{rem}\label{rem:Positivity}
			In \Cref{dfn:tanget_cone}, it is rather important to impose that $t_n$ are non-negative. If we instead consider $t_n\in \R^\nu\backslash 0$, we usually get more tangent cones, see for example the condition on $\zeta$ after \eqref{eqn:qjksdflksqdkfksdqkfkqsdhfjkhsdqkfhqskdjhfjkqshdlkf}.
		\end{rem}

		Since $\Grass(V)$ is compact, for any sequence  $(x_n,t_n)_{n\in\N}\subseteq M\times \R_+^\nu\setminus \{0\}$ such that $x_n\to x$, $t_n\to 0$, there exists a subsequence such that
		$\ker(\natural_{x_n,t_n})$ converges to some subspace of $V$. \Cref{thm:conv_kernel_mathbbG0} implies that the limit of $\ker(\natural_{x_n,t_n})$ has to be of the form $\natural_{x,0}^{-1}(L)$ for some $L\subseteq \grF{x}$ linear subspace of codimension $\dim(M)$.
		So, for all $x\in M$, $\cGF^{(0)}_x$ is nonempty.
		The set of all tangent cones is denoted by
			\begin{equation*}\begin{aligned}
					\cGF^{(0)}:=\{(L,x):x\in M,L\in \cGF^{(0)}_x\}
				\end{aligned}\end{equation*}
			We also define the set
			\begin{equation*}\begin{aligned}
					\mathbb{G}^{(0)}:=M\times \R_+^\nu\setminus \{0\}\sqcup \cGF^{(0)}\times \{0\}.
				\end{aligned}\end{equation*}
		For any $\nu$-graded basis $(V,\natural)$, we have the following injective map
		\begin{equation}\label{eqn:inclusion_mathbbG0}\begin{aligned}
				\Grass(\natural):\mathbb{G}^{(0)} & \hookrightarrow \Grass(V)\times M\times \R_+^\nu                                        \\
				(x,t)                       & \mapsto (\ker(\natural_{x,t}),x,t), &&(x,t)\in M\times \R_+^\nu\setminus \{0\} \\
				(L,x,0)                     & \mapsto (\natural_{x,0}^{-1}(L),x,0), &&(L,x,0)\in \cGF^{(0)}\times \{0\}
			\end{aligned}\end{equation}
		whose image is closed by the definition of tangent cones.
		We equip $\mathbb{G}^{(0)}$ and $\cG^{(0)}$ with the subspace topology and sub-differential structure from $\Grass(V)\times M\times \R_+^\nu$.
		A priori, $\mathbb{G}^{(0)}$ depends on the choice of the $\nu$-graded basis $(V,\natural)$. This isn't the case as the next theorem shows.
		\begin{linkthm}{G0_locally_compact}
				\begin{enumerate}
					\item\label{thm:G0_locally_compact:indep}  For any $x\in M$, the set $\cG^0_x$ doesn't depend on the choice of the $\nu$-graded basis $(V,\natural)$.
					Furthermore, the topology and the differential structure on $\mathbb{G}^{(0)}$ and $\cG^{(0)}$ don't depend on the choice of the $\nu$-graded basis $(V,\natural)$.
					\item\label{thm:G0_locally_compact:cont} The natural projection
					\begin{equation}\label{eqn:proj_mathbbG0}\begin{aligned}
							\pi_{\mathbb{G}^{(0)}} :\mathbb{G}^{(0)}\to M\times \R_+^\nu,\quad \pi_{\mathbb{G}^{(0)}}(x,t)=(x,t),\quad \pi_{\mathbb{G}^{(0)}}(L,x,0)=(x,0)
						\end{aligned}\end{equation}
					is a proper smooth surjective map.
										\item 	The inclusion $M\times \R_+^\nu\setminus \{0\}\hookrightarrow \mathbb{G}^{(0)}$ is an open smooth embedding with dense image.
				\end{enumerate}
			\end{linkthm}
			Throughout the article, the notation $\pi_{\mathbb{G}^{(0)}}$  is reserved for the projection map \eqref{eqn:proj_mathbbG0}.
		We now describe the topology of $\mathbb{G}$ in terms of sequences which follows immediately from \Crefitem{thm:G0_locally_compact}{indep}.
		\begin{prop}\label{thm:convergence_mathbbG0}
			\begin{enumerate}
				\item\label{thm:convergence_mathbbG0:part1} A sequence $(x_n,t_n)_{n\in \N}\subseteq M\times \R_+^\nu\setminus \{0\}$ converges to $(L,x,0)$ in $\mathbb{G}^{(0)}$ if and only if
					for any (or for some) $\nu$-graded basis $(V,\natural)$, $x_n\to x$, $t_n\to 0$, $\ker(\natural_{x_n,t_n})\to \natural_{x,0}^{-1}(L)$.
				\item\label{thm:convergence_mathbbG0:part2} A sequence $(L_n,x_n)\subseteq \cGF^{(0)}$ converges to $(L,x)$ in $\cGF^{(0)}$ if and only if
					for any (or for some) $\nu$-graded basis $(V,\natural)$, $x_n\to x$, $\natural_{x_n,0}^{-1}(L_n)\to \natural_{x,0}^{-1}(L)$.
			\end{enumerate}
		\end{prop}
			\begin{rem}\label{rem:connected_not_path_connected}
				In general $\mathbb{G}^{(0)}$ is not a smooth manifold.
				In fact $\mathbb{G}^{(0)}$ can exhibit pathological properties.
				If $M$ is connected, then $\mathbb{G}^{(0)}$ is connected because it is the closure of $M\times \R_+^\nu\setminus \{0\}$.
				But even though $M$ is path-connected, $\mathbb{G}^{(0)}$ generally isn't path-connected. 
			\end{rem}

	\paragraph{Tangent bundle:}
		Let $A(\mathbb{G})$ be the real vector bundle over $\mathbb{G}^{(0)}$ whose fiber over $(x,t)\in M\times \R_+^\nu\setminus \{0\}$ is equal to $T_xM$, and whose fiber over $(L,x,0)$ is equal to $\frac{\grF{x}}{L}$.
		Let $(V,\natural)$ be a $\nu$-graded basis. 
		There is a natural morphism of vector bundles $A(\natural):V\times \mathbb{G}^{(0)}\to A(\mathbb{G})$ defined as follows:
			\begin{equation}\label{eqn:Anatural}\begin{aligned}
				A(\natural)(v,x,t)=(\natural_{x,t}(v),x,t),\quad A(\natural)(v,L,x,0)=(\natural_{x,0}(v)\bmod L,x,0).
			\end{aligned}\end{equation}
		The map $A(\natural)$ is clearly surjective. 
		We equip $A(\mathbb{G})$ with the quotient topology.
		We also define $C^\infty(A(\mathbb{G}))$ to be the space of $f:A(\mathbb{G})\to \C$ such that $f\circ A(\natural)$ is smooth.
			\begin{linkthm}{ind_top_algberoid}
				\begin{enumerate}
					\item\label{thm:ind_top_algberoid:1}   The topological space $A(\mathbb{G})$ is metrizable locally compact second countable, and $C^\infty(A(\mathbb{G}))$ is a finite dimensional differential structure, 
					and $A(\natural)$ is a submersion.
					\item\label{thm:ind_top_algberoid:2} 	The differential space $A(\mathbb{G})$ is a real vector bundle over $\mathbb{G}^{(0)}$, see \Cref{dfn:vector_bundles_sub-Cartesian}.
					\item\label{thm:ind_top_algberoid:3} 	The topology and the differential structure on $A(\mathbb{G})$ don't depend on the choice of $(V,\natural)$.
				\end{enumerate}
			\end{linkthm}
				We now give some natural examples of smooth sections of the bundle $A(\mathbb{G})$.
		\begin{linkprop}{smooth_sections_AmathbbG}
			For every $k\in \Z_+^\nu$, $X\in \cF^{k}$, the section $\theta_{k}(X)$ of $A(\mathbb{G})$ defined by
			\begin{equation*}\begin{aligned}
					\theta_{k}(X)(x,t)=t^kX(x)\in T_xM,\quad   \theta_{k}(X)(L,x,0)=[X]_{k,x}\bmod L\in \frac{\grF{x}}{L} 
				\end{aligned}\end{equation*}
			is smooth.
			Also, the family $(\theta_{k}(X))_{k\in \Z_+^\nu,X\in \cF^{k}}$ generates $C^\infty(\mathbb{G}^{(0)},A(\mathbb{G}))$ as a $C^\infty(\mathbb{G}^{(0)},\R)$-module.
		\end{linkprop}
		We denote by  $A(\cGF)$ the restriction of $A(\mathbb{G})$ to $\cGF^{(0)}$.
		
		\paragraph{Tangent cones are Lie subalgebras:}
		Recall that $\grF{x}$ is a Lie algebra by \eqref{eqn:Lie_bracket_gx} and a Lie group by \eqref{eqn:BCH}.
		\begin{linkprop}{Lie_algebra}
			Let $x\in M$, $L\subseteq \grF{x}$ be a tangent cone.
			\begin{enumerate}
				\item\label{thm:Lie_algebra:Lie_subalgebra}    The linear subspace $L$ is a Lie subalgebra of $\grF{x}$ (and hence also a closed Lie subgroup of $\grF{x}$).
				\item\label{thm:Lie_algebra:adjoint_action}      If $g\in \grF{x}$, then
					$gLg^{-1}$
					is a tangent cone at $x$.
			\end{enumerate}
		\end{linkprop}
\linksec{Tangent groupoid}{bi-graded_tangent_groupoid}
	By \Cref{thm:Lie_algebra}, for each $x\in M$, following \Crefitem{ex:exs_groupoids}{quotient}, we can define the quotient groupoid of $\grF{x}$ by $\cGF^{(0)}_x$.
	We take the disjoint union of all such groupoids as $x$ varies, i.e.,
	let
	\begin{equation*}\begin{aligned}
			\cGF=\{(A,x):x\in M,A=gLh \text{ for some } g,h\in \grF{x},L\in \cGF^{(0)}_x\}
		\end{aligned}\end{equation*}
	We also define the space
	\begin{equation*}\begin{aligned}
			\mathbb{G}:=M\times M\times \R_+^\nu\setminus \{0\}\sqcup \cGF\times \{0\}.
		\end{aligned}\end{equation*}
	In this section, we equip $\mathbb{G}$ with a groupoid structure, a topology, and a differential structure making it a quasi-Lie groupoid.
	\paragraph{Groupoid structure:}
		The space of objects of $\mathbb{G}$ is $\mathbb{G}^{(0)}$.
		The structural morphisms are
		\begin{equation*}\begin{aligned}
				 & s:\mathbb{G}\to \mathbb{G}^{(0)}, &&s(y,x,t)=(x,t),&&&& s(A,x,0)=(A^{-1}A,x,0) \\
				 & r:\mathbb{G}\to \mathbb{G}^{(0)}, &&r(y,x,t)=(y,t),&&&& r(A,x,0)=(AA^{-1},x,0)\\
				&u:\mathbb{G}^{(0)}\to \mathbb{G}, &&u(x,t)=(x,x,t),&&&& u(L,x,0)=(L,x,0)\\
				&\iota:\mathbb{G}\to \mathbb{G}, &&\iota(y,x,t)=(x,y,t),&&&& \iota(A,x,0)=(A^{-1},x,0)\\
				&m:\mathbb{G}^{(2)}\to \mathbb{G}, &&m((z,y,t),(y,x,t))=(z,x,t),&&&& m((B,x,0),(A,x,0))=(BA,x,0)
			\end{aligned}\end{equation*}
		In other words, the groupoid $\mathbb{G}$ is the disjoint union of a copy of the pair groupoid $M\times M$ for each $t\in  \R_+^\nu\setminus \{0\}$, and
		a copy of the quotient groupoid of  $\grF{x}$ by $\cGF^{(0)}_x$
		for each $x\in M$.
	\paragraph{Topology and differential structure:}
		Let $(V,\natural)$ be a $\nu$-graded basis.
		The following map plays a central role in this article
		\begin{equation*}\begin{aligned}
				\cQ_V:\dom(\cQ_V)\subseteq V\times \mathbb{G}^{(0)}\to \mathbb{G}\\ (v,x,t)\mapsto(e^{\alpha_{t}(v)}\cdot x,x,t) \\ 				(v,L,x,0)\mapsto(\natural_{x,0}(v)L,x,0).
			\end{aligned}\end{equation*}
		Its domain is defined as follows:	
			Consider the map \begin{equation}\label{eqn:qsojod}\begin{aligned}
			&V\times M\times \R_+^\nu\backslash 0\to M\times M\times \R_+^\nu\backslash 0,\quad  (v,x,t)\mapsto(e^{\alpha_t(v)}\cdot x,x,t)\\
		\end{aligned}\end{equation}
		Let $U$ be the open subset of $V\times M\times \R_+^\nu\backslash 0$ such that \eqref{eqn:qsojod} is a well-defined submersion,	
						\begin{equation}\label{eqn:domprime_cQV}\begin{aligned}
						\tdom(\cQ_V)&= U\sqcup V\times M\times \{0\} &&\subseteq V\times M\times \R_+^\nu\\
						\dom(\cQ_V)&=U\sqcup V\times \cGF^{(0)}\times \{0\}&&\subseteq V\times \mathbb{G}^{(0)}.
					\end{aligned}\end{equation}
		The set $\tdom(\cQ_V)$ is open because for any $v\in V,x\in M$, when $t$ is close enough to $0$ but non-zero, $(v,x,t)\in U$ by \eqref{eqn:finit_N_condition}.
		The set $\dom(\cQ_V)$ is open because $\dom(\cQ_V)=\pi_{V\times \mathbb{G}^{(0)}}^{-1}(\tdom(\cQ_V))$,
		where \begin{equation}\label{eqn:proj_with_V_mathbbG}\begin{aligned}
			\pi_{V\times \mathbb{G}^{(0)}}:V\times \mathbb{G}^{(0)}\to V\times M\times \R_+^\nu,\quad \pi_{V\times \mathbb{G}^{(0)}}(v,\gamma)=(v,\pi_{\mathbb{G}^{(0)}}(\gamma))	
		\end{aligned}\end{equation}
		is the natural projection which is continuous by \Crefitem{thm:G0_locally_compact}{cont}.
		\begin{dfn}\label{dfn:topology_mathbbG}
			\begin{enumerate}
				\item   A subset $U\subseteq \mathbb{G}^{(0)}$ is open if and only if the following holds:
				\begin{itemize}
					\item   The subset $U\cap (M\times M\times \R_+^\nu\setminus \{0\})$ is open.
					\item   For any $\nu$-graded basis $(V,\natural)$, $\cQ_V^{-1}(U) $ is an open subset of $V\times \mathbb{G}^{(0)}$.
				\end{itemize}
				In other words, if $T$ is a topological space, then a map $f:\mathbb{G}\to T$ is continuous if and only if 
				$f_{|M\times M\times \R_+^\nu\setminus \{0\}}$ and $f\circ \cQ_V$ are continuous for every $\nu$-graded basis $(V,\natural)$.
				\item A function $f:\mathbb{G}\to \C$ is smooth if and only if the following conditions are satisfied: \begin{itemize}
					\item   The map $f_{|M\times M\times \R_+^\nu\setminus \{0\}}$ is smooth.
					\item   For any $\nu$-graded basis $(V,\natural)$, the map $f\circ \cQ_V:\dom(\cQ_V)\to \C$ is smooth, where $\dom(\cQ_V)\subseteq V\times \mathbb{G}^{(0)}$ inherits a differential structure from $V$ and $\mathbb{G}^{(0)}$.
				\end{itemize}		
			\end{enumerate}
		\end{dfn}
		\begin{linkthm}{topology_mathbbG}
			\begin{enumerate}
				\item\label{thm:topology_mathbbG:invariance}    In \Cref{dfn:topology_mathbbG}, one can replace "For any $\nu$-graded basis" by "For some $\nu$-graded basis".
				Neither the topology nor the space of smooth functions change.
				\item\label{thm:topology_mathbbG:topology} 		The topological space $\mathbb{G}$ is metrizable locally compact second countable.
				\item\label{thm:topology_mathbbG:subCartesian} 	The smooth structure $C^\infty(\mathbb{G})$ defines a finite dimensional differential structure on $\mathbb{G}$.
				\item\label{thm:topology_mathbbG:open_subset}	The inclusion $M\times M\times \R_+^\nu\setminus \{0\}\hookrightarrow \mathbb{G}$ is an open smooth embedding with dense image.
				\item\label{thm:topology_mathbbG:submersion}	For any $\nu$-graded basis $(V,\natural)$, the map $\cQ_{V}$ is a submersion.
			\end{enumerate}
		\end{linkthm}
		Let us describe the topology in terms of sequences.
		\begin{linkthm}{convergence_mathbbG}
			\begin{enumerate}
				\item\label{thm:convergence_mathbbG:sequence} A sequence $(y_n,x_n,t_n)_{n\in \N}\subseteq  M\times  M\times \R_+^\nu\setminus \{0\}$ converges to a point $(A,x,0)$ in $\mathbb{G}$ if and only if the following hold:
					\begin{enumerate}[label=\Roman*.]
						\item   The sequence $(x_n,t_n)_{n\in \N}\subseteq  M\times \R_+^\nu\setminus \{0\}$ converges to $(A^{-1}A,x,0)$ in $\mathbb{G}^{(0)}$.
						\item   For any (or for some) $\nu$-graded basis $(V,\natural)$, there exists\footnote{To be very precise in the forward implication $\implies$, the elements of the sequence $v_n$ only exist for $n$ big enough. To simplify the exposition, we will ignore this.} $(v_n)_{n\in \N}\subseteq  V$ and $v\in \natural_{x,0}^{-1}(A)$ such that $v_n\to v$ and $e^{\alpha_{t_n}(v_n)}\cdot x_n=y_n$.
					\end{enumerate}
					Furthermore, if $(y_n,x_n,t_n)$ converges to $(A,x,0)$, then for any $v\in  \natural_{x,0}^{-1}(A)$,  there exists $(v_n)_{n\in \N}\subseteq V$ such that $v_n\to v$ and $e^{\alpha_{t_n}(v_n)}\cdot x_n=y_n$.
				\item\label{thm:convergence_mathbbG:2} A sequence $(A_n,x_n)_{n\in \N}\subseteq \cGF$ converges to $(A,x)$ in $\cGF$ if and only if the following hold:
					\begin{enumerate}[label=\Roman*.]
						\item    The sequence $(A_n^{-1}A_n,x_n)_{n\in \N}\subseteq \cGF^{(0)}$ converges to $(A^{-1}A,x)$ in the topology of $\cGF^{(0)}$.
						\item    For any (or for some) $\nu$-graded basis $(V,\natural)$, there exists $(v_n)_{n\in \N}\subseteq  V$ and $v\in \natural_{x,0}^{-1}(A)$ such that $v_n\to v$ and $v_n\in \natural_{x_n,0}^{-1}(A_n)$.
					\end{enumerate}
					Furthermore, if $(A_n,x_n)$ converges to $(A,x)$, then for any $v\in  \natural_{x,0}^{-1}(A)$,  there exists $(v_n)_{n\in \N}\subseteq V$ such that $v_n\to v$ and  $v_n\in \natural_{x_n,0}^{-1}(A_n)$.
			\end{enumerate}
		\end{linkthm}
		It is rather interesting to compare \Cref{thm:convergence_mathbbG} to \Cref{prop:grass_prop}.
		In practice, it is often more useful to use the following theorem to verify the convergence in $\mathbb{G}$.
		\begin{linkthm}{convergence_mathbbG_practical}
			Let $x\in M$, $L\in \cG^0_x$ a tangent cone, $(V,\natural)$ be a $\nu$-graded basis, and $S\subseteq V$ a linear subspace such that the map 
				$(v,l)\in S\times L\mapsto \natural_{x,0}(v) l\in \mathfrak{g}_x$
			is a diffeomorphism. A sequence $(y_n,x_n,t_n)_{n\in \N}\subseteq  M\times  M\times \R_+^\nu\setminus \{0\}$ converges to a point $(\natural_{x,0}(v)L,x,0)$ in $\mathbb{G}$ where $v\in S$ if and only if the following hold:
					\begin{enumerate}[label=\Roman*.]
						\item   The sequence $(x_n,t_n)_{n\in \N}\subseteq  M\times \R_+^\nu\setminus \{0\}$ converges to $(L,x,0)$ in $\mathbb{G}^{(0)}$.
						\item   For $n$ big enough, 
						there exists $(v_n)_{n\in \N}\subseteq  S$ such that $v_n\to v$ and $e^{\alpha_{t_n}(v_n)}\cdot x_n=y_n$. 
					\end{enumerate}
		\end{linkthm}
		\begin{linkthm}{tangent_groupoid_is_smooth}
			The groupoid $\mathbb{G}\rightrightarrows \mathbb{G}^{(0)}$ is a quasi-Lie groupoid whose Lie algebroid is $A(\mathbb{G})$.
		\end{linkthm}
		By \Cref{thm:tangent_groupoid_is_smooth}, we mean that the differential structure and topology induced on $A(\mathbb{G})$ from the differential structure and topology on $\mathbb{G}$ coincide with the differential structure and topology described in \Cref{sec:cotangent_cone}.
	
		The following proposition gives a natural isomorphism of vector bundles which will be used repeatedly throughout this article.
		\begin{linkprop}{submersion_cQ}
			For any $k,l\in \C$, we have a natural smooth isomorphism 
				\begin{equation}\label{eqn:integration_denstiy_iso_exponential}\begin{aligned}
					\Omega^{k+l}(V)\simeq \cQ^*_V(\Omega^{k,l}(\mathbb{G}))\otimes \Omega^{k+l}(\ker(\odif{\cQ_V}))
				\end{aligned}\end{equation}
		\end{linkprop}
	\paragraph{Lie bracket:}
		Since $A(\mathbb{G})$ is the Lie algebroid of a quasi-Lie groupoid, it has a Lie bracket
		\begin{equation*}\begin{aligned}
				[\cdot,\cdot]:C^\infty(\mathbb{G}^{(0)},A(\mathbb{G}))\times  C^\infty(\mathbb{G}^{(0)},A(\mathbb{G}))\to C^\infty(\mathbb{G}^{(0)},A(\mathbb{G}))
			\end{aligned}\end{equation*}
		On $M\times M\times \R_+^\nu\setminus \{0\}$, the groupoid $\mathbb{G}$ is the pair groupoid. So, the Lie bracket when restricted to $M\times \R_+^\nu\setminus \{0\}$ is the usual Lie bracket on $TM$.
		By the density of $M\times \R_+^\nu\setminus \{0\}$ in $\mathbb{G}^{(0)}$, we obtain 
				\begin{equation}\label{eqn:Lie_Bracket_AG_theta}\begin{aligned}
					[\theta_{k}(X),\theta_{l}(Y)]=\theta_{k+l}([X,Y]),\quad \forall k,l\in \Z_+^\nu,X\in \cF^k,Y\in \cF^l.
				\end{aligned}\end{equation}
\section{Example of a tangent groupoid}\label{sec:tangent_groupoid_exs}
In this section, we consider an example of a $2$-graded sub-Riemannian structure. We will present the tangent cones and the tangent groupoid.
We leave the verification to the reader.
Let $N\in \N$ greater than or equal to $2$ be fixed.
Consider the $2$-graded structure of depth $(1,N)$ on $\R^2$
	\begin{equation*}\begin{aligned}
	 	\cF^{0,k}=\{f \partial_x+g x^{N-k}\partial_y:f,g\in C^\infty(\R^2,\R)\},\quad k\in \bb{1,N}.
	\end{aligned}\end{equation*}
This is the structure defined in \Crefitem{exs:exs_weighted}{Weighted_bigraded} associated to the families $\{\partial_x,\partial_y\}$ and $\{\partial_x,x^{N-1}\partial_y\}$.

\paragraph{Osculating Lie algebras:} We will describe the osculating Lie algebra by giving  a linear basis. 
In the following equations, span means linearly spanned, and the generators are linearly independent.
To ease the notation, we will use 
	\begin{equation}\label{eqn:kqsdkfjqposjfkpoqsfdf}\begin{aligned}
		&X_1=[\partial_x]_{(1,0),(x,y)},&&X_2=[\partial_y]_{(1,0),(x,y)},&&&Y=[\partial_x]_{(0,1),(x,y)},\\
		&Z_1=[x^{N-1}\partial_y]_{(0,1),(x,y)},&&Z_2=[x^{N-2}\partial_y]_{(0,2),(x,y)},\cdots,&&&Z_N=[\partial_y]_{(0,N),(x,y)}
	\end{aligned}\end{equation}
 which are elements of $\mathfrak{g}_{(x,y)}$. We have
	\begin{equation}\label{eqn:qksodfjopkqsdjkofjqskopdfjqspd}\begin{aligned}
		\mathfrak{g}_{(x,y)}=\begin{cases}
			\mathrm{span}\left(X_1,X_2, Y,Z_1\right),&\text{if } x\neq 0,\\
			\mathrm{span}\left(X_1,X_2,Y,Z_1,Z_2,\cdots,Z_N\right),&\text{if } x=0.
		\end{cases}		
	\end{aligned}\end{equation}
	The Lie algebra $\mathfrak{g}_{(x,y)}$ is commutative if $x\neq 0$.
	If $x=0$, then the only non-trivial Lie bracket are
		\begin{equation*}\begin{aligned}
			[Y,Z_k]=(N-k)Z_{k+1},\quad k\in \bb{1,N-1}.		
		\end{aligned}\end{equation*}
	\paragraph{Tangent cones at $x\neq 0$:} The tangent cones at $(x,y)$ with $x\neq 0$ are the linear subspaces:
		\begin{equation*}\begin{aligned}
				L_{\lambda}:=\mathrm{span}\{\lambda X_1-Y, \lambda x^{N-1}X_2-Z_1\},\quad 			L_{\infty}:=	 \mathrm{span}\{ X_1, X_2\} 
						\end{aligned}\end{equation*}
	where $\lambda\in \R_+$.
	A sequence $((x_n,y_n),(t_n,s_n))\in \R^2\times (\R^2\backslash 0)$ converges $(L_\lambda,(x,y),(0,0))$ in $\mathbb{G}^{(0)}$ if $x_n\to x$, $y_n\to y$, $t_n\to 0$, $s_n\to 0$ and $\frac{s_n}{t_n}\to \lambda\in [0,+\infty]$.

\paragraph{Tangent cones at $x=0$:}
	The tangent cones at $(0,y)$ are the linear subspaces:
		\begin{equation}\label{eqn:qjksdflksqdkfksdqkfkqsdhfjkhsdqkfhqskdjhfjkqshdlkf}\begin{aligned}
			&L_{\infty,\mu,\eta}=\mathrm{span}(X_1,X_2-\eta Z_N,Z_1-\mu Z_{2},\cdots,Z_{N-1}-\mu Z_{N}),&&L_{\infty}=\mathrm{span}(X_1,Z_1,\cdots,Z_N),\\
			&L_{\infty,\zeta}=\mathrm{span}(X_1,X_2-\zeta Z_1,Z_{2},\cdots,Z_{N}), 			&&L_{\lambda}:=\mathrm{span}(\lambda X_1-Y,Z_1,\cdots,Z_N),
		\end{aligned}\end{equation}
	where $\eta,\lambda\in \R_+$ and $\mu\in \R$, and $\zeta\in \R$ or $\zeta \in \R_+$ depending on the parity of $N$. A sequence $((x_n,y_n),(t_n,s_n))\in \R^2\times (\R^2\backslash 0)$ converges $(L,(0,y),(0,0))$ in $\mathbb{G}^{(0)}$ if $x_n\to 0$, $y_n\to y$, $t_n\to 0$, $s_n\to 0$, and 
	\begin{itemize}
		\item    	 $\frac{s_n}{t_n}\to \lambda$, if $L=L_\lambda$. 
		\item $\frac{t_n}{s_n^N}\to \eta$ and $\frac{x_n}{s_n}\to \mu$, if $L=L_{\infty,\mu,\eta}$. 
		\item $\frac{t_n}{s_nx_n^{N-1}}\to \zeta$ and $\frac{x_n}{s_n}\to \infty$, if $L=L_{\infty,\zeta}$. 
		\item $\frac{s_n}{t_n}\to \infty$ and $\frac{s_n^{N-k}x_n^k}{t_n}\to 0$ for all $k\in \bb{0,N-1}$, if $L=L_{\infty}$.
	\end{itemize}
	In the above conditions, we use the convention that $\frac{1}{0}=\infty$ and $\frac{0}{0}=0$.
\paragraph{Topology on $\mathbb{G}$:}
The following computation is a direct corollary of \Cref{thm:convergence_mathbbG_practical}.
A sequence $((x_n',y_n'),(x_n,y_n),(t_n,s_n))\in \R^2\times\R^2 \times(\R^2\backslash 0)$ converges to $(A,(x,y),(0,0))$ in $\mathbb{G}$ with $x\neq 0$ if and only if $(x_n,y_n)\to (x,y),(x_n',y_n')\to (x,y),(t_n,s_n)\to 0$, and the following conditions are satisfied:
\begin{itemize}
	\item    $\frac{x_n'-x_n}{t_n}\to a$ and $\frac{y_n'-y_n}{t_n}\to b$ and $\frac{s_n}{t_n}\to \lambda\in [0,+\infty[$, if $A=(aX_1+bX_2)L_{\lambda}$.
	\item $\frac{x_n'-x_n}{s_n}\to a$ and $\frac{y_n'-y_n}{s_n}\to bx^{N-1}$ and $\frac{s_n}{t_n}\to \infty$, if $A=(aY+bZ_1)L_{\infty}$.
\end{itemize}
A sequence $((x_n',y_n'),(x_n,y_n),(t_n,s_n))\in \R^2\times\R^2 \times(\R^2\backslash 0)$ converges to $(A,(0,y),(0,0))$ in $\mathbb{G}$ if and only if $(x_n,y_n)\to (0,y),(x_n',y_n')\to (0,y),(t_n,s_n)\to 0$, and the following conditions are satisfied:	
\begin{itemize}
	\item    $\frac{x_n'-x_n}{t_n}\to a$, $\frac{y_n'-y_n}{t_n}\to b$ and $\frac{s_n}{t_n}\to \lambda\in [0,+\infty[$, if $A=(aX_1+bX_2)L_{\lambda}$.
	\item  $\frac{x_n'-x_n}{s_n}\to a$, $\frac{y_n'-y_n}{s_n^N}\to b$, $\frac{t_n}{s_n^N}\to \eta$ and $\frac{x_n}{s_n}\to \mu$, if $A=(aY+bZ_N)L_{\infty,\mu,\eta}$.
	\item   $\frac{x_n'-x_n}{s_n}\to a$, $\frac{y_n'-y_n}{s_nx_n^{N-1}}\to b$, $\frac{t_n}{s_nx_n^{N-1}}\to \zeta$ and $\frac{x_n}{s_n}\to \infty$, if $A=(aY+bZ_1)L_{\infty,\zeta}$.
	\item 	  $\frac{x_n'-x_n}{s_n}\to a$, $\frac{y_n'-y_n}{t_n}\to b$, $\frac{s_n}{t_n}\to \infty$, $\frac{s_n^{N-k}x_n^k}{t_n}\to 0$ for all $k\in \bb{0,N-1}$, if $A=(aY+bX_2)L_{\infty}$.
\end{itemize}
The above computation illustrates how complicated the groupoid $\mathbb{G}\rightrightarrows \mathbb{G}^{(0)}$ can be even in very simple examples.

\linksec{Debord-Skandalis actions}{equivariance_of_norms}
	In \eqref{eqn:dilations_basic_V}, we defined $\R_+^\nu$-dilations on $V$ where $(V,\natural)$ is a $\nu$-graded basis.
	In this section, we define some dilations that will be used throughout this article.
	Let
			\begin{equation}\label{eqn:lambda_mu_product}\begin{aligned}
				&\lambda\mu:=(\lambda_1\mu_1,\cdots,\lambda_\nu\mu_\nu)\in \R_+^\nu, &&\forall \lambda,\mu\in \R_+^\nu\\
					&\lambda^{-1}:=(\lambda_1^{-1},\cdots,\lambda_{\nu}^{-1})\in (\Rpt)^\nu,&& \forall \lambda\in (\Rpt)^\nu.		
				\end{aligned}\end{equation}
			This notation will be used throughout the article. 
		\begin{linkprop}{quasi-Lie-Debord Skandalis}
			Let $\lambda\in (\Rpt)^\nu$. 
		The map 
			\begin{equation*}\begin{aligned}
					\alpha_\lambda:\mathbb{G}\to \mathbb{G},	\quad\alpha_\lambda(y,x,t)=(y,x,\lambda^{-1}t),  \quad \alpha_\lambda(A,x,0)=(\alpha_\lambda(A),x,0)	
			\end{aligned}\end{equation*}
			is a quasi-Lie groupoid automorphism of $\mathbb{G}$.
		\end{linkprop}
	  By \Cref{sec:automorphisms}, $(\Rpt)^\nu$ acts on
					$\cbbG $, $C^{-\infty}_{r,s}(\mathbb{G},\omegahalf)$, $C^*\mathbb{G}$			by $*$-algebra automorphisms.
	\paragraph{Equivariance of the $C^*$-norm:}
		Since $\alpha_\lambda:\mathbb{G}\to \mathbb{G}$ is an automorphism of quasi-Lie groupoids, by \eqref{eqn:automorphis_preserves_Cnorm}, 
			\begin{equation}\label{eqn:equiv_norm_Cstar_Debord_Skand}\begin{aligned}
				\norm{\alpha_\lambda(f)}_{C^*\mathbb{G}}=\norm{f}_{C^*\mathbb{G}}	,\quad \forall f\in C^*\mathbb{G},\ \lambda\in (\Rpt)^\nu.	
			\end{aligned}\end{equation}
	\paragraph{Equivariance of the $L^1$-norm:}
		Fix a Riemannian metric on $M$. For each $x\in M$, let $\vol_{x}\in \Omega^1(T_xM)$ be the associated $1$-density.
		Let $\omega$ be the section on $\mathbb{G}$ of $\Omega^{\frac{1}{2},-\frac{1}{2}}(\mathbb{G})$
		defined by
			\begin{equation*}\begin{aligned}
				\omega(y,x,t)=\vol_{y}^{\frac{1}{2}}\otimes \vol_{x}^{-\frac{1}{2}},\quad \forall (y,x,t)\in M\times M\times \R_+^\nu\setminus \{0\},
			\end{aligned}\end{equation*}
		and	whose value at $(A,x)\in\cGF\times \{0\}$ is defined as follows:
		For some $g\in \grF{x}$ and $L\in \cGF^{(0)}_x$, $A=gL$.
		Recall that $A(\mathbb{G})_{s(A,x,0)}=\frac{\grF{x}}{L}$ and $A(\mathbb{G})_{r(A,x,0)}=\frac{\grF{x}}{gLg^{-1}}$.
		The linear map 
			\begin{equation*}\begin{aligned}
				\frac{\grF{x}}{L}\to \frac{\grF{x}}{gLg^{-1}},\quad v\mapsto gvg^{-1}		
			\end{aligned}\end{equation*}
		is an isomorphism. Thus, it induces a section of $\omegahalf\left(\frac{\grF{x}}{gLg^{-1}}\right)\otimes\Omega^{-\frac{1}{2}}\left(\frac{\grF{x}}{L}\right)=\Omega^{\frac{1}{2},-\frac{1}{2}}(\mathbb{G})_{(A,x,0)} $
		which is $\omega(A,x,0)$.
		Notice that $\omega(A,x,0)$ doesn't depend on the choice of $g$ because for any $l\in L$, the map 
		$\grF{x}/L\to \grF{x}/L$ given by $v\mapsto lvl^{-1}$
		has determinant $1$ because $\grF{x}$ is a nilpotent Lie algebra.
		\begin{linkprop}{L1normRiemMetric}
			The section $\omega$ is smooth, i.e., $\omega\in C^\infty(\mathbb{G},\Omega^{\frac{1}{2},-\frac{1}{2}})$ and satisfies \eqref{eqn:omega_symmetry_condition}.
		\end{linkprop}

		Since $ M\times \R_+^\nu\setminus \{0\}$ is dense in $\mathbb{G}^{(0)}$, it follows that $\norm{f}_{L^1(\mathbb{G},\omega)}$ is given by 
			\begin{equation}\label{eqn:L1normMathbbG}\begin{aligned}
					\norm{f}_{L^1(\mathbb{G},\omega)}= \sup\left\{\max\left(\int_M |f|(y,x,t)\vol_{y}^{\frac{1}{2}}\vol_{x}^{-\frac{1}{2}},\int_M |f|(x,y,t)\vol_{y}^{\frac{1}{2}}\vol_{x}^{-\frac{1}{2}}\right)\right\}	,
			\end{aligned}\end{equation}
		where the supremum is over $(x,t)\in  M\times \R_+^\nu\setminus \{0\}$. 
		By \eqref{eqn:L1normMathbbG}, we have 
		\begin{equation}\label{eqn:L1norm_invariance}\begin{aligned}
			\norm{\alpha_\lambda(f)}_{L^1(\mathbb{G},\omega)}=\norm{f}_{L^1(\mathbb{G},\omega)},\quad \forall \lambda\in (\Rpt)^\nu, \ f\in \cbbG.		
		\end{aligned}\end{equation}
		
		\begin{rem}
			An amusing consequence of \eqref{eqn:L1normMathbbG} is that if $f\in \cbbG$, then the RHS of \eqref{eqn:L1normMathbbG} is finite.
			This is not obvious at all at first sight. We warn the reader though that if $f\in C^\infty_c(\mathbb{G})$ is not density-valued, then 
				$\sup\left\{\int_M |f|(y,x,t)\vol_{y}:(x,t)\in  M\times \R_+^\nu\setminus \{0\}\right\}$
				is usually infinite. Finiteness of the RHS of \eqref{eqn:L1normMathbbG} is an interplay between the topologies on $\mathbb{G}$ and on $A(\mathbb{G})$.
		\end{rem}

		\linksec{Invariant smooth functions}{smooth_functions_vanishing_Fourier_Transform}
		\paragraph{Summary:}
			In \Cref{chap:chapter_bi-garded_pseudo_diff}, we will define $\nu$-graded pseudo-differential operators as integrals. 
			To study such integrals, it is necessary that if $g\in \cbbG$, then the maps
				\begin{equation}\label{eqn:qsdkfjmsdjfkqmsdjlfjqskmdljf}\begin{aligned}
					(\Rpt)^\nu\to \cbbG,\quad \lambda\mapsto g\ast\alpha_{\lambda}(f)	\quad\text{and }\quad 	\lambda\mapsto \alpha_{\lambda}(f)\ast g
				\end{aligned}\end{equation}
			extend to $\R_+^\nu$. The extension in general doesn't exist.
			In this section, we define a subalgebra of the convolution algebra $\cbbG$ which we call the space of invariant functions.
			The main property of interest is that for such functions \eqref{eqn:qsdkfjmsdjfkqmsdjlfjqskmdljf} extend smoothly to $\R_+^\nu$.
			
		\paragraph{Invariant functions:}
		Let $(V,\natural)$ be a $\nu$-graded basis. 
			By \eqref{eqn:integration_denstiy_iso_exponential} with $k=l=\frac{1}{2}$ and \eqref{eqn:integration_along_fibers}, we have an integration along the fibers map
			\begin{equation*}\begin{aligned}
					\cQ_{V*}:C^\infty_c(\dom(\cQ_V),\Omega^1(V))\to \cbbG.
				\end{aligned}\end{equation*}
			Since a submersion is open, by \Crefitem{thm:topology_mathbbG}{submersion}, $\cQ_V$ is open.
			So, $\im(\cQ_V)$ is an open neighborhood of $\cGF\times \{0\}$. Since $\mathbb{G}=\im(\cQ_V)\cup (M\times M\times \R_+^\nu\setminus \{0\})$, by \Cref{prop:parition_of_unity}, for any $f\in \cbbG$,  there exists $g\in C^\infty_c(\dom(\cQ_V),\Omega^1(V))$
				such that 
				$f-\cQ_{V*}(g)\in C^\infty_c(M\times M\times \R_+^\nu\setminus \{0\},\omegahalf)$.
				
			\begin{dfn}\label{dfn:space_of_invariant_functions}
				The space of invariant\footnote{Invariant functions where essentially defined by Androulidakis and Skandalis in \cite{AS1}. We remark that the map $\tilde{f}\mapsto \cQ_{V*}(\tilde{f}\circ \pi_{V\times \mathbb{G}^{(0)}})$ is essentially an $X$-ray transform.
				Such transforms in the case of Lie groups have been extensively studied, see for example \cite{HelgasonBookGeometricAnalysis}} functions $\cinv$ is the space of $f\in \cbbG$ 
				 such that 
					\begin{equation}\label{eqn:sum_smooth_decompo_invariatn_function}\begin{aligned}
						f-\cQ_{V*}(\tilde{f}\circ \pi_{V\times\mathbb{G}^{(0)}})\in C^\infty_c(M\times M\times \R_+^\nu\setminus \{0\},\omegahalf), \text{ for some }\tilde{f}\in C^\infty_c(\tdom(\cQ_V),\Omega^1(V)).
					\end{aligned}\end{equation}
				We equip $\cinv$ with the quotient topology from the surjective map 
				\begin{equation}\label{eqn:quotient_map_topology_LF_Space}\begin{aligned}
					q_V:C^\infty_c(\tdom(\cQ_V),\Omega^1(V))\times C^\infty_c(M\times M\times \R_+^\nu\setminus \{0\},\Omega^\frac{1}{2})\to \cinv,\\ 
					 (\tilde{f},\dbtilde{f})\mapsto \cQ_{V*}(\tilde{f}\circ \pi_{V\times \mathbb{G}^{(0)}})+\dbtilde{f}.
				\end{aligned}\end{equation}
			\end{dfn}
			\begin{linkthm}{algebra_structure_invariant_functions}
				\begin{enumerate}
					\item\label{thm:algebra_structure_invariant_functions:invariance}    The space $\cinv$ and its topology don't depend on the choice of a $\nu$-graded basis $(V,\natural)$.
					\item\label{thm:algebra_structure_invariant_functions:space} 		The space $\cinv$ is an LF-space, and the inclusion $\cinv\hookrightarrow \cbbG$ is continuous.
					\item\label{thm:algebra_structure_invariant_functions:algebra}   The space $\cinv$ is a $*$-subalgebra of $\cbbG$. 
					\item\label{thm:algebra_structure_invariant_functions:dilation} If $\lambda\in (\Rpt)^\nu$ and $f\in \cinv$, then $\alpha_{\lambda}(f)\in  \cinv$.
				\end{enumerate}
			\end{linkthm}
			\begin{dfn}
				We denote by $C^{-\infty}_{r,s,\mathrm{inv}}(\mathbb{G},\omegahalf)$ the space of all $u\in C^{-\infty}_{r,s}(\mathbb{G},\omegahalf)$ such that 
			\begin{equation}\label{eqn:cinvdis_dfn}\begin{aligned}
				u\ast \cinv\subseteq \cinv,\quad \cinv\ast u\subseteq \cinv.
			\end{aligned}\end{equation}
			\end{dfn}
			The space $\cinvdis$ is a $*$-subalgebra of $C^{-\infty}_{r,s,\mathrm{inv}}(\mathbb{G},\omegahalf)$. Furthermore, if $\lambda\in (\Rpt)^\nu$ and $u\in \cinvdis$, then $\alpha_\lambda(u)\in \cinvdis$.
			\begin{rem}\label{rem:continuity_dis}
				If $u\in\cinvdis$, then $u\ast \cdot:\cinv\to\cinv$ is continuous by the closed graph theorem and \Cref{prop:automatic_continuity_distributions_quasi_Lie_groupoid}, and \Crefitem{thm:algebra_structure_invariant_functions}{space}.
			\end{rem}

			\begin{linkthm}{compatability_cinvdiv_cinv}
				\begin{enumerate}
					\item\label{thm:compatability_cinvdiv_cinv:inclusion_smooth_cinvdiv} If $u\in C^{\infty}_{}(M\times \R_+^\nu)$, then $\delta_{u\circ \pi_{\mathbb{G}^{(0)}}}\in \cinvdis$.
					\item\label{thm:compatability_cinvdiv_cinv:inclusion_vectorfield_cinvdiv} If $k\in \Z_+^\nu$, $X\in \cF^{k}$, then $\theta_{k}(X)\in \cinvdis$.
				\end{enumerate}
			\end{linkthm}
			In \Cref{thm:compatability_cinvdiv_cinv}, we are using the notation of \Cref{ex:distributions_quasi_Lie_groupods}.
			\begin{linkthm}{equivariance_Cn_norm}
				For any $f\in \cinv$, $g\in \cbbG$, the functions 
					\begin{equation*}\begin{aligned}
			(\Rpt)^\nu\to \cbbG,\quad \lambda\mapsto \alpha_{\lambda}(f)\ast g\quad\text{and}\quad \lambda\mapsto g\ast \alpha_{\lambda}(f)
					\end{aligned}\end{equation*}
			extend to smooth functions $\R_+^\nu\to \cbbG$ whose values at $\lambda\in \R_+^\nu$ are denoted by 
				 $\alpha_{\lambda}(f)\ast g$ and $g\ast \alpha_{\lambda}(f)$.
			Furthermore, if $g\in \cinv$, then the extensions are smooth functions $\R_+^\nu\to \cinv$.
			So, $\alpha_{\lambda}(f)$ defines an element of $\cinvdis$ for any $\lambda\in \R_+^\nu$.
			Finally,
			\begin{equation}\label{eqn:value_of_alpha_00}\begin{aligned}
				\alpha_{0}(f)= 		\delta_{u\circ\pi_{\mathbb{G}^{(0)}}},\quad u(x,t)=\int_V \tilde{f}(v,x,0)\in C^\infty(M\times \R_+^\nu),
			\end{aligned}\end{equation}
			where $\tilde{f}$ is any function as in \eqref{eqn:sum_smooth_decompo_invariatn_function}.
			\end{linkthm}
\linksec{Multi-grading on the algebra of differential operators}{bi-grading_algebra_diff}
		\paragraph{Notation and convention:}
			The space $C^\infty(M,\omegahalf)$ denotes $C^\infty(M,\Omega^\frac{1}{2}(TM))$, and $C^{-\infty}(M,\omegahalf)$ denotes the topological dual of $C^\infty_c(M,\omegahalf)$.
			This convention will be used throughout the article.
		\paragraph{Grading on differential operators:}
		Let $\DO(M)$ be the algebra of finite order differential operators $D:C^\infty(M,\omegahalf)\to C^\infty(M,\omegahalf)$.
		By finite order, we mean that there exists $n\in \N$ such that for any $f\in C^\infty(M,\Omega^{\frac{1}{2}})$ and $x\in M$, $D(f)(x)$
		only depends on the partial derivatives of $f$ at $x$ of order $\leq n$.
		If $f\in C^\infty(M)$, then the differential operator $g\in C^\infty(M,\omegahalf)\mapsto fg\in C^\infty(M,\omegahalf)$ will be denoted by $\delta_{f}$ to be consistent with the notation introduced in \Crefitem{ex:distributions_quasi_Lie_groupods}{dirac}.
		For any $k\in \Z_+^\nu$, we denote by $\DO^{k}(M)$ the $\C$-linear span of monomials $X_1\cdots X_n$ such that $X_i\in \cF^{a(i)}$ with $a(1)+\cdots+a(n)\preceq k$.
		A monomial of length $0$ is by convention $\delta_f$ where $f\in C^\infty(M)$, i.e., $ \DO^{0}(M)=C^\infty(M)$.
		By our assumption of finite order and \eqref{eqn:finit_N_condition},
			\begin{equation*}\begin{aligned}
				\DO(M)=\bigcup_{k\in \Z_+^\nu}\DO^{k}(M).
			\end{aligned}\end{equation*}
		Let
			\begin{equation*}\begin{aligned}
				 &\DO_c(M)=C_c^\infty(M)\DO(M),\quad  &&\DO^{\prec k}(M):=\sum_{l\prec k}\DO^{l}(M)\\
				&\DO^{k}_c(M)=C^\infty_c(M)\DO^{k}(M),\quad &&\DO^{\prec k}_c(M)=C^\infty_c(M)\DO^{\prec k}(M).
			\end{aligned}\end{equation*}
		The proof of the following proposition is straightforward and left to the reader.
		\begin{prop}\label{thm:properties_bi_grading_diff_op}
			\begin{enumerate}
				\item\label{thm:properties_bi_grading_diff_op:1} If $k\preceq l$, then $\DO^{k}(M)\subseteq \DO^{l}(M)$.
				\item\label{thm:properties_bi_grading_diff_op:2} For any $k,l\in \Z_+^\nu$, $\DO^{k}(M)\DO^{l}(M)\subseteq \DO^{k+l}(M)$.
				\item\label{thm:properties_bi_grading_diff_op:3} If $D\in \DO^{k}(M)$ and $f\in C^\infty(M)$, then $[D,\delta_f]\in \DO^{\prec k}(M)$. 
				\item\label{thm:properties_bi_grading_diff_op:4} If $D\in \DO^{k}(M)$, then its formal adjoint $D^*$ and formal transpose $D^t$ belong to $\DO^{k}(M)$.
				\item\label{thm:properties_bi_grading_diff_op:5} For any $k\in \Z_+^\nu$, $\DO^{k}(M)$ is a finitely generated $C^\infty(M)$-module.
			\end{enumerate}
		\end{prop}
		In \Crefitem{thm:properties_bi_grading_diff_op}{4}, the formal adjoint and formal transpose are canonically defined because we are working with half-densities.
		So, $X^t=X^*=-X$ for any vector field $X\in \cX(M)$.
	
	\paragraph{Differential operators as distributions on $\mathbb{G}$:}
		Let $k\in \Z_+^\nu$. 
		By \Crefitem{thm:compatability_cinvdiv_cinv}{inclusion_vectorfield_cinvdiv}, we have a linear map 
				$\theta_{k}:\cF^{k}\to \cinvdis$.
		We extend $\theta_k$ to a linear map 
			\begin{equation*}\begin{aligned}
				\theta_{k}:\DO^{k}(M)\to 	\cinvdis		
			\end{aligned}\end{equation*}
		For monomials of length $0$, i.e., $\delta_f$ where $f\in C^\infty(M)$, we set
		\begin{equation*}\begin{aligned}
			\theta_{k}(\delta_f):=\delta_{g\circ \pi_{\mathbb{G}^{(0)}}}\in \cinvdis, \quad g(x,t):=t^kf(x)\in C^\infty(M\times \R_+^\nu).
		\end{aligned}\end{equation*}
		Here, we use \Crefitem{thm:compatability_cinvdiv_cinv}{inclusion_smooth_cinvdiv}.
		For monomials of non-zero length, we use the formula
			\begin{equation*}\begin{aligned}
				\theta_{k}(X_1\cdots X_n)=\theta_{k-\sum_i a(i)}(\delta_1)\ast \theta_{a(1)}(X_1)\ast \cdots \ast \theta_{a(n)}(X_n),
			\end{aligned}\end{equation*}
		We define $\theta_k$ for general differential operators using linearity by writing $D$ as sum of monomials.
		\begin{linkthm}{inclusion_diff_op}
				For any $k\in \Z_+^\nu$, the map $\theta_{k}$ is well-defined, i.e., $\theta_{k}(D)$ doesn't depend on the presentation of $D $ as a sum of monomials. 
				Furthermore,
					\begin{equation}\label{eqn:theta_product}\begin{aligned}
						\theta_{k_1+k_2}(D_1D_2)=\theta_{k_1}(D_1)\ast \theta_{k_2}(D_2).
					\end{aligned}\end{equation}
		\end{linkthm}
		It is useful to remark that 
			\begin{equation}\label{eqn:homogenity_theta}\begin{aligned}
					\alpha_\lambda(\theta_{k}(D))=\lambda^{k}\theta_{k}(D),\quad \theta_k(D)^*=\theta_k(D^*),\quad \forall D\in \DO^{k}(M).		
			\end{aligned}\end{equation}
		\MyCref{thm:inclusion_diff_op} is far from trivial. When restricted to $\cG$, $\theta_k(D)$ is in some sense the principal symbol of $D$.
		One can define this restriction without talking at all about the groupoid $\mathbb{G}$. 
		Independence of this symbol on the presentation of $D$ as a sum of monomial is on the other hand not at all clear unless one has access to the groupoid $\mathbb{G}$ in which case the independence is trivial to prove!
\paragraph{Invariant functions whose Fourier transform vanish at the origin:}
		In the study of the principal symbol of pseudo-differential operators, one often needs to consider the space of compactly supported smooth functions whose Fourier transform vanishes at the origin to a certain order.
		We now define an analogue of such spaces for our differential space $\mathbb{G}$.
		\begin{dfn}\label{dfn:Schwartz_functions_vanish}
			For any $k\in \Z_+^\nu$, let 
			\begin{equation}\label{eqn:dfn_algebra_vanishing_Fourier_invariant}\begin{aligned}
				\cinvf{k}=\left\{\sum_{i=1}^n\theta_{k_i}(D_i)\ast f_i:k_i\in \Z_+^\nu,k\preceq k_i, D_i\in \DO^{k_i}(M),f_i\in \cinv\right\}
			\end{aligned}\end{equation}
		\end{dfn}
		Since $\theta_0(\delta_1)=\delta_{1}$, $\cinvf{0}=\cinv$.
		Notice that by \eqref{eqn:theta_product}, if $k,l\in \Z_+^\nu$, then 
			\begin{equation}\label{eqn:theta_product_cinfk}\begin{aligned}
				&\theta_k(\DO^{k}(M))\ast \cinvf{l}\subseteq \cinvf{k+l}, &&\\
				&\cinvf{l}\subseteq \cinvf{k}, &&\text{if } k\preceq l.
			\end{aligned}\end{equation}
		
	\begin{linkthm}{invariant_vanish_algebra}
		\begin{enumerate}
			\item\label{thm:invariant_vanish_algebra:inclusion} For any $k\in \Z_+^\nu$, $C^\infty_c(M\times M\times (\Rpt)^\nu,\omegahalf)\subseteq \cinvf{k}$.
			\item\label{thm:invariant_vanish_algebra:algebra}     For any $k,l\in \Z_+^\nu$, we have 
			\begin{equation*}\begin{aligned}
				\cinvf{k}\ast \cinvf{l}\subseteq \cinvf{k+l},\quad \cinvf{k}^*=\cinvf{k}.
			\end{aligned}\end{equation*}
			\item\label{thm:invariant_vanish_algebra:alternate} If $k\in \Z_+^\nu$, then 
			\begin{equation}\label{eqn:invariant_vanish_algebra:alternate}\begin{aligned}
				\cinvf{k}=\left\{\sum_{i=1}^n f_i\ast \theta_{k_i}(D_i):k_i\in \Z_+^\nu,k\preceq k_i, D_i\in \DO^{k_i}(M),f_i\in \cinv\right\}
			\end{aligned}\end{equation}	
		\end{enumerate}
		\end{linkthm}
		\begin{rem}\label{rem:singular_support}
			If $\nu\geq 2$, then functions in $\cinvf{k}$ satisfy some vanishing conditions away from the fiber at $t=0$.
			For example, if $k_1\neq 0$ and $f\in \cinvf{k}$, then $f(y,x,(0,t))=0$ for all  $t\in \R^{\nu-1}_{+}\backslash 0$ and $y,x\in M$.
			In particular, $C^\infty_c(M\times M\times \R^\nu\backslash 0,\omegahalf)\nsubseteq \cinvf{k}$.
		\end{rem}
		
			The following proposition will be used in \Cref{chap:chapter_bi-garded_pseudo_diff}.
	
	\begin{linkprop}{derivative_vanish_at_0}
			For any $k\in \Z_+^\nu$, $f\in \cinvf{k}$, $i\in \bb{1,\nu}$,
			\begin{equation*}\begin{aligned}
				\odv*{\alpha_{(1,\cdots,1,\lambda_i,1,\cdots,1)}(f)}{\lambda_{i}}_{\lambda_i=1}\in \cinvf{(k_1,\cdots,k_{i-1},\max(1,k_{i}),k_{i+1},\cdots,k_\nu)}.
			\end{aligned}\end{equation*}
	\end{linkprop}
\linksec{Weakly commutative sub-Riemannian structure}{weakly_comm_struct}
		
		The following condition was introduced by Street \cite{StreetBook2}.
		\begin{dfn}\label{dfn:weakly_commuting}
			A $\nu$-graded sub-Riemannian structure $\cF^{\bullet}$ is called weakly commutative if
			\begin{equation}\label{eqn:weakly_commuting}\begin{aligned}
					\cF^{k}=\cF^{(k_1,0,\cdots,0)}+\cdots+\cF^{(0,\cdots,0,k_{\nu})},\quad \forall k\in \Z_+^\nu
				\end{aligned}\end{equation}
		\end{dfn}
		We suppose for the rest of this section that $\cF$ is weakly commutative.
		If $V=\oplus_{k\in \Z_+^\nu}V^k$ is a $\nu$-graded vector space, then we denote by 
			\begin{equation*}\begin{aligned}
				 V^{\weak{i}}:=\bigoplus_{j\in \Z_+} V^{(\overbrace{0,\cdots,0}^{i-1},j,\overbrace{0\cdots,0}^{\nu-i})},\quad i\in \bb{1,\nu}.		
			\end{aligned}\end{equation*}
		By \eqref{eqn:weakly_commuting},
		\begin{equation}\label{eqn:weakly_commuting_lie_algebra}\begin{aligned}
				\grF{x}=\bigoplus_{i=1}^{\nu}\grF{x}^{\weak{i}},\quad \forall x\in M.
			\end{aligned}\end{equation}
		Furthermore, $[\grF{x}^{\weak{i}},\grF{x}^{\weak{j}}]=0$ for all $i,j\in \bb{1,\nu}$ with $i\neq j$.
		\begin{exs}\label{ex:commuting}
			\begin{enumerate}
				\item Any $1$-graded sub-Riemannian structure is weakly commutative. Similarly, any $2$-graded sub-Riemannian structure of depth $(1,N)$ is weakly commutative. 
				\item Let $X_1,\cdots,X_{d}$ and $Y_1,\cdots,Y_{d'}$ be two families of vector fields such that each family satisfies Hörmander's condition of depth $N\in \N$ and $N'\in \N$ respectively, $\cF^{\bullet}$
						the $2$-graded sub-Riemannian structure of depth $(N,N')$ generated by the two families as in \Crefitem{exs:exs_weighted}{Weighted_bigraded}. 
						Then, $\cF^{\bullet}$ is weakly commutative if and only if 
							\begin{equation}\label{eqn:weak_commut}\begin{aligned}
								[X_k,Y_l]=\sum_{i=1}^d f_i X_i+\sum_{j=1}^{d'} g_j Y_j,\quad \forall k\in \bb{1,d},l\in\bb{1,d'},		
							\end{aligned}\end{equation}
						for some $f_i,g_j\in C^\infty(M,\R)$.
						This is automatic for example if one of the two families spans the tangent space without need for commutators.
						Another example is $M=\R^2$ and the families are $\{\partial_x,x^k\partial_y\}$ and $\{\partial_y,y^{l}\partial_x\}$.
			\end{enumerate}
		\end{exs}
		The following theorem is false without the weak-commutativity condition.
		It is needed to ensure that our pseudo-differential calculus defined in \Cref{chap:chapter_bi-garded_pseudo_diff} is closed under composition.
		\begin{linkthm}{estimate_weakly_commut_convolution}
			Let $\cF^{\bullet}$ be a weakly commutative $\nu$-graded sub-Riemannian structure, $f,g\in \cinv$.
			For any $\lambda,\mu\in \R_+^\nu$, if $\lambda+\mu\in (\Rpt)^\nu$, then
					$\alpha_{\lambda}(f)\ast \alpha_{\mu}(g)\in \cinv$.
			 Furthermore, the following function is smooth:
				\begin{equation}\label{eqn:lambda_mu_map_azs}\begin{aligned}
					\{(\lambda,\mu)\in \R_+^{2\nu}:\lambda+\nu\in (\Rpt)^\nu\} \to \cinv,\quad (\lambda,\mu)\mapsto \alpha_{\lambda}(f)\ast \alpha_{\mu}(g).	
				\end{aligned}\end{equation}
		\end{linkthm}
		Weak-commutativity also allows one to build smooth functions on $\mathbb{G}$ from smaller groupoids as follows:
		Let $\tilde{\nu},\dbtilde{\nu}\in \N$ such that $\tilde{\nu}+\dbtilde{\nu}=\nu$.
		We define a $\tilde{\nu}$-graded $\tilde{\cF}^\bullet$ and a $\dbtilde{\nu}$-graded $\dbtilde{\cF}^\bullet$ sub-Riemannian structures on $M$ by the formulae
			\begin{equation}\label{eqn:extension_bigradedStructure}\begin{aligned}
				\tilde{\cF}^{k}:=\cF^{(k,0,\cdots,0)},\quad \dbtilde{\cF}^{l}:=\cF^{(0,\cdots,0,l)},\quad \forall k\in \Z_+^{\tilde{\nu}},l\in \Z_+^{\dbtilde{\nu}}.		
			\end{aligned}\end{equation}
		So, by \eqref{eqn:weakly_commuting}, $\cF$ and $\tilde{\cF}$ and $\dbtilde{\cF}$ are connected by the equation 
			\begin{equation}\label{eqn:cF_tilde_dbtilde}\begin{aligned}
				\cF^{(k,l)}=\tilde{\cF}^k+\dbtilde{\cF}^l,\quad \forall k\in \Z_+^{\tilde{\nu}},l\in \Z_+^{\dbtilde{\nu}}.
			\end{aligned}\end{equation}
		Let $\tilde{\mathbb{G}}\rightrightarrows \tilde{\mathbb{G}}^{(0)}$ and $\dbtilde{\mathbb{G}}\rightrightarrows \dbtilde{\mathbb{G}}^{(0)}$ be the associated quasi-Lie groupoids.
		\begin{linkthm}{extension_parameters}
			There exists a natural $\C$-bilinear map (thought of as a convolution product)
			\begin{equation*}\begin{aligned}
				\Phi:C^{\infty}_{c,\mathrm{inv}}(\tilde{\mathbb{G}},\omegahalf)\times C^{\infty}_{c,\mathrm{inv}}(\dbtilde{\mathbb{G}},\omegahalf) \to \cinv
			\end{aligned}\end{equation*}
			which is uniquely determined by the identity  
			\begin{equation}\label{eqn:Convolution_extension_parameters}\begin{aligned}
				\Phi(f,g)(z,x,t,s):=
					\int_{M}f(z,y,t)g(y,x,s),\quad  \forall(z,x,t,s)\in M\times M\times (\R_+^{\tilde{\nu}}\backslash 0 )\times (\R_+^{\dbtilde{\nu}}\backslash 0),
			\end{aligned}\end{equation}
			where $f\in C^{\infty}_{c,\mathrm{inv}}(\tilde{\mathbb{G}},\omegahalf)$ and $g\in C^\infty_{c,\mathrm{inv}}(\dbtilde{\mathbb{G}},\omegahalf)$.
			Furthermore, 
				\begin{equation}\label{eqn:homogenity_extension}\begin{aligned}
					\alpha_{(\lambda,\mu)}(\Phi(f,g))=\Phi(\alpha_{\lambda}(f),\alpha_{\mu}(g)),\quad \Phi(C^{\infty}_{c,\mathrm{inv},k}(\tilde{\mathbb{G}},\omegahalf),C^{\infty}_{c,\mathrm{inv},l}(\dbtilde{\mathbb{G}},\omegahalf))\subseteq \cinvf{(k,l)},	
				\end{aligned}\end{equation}
				where  $\lambda\in (\Rpt)^{\tilde{\nu}},\mu\in (\Rpt)^{\dbtilde{\nu}}$, $f\in C^{\infty}_{c,\mathrm{inv}}(\tilde{\mathbb{G}},\omegahalf),g\in C^\infty_{c,\mathrm{inv}}(\dbtilde{\mathbb{G}},\omegahalf)$, $k\in \Z_+^{\tilde{\nu}},l\in \Z_+^{\dbtilde{\nu}}$.
		\end{linkthm}
			By Symmetry, we can also define a map (still denoted by $\Phi$) 
				\begin{equation*}\begin{aligned}
				\Phi:C^{\infty}_{c,\mathrm{inv}}(\dbtilde{\mathbb{G}},\omegahalf)\times C^{\infty}_{c,\mathrm{inv}}(\tilde{\mathbb{G}},\omegahalf) \to \cinv
			\end{aligned}\end{equation*}
			which is uniquely determined by the identity  
			\begin{equation}\label{eqn:Convolution_extension_parameters_2}\begin{aligned}
				\Phi(g,f)(z,x,t,s):=
					\int_{M}g(z,y,s)f(y,x,t),\quad  \forall(z,x,t,s)\in M\times M\times (\R_+^{\tilde{\nu}}\backslash 0 )\times (\R_+^{\dbtilde{\nu}}\backslash 0),
			\end{aligned}\end{equation}
			where $f\in C^{\infty}_{c,\mathrm{inv}}(\tilde{\mathbb{G}},\omegahalf)$ and $g\in C^\infty_{c,\mathrm{inv}}(\dbtilde{\mathbb{G}},\omegahalf)$, and which satisfies
			similar identities to those in \eqref{eqn:homogenity_extension}.
			The next result says that $\Phi(f,g)$ and $\Phi(g,f)$ are equal up to lower order terms
		
		\begin{linkthm}{commutator_phi}
				If $f\in C^{\infty}_{c,\mathrm{inv}}(\tilde{\mathbb{G}},\omegahalf)$ and $g\in C^\infty_{c,\mathrm{inv}}(\dbtilde{\mathbb{G}},\omegahalf)$, then
				there exists invariant functions $h_1,\cdots,h_{\tilde{\nu}},h'_1,\cdots,h'_{\dbtilde{\nu}}\in \cinv$ such that 
					\begin{equation}\label{eqn:commutator_phi}\begin{aligned}
						\Phi(f,g)-\Phi(g,f)=\sum_{i=1}^{\tilde{\nu}} t_i h_i+\sum_{i=1}^{\dbtilde{\nu} }s_i h_i'
					\end{aligned}\end{equation}
				 
		\end{linkthm}

\chapter{Proofs of the results in Chapter \ref{chap:bi-graded_tangent_groupoid}}\label{chap:bi-graded_tangent_groupoid_proof}
		The following proposition will be used throughout this chapter. 
		Its proof is left to the reader.
		\begin{prop}\label{prop:grass_prop}
			Let $V$ be a vector space, $d\in \bb{0,\dim(V)}$, $\Grass_d(V)$ the Grassmannian manifold of linear subspaces of $V$ of dimension $d$, $(L_n)_{n\in \N}\subseteq \Grass_d(V)$ a sequence, $L\in \Grass_d(V)$.
			The following are equivalent:
			\begin{enumerate}
				\item     The sequence $(L_n)_{n\in \N}$ converges to $L$ in the Grassmannian topology.
				\item     For every sequence $(l_n)_{n\in \N}\subseteq V$ such that $l_n\in L_n$ for all $n\in \N$, if $l_n$ converges to a vector $v\in V$, then $v\in L$.
			\end{enumerate}
			Furthermore, if $L_n\to L$ in the Grassmannian topology, then for any $l\in L$, there exists a sequence $l_n\in L_n$ such that $l_n\to l$.
		\end{prop}
		If $(V,\natural)$ and $(V',\natural')$ are $\nu$-graded basis, then we write $(V',\natural')\subseteq (V,\natural)$ if $V'\subseteq V$ and $\natural_{|V'}=\natural'$.
		We call an element $v\in V$ pure if $v\in V^k$ for some $k\in \Z_+^\nu$.  
		
\section{Proofs of the results in Section \ref{sec:cotangent_cone}}\label{sec:cotangent_cone_proof}	
		\begin{linkproof}{conv_kernel_mathbbG0}
			The map $\natural_{x,0}$ preserves the $\nu$-grading on $V$ and $\grF{x}$. Hence, $\ker(\natural_{x,0})$ is a $\nu$-graded subspace of $V$.
			So, it is enough to show that if $v\in V^{k}\cap \ker(\natural_{x,0})$ for some $k\in \Z_+^\nu$, then $v\in L$.
			Since $\natural_{x,0}(v)=0$, $\natural(v)\in \cF^{\prec k}+C^\infty_x(M,\R)\cF^{k}$.
			By the definition of a $\nu$-graded basis, we can write $\natural(v)$ as a sum over some finite index sets $I$ and $J$
			\begin{equation*}\begin{aligned}
					\natural(v)=\sum_{i\in I} f_i\natural(w_i)+\sum_{j\in J}f_j'\natural(w_j')
				\end{aligned}\end{equation*}
			where $w_i\in V^{a(i)}$ for some $a(i)\in \Z_+^\nu$ with $a(i)\prec k$, and $w'_j\in V^{k}$, $f_i,f'_j\in C^\infty(M,\R)$ such that $f'_j(x)=0$.
			It follows that for any $n\in \N$,
			\begin{equation*}\begin{aligned}
					v_n:=v-\sum_{i\in I} f_i(x_n)t_n^{k-a(i)}w_i-\sum_{j\in J}f'_j(x_n)w'_j\in \ker(\natural_{x_n,t_n}).
				\end{aligned}\end{equation*}
			Since $v_n\to v$, by \Cref{prop:grass_prop}, we deduce that $v\in L$.
		\end{linkproof}
		\begin{linkproof}{G0_locally_compact}
			The map $\pi_{\mathbb{G}^{(0)}}$ is obviously continuous.
			It is surjective because $\cGF^{(0)}_x \neq \emptyset$ for every $x\in M$.
			It is proper because $\Grass(V)$ is compact.
			The inclusion $M\times \R_+^\nu\setminus \{0\}\hookrightarrow \mathbb{G}^{(0)}$ is easily seen to be an open smooth embedding.
			Its image is dense by the definition of tangent cones.
			It remains to show \Crefitem{thm:G0_locally_compact}{indep}.
			It is clear that $(V\oplus V',\natural\oplus \natural')$ is a $\nu$-graded basis.
			We will construct closed smooth embeddings
			\begin{equation*}\begin{aligned}
					\Phi:\Grass(V)\times M\times \R_+^\nu\to \Grass(V\oplus V')\times M\times \R_+^\nu, \\
					\Phi':\Grass(V')\times M\times \R_+^\nu\to \Grass(V\oplus V')\times M\times \R_+^\nu
				\end{aligned}\end{equation*}
			such that the diagram
			\begin{equation} \label{diag:smooth_structure_mathbbG0}\begin{tikzcd}
					& \Grass(V\oplus V')\times M\times \R_+^\nu  &  \\
					\Grass(V)\times M\times \R_+^\nu \arrow[ur,"\Phi"]& &\Grass(V')\times M\times \R_+^\nu \arrow[ul,"\Phi'"']   \\
					& \mathbb{G}^{(0)}\arrow[ur,"\Grass(\natural')"']\arrow[ul,"\Grass(\natural)"]\arrow[uu,"\Grass(\natural\oplus \natural')"]&
				\end{tikzcd}\end{equation}
			commutes. The existence of $\Phi$ and $\Phi'$ finishes the proof of \Crefitem{thm:G0_locally_compact}{indep}.

			By symmetry, it is enough to construct $\Phi'$.
			We claim that we can find a smooth map $T:V\times M\times \R_+^\nu\to V'$ which is linear in $V$ such that
			\begin{equation}\label{eqn:qskjfpiojqspdjfpqjsdpf}\begin{aligned}
					\natural'_{x,t}(T(v,x,t))=\natural_{x,t}(v),\quad \forall (v,x,t)\in V\times M\times \R_+^\nu.
				\end{aligned}\end{equation}
			Notice that in \eqref{eqn:qskjfpiojqspdjfpqjsdpf}, $t\in \R_+^\nu$. 
			So, \eqref{eqn:qskjfpiojqspdjfpqjsdpf} is an equality between elements of $T_xM$ if $t\neq 0$ and between elements of $\grF{x}$ if $t=0$.
			It is enough to define $T$ on a linear basis of $V$ and then extend $T$ by linearity.
			We fix a basis of $V$ of pure elements.
			If $v\in V^{k}$ is a basis element, then by the definition of a $\nu$-graded basis, there exists a finite family $(f_i)_{i\in I}\subseteq C^\infty(M,\R)$, $(v_i)_{i\in I}\subseteq V'$ such that $v_i\in V^{\prime a(i)}$ with $a(i)\preceq k$ and
					$\natural(v)=\sum_{i\in I} f_i \natural'(v_i)$.
			We define $T(v,x,t)$ by the formula
			\begin{equation}\label{eqn:map_T_change_of_basis}\begin{aligned}
					T(v,x,t)=        \sum_{i\in I} t^{k-a(i)}f_i(x) v_i,
				\end{aligned}\end{equation}
			We leave it to the reader to check that \eqref{eqn:qskjfpiojqspdjfpqjsdpf} holds.
			We extend $T$ to a map $T:(V\oplus V')\times M\times \R_+^\nu\to V'$ which is linear in $V\oplus V'$ by the formula
					$T(v+v',x,t)=T(v,x,t)+v' $.
			The map $T(\cdot,x,t)$ is surjective for any $x,t$.
			Hence, we can define the map
			\begin{equation*}\begin{aligned}
					\Phi'(L,x,t)=(T(\cdot,x,t)^{-1}(L),x,t)
				\end{aligned}\end{equation*}
			where $T(\cdot,x,t)^{-1}(L):=\{v+v'\in V\oplus V':T(v+v',x,t)\in L\}$.
			Surjectivity of $T(\cdot,x,t)$ ensures that $\Phi'$ is a well-defined closed embedding.
			Commutativity of \eqref{diag:smooth_structure_mathbbG0} follows from the identity
				$\natural_{x,t}'(T(v+v',x,t))=\natural_{x,t}(v)+\natural'_{x,t}(v')$.
		\end{linkproof}
		\begin{linkproof}{ind_top_algberoid}
			We embed the space $\mathbb{G}^{(0)}$ inside $\Grass(V)\times M\times \R_+^\nu$ using $\Grass(\natural)$.
			There is a natural real vector bundle $E\to \Grass(V)\times M\times \R_+^\nu$ whose fiber over $(L,x,t)$ is $V/L$.
			Furthermore, there is an obvious smooth submersion $V\times \Grass(V)\times M\times \R_+^\nu\to E$.
			The restriction of $E$ to $\mathbb{G}^{(0)}$ is $A(\mathbb{G})$ and the obvious projection map from $V\times \mathbb{G}^{(0)}\to A(\mathbb{G}^{(0)})$ gets identified with $A(\natural)$.
			From this we deduce \Crefitem{thm:ind_top_algberoid}{1} and \Crefitem{thm:ind_top_algberoid}{2}.
			\Crefitem{thm:ind_top_algberoid}{3} is proved by an obvious modification of the proof of \Crefitem{thm:G0_locally_compact}{indep}. 
			The details are left to the reader.
					\end{linkproof}
		\begin{linkproof}{smooth_sections_AmathbbG}
			Let $(V,\natural)$ be a $\nu$-graded basis. By the definition of a $\nu$-graded basis, $X=\sum_{i\in I}f_i\natural(v_i)$ where $f_i\in C^\infty(M,\R)$, $v_i\in V^{a(i)}$ with $a(i)\preceq k$.
			So, for any $\gamma\in \mathbb{G}^{(0)}$
				\begin{equation*}\begin{aligned}
					\theta_k(X)(\gamma)=\sum_{i\in I}\big(g_i\circ \pi_{\mathbb{G}^{(0)}}(\gamma)\big)\theta_{a(i)}(v_i)(\gamma)=\sum_{i\in I}\big(g_i\circ \pi_{\mathbb{G}^{(0)}}(\gamma)\big)A(\natural)(v_i,\gamma),
				\end{aligned}\end{equation*}
				where $g_i\in C^\infty(M\times \R_+^\nu)$ are defined by $g_i(x,t)=t^{k-a(i)}f_i(x)$.
			Hence, $\theta_k(X)$ is a smooth section of $A(\mathbb{G}^{(0)})$.
			It is clear from the description of $A(\mathbb{G})$ in the proof of \Cref{thm:ind_top_algberoid} that the family $(\theta_{k}(v))_{k\in \Z_+^\nu,v\in V^k}$ generates the $C^\infty(\mathbb{G}^{(0)},\R)$-module $C^\infty(\mathbb{G}^{(0)},A(\mathbb{G}))$.
		\end{linkproof}

		\paragraph{Multi-graded Lie basis:}
		Throughout this article, we sometimes need to lift the Lie bracket from the osculating Lie algebras $\grF{x}$ to $\nu$-graded basis. For this reason, we introduce the following notion:
		A $\nu$-graded Lie basis is a $\nu$-graded basis $(V,\natural)$ such that $V$ is equipped with a Lie bracket $[\cdot,\cdot]:V\times V\to V$ such that
			$[V^{k},V^{l}]\subseteq 
								V^{k+l}$,
		and 
		\begin{equation}\label{eqn:bi_graded_lie_basis_commutator}\begin{aligned}
			[\natural(v),\natural(w)]-\natural([v,w])\in \cF^{\prec k+l},\quad \forall v\in V^{k},w\in V^{l}.
		\end{aligned}\end{equation}
		Notice that \eqref{eqn:bi_graded_lie_basis_commutator} is trivially true if $k+l\npreceq N$ because by \eqref{eqn:finit_N_condition}, $\cF^{\prec k+l}=\cX(M)$.
		Since  $V$ is finite dimensional, $V^k\neq 0$ for only a finite number of $k\in \Z_+^\nu$.
		Together with $V^0=0$, we deduce that the Lie algebra $V$ is nilpotent
		\begin{lem}\label{lem:bigger_bi_graded_Lie}
			If $(V',\natural')$ is a $\nu$-graded basis, then there exists a $\nu$-graded Lie basis $(V,\natural)$ such that $(V',\natural')\subseteq (V,\natural)$.
		\end{lem}

		\begin{proof}
			Let $S$ be a linear basis of $V'$ consisting of pure elements.
			Let $W$ be the free Lie algebra on $S$ and $\natural_{W}:W\to \cX(M)$ the Lie algebra extension of the function $\natural_{|S}:S\to \cX(M)$.
			We also have a natural inclusion $V'\hookrightarrow  W$.
			The space $W$ inherits a $\nu$-grading from $S$ making it an infinite dimensional $\nu$-graded Lie algebra $W=\oplus_{k\in \Z_+^\nu}W^{k}$.
			By \eqref{eqn:Liebracket_cFi}, $\natural(W^k)\subseteq \cF^k$.
			Since $V'$ is finite dimensional, there exists $L\in \Z_+^\nu$ such that $V'\subseteq \oplus_{k\preceq L}W^{k}$.
			By replacing $L$ with $L+N$, we can assume that $N\preceq L$.
			Let $V=\oplus_{k\in \Z_+^\nu,k\preceq L}W^{k}$.
			We equip $V$ with a Lie bracket by identifying it with $W/\oplus_{k\npreceq L}W^{k}$, and using the fact that $\oplus_{k\npreceq L}W^{k}$ is an ideal of $W$.
			Equivalently, the lie bracket in $V$ of $v\in W^{k}$ and $w\in W^{l}$ is defined by the formula 
				\begin{equation*}\begin{aligned}
					[v,w]_V=\begin{cases}
						[v,w]_W & \text{ if }k+l\preceq L\\
						0 & \text{ otherwise.}
					\end{cases}		
				\end{aligned}\end{equation*}
			The pair $(V,\natural)$ is a $\nu$-graded Lie basis. 
			It is a $\nu$-graded basis because $(V',\natural')\subseteq (V,\natural)$.
			It satisfies a stronger form of \eqref{eqn:bi_graded_lie_basis_commutator}, namely that
				\begin{equation*}
					[\natural(v),\natural(w)]_V=\natural([v,w]_V),\quad \forall v\in V^{k},w\in V^{l}\text{ such that }k+l\preceq L.\qedhere
				\end{equation*}
		\end{proof}

		\begin{rem}\label{rem:natural_x_0_is_a_lie_algebra_homo}
			If $(V,\natural)$ is a $\nu$-graded Lie basis, then for any $x\in M$, by \eqref{eqn:bi_graded_lie_basis_commutator}, the linear map $\natural_{x,0}:V\to \grF{x}$ a Lie algebra and a Lie group homomorphism.
		\end{rem}

		\begin{linkproof}{Lie_algebra}
			Let $(V,\natural)$ be a $\nu$-graded Lie basis.
			We claim that there exists a smooth map $T:V\times V\times M\times \R_+^\nu\to V$ which is bilinear in $V\times V$ such that
			\begin{equation}\label{eqn:T_Lie_Bracket}\begin{aligned}
					&\natural_{x,t}(T(v_1,v_2,x,t))=[\natural(\alpha_{t}(v_1)),\natural(\alpha_{t}(v_2))](x),\quad &&\forall  (v_1,v_2,x,t)\in V\times V\times M\times    \R_+^\nu\setminus \{0\}\\
					&T(v_1,v_2,x,0)=[v_1,v_2],\quad &&\forall (v_1,v_2,x)\in V\times V\times M.
				\end{aligned}\end{equation}
			By bilinearity, it is enough to define $T$ on a basis of $V$. We fix a basis of pure elements.
			Let $v_1\in V^{k_1}$ and $v_2\in V^{k_2}$ be basis elements. By \eqref{eqn:bi_graded_lie_basis_commutator}, there exists $f_i\in C^\infty(M,\R)$ and $w_i\in V^{a(i)}$ such that $a(i)\prec k_1+k_2$ and
			\begin{equation*}\begin{aligned}
					[\natural(v_1),\natural(v_2)]=\natural([v_1,v_2])+\sum_i f_i \natural (w_i)
				\end{aligned}\end{equation*}
			We define
			\begin{equation}\label{eqn:qsdjhlqusdhiqshldloqsdoiuoqsjdojoqsjd}\begin{aligned}
					T(v_1,v_2,x,t)= [v_1,v_2]+     \sum_i f_i(x)t^{k_1+k_2-a(i)} w_i
				\end{aligned}\end{equation}
			It is straightforward to check that $T$ satisfies \eqref{eqn:T_Lie_Bracket}.

			Let $L\subseteq \grF{x}$ be a tangent cone at $x$. 
			So, there exists $(x_n,t_n)_{n\in \N}\subseteq M\times \R_+^\nu\setminus \{0\}$ such that $x_n\to x$, $t_n\to 0$ and
			$\ker(\natural_{x_n,t_n})\to \natural_{x,0}^{-1}(L)$.
			Since $\natural_{x,0}:V\to \mathfrak{g}_x$ is a Lie algebra homomorphism, it suffices to prove that $\natural_{x,0}^{-1}(L)$ is a Lie subalgebra of $V$.
			Let $v,w\in \natural_{x,0}^{-1}(L)$.
			By \Cref{prop:grass_prop}, there exists sequences $v_n, w_n\in  \ker(\natural_{x_n,t_n})$ such that $v_n\to v$ and $w_n\to w$.
			The vector fields $\natural(\alpha_{t_n}(v_n))$ and $\natural(\alpha_{t_n}(w_n))$ vanish at $x_n$.
			So, their Lie bracket vanishes at $x_n$.
			Hence, $T(v_n,w_n,x_n,t_n)\in \ker(\natural_{x_n,t_n})$.
			By continuity of $T$ and \Cref{prop:grass_prop}, $[v,w]=T(v,w,x,0)\in \natural_{x,0}^{-1}(L)$.

			We now  prove \Crefitem{thm:Lie_algebra}{adjoint_action}. 
			Recall that the tangent space $T_L\Grass(V)$ of the Grassmannian manifold $\Grass(V)$ at a subspace $L\subseteq V$ naturally identifies with $\Hom(L,V/L)$ as follows: Let $(L_t)_{t\in ]-\epsilon,\epsilon[}\subseteq \Grass(V)$ be a smooth path of subspaces for some $\epsilon>0$ with $L_0=L$.
			The derivative $\odv{L_t}{t}_{t=0}$ is the linear map $L\to V/L$ defined by 
				\begin{equation}\label{eqn:tangent_space_Grassmannian}\begin{aligned}
					\odv{L_t}{t}_{t=0}(l)=c'(0) \bmod L,
				\end{aligned}\end{equation}
			 where $c:]-\epsilon,\epsilon[\to V$ is any smooth path such that $c(t)\in L_t$ for all $t\in ]-\epsilon,\epsilon[$ and $c(0)=l$.
			Notice that $\odv{L_t}{t}_{t=0}(l)$ doesn't depend on  the choice of $c$ because if $c$ is a smooth path such that $c(t)\in L_t$ for all $t$ and $c(0)=0$, then $c'(0)=\lim_{t\to 0}\frac{c(t)}{t}\in L$ by \Cref{prop:grass_prop}.
			
			Now, let $v\in V$ be fixed. On the manifold $\Grass(V)\times M\times \R_+^\nu$, we define the vector field
				\begin{equation*}\begin{aligned}
					Z(L,x,t)=(w\mapsto T(w,v,x,t)\bmod L,\natural(\alpha_{t}(v))(x),0) \in T_{(L,x,t)}\Grass(V)\times M\times \R_+^\nu.
				\end{aligned}\end{equation*}
			On $\Grass(V)\times M\times \{0\}$, $Z(L,x,0)=(w\mapsto [w,v]\bmod L,0,0)$.
			So,\footnote{Here, the sign convention from \eqref{eqn:BCH} is used.} 
				\begin{equation*}\begin{aligned}
					e^Z\cdot (L,x,0)=(vLv^{-1},x,0).
				\end{aligned}\end{equation*}
			We now see the space $\mathbb{G}^{(0)}$ as a subspace of $\Grass(V)\times M\times \R_+^\nu$ using the map $\Grass(\natural)$ from \eqref{eqn:inclusion_mathbbG0}.
			To prove \Crefitem{thm:Lie_algebra}{adjoint_action}, it suffices to show that the flow of $Z$ leaves invariant the subset $\cGF^{(0)}\times \{0\}$.
			The set $\cGF^{(0)}\times \{0\}$ is the limit set of $M\times \R_+^\nu\setminus \{0\}$.
			Hence, it suffices to show that the flow of $Z$ preserves $M\times \R_+^\nu\setminus \{0\}$.
			In fact $Z$ is tangent to $M\times \R_+^\nu\setminus \{0\}$.
			To see this, we compute the differential of $\Grass(\natural)$ at a point $(x,t)\in M\times \R_+^\nu\setminus \{0\}$.
			We claim that for any $\xi \in T_xM$, $\odif{\Grass(\natural)}_{(x,t)}(\xi,0)=(\phi,\xi,0)$ where $\phi:\ker(\natural_{x,t})\to V/\ker(\natural_{x,t})$ is the unique linear map which satisfies 
			\begin{equation*}\begin{aligned}
				\natural_{x,t}(\phi(w))=[\natural(\alpha_{t}(w)),X](x)	,\quad \forall w\in 	\ker(\natural_{x,t})
			\end{aligned}\end{equation*}
			where $X\in \cX(M)$ is any vector field which satisfies $X(x)=\xi$.
			The claim finishes the proof of \Crefitem{thm:Lie_algebra}{adjoint_action} because by \eqref{eqn:T_Lie_Bracket}, we have
				\begin{equation*}\begin{aligned}
					\odif{\Grass(\natural)}_{(x,t)}(\natural(\alpha_{t}(v))(x),0)= Z(\Grass(\natural)(x,t)).
				\end{aligned}\end{equation*}
			Let us prove our claim.
			We fix $x_0\in M$, $t_0\in \R^\nu_+\backslash 0$, $X\in \cX(M)$. 
			Let $v_1,\cdots,v_{\dim(M)}\in V$ such that $\natural_{x_0,t_0}(v_1),\cdots,\natural_{x_0,t_0}(v_{\dim(M)})$ form a linear basis of $T_{x_0}M$.
			So, the vector fields $\natural(\alpha_{t_0}(v_1)),\cdots,\natural(\alpha_{t_0}(v_{\dim(M)}))$ linearly span the tangent bundle in a neighborhood $U$ of $x_0$.
			For any $w\in V$, $\natural(\alpha_{t_0}(w))=\sum_{i}f_i \natural(\alpha_{t_0}(v_i))$ for some unique smooth functions $f_i\in C^\infty(U)$.
			For any $x\in U$, a generic element of $\ker(\natural_{x,t_0})$ is then given by $w-\sum_{i}f_i(x)v_i$.
			By the identification of the tangent space of Grassmannian manifolds in \eqref{eqn:tangent_space_Grassmannian}, $\odif{\Grass(\natural)}_{(x_0,t_0)}(X(x_0),0)=(\phi,X(x_0),0)$ where $\phi$ is the linear map 
				\begin{equation*}\begin{aligned}
						w-\sum_{i}f_i(x_0)v_i&\mapsto \odv*{\left(w-\sum_{i}f_i(e^{tX}\cdot x_0)v_i\right)}{t}_{t=0}=-\sum_{i}X(f_i)(x_0)v_i \mod \ker(\natural_{x_0,t_0}).
				\end{aligned}\end{equation*}
			We have
				\begin{equation*}\begin{aligned}
					\left[\natural\left(\alpha_{t_0}\left(w-\sum_{i}f_i(x_0)v_i\right)\right),X\right](x_0)&=[\natural(\alpha_{t_0}(w)),X](x_0)-\sum_{i}f_i(x_0)[\natural(\alpha_{t_0}(v_i)),X](x_0)\\
					&=-\sum_{i}X(f_i)(x_0)\natural_{x_0,t_0}(v_i).
				\end{aligned}\end{equation*}
			The claim follows.
		\end{linkproof}

		While proving \Crefitem{thm:Lie_algebra}{adjoint_action}, we also proved the following: 
		\begin{lem}\label{lem:r_smooth}
			Let $(V,\natural)$ be a $\nu$-graded Lie basis. The following function is smooth 
				\begin{equation*}\begin{aligned}
					r\circ \cQ_V:\dom(\cQ_V)&\to \mathbb{G}^{(0)}&\\
					(v,x,t)&\mapsto (e^{\alpha_{t}(v)}\cdot x,t)\\
					(v,L,x,0)&\mapsto (\natural_{x,0}(v)L\natural_{x,0}(v)^{-1},x,0).
				\end{aligned}\end{equation*}
		\end{lem}
		\begin{proof}
			The map $r\circ \cQ_V$ for a fixed $v$ is the restriction to $\mathbb{G}^{(0)}$ of the flow at time $1$ of $Z$.
			Since $Z$ depends smoothly on $v$, it follows that $r\circ \cQ_V$ is smooth.
		\end{proof}
\section{Key lemmas}\label{sec:prelim_results}
		\paragraph{Summary:}
		We will prove two Lemmas. The first allows us to pass from one $\nu$-graded basis to another. The second lifts composition on $\mathbb{G}$ to $\nu$-graded Lie basis.
		Both lemmas can be seen as refinements of the maps $T$ constructed in the proofs of \Cref{thm:G0_locally_compact} and \Cref{thm:Lie_algebra}.
		The Lemmas are proved in the reverse order.
		\paragraph{Notation and convention:}
		Throughout this Chapter and \Cref{chap:chapter_bi-garded_pseudo_diff_proof}, we will make use of the following auxiliary $(\Rpt)^\nu$-actions.
		Let $\lambda\in (\Rpt)^\nu$. 
		Recall the notation from \eqref{eqn:lambda_mu_product}.
		We define
		\begin{equation}\label{eqn:dilations_VMR2}\begin{aligned}
				&\alpha_\lambda:M\times \R_+^{\nu}\to M\times \R_+^\nu, &&\alpha_\lambda(x,t)=(x,\lambda^{-1}t)\\
				&\alpha_\lambda:V\times M\times \R_+^{\nu}\to V\times M\times \R_+^\nu, &&\alpha_\lambda(v,x,t)=(\alpha_\lambda(v),\alpha_\lambda(x,t))\\
				&\alpha_\lambda:V\times V\times M\times \R_+^{\nu}\to V\times V\times M\times \R_+^\nu, &&\alpha_\lambda(v,w,x,t)=(\alpha_\lambda(v),\alpha_{\lambda}(w),\alpha_\lambda(x,t))\\
							&\alpha_\lambda:\mathbb{G}^{(0)}\to \mathbb{G}^{(0)},	&&\alpha_\lambda(x,t)=(x,\lambda^{-1}t),  \quad \alpha_\lambda(L,x,0)=(\alpha_\lambda(L),x,0)\\
				&\alpha_\lambda:V\times \mathbb{G}^{(0)}\to V\times \mathbb{G}^{(0)}, &&\alpha_\lambda(v,\gamma)=(\alpha_\lambda(v),\alpha_\lambda(\gamma))\\
					&\alpha_\lambda:V\times V\times \mathbb{G}^{(0)}\to V\times V\times \mathbb{G}^{(0)}, &&\alpha_\lambda(v,w,\gamma)=(\alpha_\lambda(v),\alpha_{\lambda}(w),\alpha_\lambda(\gamma))
		\end{aligned}\end{equation}
		To summarize, $\alpha_\lambda$ acts diagonally on all variables. It acts trivially on the variable $M$, on $V$ by the dilation from \eqref{eqn:dilations_basic_V}, on $\Grass(\mathfrak{g}_x)$ using the dilation on $\mathfrak{g}_x$, on $\R_+^\nu$ by $t\mapsto \lambda^{-1}t$.
		\begin{rem}\label{rem:map_alpha_lambda_G0_well_defined}
			The map $\alpha_\lambda:\mathbb{G}^{(0)}\to \mathbb{G}^{(0)}$ is well-defined diffeomorphism. To check that it is well-defined, we need to show that if $L\subseteq \mathfrak{g}_x$ is a tangent cone at $x$, then $\alpha_\lambda(L)$ is also a tangent cone at $x$.
			This is because if $(x_n,t_n)_{n\in \N}\subseteq M\times \R_+^\nu\backslash 0$ is a sequence such that $x_n\to x$, $t_n\to 0$, $\ker(\natural_{x_n,t_n})\to \natural_{x,0}^{-1}(L)$, then
					$\ker(\natural_{x_n,\lambda^{-1}t_n})=\alpha_{\lambda}(\ker(\natural_{x_n,t_n}))\to \alpha_\lambda(L)$.
				The map $\alpha_\lambda:\mathbb{G}^{(0)}\to \mathbb{G}^{(0)}$ is a diffeomorphism because it is the restriction of the diffeomorphism $\Grass(V)\times M\times \R_+^\nu\to \Grass(V)\times M\times \R_+^\nu$ given by $(L,x,t)\mapsto (\alpha_\lambda(L),x,\lambda^{-1}t)$.
		\end{rem}

		We will usually have partially defined maps $f:\dom(f)\subseteq X\to Y$ between spaces $X$ and $Y$ which are equipped with $(\Rpt)^\nu$-actions.
		We will say that $f$ is equivariant if $\dom(f)$ is equivariant and $f(\alpha_\lambda(x))=\alpha_\lambda(f(x))$ for all $x\in \dom(f)$ and $\lambda\in (\Rpt)^\nu$.
		\begin{linklem}{change_of_basis_map}
		Let $(V,\natural)$ and $(V',\natural')$ be $\nu$-graded basis such that $(V',\natural')\subseteq (V,\natural)$. There exists smooth equivariant submersions 
			\begin{equation*}\begin{aligned}
				\tilde{\phi}:\dom(\tilde{\phi})\subseteq V\times M\times \R_+^\nu\to V'\times M\times \R_+^\nu, \quad \phi:\dom(\phi)\subseteq V\times \mathbb{G}^{(0)}\to V'\times \mathbb{G}^{(0)},
			\end{aligned}\end{equation*}
		defined on open neighborhoods of $V\times M\times \{0\}$ and $V\times \cG^{(0)}\times \{0\}$ respectively such that:
		\begin{enumerate}
			\item If $(v,x)\in V'\times M$, then $\tilde{\phi}(v,x,0)=(v,x,0)$.\footnote{In fact, the map we construct satisfies the stronger condition that $\tilde{\phi}(v,x,t)=(v,x,t)$ for all $(v,x,t)\in (V'\times M\times \R_+^\nu)\cap \dom(\tilde{\phi})$. We won't need this.}
				\item One has
					\begin{equation}\label{eqn:ksqldjflqdsjfmqsdjf}\begin{aligned}
						\dom(\tilde{\phi})\subseteq \tdom(\cQ_V),\quad  \im(\tilde{\phi})\subseteq \tdom(\cQ_{V'}), \quad
						\dom(\phi)\subseteq \dom(\cQ_V),\quad  \im(\phi)\subseteq \dom(\cQ_{V'})
					\end{aligned}\end{equation}
				 and the following diagram commutes 
				\begin{equation}\label{eqn:comm_diag_phi_hat}\begin{aligned}
					\begin{tikzcd}
					\dom(\phi)\arrow[dr,"\cQ_{V}"']\arrow[r,"\phi"]& V'\times \mathbb{G}^{(0)}\arrow[d,"\cQ_{V'}"]\\&\mathbb{G}
					\end{tikzcd}.
				\end{aligned}\end{equation}
				\item One has $\dom(\phi)=\pi_{V\times \mathbb{G}^{(0)}}^{-1}(\dom(\tilde{\phi}))$, and the following diagram
				
				\begin{equation}\label{eqn:ksqldjflqdsjfmqsdjf_2}\begin{tikzcd}
					\dom(\phi)\arrow[d,"\pi_{V\times \mathbb{G}^{(0)}}"']\arrow[r, "\phi"]    &   V'\times \mathbb{G}^{(0)} \arrow[d,"\pi_{V'\times \mathbb{G}^{(0)}}"]\\
												\dom(\tilde{\phi})\arrow[r, "\tilde{\phi}"]& V'\times M\times \R_+^\nu
				\end{tikzcd}\end{equation}
				is a pullback diagram.
			\end{enumerate}
		\end{linklem}
		To summarize, the map $\phi$ allows one to pass from $\cQ_{V'}$ to $\cQ_V$. Furthermore, it is the pullback of some submersion $\tilde{\phi}$ which doesn't live on the blow up (and thus we can use standard differential geometry tools like the inverse mapping theorem on $\tilde{\phi}$).

		To state the following lemma, we need a few preliminaries: Let $(V,\natural)$ be a $\nu$-graded Lie basis.
		We wish to lift the product map $\mathbb{G}^{(2)}\to \mathbb{G}$ through $\cQ_V$.
		Notice that 
			\begin{equation*}\begin{aligned}
				&\{(v,\gamma),(w,\eta)\in \dom(\cQ_V)\times \dom(\cQ_V):(\cQ_V(v,\gamma),\cQ_V(w,\eta))\in \mathbb{G}^{(2)}\}\\
				=&\{(v,\gamma),(w,\eta)\in \dom(\cQ_V)\times \dom(\cQ_V):\gamma=r\circ \cQ_V(w,\eta)\}.
			\end{aligned}\end{equation*}
		So, this set can be identified with the following subset of $V\times V\times \mathbb{G}^{(0)}$
			\begin{equation*}\begin{aligned}
		\dom(\cQ_V^2)=\{(v,w,\eta)\in V\times V\times \mathbb{G}^{(0)}:(v,r\circ \cQ_V(w,\eta)),(w,\eta)\in \dom(\cQ_V)\}.	
			\end{aligned}\end{equation*}
		Recall that $\dom(\cQ_V)=\pi_{V\times \mathbb{G}^{(0)}}^{-1}(\tdom(\cQ_V))$. Similarly, if we define
			\begin{equation*}\begin{aligned}
				\tdom(\cQ_V^2)&=\{(v,w,x,t)\in V\times V\times M\times \R_+^\nu:(v,e^{\alpha_{t}(w)}\cdot x,t),(w,x,t)\in \tdom(\cQ_V)\},
			\end{aligned}\end{equation*}
			then $\dom(\cQ_V^2)=\pi_{V\times V\times \mathbb{G}^{(0)}}^{-1}(\tdom(\cQ_V^2))$ where $\pi_{V\times V\times \mathbb{G}^{(0)}}:V\times V\times \mathbb{G}^{(0)}\to V\times V\times M\times \R_+^\nu$ is the natural projection, see \eqref{eqn:proj_with_V_mathbbG}.
		The set $\tdom(\cQ_V^2)$ is open because $\tdom(\cQ_V)$ is open.
		Since  $\pi_{V\times V\times \mathbb{G}^{(0)}}$ is continuous, $\dom(\cQ_V^2)$ is open.
		Let \begin{equation}\label{eqn:cQv2}\begin{aligned}
					\cQ_V^2:\dom(\cQ_V^2)\subseteq 		V\times V\times \mathbb{G}^{(0)}\to \mathbb{G}\\
					(v,w,x,t)\mapsto (e^{\alpha_{t}(v)}e^{\alpha_{t}(w)}\cdot x,x,t)           \\
					(v,w,L,x,0)\mapsto(\natural_{x,0}(v\cdot w)L,x,0),
					\end{aligned}\end{equation}
		Equivalently, 
			\begin{equation}\label{eqn:QV2_is_QV_times_QV}\begin{aligned}
				\cQ_V^2(v,w,\gamma)=\cQ_V(v,r\circ\cQ_V(w,\gamma) )\cQ_V(w,\gamma)		
			\end{aligned}\end{equation}
		The sets $\tdom(\cQ_V^2)$, $\dom(\cQ_V^2)$ and the maps $\pi_{V\times V\times \mathbb{G}^{(0)}}$, $\cQ_V^2$ are $(\Rpt)^\nu$-equivariant. 
		\begin{linklem}{composition_bisub}
		Let $(V,\natural)$ be a $\nu$-graded Lie basis. 
		There exists equivariant smooth submersions
		\begin{equation*}\begin{aligned}
			\tilde{\psi}:\dom(\tilde{\psi})\subseteq V\times V\times M\times \R_+^\nu\to V	\times M\times \R_+^\nu,\quad
			\psi:\dom(\psi)\subseteq V\times V\times \mathbb{G}^{(0)}\to V	\times \mathbb{G}^{(0)}
				\end{aligned}\end{equation*}
		defined on some open neighborhood of $V\times V\times M\times \{0\}$ and $V\times V\times \cG^{(0)}\times \{0\}$ such that: 
		\begin{enumerate}
		\item\label{thm:composition_bisub:3} For any  $(v,w,x)\in V\times V\times M $, $\tilde{\psi}(v,w,x,0)=(v\cdot w,x,0)$.
		\item\label{thm:composition_bisub:4} For any $(v,x,t)\in V\times M\times \R_+^\nu$, we have $\tilde{\psi}(v,0,x,t)=\tilde{\psi}(0,v,x,t)=(v,x,t)$ whenever $(v,0,x,t)$ or $(0,v,x,t)$ is in the domain of $\tilde{\psi}$.
		\item\label{thm:composition_bisub:5}
		One has, 
			\begin{equation}\label{eqn:domains_psi}\begin{aligned}
				\dom(\tilde{\psi})\subseteq \tdom(\cQ_V^2),\ \im(\tilde{\psi})\subseteq \tdom(\cQ_V),\ \dom(\psi)\subseteq \dom(\cQ_V^2),\ \im(\psi)\subseteq \dom(\cQ_V),
			\end{aligned}\end{equation}
		and the following diagram commutes  
		\begin{equation}\label{eqn:diag_composition_bisub}\begin{aligned}
			\begin{tikzcd}\dom(\psi)\arrow[r,"\psi"]\arrow[dr,"\cQ_{V}^2"']&V\times \mathbb{G}^{(0)}\arrow[d,"\cQ_{V}"]\\& \mathbb{G}
			   \end{tikzcd}.
		\end{aligned}\end{equation}	
		\item One has $\dom(\psi)=\pi_{V\times V\times \mathbb{G}^{(0)}}^{-1}(\dom(\tilde{\psi}))$ and the following diagram		
			\begin{equation}\label{eqn:diag_comb_sub_psi}
				\begin{tikzcd}
			\dom(\psi)\arrow[d,"\pi_{V\times V\times \mathbb{G}^{(0)}}"']\arrow[r, "\psi"]    &   V\times \mathbb{G}^{(0)} \arrow[d,"\pi_{V\times V\times \mathbb{G}^{(0)}}"]\\
										\dom(\tilde{\psi})\arrow[r, "\tilde{\psi}"]& V\times M\times \R_+^\nu
		\end{tikzcd}\end{equation}
		is a pullback diagram.
		\end{enumerate}
		\end{linklem}
		\begin{rem}\label{rem:psi_hat}
			Commutativity of \eqref{eqn:diag_composition_bisub} and that Diagram \eqref{eqn:diag_comb_sub_psi} imply that $\psi$ and $\tilde{\psi}$ are of the form 
			\begin{equation}\label{eqn:composition_bisub_tilde_psi}\begin{aligned}
						\tilde{\psi}(v,w,x,t)=(\hat{\psi}(v,w,x,t),x,t),\quad \psi(v,w,\gamma)=(\hat{\psi}(v,w,\pi_{\mathbb{G}^{(0)}}(\gamma)),\gamma)
					\end{aligned}\end{equation}
			for some map $\hat{\psi}:\dom(\tilde{\psi})\to V$ which satisfies
			\begin{equation}\label{eqn:product_phi_compsotion_bisub}\begin{aligned}
					e^{\alpha_{t}(\hat{\psi}(v,w,x,t))}\cdot x=e^{\alpha_{t}(v)} e^{\alpha_{t}(w)}\cdot x,\text{ and } \hat{\psi}(v,w,x,0)=v\cdot w.
			\end{aligned}\end{equation}
			In some proofs later on in this section, we will make use of the map $\hat{\psi}$.
		\end{rem}
		Throughout this chapter, we will be consistent in reserving $\tilde{\phi},\phi$ for the maps obtained by applying \Cref{thm:change_of_basis_map} and $\tilde{\psi},\psi$ for the maps obtained by applying \Cref{thm:composition_bisub}.
	\begin{linkproof}{composition_bisub}
		\begin{lem}\label{lem:partial_proof_thm_compo_bisub}
				To prove \Cref{thm:composition_bisub}, it is enough to find an $(\Rpt)^\nu$-equivariant open neighborhood $U\subseteq V\times M\times \R_+^\nu$ of $V\times M\times  \{0\}$ and a linear map
				$w\in V\mapsto A_w\in \cX(U)$ such that: 
		\begin{enumerate}
			\item For any $\lambda\in (\Rpt)^\nu$, $w\in V$, $\alpha_{\lambda*}(A_w)=A_{\alpha_\lambda(w)}$.
			\item If $(v,x)\in V\times M$, then $A_w(v,x,0)=\left(\odv*{\BCH(\tau w,v)}{\tau}_{\tau=0},0,0\right)$.
			\item\label{lem:partial_proof_thm_compo_bisub:3} If $(v,x,t)\in U$, $w\in V$, then $e^{\tau A_w}\cdot (v,x,t)=(v_\tau,x,t)$ and $e^{\alpha_t(v_\tau)}\cdot x=e^{\tau \alpha_{t}(w)}e^{\alpha_{t}(v)}\cdot x$ whenever the flow of $A_w$ is well-defined at $\tau \in \R$.
			\item\label{lem:partial_proof_thm_compo_bisub:ksqdjfmkjsdqklùjfkljqsdljfmdsqmjfqdsfjqml} Whenever the flow is well-defined, $e^{\tau A_w}\cdot (0,x,t)=(\tau w,x,t)$.
		\end{enumerate}
		\end{lem}
		\begin{proof}
			Let 
				$\tilde{\psi}(w,v,x,t)=e^{A_w}\cdot (v,x,t)$ whose domain is the set of points at which the flow exists. 
				The map $\tilde{\psi}$ is of the form $\tilde{\psi}(w,v,x,t)=(\hat{\psi}(v,w,x,t),x,t)$ for some map $\hat{\psi}$ which satisfies \eqref{eqn:product_phi_compsotion_bisub}.
			We define $\psi$ by  \eqref{eqn:composition_bisub_tilde_psi}.
			By \eqref{eqn:product_phi_compsotion_bisub}, the map $\tilde{\psi}$ is a submersion at any point in $\{0\}\times \{0\}\times M\times \{0\}$.
			Since $\tilde{\psi}$ is equivariant, the set of points at which $\tilde{\psi}$ is a submersion is an $(\Rpt)^\nu$-equivariant open neighborhood of $V\times V\times M\times \{0\}$. 
			By reducing $\dom(\tilde{\psi})$ to the set of the points at which $\tilde{\psi}$ is a submersion, 
			we can suppose that $\tilde{\psi}$
			is a submersion.
			By replacing $\dom(\tilde{\psi})$ with $\dom(\tilde{\psi})\cap \tdom(\cQ_V^2) \cap \tilde{\psi}^{-1}(\tdom(\cQ_V))$,
			we can further suppose that $\dom(\tilde{\psi})\subseteq \tdom(\cQ_V^2)$ and $\im(\tilde{\psi})\subseteq \tdom(\cQ_V)$.
			We now take the domain of $\psi$ to be $\pi_{V\times V\times \mathbb{G}^{(0)}}^{-1}(\dom(\tilde{\psi}))$. 
			Diagram \eqref{eqn:diag_comb_sub_psi} is easily seen to be a pullback diagram.
			So, by \Cref{lem:pullback_local_diff}, $\psi$ is a submersion. 
			Since $\dom(\cQ_V^2)=\pi_{V\times V\times \mathbb{G}^{(0)}}^{-1}(\tdom(\cQ_V^2))$ and $\dom(\cQ_V)=\pi_{V\times \mathbb{G}^{(0)}}^{-1}(\tdom(\cQ_V))$, \eqref{eqn:domains_psi} follows.				
			Commutativity of \eqref{eqn:diag_composition_bisub} follows from \eqref{eqn:product_phi_compsotion_bisub} and \Cref{rem:natural_x_0_is_a_lie_algebra_homo}.
			Finally, \Crefitem{thm:composition_bisub}{4} follows from Condition \ref{lem:partial_proof_thm_compo_bisub:ksqdjfmkjsdqklùjfkljqsdljfmdsqmjfqdsfjqml} and that $A_0=0$.
		\end{proof}
			
		Let $T:V\times V\times M\times \R_+^\nu\to V$ be a bilinear map as in the proof of \Cref{thm:Lie_algebra} which satisfies \eqref{eqn:T_Lie_Bracket}.
		One can check that by the definition of $T$ in \eqref{eqn:qsdjhlqusdhiqshldloqsdoiuoqsjdojoqsjd} that $T$ is $(\Rpt)^\nu$-equivariant.
		By replacing $T$, by $(w,v,x,t)\mapsto \frac{T(w,v,x,t)-T(v,w,x,t)}{2}$, we can further suppose that 
			\begin{equation}\label{eqn:T_anti_symmetric_map}\begin{aligned}
				T(w,v,x,t)=-T(v,w,x,t),\quad \forall (v,w,x,t)\in V\times V\times M\times \R_+^\nu.		
			\end{aligned}\end{equation}
		Let $f_\tau,h_\tau:V\times M\times \R_+^\nu\to \cL(V)$ for $\tau\in\R$ be locally-defined smooth solutions of the first order differential equation (matrix differential equation)
				\begin{equation}\label{eqn:f_tau_h_tau}\begin{aligned}
				f_0(v,x,t)w=w,\quad	\pdv*{f_{\tau}(v,x,t)}{\tau}w=f_\tau(v,x,t)T(w,v,e^{-\tau \alpha_t(v)}\cdot x,t),\quad h_\tau=\int_{0}^\tau f_{s}ds.	
				\end{aligned}\end{equation}
				Because the differential equation defining $f$ is a matrix differential equation, the maps $f_{\tau}(v,x,t)$ and $h_\tau(v,x,t)$ are defined whenever $e^{-\tau\alpha_t(v)}\cdot x$ is well-defined.
				Let $h_{1}^{-1}(v,x,t)$ be the linear inverse of $h_{1}(v,x,t)$ if it exists.
				By \eqref{eqn:T_Lie_Bracket}, $f_\tau(v,x,0)=e^{\tau\mathrm{ad}_v}$, where $\mathrm{ad}_v:V\to V$ is the linear map $\mathrm{ad}_v(w)=[w,v]$.
				Hence, $h_1(v,x,0)=\frac{e^{\mathrm{ad}_v}-1}{\mathrm{ad}_v}$.
				Since $\mathrm{ad}_v$ is a nilpotent linear map, $h_{1}(v,x,0)$ is invertible.
				The formula for the derivative of the exponential of matrices \cite[Theorem 5 Section 1.2]{RossmannBook} implies that
					\begin{equation}\label{eqn:qksodforijhaozeihrqskodfqsidfôiqsdiof3}\begin{aligned}
						h_{1}^{-1}(v,x,0)w=\odv*{\BCH(\tau w,v)}{\tau}_{\tau=0}.	
					\end{aligned}\end{equation}
				Let
				\begin{equation}\label{eqn:Aw}\begin{aligned}
						&U=\left\{(v,x,t)\in V\times  M\times \R_+^\nu:e^{\alpha_t(v)}\cdot x\text{ and }\ h_{1}^{-1}\left( v,e^{\alpha_t(v)}\cdot x,t\right)\text{ exist}\right\}\\
						&A_w(v,x,t)=(h_{1}^{-1}(v,e^{\alpha_t(v)}\cdot x,t)w,0,0)\in \cX(U).		
					\end{aligned}\end{equation}
				The set $U$ is an open neighborhood of $V\times M\times \{0\}$. 
				Let us check that $A_w$ satisfies the required conditions.
				Equivariance of $T$ implies that
				\begin{equation*}\begin{aligned}
					f_{\tau}(\alpha_\lambda(v,x,t))\alpha_\lambda(w)=\alpha_\lambda(f_\tau(v,x,t)w),	\quad 	h_{\tau}(\alpha_\lambda(v,x,t))\alpha_\lambda(w)=\alpha_\lambda(h_\tau(v,x,t)w)			
				\end{aligned}\end{equation*} 
				which implies that $U$ is $(\Rpt)^\nu$-equivariant and $\alpha_{\lambda*}(A_w)=A_{\alpha_\lambda(w)}$.	
				By \eqref{eqn:qksodforijhaozeihrqskodfqsidfôiqsdiof3}, we see that 
			$A_w(v,x,0)=\left(\odv*{\BCH(\tau w,v)}{\tau}_{\tau=0},0,0\right)$.
			By \eqref{eqn:T_anti_symmetric_map}, $f_\tau(sv,x,t)v=v$ for all $s\in \R$ whenever $f_\tau$ is well-defined. 
			It follows that $A_w(s w,x,t)=(w,0,0)$ for all $s\in \R$ from which \Crefitem{lem:partial_proof_thm_compo_bisub}{ksqdjfmkjsdqklùjfkljqsdljfmdsqmjfqdsfjqml} follows.

			It remains to check \Crefitem{lem:partial_proof_thm_compo_bisub}{3}.
			On the manifold $V\times M\times \R_+^\nu$, we define the vector fields: 
			\begin{equation*}\begin{aligned}
				Y_w(v,x,t)=(w,0,0),\quad 
					X_w(v,x,t)=(0,\natural(\alpha_{t}(w))(x),0),\quad 
					Z(v,x,t)=(0,\natural(\alpha_{t}(v))(x),0),
			\end{aligned}\end{equation*}
			where $w\in V$ is a parameter.
			We have 
				\begin{align}
					[Y_w,Z]=X_w,\quad	e^{\tau X_w}\cdot (v,x,t)=(v,e^{\tau \alpha_{t}(w)}\cdot x,t),\quad e^{\tau Z}\cdot (v,x,t)=(v,e^{\tau\alpha_{t}(v)}\cdot x,t)\label{eqn:flow_of_X_Z}
				\end{align}	
			By \eqref{eqn:T_Lie_Bracket} and \eqref{eqn:flow_of_X_Z}, we have 
					\begin{equation}\label{eqn:Bracket_X_Z}\begin{aligned}
						[X_{w},Z](v,x,t)=X_{T(w,v,x,t)}(v,x,t).
					\end{aligned}\end{equation}		
			The flow of $A_w$ is given by $e^{\tau A_w}\cdot (v,x,t)=(v_\tau,x,t)$ for some $v_\tau \in V$ because $A_{w}$ is tangent to $V$. 
					By the definition of pushforward of vector fields, 
			\begin{equation}\label{eqn:qskdjfkqdsfkmlqdsjf}\begin{aligned}
				(v_\tau,e^{\alpha_t(v_\tau)}\cdot x,t)=e^Z\cdot (v_\tau,x,t)=e^Ze^{\tau A_w}\cdot (v,x,t)&=e^Ze^{\tau A_w}e^{-Z}e^{Z}\cdot (v,x,t)\\&=e^{\tau e_*^{Z}(A_w)}\cdot (v,e^{\alpha_{t}(v)}\cdot x,t).	
			\end{aligned}\end{equation}
			We now compute the vector field $e_*^Z(A_w)$.
			We claim that
			\begin{equation}\label{eqn:X_Y_identity}\begin{aligned}
				e^{\tau Z}_*(X_w)(v,x,t)=X_{f_\tau(v,x,t)w}(v,x,t),\ 				e^{\tau Z}_*(Y_w)(v,x,t)=Y_w(v,x,t)+X_{h_\tau(v,x,t)w}(v,x,t)								.
			\end{aligned}\end{equation}
		To see this, let $K_{w,\tau}(v,x,t)=X_{f_\tau(v,x,t)w}(v,x,t)$.
		We have 
			\begin{equation*}\begin{aligned}
				\odv*{e^{-\tau Z}_*(K_{w,\tau})}{\tau}=e^{-\tau Z}_*\left(\odv*{K_{w,\tau}}{\tau}-[K_{w,\tau},Z]\right).
			\end{aligned}\end{equation*}
		By \eqref{eqn:Bracket_X_Z}, we have
			\begin{equation*}\begin{aligned}
				[K_{w,\tau},Z](v,x,t)=X_{T(f_\tau(v,x,t)w,v,x,t)}(v,x,t)-X_{Z(f_{\tau})(v,x,t)w}(v,x,t).	
			\end{aligned}\end{equation*}
		To compute $Z(f_\tau)$, notice that 
			\begin{equation*}\begin{aligned}
				f_{\tau'+\tau}(v,e^{\tau'\alpha_t(v)}\cdot x,t)=f_{\tau'}(v,e^{\tau'\alpha_t(v)}\cdot x,t)f_{\tau}(v,x,t).		
			\end{aligned}\end{equation*}
		By taking derivative at $\tau'=0$, we deduce that
				\begin{equation*}\begin{aligned}
				Z(f_\tau)(v,x,t)w+\pdv*{f_\tau(v,x,t)w}{\tau}=T(f_\tau(v,x,t)w,v,x,t).	
			\end{aligned}\end{equation*}
		Hence, $\odv*{e^{-\tau Z}_*(K_{w,\tau})}{\tau}=0$. So, $e^{-\tau Z}_*(K_{w,\tau})=K_{w,0}=X_w$. 
		This proves the first identity of \eqref{eqn:X_Y_identity}.
		The second identity of \eqref{eqn:X_Y_identity} follows from the following:
			\begin{equation*}\begin{aligned}
				\odv*{X_{h_\tau(v,x,t)w}}{\tau}(v,x,t)=X_{f_\tau(v,x,t)w}(v,x,t)=e^{\tau Z}_{*}\left(X_w\right)=e^{\tau Z}_{*}\left([Y_w,Z]\right)=\odv*{e^{\tau Z}_*(Y_w)}{\tau}.
			\end{aligned}\end{equation*}
		Since $A_w(v,x,t)=Y_{h_{1}^{-1}(v,e^{\alpha_t(v)}\cdot x,t)w}(v,x,t)$, the second identity of \eqref{eqn:X_Y_identity} implies that
			\begin{equation*}\begin{aligned}
				e_*^{Z}(A_w)(v,x,t)=Y_{h_1^{-1}(v,x,t)w}(v,x,t)+X_w(v,x,t)=(h_1^{-1}(v,x,t)w,\natural(\alpha_t(w))(x),0).
			\end{aligned}\end{equation*}
		So, $e^{\tau e_*^{Z}(A_w)}\cdot (v,e^{\alpha_{t}(v)}\cdot x,t)=(\xi_\tau,e^{\tau \alpha_{t}(w)}e^{\alpha_{t}(v)}\cdot x,t)$ for some $\xi_\tau\in V$.
		Hence, by \eqref{eqn:qskdjfkqdsfkmlqdsjf}, 
			\begin{equation*}\begin{aligned}
				(v_\tau,e^{\alpha_{t}(v_\tau)}\cdot x,t)=(\xi_\tau,e^{\tau \alpha_{t}(w)}e^{\alpha_{t}(v)}\cdot x,t).	
			\end{aligned}\end{equation*}
		In particular, $e^{\alpha_{t}(v_\tau)}\cdot x=e^{\tau \alpha_{t}(w)}e^{\alpha_{t}(v)}\cdot x$.
	\end{linkproof}
	\begin{linkproof}{change_of_basis_map}
		By replacing $(V,\natural)$ with a bigger $\nu$-graded basis using \Cref{lem:bigger_bi_graded_Lie}, we can suppose that $(V,\natural)$ is a $\nu$-graded Lie basis.
		Let $\tilde{\psi}$ be as in \Cref{thm:composition_bisub}, and $\hat{\psi}$ as in \eqref{eqn:composition_bisub_tilde_psi}, $T:V\times M\times \R_+^\nu\to V'$ as in the proof of \Cref{thm:G0_locally_compact} which satisfies \eqref{eqn:qskjfpiojqspdjfpqjsdpf}.
		By looking at the way $T$ is constructed, we see that we can suppose that $T(v,x,0)=v$ for all $v\in V'$.
		Note that the map $T$ defined using \eqref{eqn:map_T_change_of_basis} is $(\Rpt)^\nu$-equivariant.
		\begin{lem}\label{lem:sklqdjfkjqsdfkjqsdmkjfkmsqdljf}
			To prove \Cref{thm:change_of_basis_map}, it is enough to find an equivariant smooth map $\hat{\phi}:\dom(\hat{\phi})\subseteq V\times M\times \R_+^\nu\to V'$ defined on an equivariant open neighborhood of $V\times M\times \{0\}$
			such that: 
			\begin{enumerate}
				\item\label{lem:sklqdjfkjqsdfkjqsdmkjfkmsqdljf:1} If $(v,x,t)\in \dom(\hat{\phi})$, then $(-v,\hat{\phi}(v,x,t),x,t)\in \dom(\tilde{\psi})$ and $T(\tilde{\psi}(-v,\hat{\phi}(v,x,t),x,t))=0$.	
				\item\label{lem:sklqdjfkjqsdfkjqsdmkjfkmsqdljf:2} If $v\in V'$ and $x\in M$, then $\hat{\phi}(v,x,0)=v$.	
			\end{enumerate}
		\end{lem}
		\begin{proof}
			If $(v,x,t)\in \dom(\hat{\phi})$, then by \eqref{eqn:qskjfpiojqspdjfpqjsdpf}, $\hat{\psi}(-v,\hat{\phi}(v,x,t),x,t)\in \ker(\natural_{x,t})$.
			By \Cref{rem:natural_x_0_is_a_lie_algebra_homo} and \eqref{eqn:product_phi_compsotion_bisub}, we deduce 
				\begin{equation}\label{eqn:sqdklfkjklsqdjf}\begin{aligned}
					e^{\alpha_t(\hat{\phi}(v,x,t))}\cdot x=e^{\alpha_t(v)}\cdot x,\text{ and } \natural_{x,0}(\hat{\phi}(v,x,0))=\natural_{x,0}(v).
				\end{aligned}\end{equation}
			We now define $\tilde{\phi}(v,x,t)=(\hat{\phi}(v,x,t),x,t)$ and $\phi(v,\gamma)=(\hat{\phi}(v,\pi_{\mathbb{G}^{(0)}}(\gamma)),\gamma)$.
			We proceed as in the proof of \Cref{lem:partial_proof_thm_compo_bisub} to reduce the domain of $\tilde{\phi}$ so that $\tilde{\phi}$ and $\phi$ are submersions,
			and \eqref{eqn:ksqldjflqdsjfmqsdjf} holds, and Diagram \eqref{eqn:ksqldjflqdsjfmqsdjf_2} is a pullback diagram.
			Finally, Diagram \eqref{eqn:comm_diag_phi_hat} commutes by \eqref{eqn:sqdklfkjklsqdjf}.
		\end{proof}

		We will define $\hat{\phi}(v,x,t)$ using the implicit function theorem to be the unique $v'\in V'$ which solves the equation 
				$T(\tilde{\psi}(-v,v',x,t))=0$.
		More precisely, consider the map 
			\begin{equation*}\begin{aligned}
				\Psi:\dom(\tilde{\psi})\cap (V\times V'\times M\times \R_+^\nu)\to V\times V'\times M\times \R_+^\nu,\ \Psi(v,v',x,t)=(v,T(\tilde{\psi}(v,v',x,t)),x,t) 		
			\end{aligned}\end{equation*}
		Since $\tilde{\psi}$ and $T$ are equivariant, $\Psi$ is equivariant.
		Furthermore, for any $x\in M$, $\odif{\Psi}_{(0,0,x,0)}$ is bijective and $\Psi(0,0,x,0)=(0,0,x,0)$.
		\begin{theorem}[Global inverse mapping theorem]\label{thm:inverse_mapping_thm}
			Let $f:M_1\to M_2$ be a smooth map between smooth manifolds, $A\subseteq M_1$ a closed subset. 
			If for every $x\in A$, $\odif{f}_x$ is bijective, and $f_{|A}:A\to M_2$ is injective and proper, then 
			there exists an open neighborhood $U$ of $A$ such that $f_{|U}$ is an open smooth embedding. 
		\end{theorem}
		\begin{proof}	
			By the inverse mapping theorem, and by replacing $M_1$ with an open neighborhood of $A$, we can assume that $f$ is a local diffeomorphism.
			We only need to find an open neighborhood $U$ of $A$ such that $f_{|U}$ is injective.
			Let $K\subseteq A$ be compact.
			We can find $U_n\subseteq M_1$ a sequence of relatively compact open neighborhoods of $K$ such that $K=\bigcap_n U_n$ and $\overline{U_{n+1}}\subseteq U_n$.
			Since $\overline{U_n}$ is compact, $f_{|\overline{U_n}\cup A}$ is proper.
			We claim that there exists $n$ such that $f_{|\overline{U_n}\cup A}$ is injective.
			If not, then there exists sequences $(x_n)_{n\in \N}$, $(y_n)_{n\in \N}$ such that $x_n\neq y_n$, $x_n\in \overline{U_n}$, $y_n\in \overline{U_n}\cup A$, $f(x_n)=f(y_n)$.
			Since $U_n$ are decreasing and relatively compact, by passing to a subsequence, we can suppose that $x_n\to x\in K$.
			Since $f_{|\overline{U_n}\cup A}$ is proper and $f(y_n)$ converges to $f(x)$, we can find a subsequence of $y_n$ which converges.
			So without loss of generality, we can suppose that $y_n\to y\in A$.
			So, $f(x)=f(y)$. By injectivity of $f_{|A}$, $x=y$.
			Since $f$ is a local diffeomorphism at $x$, $x_n=y_n$ for $n$ big enough which is a contraction.

			Let $(K_n)_{n\in\N}$ be an increasing sequence of compact subsets of $M_2$ such that $M_2=\bigcup_{n\in \N}K_n$.
			We will define by recurrence an increasing sequence $(V_n)_{n\in \N}$ of relatively compact open subsets of $M_1$ such that $f_{|\overline{V_n}\cup A}$ is injective and proper, and $f^{-1}(K_n)\cap A\subseteq V_n$.
			The initial step follows from what we just proved applied to $f^{-1}(K_1)\cap A$ which is compact by properness of $f_{|A}$.
			The induction step also follows from what we just proved by replacing $A$ with $\overline{V_n}\cup A$ and using the compact set $\overline{V_n} \cup(f^{-1}(K_{n+1})\cap A)$.
			The set $\bigcup_{n}V_n$ is an open neighborhood of $A$ on which $f$ is injective.
		\end{proof}
		By \Cref{thm:inverse_mapping_thm}, $\Psi$ is an open embedding on an open neighborhood $W\subseteq V\times V'\times M\times \R_+^\nu$ of $\{(0,0,x,0):x\in M\}$.
		Let $\norm{\cdot}$ be a norm on $V$. 
		For any $t\in \R_+^\nu$, let $\norm{t}_\infty:=\max\{t_i:i\in\bb{1,\nu}\}$. 
			There exists $g:M\to \Rpt $ a continuous function such that 
				\begin{equation*}\begin{aligned}
					\left\{(v,v',x,t)\in V\times V'\times M\times \R_+^\nu:\max(\norm{v},\norm{v'},\norm{t}_\infty)<2g(x)\right\}\subseteq W
				\end{aligned}\end{equation*}	
			Let
				\begin{equation*}\begin{aligned}
					U=\{(v,v',x,t)\in V\times V'\times M\times \R_+^\nu:\max\left(\norm{\alpha_{t/g(x)}(v)},\norm{\alpha_{t/g(x)}(v')}\right)<2g(x) \}
				\end{aligned}\end{equation*}
			The set $U$ is $(\Rpt)^\nu$-equivariant open neighborhood of $V\times V'\times M\times \{0\}$.
		We claim that $\Psi$ is an open smooth embedding on $U$.
		We will show that if $(v,v',x,t),(w,w',x,t)\in U$, then there exists $\lambda\in (\Rpt)^\nu$ such that $\alpha_\lambda(v,v',x,t)$, $\alpha_\lambda(w,w',x,t)\in W$.
		Equivariance of $\Psi$ would imply that $\Psi$ is injective on $U$ and by taking $(v,v',x,t)=(w,w',x,t)$, we deduce that $U\subseteq \bigcup_{\lambda\in(\Rpt)^\nu}\alpha_\lambda(W)$
		which by equivariance of $\Psi$ implies that $\Psi$ is an open embedding on $U$.

		We now prove our claim. Let $(v,v',x,t),(w,w',x,t)\in U$. Let $\epsilon>0$ to be chosen later. 
			We take $\lambda\in (\Rpt)^\nu$ to be $\lambda_i=t_i/g(x)$ if $t_i> 0$ and $\epsilon$ if $t_i=0$.
			We have 
				\begin{equation*}\begin{aligned}
					\alpha_{\lambda}(v,v',x,t)=(\alpha_{\lambda}(v),\alpha_{\lambda}(v'),x,s),\quad \alpha_{\lambda}(w,w',x,t)=(\alpha_{\lambda}(w),\alpha_{\lambda}(w'),x,s),
				\end{aligned}\end{equation*}
			where $s_i=g(x)$ if $t_i>0$ and $s_i=0$ if $t_i=0$.
			Hence, $\norm{s}_{\infty}<2g(x)$.
			As $\epsilon\to 0^+$, $\norm{\alpha_{\lambda}(v)}\to \norm{\alpha_{t/g(x)}(v)}$. Same with $v'$, $w$, $w'$. 
			So, for $\epsilon$ small enough, $\alpha_{\lambda}(v,v',x,t),\alpha_{\lambda}(w,w',x,t)\in W$.

		We can now define $\hat{\phi}$ by the formula 
			$\Psi(-v,\hat{\phi}(v,x,t),x,t)=(-v,0,x,t)$
		whose domain is 
			\begin{equation*}\begin{aligned}
				\dom(\hat{\phi}):=\{(v,x,t)\in V\times M\times \R_+^\nu:(-v,0,x,t)\in \Psi(U)\}.
			\end{aligned}\end{equation*}
		By equivariance of $U,\Psi,T$, we deduce that $\hat{\phi}$ is equivariant.
		For any $x\in M$, since $(0,0,x,0)\in U$, we deduce that $(0,x,0)\in \dom(\hat{\phi})$ which
		by equivariance of $\dom(\hat{\phi})$ implies that $V\times M\times \{0\}\subseteq  \dom(\hat{\phi})$.
		By construction, $\hat{\phi}$ satisfies \Crefitem{lem:sklqdjfkjqsdfkjqsdmkjfkmsqdljf}{1}.
		By \eqref{eqn:product_phi_compsotion_bisub}, if $v\in V'$, then $\Psi(-v,v,x,0)=(-v,0,x,0)$. \Crefitem{lem:sklqdjfkjqsdfkjqsdmkjfkmsqdljf}{2} follows.
	\end{linkproof}

\section{Proofs of the results in Section \ref{sec:bi-graded_tangent_groupoid}}\label{sec:bi-graded_tangent_groupoid_proof}
	\begin{linkproof}{topology_mathbbG}
		Let $(V,\natural)$ and $(V',\natural')$ be $\nu$-graded basis, $X$ a topological space, $f:\mathbb{G}\to X$ a function such that $f_{|M\times M\times \R_+^\nu\setminus \{0\}}$ is continuous.
		We will show that $f\circ \cQ_V$ is continuous if and only if $f\circ \cQ_{V'}$ is continuous.
		By passing through the intermediary $\nu$-graded basis $(V\oplus V',\natural\oplus \natural')$, we can suppose without loss of generality that $(V',\natural')\subseteq (V,\natural)$.
		\begin{itemize}
			\item    Suppose $f\circ \cQ_V$ is continuous. Since $\cQ_V$ restricted to $V'\times M\times \R_+^\nu$ coincides with $\cQ_{V'}$ on their common domain,
					it follows that $f\circ \cQ_{V'}$ is continuous on a neighborhood of $V'\times \cGF^{(0)}\times \{0\}$.
					The map $f\circ \cQ_{V'}$ is continuous on $\dom(\cQ_{V'})\cap V'\times M\times \R_+^\nu\setminus \{0\}$ because $f_{|M\times M\times \R_+^\nu\setminus \{0\}}$ is continuous.
					So, $f\circ \cQ_{V'}$ is continuous on $\dom(\cQ_{V'})$.
			\item 	Suppose $f\circ \cQ_{V'}$ is continuous.
			Let $\phi$ be as in \Cref{thm:change_of_basis_map}. 
			By \eqref{eqn:comm_diag_phi_hat}, $f\circ \cQ_V$ is continuous on $\dom(\phi)$. 
			The map $f\circ \cQ_V$ is continuous on $\dom(\cQ_V)\cap V\times M\times \R_+^\nu\setminus \{0\}$ because $f_{|M\times M\times \R_+^\nu\setminus \{0\}}$ is continuous.
			Since $V\times \cGF^{(0)}\times \{0\}\subseteq \dom(\phi)$, we deduce that $f\circ \cQ_V$ is continuous on $\dom(\cQ_V)$.
		\end{itemize}
		By an identical argument, one shows that if $f:\mathbb{G}\to \C$ is a map such that $f_{|M\times M\times \R_+^\nu\setminus \{0\}}$ is smooth, then 
		$f\circ \cQ_V$ is smooth if and only if $f\circ \cQ_{V'}$ is smooth.
		This finishes the proof of \Crefitem{thm:topology_mathbbG}{invariance}.

		By the Cauchy-Lipschitz theorem, $M\times M\times \R_+^\nu\setminus \{0\}$ is an open subset of $\mathbb{G}$.
		Since $V\times M\times \R_+^\nu\setminus \{0\}$ is an open dense subset of $V\times \mathbb{G}^{(0)}$, 
		one deduces that $M\times M\times \R_+^\nu\setminus \{0\}$ is a dense subset of $\mathbb{G}$. 
		This finishes the proof of \Crefitem{thm:topology_mathbbG}{open_subset}.
		Fix $(V,\natural)$ any $\nu$-graded \textit{Lie} basis.
		Let 
			\begin{equation*}\begin{aligned}
				\cQ_V\sqcup \mathrm{Id}:\dom(\cQ_V)\sqcup (M\times M\times \R_+^\nu\backslash 0)\to \mathbb{G}
			\end{aligned}\end{equation*}
		 be the obvious surjective map.
		By definition, $\mathbb{G}$ is equipped with the quotient topology and smooth structure.
		Therefore, to prove \Crefitem{thm:topology_mathbbG}{topology} and \Crefitem{thm:topology_mathbbG}{subCartesian}, it is enough to show that the equivalence relation on $\dom(\cQ_V)\sqcup (M\times M\times \R_+^\nu\backslash 0)$ of having the same image under $\cQ_V\sqcup \mathrm{Id}$ satisfies the hypothesis of \Cref{thm:quotient_differential_space}.
		This would also imply \Crefitem{thm:topology_mathbbG}{submersion} in the case where the $\nu$-graded basis is Lie. The general case follows from the case of $\nu$-graded Lie basis by \Cref{lem:bigger_bi_graded_Lie} together with 
		\eqref{eqn:comm_diag_phi_hat} and \Cref{prop:composition_submersion}.
	
		Let
		\begin{equation*}\begin{aligned}
				R&=\{(v,\gamma,w,\eta)\in \dom(\cQ_V)\times \dom(\cQ_V):\cQ_V(v,\gamma)=\cQ_V(w,\eta)\}\\
				&=\{(v,x,t,w,x,t):e^{\alpha_t(-w)}e^{\alpha_t(v)}\cdot x= x\}\cup \{(v,L,x,0,w,L,x,0):w^{-1}\cdot v\in \natural_{x,0}^{-1}(L) \}	.	
			\end{aligned}\end{equation*}
		Since the map $\cQ_V$ is a submersion at any point in $\dom(\cQ_V)\cap (M\times M\times \R_+^\nu\backslash 0)$, see \eqref{eqn:domprime_cQV}, it suffices to prove: 
		\begin{enumerate}
			\item     The set $R$ is closed in $\dom(\cQ_V)\times \dom(\cQ_V)$.
		\item 	The projection $(v,\gamma,w,\eta)\in R\mapsto (v,\gamma)\in V\times \mathbb{G}^{(0)}$ is a submersion at every point in the set $R\cap (V\times \cG^{(0)}\times \{0\}\times V\times \cG^{(0)}\times \{0\})$.
		\end{enumerate}
					\begin{lem}[Period bounding lemma \cite{period2}]\label{lem:periodbounding}
				If $X\in \cX(M)$, $K\subseteq M$ a compact subset, then there exists $\epsilon>0$ such that for any $x\in K$, $\tau\in ]0,\epsilon[$, if $e^{\tau X}\cdot x=x$, then $X(x)=0$.
			\end{lem}
			\begin{proof}
				If $\epsilon$ doesn't exist, then there exists $(x_n,\epsilon_n)_{n\in \N}\subseteq K\times \Rpt$ such that $\epsilon_n\to 0$, $X(x_n)\neq 0$ and $e^{\epsilon_n X}\cdot x_n=x_n$. 
				By passing to a subsequence, we can suppose that $x_n\to x\in K$.
				We fix some chart near $x$. 
				Since the map $\chi:\R\times M\to M$ which maps $(\tau,x)$ to $e^{\tau X}\cdot x$ is a smooth map on its domain of definition, we can suppose that the orbit $e^{\tau X}\cdot x_n$ for $0\leq \tau\leq \epsilon_n$ lies in the fixed chart.
				The function 
					\begin{equation*}\begin{aligned}
					\R\to\R,\quad    \tau\mapsto \langle e^{\tau X}\cdot x_n,X(x_n)\rangle
					\end{aligned}\end{equation*}
				is a smooth function whose values at $0$ and at $\epsilon_n$ agree. 
				Here, $\langle \cdot,\cdot\rangle$ is the inner product inside the coordinate chart we fixed.
				By Rolle's theorem, there exists $0<\tau_n<\epsilon_n$ such that 
					\begin{equation*}\begin{aligned}
						\langle X(e^{\tau_nX}\cdot x_n),X(x_n)\rangle=0.
					\end{aligned}\end{equation*}
				Since $X(x_n)\neq 0$, we can by taking a subsequence suppose that 
					$\frac{X(x_n)}{\norm{X(x_n)}}\to v\neq 0$.
					Since 
					$X(e^{\tau_nX}\cdot x_n)=\odif{\chi}_{(\tau_n,x_n)}(0,X(x_n))$
				and since $\odif{\chi}_{(\tau_n,x_n)}$ converges to the identity map as $n\to +\infty$, it follows that 
	\begin{equation*}\begin{aligned}
							\frac{X(e^{\tau_nX}\cdot x_n)}{\norm{X(x_n)}}\to v.		
	\end{aligned}\end{equation*}
				Hence, $\langle v,v\rangle=0$ which is a contraction.			
			\end{proof}
		\begin{lem}[{\cite{Debord2013}}]\label{thm:periodicbounding}
			Let $\norm{\cdot}$ a Euclidean norm on $V$, $K\subseteq M$ compact. 
			There exists $\epsilon>0$ such that for any $x\in K$, $v\in V$, if 
			$e^v\cdot x=x$, then either $\natural(v)(x)=0$ or $\norm{v}\geq \epsilon$.
		\end{lem}
		\begin{proof}
			On $V\times M$, there is a natural vector field $X(v,x)=(0,\natural(v)(x))$
			whose flow is $e^{\tau X}\cdot (v,x)=(v,e^{\tau v}\cdot x)$.
			The result follows from \Cref{lem:periodbounding} applied to $X$ and $ \{v\in V:\norm{v}\leq 1\}\times K$.
		\end{proof}
		
				We now prove that $R$ is closed. It is clear that $R$ is closed away from $t=0$. 
		By \Crefitem{thm:convergence_mathbbG0}{part2}, $\{(v,L,x,0,w,L,x,0):w^{-1}\cdot v\in \natural_{x,0}^{-1}(L) \}$ is closed.
		So, we need to show that if $(v_n)_{n\in \N},(w_n)_{n\in \N}\subseteq V$, $(x_n)_{n\in \N}\subseteq M$, $(t_n)_{n\in \N}\subseteq \R_+^\nu\backslash 0$ are sequences such that 
			\begin{equation*}\begin{aligned}
				v_n\to v,\ w_n\to w,\ x_n\to x,\ t_n\to 0,\ \ker(\natural_{x_n,t_n})\to \natural_{x,0}^{-1}(L),\ e^{\alpha_{t_n}(-w_n)}e^{\alpha_{t_n}(v_n)}\cdot x_n= x_n,
			\end{aligned}\end{equation*}
		 then $w^{-1}\cdot v\in \natural_{x,0}^{-1}(L)$.
		Let $\hat{\psi}$ as in \eqref{eqn:composition_bisub_tilde_psi}. 
		Since $(-w_n,v_n,x_n,t_n)\to (-w,v,x,0)\in \dom(\hat{\psi})$, we deduce that for $n$ big enough, $(-w_n,v_n,x_n,t_n)\in \dom(\hat{\psi})$.
		By \eqref{eqn:product_phi_compsotion_bisub}, $e^{\alpha_{t_n}(\hat{\psi}(-w_n,v_n,x_n,t_n))}\cdot x_n=x_n$.
		Since $\hat{\psi}(-w_n,v_n,x_n,t_n)\to \hat{\psi}(-w,v,x,0)=w^{-1}\cdot v$, we deduce from \Cref{thm:periodicbounding} that for $n$ big enough, $\hat{\psi}(-w_n,v_n,x_n,t_n)\in \ker(\natural_{x_n,t_n})$.
		By \Cref{prop:grass_prop}, $w^{-1}\cdot v\in \natural_{x,0}^{-1}(L)$. This finishes the proof that $R$ is closed.

		Let $v_0,w_0\in V$, $(L_0,x_0)\in \cG^{(0)}$ such that $w_0^{-1}\cdot v_0\in \natural_{x_0,0}^{-1}(L_0)$.
		Consider the smooth manifold 
			\begin{equation*}\begin{aligned}
				A=\{(v,w,L,x,t)\in V\times V\times \Grass(V)\times M\times \R_+^\nu:(-w,v,x,t)\in \dom(\hat{\psi}),\hat{\psi}(-w,v,x,t)\in L\}.		
			\end{aligned}\end{equation*}
		Take any smooth map $g:V\to \R^{\dim(V)-\dim(M)}$ which has the property that its restriction to $v_0\cdot \natural_{x_0,0}^{-1}(L_0)$ is a local diffeomorphism at $w_0$.
		The differential of the map 
			\begin{equation*}\begin{aligned}
				A\to V\times \Grass(V)\times M\times \R_+^\nu\times \R^{\dim(V)-\dim(M)},\quad (v,w,L,x,t)\mapsto (v,L,x,t,g(w))		
			\end{aligned}\end{equation*}
		at $(v_0,w_0,\natural_{x_0,0}^{-1}(L_0),x_0,0)$ is a bijection. It follows that it is a local diffeomorphism at the point $(v_0,w_0,\natural_{x_0,0}^{-1}(L_0),x_0,0)$.
		Therefore, by restricting to $\mathbb{G}^{(0)}$ using the embedding $\Grass(\natural)$ from \eqref{eqn:inclusion_mathbbG0}, the map $(v,\gamma,w,\eta)\in R\mapsto (v,\gamma,g(w))\in V\times \mathbb{G}^{(0)}\times \R^{\dim(V)-\dim(M)}$ is a local diffeomorphism at $(v_0,L_0,x_0,0,w_0,L_0,x_0,0)$.
		This finishes the proof of \Cref{thm:topology_mathbbG}.
	\end{linkproof}

		 \begin{linkproof}{convergence_mathbbG}
			Both parts are proved similarly.
			One direction follows from continuity of $\cQ_{V}$ for any $\nu$-graded basis $(V,\natural)$ which follows from \Crefitem{thm:topology_mathbbG}{invariance}.
			The other follows from the fact that $\cQ_V$ is open which follows from \Crefitem{thm:topology_mathbbG}{submersion}.
		 \end{linkproof}

	\begin{linkproof}{convergence_mathbbG_practical}
		By \Cref{thm:convergence_mathbbG}, it suffices to prove that if $(y_n,x_n,t_n)\to (\natural_{x,0}(v)L,x,0)$ and $(x_n,t_n)\to (L,x,0)$, then the equation $e^{\alpha_{t_n}(v_n)}\cdot x_n=y_n$ admits a solution in $S$ for $n$ big enough that satisfies $v_n\to v$.
		By \Cref{lem:bigger_bi_graded_Lie}, we can suppose that $(V,\natural)$ is a $\nu$-graded Lie basis.
		Notice that by our hypothesis on $S$, $\dim(S)=\dim(M)$.
		The map 
			\begin{equation*}\begin{aligned}
				S\times \natural_{x,0}^{-1}(L)\to V,\quad (s,w)\mapsto sw		
			\end{aligned}\end{equation*}
		is a diffeomorphism because of \Cref{rem:natural_x_0_is_a_lie_algebra_homo}, and that $\natural_{x,0}^{-1}(L)$ is a subgroup of $V$ which contains the subgroup $\ker(\natural_{x,0})$.
		Let $E\to \Grass(V)$ be the tautological real vector bundle whose fiber over $L'\in \Grass(V)$ is $L'$, and $\hat{\psi}$ as in \eqref{eqn:composition_bisub_tilde_psi}.
		By our hypothesis on $S$ and \eqref{eqn:product_phi_compsotion_bisub}, the map 
			\begin{equation}\label{eqn:qsdfkkqsdflkqsdmlfkmùlqsdklfmqskdùf}\begin{aligned}
					S\times E\times M\times \R_+^\nu\to  V\times \Grass(V)\times M\times \R_+^\nu,\quad (s,l,L',y,t)\mapsto (\hat{\psi}(s,l,y,t),L',y,t)
			\end{aligned}\end{equation}
		is a local diffeomorphism at $(v,0,\natural_{x,0}^{-1}(L),x,0)$.
		Now, let $v_n\in V$ be a sequence such that $y_n=e^{\alpha_{t_n}(v_n)}\cdot x_n$ and $v_n\to v$ which exists by \Cref{thm:convergence_mathbbG}.
		So, for $n$ big enough, $(v_n,\ker(\natural_{x_n,t_n}),x_n,t_n)$ is in the image of \eqref{eqn:qsdfkkqsdflkqsdmlfkmùlqsdklfmqskdùf}.
		Hence, there exists a unique $s_n\in S$ and $l_n\in \natural_{x_n,t_n}^{-1}(\ker(\natural_{x_n,t_n}))$ such that $\hat{\psi}(s_n,l_n,x_n,t_n)=v_n$.
		Furthermore, $(s_n,l_n,\ker(\natural_{x_n,t_n}),x_n,t_n)\to (v,0,\natural_{x,0}^{-1}(L),x,0)$.
		So, $s_n\to v$.
		By \eqref{eqn:product_phi_compsotion_bisub}, $y_n=e^{\alpha_{t_n}(v_n)}\cdot x_n=e^{\alpha_{t_n}(s_n)}e^{\alpha_{t_n}(l_n)}\cdot x_n=e^{\alpha_{t_n}(s_n)}\cdot x_n$.
	\end{linkproof}
		 \begin{linkproof}{tangent_groupoid_is_smooth}
		The map $s$ is smooth because $s\circ \cQ_V:\dom(\cQ_V)\subseteq V\times \mathbb{G}^{(0)}\to \mathbb{G}^{(0)}$ is the projection onto the second coordinate. It is a submersion by \Cref{prop:composition_submersion} and \Crefitem{thm:topology_mathbbG}{submersion}.
		The identity map $u:\mathbb{G}^{(0)}\to \mathbb{G}$ is smooth because $u(\gamma)=\cQ_V(0,\gamma)$ and $\cQ_V$ is smooth. 
		Let us check that the inverse map $\iota:\mathbb{G}\to \mathbb{G}$ is smooth.
		It suffices to prove that $\iota\circ \cQ_{V}$ is smooth. 
		By \Cref{lem:r_smooth}, the following map is smooth 
			\begin{equation}\label{eqn:hat_iota}\begin{aligned}
					\hat{\iota}:\dom(\cQ_V)\to  V\times \mathbb{G}^{(0)},\quad \hat{\iota}(v,\gamma)=(-v,r\circ \cQ_V(v,\gamma)).
			\end{aligned}\end{equation}
		One can easily check that the diagram 
			\begin{equation}\label{eqn:comm_diag_hat_iota}\begin{aligned}
				\begin{tikzcd}
					\dom(\cQ_V)\arrow[d,"\cQ_{V}"]\arrow[r,"\hat{\iota}"]&V\times \mathbb{G}^{(0)}\arrow[d,"\cQ_{V}"]\\\mathbb{G}\arrow[r,"\iota"]&\mathbb{G}
					\end{tikzcd}
			\end{aligned}\end{equation}
		commutes. It follows that $\iota$ is smooth. Since $r=s\circ \iota$, it follows that $r$ is smooth.

		Let us show that the multiplication map $m:\mathbb{G}^{(2)}\to \mathbb{G}$ is smooth.
		The restriction of $m$ to the $M\times M\times \R_+^\times$ part of $\mathbb{G}$ is smooth. 
		So, we need to show that the map
			\begin{equation*}\begin{aligned}
				\{(v,\gamma,w,\eta)\in \dom(\cQ_V)\times \dom(\cQ_V):(\cQ_V(v,\gamma),\cQ_V(w,\eta))\in \mathbb{G}^{(2)}\}\to \mathbb{G}\\ (v,\gamma,w,\eta)\mapsto \cQ_V(v,\gamma)\cQ_V(w,\eta)				
			\end{aligned}\end{equation*}
		is smooth. This map is precisely the map $\cQ_V^2$ from \eqref{eqn:QV2_is_QV_times_QV}.
		Near $t=0$, it is smooth by \eqref{eqn:diag_composition_bisub}.
		Away from $t=0$, it is smooth by Cauchy-Lipschitz theorem, see \eqref{eqn:cQv2}.
		Hence, $m$ is smooth.

		It remains to show that $\ker(\odif{s})$ is the same as $A(\mathbb{G})$ constructed in \Cref{sec:cotangent_cone}.
		They clearly are equal as sets.
		Let $(V,\natural)$ be a $\nu$-graded basis. The map $\odif{\cQ_V}:\ker(\odif{(s\circ \cQ_V)})\to \cQ_V^*(\ker(\odif{s}))$ is a smooth submersion because $s$ and $\cQ_V$ are submersions.
		We now restrict it to $\{0\}\times \mathbb{G}^{(0)}$, and we get the map $A(\natural)$ from \eqref{eqn:Anatural}. The result follows.
		\end{linkproof}
		\begin{linkproof}{submersion_cQ}
		By applying \Cref{prop:cor_triangle_isom} to the submersions $\cQ_V:\dom(\cQ_V)\to \mathbb{G}$ and $s:\mathbb{G}\to \mathbb{G}^{(0)}$, 
		we deduce that for any $k\in \C$, we have a natural isomorphism 
			\begin{equation}\label{eqn:skqodjfosdjf1}\begin{aligned}
				\Omega^k(V)\simeq \Omega^k(\ker(\odif{\cQ_V}))\otimes \cQ_V^*(\Omega^{k,0}(\mathbb{G})).
			\end{aligned}\end{equation}
		Let $\hat{\iota}$ as in \eqref{eqn:hat_iota}.
		By \eqref{eqn:comm_diag_hat_iota}, the following diagram commutes 
		\begin{equation} \begin{tikzcd}
			V\times \mathbb{G}^{(0)} \arrow[rd,"r\circ \cQ_V"']\arrow[r, "\hat{\iota}"]    &  V\times \mathbb{G}^{(0)}  \arrow[d,"s\circ \cQ_V"]\\
										  & \mathbb{G}^{(0)} 
		\end{tikzcd}.\end{equation}
		Since $\hat{\iota}$ is an open embedding, $\ker(\odif{(r\circ \cQ_V)})\simeq \hat{\iota}^*(\ker(\odif{(s\circ \cQ_V)}))\simeq \hat{\iota}^*(V)\simeq V$.
		By \Cref{prop:cor_triangle_isom} applied to $r$ and $\cQ_V$, we deduce that 
		\begin{equation}\label{eqn:skqodjfosdjf2}\begin{aligned}
			\Omega^k(V)\simeq \Omega^k(\ker(\odif{\cQ_V}))\otimes \cQ_V^*(\Omega^{0,k}(\mathbb{G})).
		\end{aligned}\end{equation}
		Taking the tensor product of \eqref{eqn:skqodjfosdjf1} and \eqref{eqn:skqodjfosdjf2}, the result follows.
				\end{linkproof}
\section{Proofs of the results in Section \ref{sec:equivariance_of_norms}}\label{sec:equivariance_of_norms_proof}

		\begin{linkproof}{quasi-Lie-Debord Skandalis}
			The map $\alpha_\lambda:\mathbb{G}\to \mathbb{G}$ is well-defined because of \Cref{rem:map_alpha_lambda_G0_well_defined}.	
			It is easily seen to be a groupoid automorphism.
			So, we only need to show that it is smooth.	
			The following diagram commutes 
			\begin{equation} \begin{tikzcd}
				V\times \mathbb{G}^{(0)}\arrow[d,"\cQ_V"]\arrow[r, "\alpha_\lambda"]    &V\times \mathbb{G}^{(0)}    \arrow[d,"\cQ_V"]\\
								\mathbb{G}			\arrow[r, "\alpha_\lambda"]& \mathbb{G}
			\end{tikzcd}\end{equation}
			Hence, $\alpha_\lambda$ is smooth.	
		\end{linkproof}

	\begin{linkproof}{L1normRiemMetric}
		Let $(V,\natural)$ be a $\nu$-graded basis. 
		We need to show that $\omega\circ \cQ_V$ is a smooth section of $\cQ^*_V(\Omega^{\frac{1}{2},-\frac{1}{2}}(\mathbb{G}))$.
		By \Cref{thm:submersion_cQ}, we have a natural isomorphism 
				$\C\simeq \cQ^*_V(\Omega^{\frac{1}{2},-\frac{1}{2}}(\mathbb{G}))$.
		We define a smooth function $f:\tdom(\cQ_V)\to \C$ as follows:
		If $(v,x,t)\in \tdom(\cQ_V)$, then $y\mapsto e^{\alpha_{t}(v)}\cdot y$ is a local diffeomorphism at $x$ which maps $x$ to $e^{\alpha_{t}(v)}\cdot x$.
		So, its differential at $x$ is an isomorphism $L:T_xM\to T_{e^{\alpha_{t}(v)}}M$.
		Let 
			\begin{equation*}\begin{aligned}
				f(v,x,t)=L^*\left(\vol_{e^{\alpha_{t}(v)}\cdot x}^\frac{1}{2}\right)\vol_{x}^{-\frac{1}{2}}
			\end{aligned}\end{equation*}
		The function $f$ is clearly smooth.
		One can check that $f\circ \pi_{V\times \mathbb{G}^{(0)}}:\dom(\cQ_V)\to \C$ corresponds to $w\circ \cQ_V$ using the isomorphism from \Cref{thm:submersion_cQ}.
		This shows that $\omega$ is smooth.
		It is obvious that $\omega$ satisfies \eqref{eqn:omega_symmetry_condition} on $M\times M\times \R_+^\nu\backslash 0$.
		By density, it follows that $\omega$ satisfies \eqref{eqn:omega_symmetry_condition}.
	\end{linkproof}

\section{Proofs of the results in Section \ref{sec:smooth_functions_vanishing_Fourier_Transform}}\label{sec:smooth_functions_vanishing_Fourier_Transform_proof}
	
\begin{linkproof}{algebra_structure_invariant_functions}
		For any $\nu$-graded basis $(V,\natural)$, let
			\begin{equation*}\begin{aligned}
				\mathcal{A}_V:=C^\infty_c(\tdom(\cQ_V),\Omega^1(V))\times C^\infty_c(M\times M\times \R_+^\nu\setminus \{0\},\Omega^\frac{1}{2})
			\end{aligned}\end{equation*}
		equipped with the product topology.
		Given two $\nu$-graded basis $(V,\natural)$ and $(V',\natural')$, we will construct a continuous linear map 
			$T_{V',V}:\mathcal{A}_V\to \mathcal{A}_{V'}$
		such that $q_{V'}\circ T_{V',V}=q_V$.
		Constructing the map $T_{V',V}$ finishes the proof of \Crefitem{thm:algebra_structure_invariant_functions}{invariance}.
		Using the $\nu$-graded basis $(V\oplus V',\natural\oplus \natural')$, we only need to define $T_{V\oplus V',V}$ and $T_{V',V\oplus V'}$ because then we can take $T_{V',V}:=T_{V',V\oplus V'}\circ T_{V\oplus V',V}$.
		So, without loss of generality we can suppose that $(V',\natural')\subseteq (V,\natural)$. 
		We will define $T_{V,V'}$ and $T_{V',V}$.
		Let $\tilde{\phi}$ and $\phi$ be as in \Cref{thm:change_of_basis_map}.
		The following diagram commutes 
		\begin{equation*} \begin{tikzcd}
			\dom(\tilde{\phi})\subseteq V\times M\times \R_+^\nu\arrow[rd,"\pi"']\arrow[r, "\tilde{\phi}"]    & V'\times M\times \R_+^\nu   \arrow[d,"\pi'"]\\
										  &M\times \R_+^\nu 
		\end{tikzcd},\end{equation*} 
		where $\pi,\pi'$ are the obvious projections.
		By \Cref{prop:cor_triangle_isom}, 
			$\Omega^1(\ker(\odif{\tilde{\phi}}))\simeq \Omega^1(V/V')$.
		In particular, we have an integration along the fibers map 
			\begin{equation*}\begin{aligned}
				\tilde{\phi}_*:C^\infty_c(\dom(\phi),\Omega^1(V))\to C^\infty_c(V'\times M\times \R_+^\nu,\Omega^1(V'))	
			\end{aligned}\end{equation*}
		Since $\phi$ is the pullback of $\tilde{\phi}$ as in \Cref{lem:pullback_local_diff}, it follows that we have an integration map 
			\begin{equation*}\begin{aligned}
				\phi_*:C^\infty_c(\dom(\phi),\Omega^1(V))\to C^\infty_c(V'\times \mathbb{G}^{(0)},\Omega^1(V'))
			\end{aligned}\end{equation*}
			which satisfies
			\begin{equation}\label{eqn:phi_hat_phi_tilde_pushforward}\begin{aligned}
				\phi_*(g\circ \pi_{V\times \mathbb{G}^{(0)}})=\tilde{\phi}_*(g)\circ \pi_{V'\times \mathbb{G}^{(0)}},\quad \forall g\in C^\infty_c(\dom(\tilde{\phi}),\Omega^1(V)).		
			\end{aligned}\end{equation}
		Furthermore, by \eqref{eqn:comm_diag_phi_hat}, 
			\begin{equation}\label{eqn:qksjdkofjqsodjfkqsjdjfkqs}\begin{aligned}
				\cQ_{V'*}\circ \phi_*=\cQ_{V*}.		
			\end{aligned}\end{equation}
		We can now define $T_{V,V'}$ and $T_{V',V}$.
		\begin{itemize}
			\item   Let $\chi\in C^\infty(V\times M\times \R_+^\nu)$ equal to $1$ in a neighborhood of $V\times M\times \{0\}$ and $\supp(\chi)\subseteq \dom(\tilde{\phi})$.
					By \eqref{eqn:phi_hat_phi_tilde_pushforward} and \eqref{eqn:qksjdkofjqsodjfkqsjdjfkqs}, 
					we can take 
						\begin{equation*}\begin{aligned}
							T_{V',V}(g,h)=\left(\tilde{\phi}_*(\chi g),h+\cQ_{V*}(((1-\chi)g)\circ \pi_{V\times \mathbb{G}^{(0)}})\right).
						\end{aligned}\end{equation*}
			\item  	Let $\chi\in C^\infty(V'\times M\times \R_+^\nu)$ equal to $1$ in a neighborhood of $V'\times M\times \{0\}$ and $\supp(\chi)\subseteq \im(\tilde{\phi})$.
					Let $k\in C^\infty(\dom(\tilde{\phi}),\Omega^1(V/V'))$ such that $\tilde{\phi}_{|\supp(k)}:\supp(k)\to \im(\tilde{\phi})$ is proper and 
					$\tilde{\phi}_*(k)\in C^\infty(\im(\tilde{\phi}))$ is equal to $1$.
					The function $k$ exists because $\tilde{\phi}$ is a submersion.
					We define 
						\begin{equation*}\begin{aligned}
							T_{V,V'}(g,h)=\left(k ((\chi g)\circ \tilde{\phi}),h+\cQ_{V'*}(((1-\chi)g)\circ \pi_{V'\times \mathbb{G}^{(0)}})\right),
						\end{aligned}\end{equation*}
						where $k ((\chi g)\circ \tilde{\phi})$ is the pointwise multiplication of $k$ and $(\chi g)\circ \tilde{\phi}\in C^\infty(\dom(\tilde{\phi}),\Omega^1(V'))$.
					To see that $q_V\circ T_{V',V}=q_{V'}$, let $l=k( (\chi g)\circ \tilde{\phi})$. We have
						\begin{equation*}\begin{aligned}
							\cQ_{V*}\left(l\circ \pi_{V\times \mathbb{G}^{(0)}}\right)&=\cQ_{V'*}\left( \phi_*\left(l\circ \pi_{V\times \mathbb{G}^{(0)}}\right)\right)&&\text{by }\eqref{eqn:qksjdkofjqsodjfkqsjdjfkqs}\\
							&=\cQ_{V'*}\Big( \phi_*\Big((k\circ \pi_{V\times \mathbb{G}^{(0)}})  ((\chi g)\circ \pi_{V'\times \mathbb{G}^{(0)}}\circ \phi)\Big)\Big)&&\text{by }\eqref{eqn:ksqldjflqdsjfmqsdjf_2}\\
							&=\cQ_{V'*}\left( \phi_*\left(k\circ \pi_{V\times \mathbb{G}^{(0)}} \right) (\chi g)\circ \pi_{V'\times \mathbb{G}^{(0)}}\right)&&\\
							&=\cQ_{V'*}\left(( \tilde{\phi}_*(k) \chi g)\circ \pi_{V'\times \mathbb{G}^{(0)}}\right)&&\text{by }\eqref{eqn:phi_hat_phi_tilde_pushforward}\\
							&=\cQ_{V'*}\left((\chi g)\circ \pi_{V'\times \mathbb{G}^{(0)}}\right).
						\end{aligned}\end{equation*}
					
		\end{itemize}		
		We have thus proved \Crefitem{thm:algebra_structure_invariant_functions}{invariance}.
		Continuity of the inclusion $\cinv\hookrightarrow \cbbG$ is obvious.
		Since $\cinv$ is a quotient of an LF space, we only need to check that $\ker(q_V)$ is closed which is straightforward to see.
		This finishes the proof of \Crefitem{thm:algebra_structure_invariant_functions}{space}.
		By equivariance of $\cQ_V$, we have
		$\alpha_{\lambda}(q_V(g,h))=q_V(\alpha_\lambda(g),\alpha_{\lambda}(h))$. 
		This proves \Crefitem{thm:algebra_structure_invariant_functions}{dilation}.
		\begin{lem}\label{lem:change_open_set_invariant_function}
			Let $(V,\natural)$ be a $\nu$-graded basis, $f\in \cinv$. 
			If $U\subseteq \tdom(\cQ_V)$ is an open neighborhood of $V\times M\times \{0\}$, then there exists $(\tilde{f},\dbtilde{f})\in q_V^{-1}(f)$ such that
			 $\supp(\tilde{f})\subseteq U$.
		\end{lem}
		\begin{proof}
			Let $\chi\in C^\infty(V\times M\times \R_+^\nu)$ which is equal to $1$ in a neighborhood of $V\times M\times \{0\}$ and $\supp(\chi)\subseteq U$.
			If $(\tilde{f},\dbtilde{f})\in q_V^{-1}(f)$ 
			then, $\cQ_{V*}((\tilde{f}(1-\chi))\circ \pi_{V\times \mathbb{G}^{(0)}})\in C^\infty_c(M\times M\times \R_+^\nu\setminus \{0\},\omegahalf)$.
			So, $(\tilde{f}\chi,\dbtilde{f}+\cQ_{V*}((\tilde{f}(1-\chi))\circ \pi_{V\times \mathbb{G}^{(0)}}))\in q_V^{-1}(f)$.
		\end{proof}
		We fix a $\nu$-graded Lie basis $(V,\natural)$.
		The space $C^\infty_c(M\times M\times \R_+^\nu\setminus \{0\},\omegahalf)$ is two-sided ideal of $\cbbG$.
		So, to prove \Crefitem{thm:algebra_structure_invariant_functions}{algebra},	we only need to consider elements of the form $\cQ_{V*}(g\circ \pi_{V\times \mathbb{G}^{(0)}})$.
		By \Cref{lem:change_open_set_invariant_function}, we can suppose that $\supp(g)\subseteq \tdom(\cQ_V)$.
		Let $\hat{\iota}$ as in \eqref{eqn:hat_iota}.
		We also define 
			\begin{equation*}\begin{aligned}
				\tilde{\iota}:\tdom(\cQ_V)\to V\times M\times \R_+^\nu,\quad \tilde{\iota}(v,x,t)=(-v,e^{\alpha_{t}(v)}\cdot x,t).
			\end{aligned}\end{equation*}
		Both $\tilde{\iota}$ and $\hat{\iota}$ are open embeddings, and $\hat{\iota}$ is the pullback of $\tilde{\iota}$ as in \Cref{lem:pullback_local_diff}. 
		So, by \eqref{eqn:integration_pullback} and \eqref{eqn:comm_diag_hat_iota},
			\begin{equation*}\begin{aligned}
				\iota_*(\cQ_{V*}(g\circ \pi_{V\times \mathbb{G}^{(0)}}))=\cQ_{V*}(\hat{\iota}_*(g\circ \pi_{V\times \mathbb{G}^{(0)}}))=\cQ_{V*}(\tilde{\iota}_*(g)\circ \pi_{V\times \mathbb{G}^{(0)}}).
			\end{aligned}\end{equation*}
		Hence, $\cinv$ is closed under taking adjoints.
		
		Let us prove that $\cinv$ is closed under taking convolution.
		Let $\tilde{\psi},\psi$ as in \Cref{thm:composition_bisub}.
		By \eqref{eqn:composition_bisub_tilde_psi}, the following diagram commutes:
		\begin{equation*} \begin{tikzcd}
			\dom(\tilde{\psi})\subseteq V\times V\times M\times \R_+^\nu\arrow[rd,"\pi"']\arrow[r, "\tilde{\psi}"]    & V\times M\times \R_+^\nu  \arrow[d,"\pi'"]\\
										  &M\times \R_+^\nu
		\end{tikzcd},\end{equation*} 
		where $\pi,\pi'$ are the obvious projections. By \Cref{prop:cor_triangle_isom}, we deduce that $\ker(\odif{\tilde{\psi}})\simeq \Omega^1(V)$.
		Hence, we have an integration along the fibers map 
			\begin{equation*}\begin{aligned}
				\tilde{\psi}_*:C^\infty_c(\dom(\tilde{\psi}),\Omega^2(V))\to C^\infty_c(V\times M\times \R_+^\nu,\Omega^1(V)).		
			\end{aligned}\end{equation*}
		Since $\psi$ is the pullback of $\tilde{\psi}$, by \Cref{lem:pullback_local_diff}, we have an integration along the fibers map 
		\begin{equation*}\begin{aligned}
			\psi_*:C^\infty_c(\dom(\psi),\Omega^2(V))\to C^\infty_c(V\times \mathbb{G}^{(0)},\Omega^1(V))
		\end{aligned}\end{equation*}
		which satisfies
			\begin{equation}\label{eqn:integration_pullback_hat_psi}\begin{aligned}
				\psi_*(g\circ \pi_{V\times \mathbb{G}^{(0)}})=\tilde{\psi}_*(g)\circ \pi_{V\times \mathbb{G}^{(0)}},\quad \forall g\in C^\infty_c(\dom(\tilde{\psi}),\Omega^2(V)).		
			\end{aligned}\end{equation}
		By  \eqref{eqn:diag_composition_bisub} and \Cref{prop:cor_triangle_isom}, we deduce that 
			\begin{equation*}\begin{aligned}
				\Omega^1(\ker(\odif{\cQ_V^2}))\simeq \psi^*(\Omega^1(\ker(\odif{\cQ_V})))\otimes \Omega^1(\ker(\odif{\psi}))\simeq \psi^*(\Omega^1(\ker(\odif{\cQ_V})))\otimes \Omega^1(V).
			\end{aligned}\end{equation*}
		We deduce from pulling back \eqref{eqn:integration_denstiy_iso_exponential} with $k=l=\frac{1}{2}$ using $\psi$ that
			\begin{equation*}\begin{aligned}
				\cQ_V^{2*}(\Omega^{\frac{1}{2}}(\mathbb{G}))\otimes \Omega^1(\ker(\odif{\cQ_V^2}))	&\simeq \psi^{*}(\cQ_V^*(\Omega^{\frac{1}{2}}(\mathbb{G})))\otimes \Omega^1(\ker(\odif{\cQ_V^2}))\\
				&\simeq  \psi^{*}(\cQ_V^*(\Omega^{\frac{1}{2}}(\mathbb{G})))\otimes\psi^*(\Omega^1(\ker(\odif{\cQ_V})))\otimes \Omega^1(V)\\
				&\simeq \psi^*(\Omega^1(V))\otimes \Omega^1(V)\simeq \Omega^2(V).
			\end{aligned}\end{equation*}
		In particular, we have an integration along the fibers map 
			\begin{equation*}\begin{aligned}
				\cQ_{V*}^{2}:C^\infty_c(\dom(\cQ_V^2),\Omega^2(V))\to \cbbG	
			\end{aligned}\end{equation*}		
		which by \eqref{eqn:diag_composition_bisub} satisfies 
			\begin{equation}\label{eqn:klqsjdkfjqsmd}\begin{aligned}
						\cQ_{V*}^2=\cQ_{V*}\circ \psi_*
			\end{aligned}\end{equation}
			In the following lemma, we insist on equivariance. More precisely, we prove that $U$ satisfies \eqref{eqn:stron_equivariance_U}.
			 This won't be needed in the proof of this theorem, but it will be needed later on in the proof of \Cref{thm:equivariance_Cn_norm} and \Cref{thm:estimate_weakly_commut_convolution}.
		\begin{lem}\label{lem:support_convolution_invariant_functions}
			For any $(\Rpt)^\nu$-equivariant open neighborhood $L\subseteq V\times V\times M\times \R_+^\nu$ of $V\times V\times M\times \{0\}$,
			there exists an open neighborhood $U\subseteq \tdom(\cQ_V)$ of $V\times M\times \{0\}$ such that if $(w,x,t)\in U$ and $(v,e^{\alpha_{t}(w)}\cdot x,t)\in U$, then $(v,w,x,t)\in L$. Furthermore, $U$ satisfies the following: 
				\begin{equation}\label{eqn:stron_equivariance_U}\begin{aligned}
					\forall v\in V,x\in M,\lambda,t\in \R_+^\nu, \text{ if }(v,x,\lambda t)\in U,\text{ then }(\alpha_\lambda(v),x,t)\in U.
				\end{aligned}\end{equation}
			In particular, $U$ is $(\Rpt)^\nu$-equivariant.
		\end{lem}
		\begin{proof}
			Let $\norm{\cdot}$ be a Euclidean norm on $V$.
			Since $\tdom(\cQ_V)$ and $L$ are equivariant open neighborhoods of $\{0\}\times M\times \{0\}\subseteq V\times M\times \R_+^\nu$ and $\{0\}\times \{0\}\times M\times \{0\}\subseteq V\times V\times  M\times \R_+^\nu$ respectively, it 
			follows that there exists a continuous function $\chi:M\to \Rpt$ such that 
				\begin{equation*}\begin{aligned}
					&\{(v,x,t)\in V\times M\times \R_+^\nu:\norm{\alpha_t(v)}<2\chi(x)\}\subseteq \tdom(\cQ_V)\\ &\{(v,w,x,t)\in V\times M\times \R_+^\nu:\max(\norm{\alpha_t(v)},\norm{\alpha_t(w)})<2\chi(x)\}\subseteq L.
				\end{aligned}\end{equation*}
			By a compactness argument, there exists a continuous function $\kappa:M\to \Rpt$ such that for any $x\in M$, $v\in V$ if $\norm{v}< \kappa(x)$, then $\chi(e^{v}\cdot x)<2\chi(x)$.
			We take $U=\{(v,x,t):\norm{\alpha_t(v)}<\min(\chi(x),\kappa(x))\}$.
		\end{proof}
		By \Cref{lem:support_convolution_invariant_functions} applied to $L=\dom(\tilde{\psi})$, we get an equivariant open set $U\subseteq \tdom(\cQ_V)$.
		Let $g_1,g_2\in C^\infty_c(\tdom(\cQ_V),\Omega^1(V))$ such that $\supp(g_1)\subseteq U$ and $\supp(g_2)\subseteq U$.
		We define
			\begin{equation*}\begin{aligned}
				\quad h(v,w,x,t)=g_1(v,e^{\alpha_{t}(w)}\cdot x,t)g_2(w,x,t)\in C^\infty_c(\dom(\tilde{\psi}),\Omega^2(V))		
			\end{aligned}\end{equation*}	
		We have
			\begin{equation}\label{eqn:skijokfojspqkodfjpqskdjfqjspdf}\begin{aligned}
				\cQ_{V*}(g_1\circ \pi_{V\times \mathbb{G}^{(0)}})	\ast \cQ_{V*}(g_2\circ \pi_{V\times \mathbb{G}^{(0)}})&=	\cQ_{V*}^2(h\circ \pi_{V\times V\times \mathbb{G}^{(0)}})&&\text{by }\eqref{eqn:QV2_is_QV_times_QV}\\
				&=\cQ_{V*}(\psi_*(h \circ \pi_{V\times V\times \mathbb{G}^{(0)}}))&&\text{by }\eqref{eqn:klqsjdkfjqsmd}\\
				&=\cQ_{V*}(\tilde{\psi}_*(h)\circ \pi_{V\times \mathbb{G}^{(0)}})&&\text{by }\eqref{eqn:integration_pullback_hat_psi}
			\end{aligned}\end{equation}
		So, $\cinv$ is closed under convolution.
	\end{linkproof}
	\begin{rem}\label{rem:product_and _adjoint_invariant_functions}
		In the proof of \Crefitem{thm:algebra_structure_invariant_functions}{algebra}, using \eqref{eqn:product_phi_compsotion_bisub} and \eqref{eqn:composition_bisub_tilde_psi}, we proved that if $(V,\natural)$ is a $\nu$-graded Lie basis, then
		\begin{itemize}
			\item     for any $g_1,g_2\in C^\infty_c(\tdom(\cQ_V),\Omega^1(V))$, we can find a $g\in C^\infty_c(\tdom(\cQ_V),\Omega^1(V))$ such that 
				\begin{equation*}\begin{aligned}
					\cQ_{V*}(g_1\circ\pi_{V\times \mathbb{G}^{(0)}})\ast \cQ_{V*}(g_2\circ \pi_{V\times \mathbb{G}^{(0)}})- \cQ_{V*}(g\circ \pi_{V\times \mathbb{G}^{(0)}})\in C^\infty_c(M\times M\times \R^\nu_+\backslash 0,\omegahalf),\\
					g(v,x,0)=\int_{w\in V}g_1(w,x,0)g_2(w^{-1}v,x,0),\quad \forall v\in V,x\in M
				\end{aligned}\end{equation*}			
			\item for any $g\in C^\infty_c(\tdom(\cQ_V),\Omega^1(V))$, we can find an $h\in C^\infty_c(\tdom(\cQ_V),\Omega^1(V))$ such that 
				\begin{equation*}\begin{aligned}
					\cQ_{V*}(g\circ\pi_{V\times \mathbb{G}^{(0)}})^*- \cQ_{V*}(h\circ \pi_{V\times \mathbb{G}^{(0)}})\in C^\infty_c(M\times M\times \R^\nu_+\backslash 0,\omegahalf)\\
					h(v,x,0)=\overline{g(-v,x,0)},\quad \forall v\in V,x\in M
				\end{aligned}\end{equation*}
		\end{itemize}
	\end{rem}
	
	\begin{linkproof}{compatability_cinvdiv_cinv}
		Since $\delta_{u\circ \pi_{\mathbb{G}^{(0)}}}^*=\delta_{\overline{u}\circ \pi_{\mathbb{G}^{(0)}}}$ and $\theta_k(X)^*=-\theta_k(X)$, and using \Crefitem{thm:algebra_structure_invariant_functions}{algebra},
		we only need to consider 
		right multiplication by invariant functions.
		By \Crefitem{prop:proper_support_distribution}{2}, $ C^\infty_c(M\times M\times \R_+^\nu\backslash 0,\omegahalf)$ is a two-sided ideal in $C^{-\infty}_{r,s}(\mathbb{G},\omegahalf)$.
		So in the proof of \Cref{thm:compatability_cinvdiv_cinv}, we only need to consider invariant functions of the form $\cQ_{V*}(g\circ \pi_{V\times \mathbb{G}^{(0)}})$.
		Let $(V,\natural)$ be a $\nu$-graded Lie basis. We have	
				\begin{equation*}\begin{aligned}
				\delta_{u\circ \pi_{\mathbb{G}^{(0)}}}\ast \cQ_{V*}(g\circ \pi_{V\times \mathbb{G}^{(0)}})=\cQ_{V*}(l\circ \pi_{V\times \mathbb{G}^{(0)}}),
			\end{aligned}\end{equation*}
		where $l(v,x,t)=u(e^{\alpha_{t}(v)}\cdot x,t) g(v,x,t)$.
		This proves \Crefitem{thm:compatability_cinvdiv_cinv}{inclusion_smooth_cinvdiv}.
		For \Crefitem{thm:compatability_cinvdiv_cinv}{inclusion_vectorfield_cinvdiv}, let $k\in \Z_+^\nu$ and $X\in \cF^{k}$.
		By the definition of a $\nu$-graded basis, we can find a finite family $(f_i)_{i\in I}\subseteq C^\infty(M)$, $(v_{i})_{i\in I}\subseteq V$ such that $v_i\in V^{a(i)}$ with $a(i)\preceq k$ and 
			$\natural(X)=\sum_{i\in I}	f_i \natural(v_i)$.
		In particular, 
			\begin{equation}\label{eqn:qskdofjpokqsdjfmjkmds}\begin{aligned}
				\theta_{k}(X)=		\sum_{i\in I}	\delta_{\left(f_i t^{k-a(i)}\right)\circ \pi_{\mathbb{G}^{(0)}}}\ast \theta_{a(i)}(\natural(v_i)).	
			\end{aligned}\end{equation}
		So, by \Crefitem{thm:compatability_cinvdiv_cinv}{inclusion_smooth_cinvdiv}, we need to show that if $v_0\in V^{k}$, then $\theta_k(\natural(v_0))\ast \cinv\subseteq \cinv$.
		Let $X=\natural(v_0)$, $\tilde{\psi}$ and $\psi$ be as in \Cref{thm:composition_bisub}, $\tilde{X},\hat{X}$ be partially-defined vector fields on $V\times M\times \R_+^\nu$ and $V\times \mathbb{G}^{(0)}$ respectively by the formulas 
			\begin{equation}\label{eqn:kqlsdjkfmjqsmldjflmjsqdklmfjkmlsdqjkmlfj}\begin{aligned}
				\tilde{X}(v,x,t)= \odv*{\tilde{\psi}(\tau v_0,v,x,t)}{\tau}_{\tau=0},\quad 		\hat{X}(v,\gamma)= \odv*{\psi(\tau v_0,v,\gamma)}{\tau}_{\tau=0}.
			\end{aligned}\end{equation}
		Here, we use \Crefitem{thm:composition_bisub}{4} to ensure that $\tilde{X}$ and $\hat{X}$ are vector fields.
		The vector fields $\tilde{X}$ and $\hat{X}$ are well-defined in open neighborhoods of $V\times M\times \{0\}$ and $V\times \cGF^{(0)}\times \{0\}$.
		We claim that 
			\begin{equation}\label{eqn:qklsdfjas}\begin{aligned}
				\odif{\cQ_V}(\hat{X})=\theta_{k}(X)_R\circ \cQ_V		
			\end{aligned}\end{equation}
		where $\theta_{k}(X)_R$ is the right-invariant vector field associated to $ \theta_{k}(X)$ by \Cref{prop:left_right_invariant_sections}.
		By density of $V\times M\times \R_+^\nu\backslash 0$ in $V\times \mathbb{G}^{(0)}$, it is enough to check \eqref{eqn:qklsdfjas} for $t\neq 0$.
		We have for $t\neq 0$,
			\begin{equation*}\begin{aligned}
				\odif{\cQ_V}(\hat{X})(v,x,t)&=\odv*{\cQ_V(\psi(\tau v_0,v,x,t))}{\tau}_{\tau=0}\\
				&=	\odv*{(e^{\tau t^{k} v_0}e^{\alpha_{t}(v)}\cdot x,x,t)}{\tau}_{\tau=0}=\theta_{k}(X)_R(e^{\alpha_t(v)}\cdot x,x,t)=\theta_{k}(X)_R(\cQ_V(v,x,t)),
			\end{aligned}\end{equation*}
		where we used \eqref{eqn:diag_composition_bisub} in the second equality and \eqref{eqn:left_right_invariant_vector_fields_example} in the third equality.
		So, it follows that if $g\in C^\infty_c(\tdom(\cQ_V),\Omega^1(V))$ and $\supp(g)$ is a subset of the domain of definition of $\tilde{X}$ which we can suppose by \Cref{lem:change_open_set_invariant_function}, then
			\begin{equation}\label{eqn:qjkshdmfjkqsdjfjqmsdfjmqsdjfmqsd}\begin{aligned}
		\theta_{k}(X)\ast \cQ_{V*}(g\circ \pi_{V\times \mathbb{G}^{(0)}})=		\cL_{\theta_{k}(X)_R}\left(\cQ_{V*}(g\circ \pi_{V\times \mathbb{G}^{(0)}})\right)
		&=\cQ_{V*}\left(\cL_{\hat{X}}( g\circ \pi_{V\times \mathbb{G}^{(0)}})\right)\\
		&=\cQ_{V*}\left(\cL_{\tilde{X}}( g)\circ \pi_{V\times \mathbb{G}^{(0)}}\right)
			\end{aligned}\end{equation}
		This finishes the proof of \Crefitem{thm:compatability_cinvdiv_cinv}{inclusion_vectorfield_cinvdiv}.
	\end{linkproof}
	\begin{rem}\label{rem:more_details_on_proof}
		The proof of \Crefitem{thm:compatability_cinvdiv_cinv}{inclusion_vectorfield_cinvdiv} shows that if $(V,\natural)$ is a $\nu$-graded Lie basis, $v_0\in V^k$, and $f\in C^\infty_c(\tdom(\cQ_V),\Omega^1(V))$, then
		we can find $g\in C^\infty_c(\tdom(\cQ_V),\Omega^1(V))$ such that 
			\begin{equation*}\begin{aligned}
				\cQ_{V*}(g\circ \pi_{V\times \mathbb{G}^{(0)}})-\theta_{k}(\natural(v_0))\ast \cQ_{V*}(f\circ \pi_{V\times \mathbb{G}^{(0)}})		\in C^\infty_c(M\times M\times \R^\nu_+\backslash 0,\omegahalf)\\
				g(v,x,0)=\mathcal{L}_{X}(f)(v,x,0),
			\end{aligned}\end{equation*}
		where $X\in \cX(V)$ is the right-invariant vector field on $V$ associated to $v_0$.
	\end{rem}

	\begin{linkproof}{equivariance_Cn_norm}
		Let $(V,\natural)$ be a $\nu$-graded Lie basis.
		A critical observation in the proof is that $\alpha_\lambda:V\to V$ is well-defined for $\lambda\in \R_+^\nu$.
		By taking adjoint, it suffices to only consider the function 
				$ \lambda\mapsto  g\ast\alpha_\lambda(f)$.		
		We will show that for each $C>0$, the function $\lambda\mapsto g\ast \alpha_\lambda(f)$ extends smoothly to $\lambda\in [0,C]^\nu$.
		So, we fix $C>0$.
		We first show that the family $(g\ast\alpha_\lambda(f))_{\lambda\in ]0,C]^\nu}$ is uniformly compactly supported.
		By \eqref{eqn:supp_product_functions} and \eqref{eqn:support_automorphism}, 
			$\supp(g\ast \alpha_\lambda(f))\subseteq \supp(g) \alpha_\lambda(\supp(f))$.
		Therefore, it suffices to show that $\overline{\bigcup_{\lambda\in ]0,C]^\nu}  \alpha_\lambda(\supp(f))}$ is a proper subset of $\mathbb{G}$.
		\begin{lem}\label{thm:proper_alpha_lambda_mu}
			If $K\subseteq \mathbb{G}$ is a compact subset, then
					$\overline{\bigcup_{\lambda\in ]0,C]^\nu}  \alpha_\lambda(K)}$ 
			 is a proper subset of $\mathbb{G}$.
			 Furthermore, if $K\subseteq M\times M\times \R_+^\nu\backslash 0$, then $\overline{\bigcup_{\lambda\in ]0,C]^\nu}  \alpha_\lambda(K)} \subseteq M\times M\times \R_+^\nu\backslash 0$.
		\end{lem}
		\begin{proof}
			By symmetry between $s$ and $r$, it is enough to show that if $(\gamma_n)_{n\in \N}\subseteq \overline{\bigcup_{\lambda\in ]0,C]^\nu}  \alpha_\lambda(K)}$ is a sequence such that $s(\gamma_n)$ converges in $\mathbb{G}^{(0)}$, then $(\gamma_n)_{n\in \N}$ has a convergent subsequence.
			Since $\mathbb{G}$ is metrizable, we can without loss of generality suppose that $(\gamma_n)_{n\in \N}\subseteq \bigcup_{\lambda\in ]0,C]^\nu}  \alpha_\lambda(K)$.
			By passing to a subsequence, we can suppose that $\gamma_n=\alpha_{\lambda_n}(\eta_n)$ with $\eta_n\in K$ converges to $\eta\in K$ and  $\lambda_n\in ]0,C]^\nu$ converges to $\lambda\in [0,C]^\nu$.
			We can also suppose that either $\eta_n\in M\times M\times \R_+^\nu\backslash 0$ for all $n\in \N$, or $\eta_n\in \cG^{(0)}\times \{0\}$ for all $n\in \N$.
			The proof is now divided into the two cases:
			\begin{enumerate}
				\item     If $\eta_n=(A_n,x_n,0)$ converges to $\eta=(A,x,0)$, then by \Crefitem{thm:convergence_mathbbG}{2}, $x_n\to x$, $A_n=\natural_{x_n,0}(v_nL_n)$, $A=\natural_{x,0}(vL)$ where $L_n\in \Grass(V)$, $v_n\in V$, $L_n\to L$ and $v_n\to v$.
						Convergence of $s(\alpha_{\lambda_n}(\eta_n))$ in $\mathbb{G}^{(0)}$ implies that $\alpha_{\lambda_n}(L_n)$ converges to some subspace $L'\in \Grass(V)$. It follows that $\alpha_{\lambda_n}(\eta_n)=(\natural_{x_n,0}(\alpha_{\lambda_n}(v_n)\alpha_{\lambda_n}(L_n)),x_n,0)\to (\natural_{x,0}(\alpha_\lambda(v)L'),x,0)$. 
				\item If $\eta_n=(y_n,x_n,t_n)$, then either $\eta\in M\times M\times \R_+^\nu\backslash 0$ or $\eta\in \cG^{(0)}\times \{0\}$. 
				\begin{enumerate}
					\item If $\eta=(y,x,t)$, then $y_n\to y$, $x_n\to x$, $t_n\to t$. Convergence of $s(\alpha_{\lambda_n}(\eta_n))$ in $\mathbb{G}^{(0)}$ implies that $\lambda_n^{-1}t_n\to s$ for some $s\in \R_+^\nu$.
					This implies that $t=\lambda s$. So, $s\neq 0$. Hence, $\alpha_{\lambda_n}(y_n,x_n,t_n)=(y_n,x_n,\lambda_n^{-1}t_n)\to (y,x,s)$.
					\item If   $\eta=(A,x,0)$, then by \Crefitem{thm:convergence_mathbbG}{sequence}, $x_n\to x$, $t_n\to 0$, $\ker(\natural_{x_n,t_n})\to L$ in $\Grass(V)$, $y_n=\exp(\alpha_{t_n}(v_n))\cdot x_n$ for some $v_n\in V$, $v_n\to v$, $A=\natural_{x,0}(vL)$.
							Convergence of $s(\alpha_{\lambda_n}(\eta_n))$ in $\mathbb{G}^{(0)}$ implies that $\ker(\natural_{x_n,\lambda_{n}^{-1}t_n})\to L'$ for some $L'\in \Grass(V)$.
							Since $y_n=\exp(\alpha_{\lambda_{n}^{-1}t_n}(\alpha_{\lambda_n}(v_n)))\cdot x_n$ and $\alpha_{\lambda_n}(v_n)\to \alpha_{\lambda}(v)$, it follows from \Crefitem{thm:convergence_mathbbG}{sequence} that $(y_n,x_n,\lambda_{n}^{-1}t_n)$ converges to $(\natural_{x,0}(\alpha_{\lambda}(v)L'),x,0)$.
				\end{enumerate}
				If $K\subseteq M\times M\times \R_+^\nu\backslash 0$, then only Case 2.a occurs. \qedhere
			\end{enumerate}
		\end{proof}
		By \Cref{thm:proper_alpha_lambda_mu}, we only need to show that the function $\lambda\mapsto g\ast \alpha_\lambda(f)$ extends to a smooth function $[0,C]^\nu\to C^\infty(\mathbb{G},\omegahalf)$.
		By decomposing $f$ as in \eqref{eqn:sum_smooth_decompo_invariatn_function}, there are two cases to consider:
		\begin{itemize}
			\item   The first case is when $f\in C^\infty_c(M\times M\times \R_+^\nu\backslash 0,\omegahalf)$.
					By \Cref{thm:proper_alpha_lambda_mu}, there exists a compact subset $K\subseteq M\times M\times \R_+^\nu\backslash 0$ such that $\supp(g\ast \alpha_\lambda(f))\subseteq K$ for all $\lambda\in ]0,C]^\nu$.
					On $M\times M\times \R_+^\nu\backslash 0$, we have 
						\begin{equation*}\begin{aligned}
								 g\ast \alpha_\lambda(f)(z,x,t)=\int_M g(z,y, t)f(y,x,\lambda t)		
						\end{aligned}\end{equation*}
					This clearly extends smoothly to $[0,C]^\nu$. 
					Furthermore, $\alpha_{0}(f)=0$. So, \eqref{eqn:value_of_alpha_00} follows.
			\item The second case is when $f=\cQ_{V*}(\tilde{f}\circ \pi_{V\times \mathbb{G}^{(0)}})$ where $\tilde{f}\in C^\infty_c(\tdom(\cQ_V),\Omega^1(V))$.
			We have
				\begin{equation*}\begin{aligned}
					g\ast f (\gamma)=\int_V h(v,\gamma)g(\gamma \cQ_V(v,s(\gamma))^{-1})\tilde{f}(v,\pi_{\mathbb{G}^{(0)}}\circ s(\gamma))	,
				\end{aligned}\end{equation*}
			where $h(v,\gamma):\omegahalf A(\mathbb{G})_{r(\cQ_V(v,s(\gamma)))} \to \omegahalf A(\mathbb{G})_{s(\gamma)}$ is the canonical isomorphism which comes from $\cQ_V^*(\Omega^{\frac{1}{2},0}(\mathbb{G}))\simeq \cQ_V^*(\Omega^{0,\frac{1}{2}}(\mathbb{G}))$, see \Cref{thm:submersion_cQ}.
			Let $\beta_{\lambda}:M\times \R_+^\nu\to M\times \R_+^\nu$ for $\lambda\in \R_+^\nu$ be the map $(x,t)\mapsto (x,\lambda t)$.
			Since $\alpha_\lambda(f)=\cQ_{V*}(\alpha_\lambda(\tilde{f})\circ \pi_{V\times \mathbb{G}^{(0)}})$,
			we have
				\begin{equation*}\begin{aligned}
					g\ast \alpha_\lambda(f) (\gamma)&=\int_V h(v,\gamma)g(\gamma \cQ_V(v,s(\gamma))^{-1}) \alpha_\lambda(\tilde{f})(v,\pi_{\mathbb{G}^{(0)}}\circ s(\gamma))\\
					&=	\int_V h(\alpha_\lambda(v),\gamma)g(\gamma \cQ_V(\alpha_\lambda(v),s(\gamma))^{-1}) \tilde{f}(v,\beta_{\lambda}\circ \pi_{\mathbb{G}^{(0)}}\circ s(\gamma)),		
				\end{aligned}\end{equation*}
			This clearly extends to $\lambda\in [0,C]^\nu$. Furthermore, 
				$g\ast \alpha_{0}(f)(\gamma)=	g(\gamma)	\int_V  \tilde{f}(v,\beta_{0}\circ \pi_{\mathbb{G}^{(0)}}\circ s(\gamma))$.
			This proves \eqref{eqn:value_of_alpha_00} in this case.
		\end{itemize}
		We have thus proved that $g\ast \alpha_\lambda(f)$ extends to a smooth map $\R_+^\nu\to \cbbG$.
		Now suppose that $g\in \cinv$.
		Again by decomposing $f$ and $g$ as in \eqref{eqn:sum_smooth_decompo_invariatn_function}, we reduce to the case where 
		$f=\cQ_{V*}(\tilde{f}\circ \pi_{V\times \mathbb{G}^{(0)}})$ and $g=\cQ_{V*}(\tilde{g}\circ \pi_{V\times \mathbb{G}^{(0)}})$ where $\tilde{f},\tilde{g}\in C^\infty_c(\tdom(\cQ_V),\Omega^1(V))$.
		Let $\psi$ be as in \Cref{thm:composition_bisub}, $\tilde{\psi}_\lambda$ the map $(v,w,x,t)\mapsto\psi(v,\alpha_\lambda(w),x,t)$.
		\begin{lem}\label{lem:qjksdfqhsdljkflqsldqsdfqsdf}
			There exists an open $(\Rpt)^\nu$-equivariant neighborhood $L\subseteq \dom(\tilde{\psi})$ of $V\times V\times M\times \{0\}$ such that for any $\lambda\in [0,C]^\nu$, $\tilde{\psi}_\lambda$
			is a submersion on $L_\lambda:=\{(v,w,x,t):(v,\alpha_\lambda(w),x,t)\in L\}$.
		\end{lem}
		\begin{proof}
			Without loss of generality, we suppose $1\leq C$.
			Let $L$ be the set of points $(v,w,x,t)$ such that for any $\lambda\in [0,C]^\nu$, $(v,\alpha_\lambda(w),x,t)\in \dom(\tilde{\psi})$ and the differential of $\tilde{\psi}$ at $(v,\alpha_\lambda(w),x,t)$ is surjective as a map $V\times \{0\}\times T_xM\times \R\to V\times T_xM\times \R$.
			The set $L\subseteq \dom(\tilde{\psi})$ because $1\leq C$.
			It is equivariant by equivariance of $\tilde{\psi}$.
			It is open by compactness of $[0,C]^\nu$.
			By \Crefitem{thm:composition_bisub}{3}, it contains $V\times V\times M\times \{0\}$.
		\end{proof}
		We apply \Cref{lem:support_convolution_invariant_functions} with $L$ to get an equivariant open set $U\subseteq \tdom(\cQ_V)$.
		By \Cref{lem:support_convolution_invariant_functions}, 
		we can suppose that $\tilde{f}$ and $\tilde{g}$ are supported in $U$. 
		Since $U$ is equivariant, $\alpha_\lambda(\tilde{f})$ is also supported in $U$ for all $\lambda\in (\Rpt)^\nu$. 
		Consider
			\begin{equation*}\begin{aligned}
				&h_{\lambda}(v,w,x,t)=\tilde{g}(v,e^{\alpha_{t}(w)}\cdot x,t)\alpha_\lambda(\tilde{f})(w,x,t)\in C^\infty_c(L,\Omega^2(V)), &&\lambda\in ]0,C]^\nu	 \\
				&k_{\lambda}(v,w,x,t)=\tilde{g}(v,e^{\alpha_{\lambda t}(w)}\cdot x,t)\tilde{f}(w,x,\lambda t)\in C^\infty_c(L_\lambda,\Omega^2(V)), &&\lambda\in [0,C]^\nu.
			\end{aligned}\end{equation*}
		We remark that to show that $k_\lambda$ is supported in $L_\lambda$, we use \eqref{eqn:stron_equivariance_U}.
		As in \eqref{eqn:skijokfojspqkodfjpqskdjfqjspdf}, we have
			\begin{equation}\label{eqn:qjskdofjoksqdjfjqsdjfmqsidjf}\begin{aligned}
				g\ast \alpha_\lambda(f)=\cQ_{V*}(\tilde{\psi}_* (h_{\lambda})\circ \pi_{V\times \mathbb{G}^{(0)}}).
			\end{aligned}\end{equation}
		So, we need to prove that $\lambda\in ]0,C]^\nu\to C^\infty_c(\tdom(\cQ_V),\Omega^1(V))$ extends smoothly to $[0,C]\nu$.	
		To this end,
			\begin{equation}\label{eqn:psi_h_alpha_lambda}\begin{aligned}
				\tilde{\psi}_* (h_{\lambda})(v',x,t)&=\int_{\tilde{\psi}(v,w,x,t)=(v',x,t)}	\tilde{g}(v,e^{\alpha_{t}(w)}\cdot x,t)\alpha_\lambda(\tilde{f})(w,x,t)\\
				&=\int_{\tilde{\psi}(v,\alpha_\lambda(w),x,t)=(v',x,t)}	\tilde{g}(v,e^{\alpha_{\lambda t}(w)}\cdot x,t)\tilde{f}(w,x,\lambda t)=\tilde{\psi}_{\lambda*}(k_\lambda)(v',x,t).
			\end{aligned}\end{equation}
		The map $\tilde{\psi}_{\lambda}$ and $k_\lambda$ extend smoothly to $\lambda\in [0,C]^\nu$.
		Since the maps $(k_\lambda)_{[0,C]^\nu}$ are uniformly compactly supported, the result follows.
	\end{linkproof}

\section{Proofs of the results in Section \ref{sec:bi-grading_algebra_diff}}\label{sec:bi-grading_algebra_diff_proof}
	\begin{linkproof}{inclusion_diff_op}

		Let $f\in \cbbG$. 
		By writing $D$ as sum of monomials, and using  \Cref{ex:Right_Left_invariant_vector_field_Pair}, we see that $\theta_{k}(D)\ast f (y,x,t)=t^k D_y(f)(y,x,t)$, 
		where $D_y$ means that $D$ acts on the $y$-variable.
		It follows that $\theta_{k}(D)\ast f $ on $M\times M\times \R_+^\nu\backslash 0$ doesn't depend on the presentation of $D$ as sum of monomials.
		By density of  $M\times M\times \R_+^\nu\backslash 0$ in $\mathbb{G}$, independence of the presentation of $D$ as sum of monomials follows.
		For any $k\in \Z_+^\nu$, the distribution $\theta_k(\delta_1)$ is in the center of $\cinv$, see \Crefitem{ex:distributions_quasi_Lie_groupods}{dirac}. 
		From this \eqref{eqn:theta_product} follows.
	\end{linkproof}
	
	\begin{linkproof}{invariant_vanish_algebra}
		If $f\in C^\infty_c(M\times M\times (\Rpt)^\nu,\omegahalf)$ and $l\in \Z_+^\nu$, then $f=\theta_l(\delta_1)\ast g$, where $g\in C^\infty_c(M\times M\times (\Rpt)^\nu,\omegahalf)$ is defined by $g(y,x,t)=t^{-l}f(y,x,t)$.
		This proves \Crefitem{thm:invariant_vanish_algebra}{inclusion}.
		Let us prove $\cinvf{k}^*=\cinvf{k}$.
		We need to show that if $f\in \cinv$, $D\in \DO^k(M)$ for some $k\in \Z_+^\nu$, then $f\ast \theta_k(D)\in \cinvf{k}$.
		By writing $D$ as sum of monomials, we can suppose that $D$ is a monomial.
		By using the fact that $\theta_l(\delta_1)$ is in the center of $\cinvdis$ for any $l\in \Z_+^\nu$, we deduce that it suffices to show that if $X\in \cF^k$, then $f\ast \theta_k(X)\in \cinvf{k}$.
		Let $(V,\natural)$ be a $\nu$-graded Lie basis.
		We fix a linear basis of $V$ of pure elements $v_1,\cdots,v_n$ such that $v_i\in V^{\omega(i)}$ for all $i$.
		By the definition of a $\nu$-graded basis, we can write $\natural(X)=\sum_{i=1}^n f_i\natural(v_i)$ with $f_i=0$ if $\omega(i)\npreceq k$.
		So, $\theta_k(X)=\sum_{i=1}^n \theta_{k-\omega(i)}(\delta_1) \ast\delta_{f_i\circ \pi_{\mathbb{G}^{(0)}}}\ast\theta_{\omega(i)}(v_i)$. Hence,
			\begin{equation*}\begin{aligned}
				f\ast \theta_k(X)=\sum_{i=1}^n	\theta_{k-\omega(i)}(\delta_1) \ast f\ast \delta_{f_i\circ \pi_{\mathbb{G}^{(0)}}}\ast \theta_{\omega(i)}(\natural(v_i)).
			\end{aligned}\end{equation*}
		By replacing $f$ with $f\ast \delta_{f_i\circ \pi_{\mathbb{G}^{(0)}}}$ (here we use \Crefitem{thm:compatability_cinvdiv_cinv}{inclusion_smooth_cinvdiv}), we deduce that it suffices to show that for any $i\in \bb{1,n}$, if $f\in \cinv$, then $f\ast \theta_{\omega(i)}(\natural(v_i))\in \cinvf{\omega(i)}$.

		Let $\tilde{\psi}$ be as in \Cref{thm:composition_bisub}.
		For each $i\in \bb{1,n}$, as in \eqref{eqn:kqlsdjkfmjqsmldjflmjsqdklmfjkmlsdqjkmlfj}, let $\tilde{X}_i,\tilde{Y}_i$ be the partially-defined vector fields on $V\times M\times \R_+^\nu$ defined by
			\begin{equation}\label{eqn:tildeX_vector_fields}\begin{aligned}
				&\tilde{X}_i(v,x,t)= \odv*{\tilde{\psi}(\tau v_i,v,x,t)}{\tau}_{\tau=0}, 		&&\tilde{Y}_i(v,x,t)= \odv*{\tilde{\psi}(v,\tau v_i,x,t)}{\tau}_{\tau=0}
			\end{aligned}\end{equation}
		where we use \Crefitem{thm:composition_bisub}{4} to make sure that $\tilde{X}_i$ and $\tilde{Y}_i$ are vector fields.
		The vector fields $\tilde{X}_i$ and $\tilde{Y}_i$ are tangent to $V$.
		By equivariance of $\tilde{\psi}$, the vector fields are well-defined in an equivariant neighborhood of $V\times M\times \{0\}$, and $\alpha_\lambda(\tilde{X}_i)=\lambda^{\omega(i)}\tilde{X}_i$ and $\alpha_\lambda(\tilde{Y}_i)=\lambda^{\omega(i)}\tilde{Y}_i$.
		By \Crefitem{thm:composition_bisub}{3}, on $V\times M\times \{0\}$,
		$\tilde{X}_i$ and $\tilde{Y}_i$ are the right and left invariant vector fields on $V$ associated to $v_i$.
		So, on $V\times M\times \{0\}$, the vector fields $\tilde{X}_1,\cdots,\tilde{X}_n$ are linearly independent.
		Let $U\subseteq V\times M\times \R_+^\nu$ be the set of point at which the vector fields $\tilde{X}_1,\cdots,\tilde{X}_n$ are well-defined and linearly independent, and the vector fields $\tilde{Y}_1,\cdots,\tilde{Y}_n$ are well-defined.
		The set $U$ is an equivariant neighborhood of $V\times M\times \{0\}$.
		On $U$, the vector fields $\tilde{X}_1,\cdots,\tilde{X}_n$ linearly span the vector fields tangent to $V$, so there exists a unique family $(P_{i,j})_{i,j\subseteq \bb{1,n}}\subseteq C^\infty(U)$ such that 
			\begin{equation*}\begin{aligned}
				\tilde{Y}_i=\sum_{j=1}^n P_{i,j}\tilde{X}_j.		
			\end{aligned}\end{equation*}
		Again, by equivariance, $\alpha_{\lambda}(P_{i,j})=\lambda^{\omega(i)-\omega(j)}P_{i,j}$.
		On $U$, the differential operators $\cL_{P_{i,j}\tilde{X}_j}(\cdot)$ and $\cL_{\tilde{X}_j}(P_{i,j}\cdot )$ acting on $C^\infty(U,\Omega^1(V))$ have the same classical principal symbol.
		So, there exists a smooth function $H_{i,j}\in C^\infty(U)$ such that 
			\begin{equation*}\begin{aligned}
				\cL_{P_{i,j}\tilde{X}_j}(g)=\cL_{\tilde{X}_j}(P_{i,j}g)+H_{i,j}g,\quad \forall g\in C^\infty(U,\Omega^1(V)).
			\end{aligned}\end{equation*}
		By equivariance, $\alpha_{\lambda}(H_{i,j})=\lambda^{\omega(i)}H_{i,j}$.
		By replacing $U$ with a smaller open neighborhood of $V\times M\times \{0\}$, we can suppose that $U$ satisfies \eqref{eqn:stron_equivariance_U}, see proof of \Cref{lem:support_convolution_invariant_functions}. 
		\begin{lem}\label{lem:equiv_functions_U}
			For any $k\in \Z^\nu$, let $\max(k,0):=(\max(k_1,0),\cdots,\max(k_\nu,0))\in \Z_+^\nu$. If $P\in C^\infty(U)$ satisfies $\alpha_\lambda(P)=\lambda^k(P)$ for all $\lambda\in (\Rpt)^\nu$, then there exists $P'\in C^\infty(U)$ such that $P(v,x,t)=t^{\max(k,0)}P'(v,x,t)$.
 		\end{lem}
		\begin{proof}
			If $t\in (\Rpt)^\nu$, let $\tilde{t},\dbtilde{t}\in (\Rpt)^\nu$ be defined by 
				\begin{equation*}\begin{aligned}
					\tilde{t}_i=\begin{cases}
					t_i \quad \text{if }k_i\geq 0\\
					1 \quad \text{if }k_i< 0
					\end{cases},\quad  \dbtilde{t}_i=\begin{cases}
						1 \quad \text{if }k_i\geq 0\\
						t_i \quad \text{if }k_i< 0
						\end{cases}.
				\end{aligned}\end{equation*}
			We have $\tilde{t}\dbtilde{t}=t$. So, $P(v,x,t)=\alpha_{\tilde{t}}(P)(\alpha_{\tilde{t}}(v),x,\dbtilde{t})=\tilde{t}^kP(\alpha_{\tilde{t}}(v),x,\dbtilde{t})=t^{\max(k,0)}P(\alpha_{\tilde{t}}(v),x,\dbtilde{t})$. We can thus take $P'(v,x,t)=P(\alpha_{\tilde{t}}(v),x,\dbtilde{t})$
			which extends smoothly to $t\in \R_+^\nu$.
			Notice that $P'$ is well-defined on $U$ because of \eqref{eqn:stron_equivariance_U}.
		\end{proof}
		Now, we will show that $f\ast \theta_{\omega(i)}(\natural(v_i))\in \cinvf{\omega(i)}$. 
		By decomposing $f$ as in \eqref{eqn:sum_smooth_decompo_invariatn_function}, we have two cases to consider:
		     If $f\in C^\infty_c(M\times M\times \R_+^\nu\backslash 0,\omegahalf)$, then 
			$f\ast \theta_{\omega(i)}(\natural(v_i))(y,x,t)=t^{\omega(i)}\natural(v_i)_x(f)(y,x,t)$.
		Let $g\in C^\infty_c(M\times M\times \R_+^\nu\backslash 0,\omegahalf)$ be defined by $g(y,x,t)=\natural(v_i)_x(f)(y,x,t)$.
		Therefore,
			$f\ast \theta_{\omega(i)}(\natural(v_i))=\theta_{\omega(i)}(\delta_1)\ast g\in \cinvf{\omega(i)}$.
		
			The second case is when $f=\cQ_{V*}(\tilde{f}\circ \pi_{V\times \mathbb{G}^{0}})$ for some $\tilde{f}\in C^\infty_c(\tdom(\cQ_V),\Omega^1(V))$.
			We can suppose by \Cref{lem:change_open_set_invariant_function}, that $\supp(\tilde{f})\subseteq U$.
			Like in \eqref{eqn:qjkshdmfjkqsdjfjqmsdfjmqsdjfmqsd}, we have 
				\begin{equation*}\begin{aligned}
				&f\ast \theta_{\omega(i)}(\natural(v_i))\\
				&=\cQ_{V*}(\cL_{\tilde{Y}_{i}}(\tilde{f})\circ \pi_{V\times \mathbb{G}^{(0)}})\\
				&=\cQ_{V*}\left(\cL_{\sum_{j=1}^n P_{i,j}\tilde{X}_{j}}(\tilde{f})\circ \pi_{V\times \mathbb{G}^{(0)}}\right)\\
				&=\cQ_{V*}\left(\left(\sum_{j=1}^n \cL_{\tilde{X}_{j}}(P_{i,j}\tilde{f})+ H_{i,j}\tilde{f}\right)\circ \pi_{V\times \mathbb{G}^{(0)}}\right)\\
				&=\cQ_{V*}\left(\left(\sum_{j=1}^n \cL_{\tilde{X}_{j}}(t^{\max(\omega(i)-\omega(j),0)}P_{i,j}'\tilde{f})+ t^{\omega(i)}H'_{i,j}\tilde{f}\right)\circ \pi_{V\times \mathbb{G}^{(0)}}\right)\\
				&=\sum_{j=1}^n  \theta_{\omega(j)}(X_j)\ast \theta_{\max(\omega(i)-\omega(j),0)}(\delta_1)\ast \cQ_{V*}\left(\left(P_{i,j}'\tilde{f}\right)\circ \pi_{V\times \mathbb{G}^{(0)}}\right)\\
				&+ \sum_{j=1}^n\theta_{\omega(i)}(\delta_1)\ast \cQ_{V*}\left(\left(H_{i,j}'\tilde{f}\right)\circ \pi_{V\times \mathbb{G}^{(0)}}\right).
				\end{aligned}\end{equation*}
				By \eqref{eqn:theta_product_cinfk}, the first sum belongs to $\cinvf{\omega(j)+\max(\omega(i)-\omega(j),0)}$ and the second belongs to $\cinvf{\omega(i)}$.
				Since, $\omega(i)\preceq \omega(j)+\max(\omega(i)-\omega(j),0)$ for all $i,j\in\bb{1,n}$, by \eqref{eqn:theta_product_cinfk}, the first sum also belongs to $\cinvf{\omega(i)}$.
				This finishes the proof of  $\cinvf{k}^*=\cinvf{k}$ for all $k\in \Z_+^\nu$.
  The identity $\cinvf{k}^*=\cinvf{k}$ together with \Crefitem{thm:algebra_structure_invariant_functions}{algebra} imply \Crefitem{thm:invariant_vanish_algebra}{alternate}.
		The inclusion $\cinvf{k}\ast \cinvf{l}\subseteq \cinvf{k+l}$ easily follows from \Crefitem{thm:invariant_vanish_algebra}{alternate}
	\end{linkproof}
	
	\begin{linkproof}{derivative_vanish_at_0}
		If $l\in \Z_+^\nu$, $D\in \DO^l(M)$ and $f\in \cinv$, then 
		\begin{equation*}\begin{aligned}
			\alpha_{(1,\cdots,1,\lambda_i,1,\cdots,1)}(\theta_l(D)\ast f)=\lambda_i^{l_i}\theta_l(D)\ast f.
		\end{aligned}\end{equation*}
		So, 
		\begin{equation*}\begin{aligned}
			\odv*{\alpha_{(1,\cdots,1,\lambda_i,1,\cdots,1)}(\theta_l(D)\ast f)}{\lambda_{i}}_{\lambda_i=1}=\theta_l(D)\ast \left(l_if+\odv*{\alpha_{(1,\cdots,1,\lambda_i,1,\cdots,1)}(f)}{\lambda_{i}}_{\lambda_i=1}\right).		
		\end{aligned}\end{equation*}
		If $l_i\neq 0$, then $\odv*{\alpha_{(1,\cdots,1,\lambda_i,1,\cdots,1)}(\theta_l(D)\ast f)}{\lambda_{i}}_{\lambda_i=1}\in \cinvf{l}$ which is what we need to prove.
		If $l_i=0$, then it suffices to show that $\odv*{\alpha_{(1,\cdots,1,\lambda_i,1,\cdots,1)}(f)}{\lambda_{i}}_{\lambda_i=1}\in \cinvf{(0,\cdots,0,1,0,\cdots,0)}$.
		Let $(V,\natural)$ be a $\nu$-graded Lie basis. 
		By decomposing $f$ as in \eqref{eqn:dfn_algebra_vanishing_Fourier_invariant}, there are two cases to consider:
		  If $f\in C^\infty_c(M\times M\times \R_+^\nu\backslash 0,\omegahalf)$, then 
			\begin{equation*}\begin{aligned}
				\odv*{\alpha_{(1,\cdots,1,\lambda_i,1,\cdots,1)}(f)}{\lambda_{i}}_{\lambda_i=1}= t_i\pdv*{f}{t_{i}}=\theta_{(0,\cdots,0,1,0,\cdots,0)}(1)\ast\pdv*{f}{t_{i}} \in \cinvf{(0,\cdots,0,1,0,\cdots,0)}.
			\end{aligned}\end{equation*}
				If $f=\cQ_{V*}(\tilde{f}\circ \pi_{V\times \mathbb{G}^{(0)}})$ for some $\tilde{f}\in C^\infty_c(\tdom(\cQ_V),\Omega^1(V))$, then 
				\begin{equation}\label{eqn:HTYUIOOP1}\begin{aligned}
					\odv*{\alpha_{(1,\cdots,1,\lambda_i,1,\cdots,1)}(f)}{\lambda_{i}}_{\lambda_i=1}=\cQ_{V*}\left(\odv*{\alpha_{(1,\cdots,1,\lambda_i,1,\cdots,1)}(\tilde{f})}{\lambda_{i}}_{\lambda_i=1}\circ \pi_{V\times \mathbb{G}^{(0)}}\right)
				\end{aligned}\end{equation}
			Let $v_1,\cdots,v_n\in V$ be a pure basis such that $v_j\in V^{\omega(j)}$, $\tilde{X}_1,\cdots,\tilde{X}_n$ as in \eqref{eqn:tildeX_vector_fields},
			 $(v_j^*)_{j\in \bb{1,n}}\subseteq C^\infty(V\times M\times \R_+^\nu)$ the smooth functions $v_j^*(\sum_{k=1}^n c_k v_k,x,t)=c_j$.
			Notice that $\alpha_\lambda(v_j^*)=\lambda^{-\omega(j)}v_j^*$.
			Let $U\subseteq V\times M\times \R_+^\nu$ 
			be the open set at which the vector fields $\tilde{X}_1,\cdots,\tilde{X}_n$ are well-defined and linearly independent, and the matrix $(\tilde{X}_j(v_k^*))_{j,k\in \bb{1,n}}$ is invertible.
			By equivariance of $(v_j)_{j\in \bb{1,n}}$ and $(\tilde{X}_j)_{j\in \bb{1,n}}$, it follows that $U$ is equivariant.
			On $V\times M\times \{0\}$, the vector fields $(\tilde{X}_j)_{j\in \bb{1,n}}$ are the right-invariant vector fields on $V$ associated to $(v_j)_{j\in \bb{1,n}}$.
			Since $V$ is graded nilpotent, the matrix $(\tilde{X}_j(v_k^*))_{j,k\in \bb{1,n}}$ is (under some permutation of indices) an upper triangular matrix with $1$ on the diagonal.
			Hence, $V\times M\times \{0\}\subseteq U$.
			By reducing $U$ if necessary, we can suppose that $U$ satisfies \eqref{eqn:stron_equivariance_U}, see proof of \Cref{lem:support_convolution_invariant_functions}.
			Since $U$ satisfies \eqref{eqn:stron_equivariance_U}, \Cref{lem:equiv_functions_U} can be used.
			By \Cref{lem:change_open_set_invariant_function}, we can suppose that $\supp(\tilde{f})\subseteq U$.
			Let $(A_{j,k})_{j,k\in \bb{1,n}}\subseteq C^\infty(U)$ be the inverse of $(\tilde{X}_j(v_j^*))_{j,k\in \bb{1,n}}$.
			We have $\alpha_\lambda(A_{j,k})=\lambda^{\omega(j)-\omega(k)}A_{j,k}$.
			On $U$, there exists unique smooth functions $(P_{i,j})_{i\in \bb{1,\nu}, j\in \bb{1,n}}\subseteq C^\infty(U)$ and $(H_i)_{i\in \bb{1,\nu}}\subseteq C^\infty(U)$ such that 
				\begin{equation}\label{eqn:qksdjfkoqsdjofjoqsdjjfoqjsdfjmqsdf}\begin{aligned}
					\odv*{\alpha_{(1,\cdots,1,\lambda_i,1,\cdots,1)}(g)}{\lambda_{i}}_{\lambda_i=1}=t_i\pdv*{g}{t_{i}}+ \sum_{j=1}^n \cL_{\tilde{X}_j}(P_{i,j} g)+H_ig,\quad \forall g\in C^\infty(U,\Omega^1(V)),i\in\bb{1,\nu}.
				\end{aligned}\end{equation}
			By replacing $g$ with $v_k^*g$, we deduce that
				\begin{equation}\label{eqn:qskdjjflkoqjsdmjfmqsdf}\begin{aligned}
					P_{i,j}=-\sum_{k=1}^{n}\omega(k)_{i}v_k^*A_{k,j}.
				\end{aligned}\end{equation}
			So, by \eqref{eqn:HTYUIOOP1} and \eqref{eqn:qksdjfkoqsdjofjoqsdjjfoqjsdfjmqsdf}, we have 
				\begin{equation*}\begin{aligned}
					\odv*{\alpha_{(1,\cdots,1,\lambda_i,1,\cdots,1)}(f)}{\lambda_{i}}_{\lambda_i=1}&=\theta_{(0,\cdots,0,1,0,\cdots,0)}(\delta_1)\ast \cQ_{V*}\left(\left(\pdv*{\tilde{f}}{t_{i}}\right)\circ \pi_{V\times\mathbb{G}^{(0)}}\right)\\
					&-\sum_{j,k=1}^n\omega(k)_i \theta_{\omega(j)}(X_j)\ast \cQ_{V*}\left(\left(v_k^*A_{k,j}\tilde{f}\right) \circ \pi_{V\times\mathbb{G}^{(0)}}\right)\\&+\cQ_{V*}\left(\left(H_i\tilde{f}\right)\circ \pi_{V\times\mathbb{G}^{(0)}}\right)
				\end{aligned}\end{equation*}
			It is clear that $\theta_{(0,\cdots,0,1,0,\cdots,0)}(\delta_1)\ast \cQ_{V*}\left(\left(\pdv*{\tilde{f}}{t_{i}}\right)\circ \pi_{V\times\mathbb{G}^{(0)}}\right)\in \cinvf{(0,\cdots,0,1,0,\cdots,0)}$.
			For the second term, since $\alpha_{\lambda}(A_{k,j})=\lambda^{\omega(k)-\omega(j)}$, by \Cref{lem:equiv_functions_U}, $\cQ_{V*}\left(\left(v_k^*A_{k,j}\tilde{f}\right) \circ \pi_{V\times\mathbb{G}^{(0)}}\right)\in \cinvf{\max(\omega(k)-\omega(j),0)}$.
			Since $\omega(k)\preceq \omega(j)+\max(\omega(k)-\omega(j),0)$, by \eqref{eqn:theta_product_cinfk}, we have
				\begin{equation*}\begin{aligned}
					\theta_{\omega(j)}(X_j)\ast \cQ_{V*}\left(\left(v_k^*A_{k,j}\tilde{f}\right) \circ \pi_{V\times\mathbb{G}^{(0)}}\right)\in \cinvf{\omega(j)+\max(\omega(k)-\omega(j),0)}\subseteq \cinvf{\omega(k)}.
				\end{aligned}\end{equation*}
			So, since we multiply with the factor $\omega(k)_i$, only the terms with $\omega(k)_i>0$ count, which are terms such that $(0,\cdots,0,1,0,\cdots,0)\preceq \omega(k)$.
			So, by \eqref{eqn:theta_product_cinfk}, the second sum also belongs to $\cinvf{(0,\cdots,0,1,0\cdots,0)}$.
			To prove that $\cQ_{V*}\left(\left(H_i\tilde{f}\right)\circ \pi_{V\times\mathbb{G}^{(0)}}\right)\in \cinvf{(0,\cdots,0,1,0,\cdots,0)}$,
			by \Cref{lem:equiv_functions_U}, it suffices to show that $H_i$ can be written as a sum of functions of the form $BC$ with $\alpha_\lambda(B)=\lambda^l B$ and $0<l_i$ for some $l\in \Z^\nu$.
			Let $\mu\in \Omega^1(V)$ be a non-zero density. We see $\mu$ as a constant function in $C^\infty(U,\Omega^1(V))$.
			We have $\alpha_{\lambda}(\mu)=\lambda^{-\sum_{j=1}^n \omega(j)}\mu$.
			So, by replacing $g$ with $\mu$ in \eqref{eqn:qksdjfkoqsdjofjoqsdjjfoqjsdfjmqsdf},
			we deduce that 
				\begin{equation*}\begin{aligned}
					H_i&=\sum_{j=1}^n-\frac{\cL_{\tilde{X}_j}(P_{i,j}\mu)}{\mu}-\omega(j)_i=\sum_{j,k=1}^n \omega(k)_{i}v_k^* \frac{\cL_{\tilde{X}_j}(A_{k,j}\mu)}{\mu}+(\omega(k)_{i}-\omega(j)_i)\cL_{\tilde{X}_j}(v_k^*)A_{k,j}.
				\end{aligned}\end{equation*}
			For the first term, we have
				\begin{equation*}\begin{aligned}
					 \alpha_\lambda\left(\frac{\cL_{\tilde{X}_j}(A_{k,j}\mu)}{\mu}\right)=\lambda^{\omega(k)}\frac{\cL_{\tilde{X}_j}(A_{k,j}\mu)}{\mu}
				\end{aligned}\end{equation*}
			Since we only sum over terms with $\omega(k)_i>0$, we take $B=\frac{\cL_{\tilde{X}_j}(A_{k,j}\mu)}{\mu}$ and $C=1$.
			For the second term either $\omega(k)_i<\omega(j)_i$ or $\omega(k)_i>\omega(j)_i$ (otherwise $\omega(k)_{i}-\omega(j)_i=0$). In the first case we take $B=\cL_{\tilde{X}_j}(v_k^*)$ and $C=A_{k,j}$. In the second case, we take $B=A_{k,j}$ and $C=\cL_{\tilde{X}_j}(v_k^*)$.
	\end{linkproof}
\section{Proofs of the results in Section \ref{sec:weakly_comm_struct}}\label{sec:weakly_comm_struct_proof}
	In this section, we suppose that $\cF^{\bullet}$ is weakly commutative.
	\begin{linkproof}{estimate_weakly_commut_convolution}
	\begin{lem}\label{lem:weak_commutative_graded_basis}
		There exists a $\nu$-graded Lie basis $(V,\natural)$ such that $V=\oplus_{i\in \bb{1,\nu}}V^{\weak{i}}$.
	\end{lem}
	\begin{proof}
		Let $(V,\natural)$ be a $\nu$-graded Lie basis of $\cF$.
		By weak commutativity of $\cF$, we deduce that $(\oplus_{i\in \bb{1,\nu}}V^{\weak{i}},\natural_{|\oplus_{i\in \bb{1,\nu}}V^{\weak{i}}})$ is a $\nu$-graded basis of $\cF$.
		We define a Lie bracket on $\oplus_{i\in \bb{1,\nu}}V^{\weak{i}}$ as follows: If $v,w\in V^{\weak{i}}$, then $[v,w]$ is equal to their Lie bracket in $V$.
		If $v\in V^{\weak{i}},w\in V^{\weak{j}}$ with $i\neq j$, then $[v,w]=0$.
		One can check that $(\oplus_{i\in \bb{1,\nu}}V^{\weak{i}},\natural_{|\oplus_{i\in \bb{1,\nu}}V^{\weak{i}}})$ is a $\nu$-graded Lie basis.
	\end{proof}
	It suffices to show that  for any $C_1,C_2\in \R$ such that $0<C_2<2C_1$, \eqref{eqn:lambda_mu_map_azs} to a smooth map 
	$\{(\lambda,\mu)\in [0,C_1]^{2\nu}:\lambda+\mu\in [C_2,2C_1]^\nu\}\to \cinv$. We fix $C_1$ and $C_2$.
	\begin{lem}\label{lem:suppott_convolution_alpha_lambda_mu}
		If $K_1\subseteq M\times M\times \R_+^\nu\backslash 0$ and $K_2\subseteq \mathbb{G}$ are compact subsets,  
			\begin{equation*}\begin{aligned}
				K=\bigcup_{\{(\lambda,\mu)\in ]0,C_1]^{2\nu}:\lambda+\mu\in [C_2,2C_1]^\nu\}}\alpha_{\lambda}(K_1)\ast \alpha_{\mu}(K_2),
			\end{aligned}\end{equation*}
		then the closure of $K$ in $\mathbb{G}$ is a compact subset of $M\times M\times \R_+^\nu\backslash 0$.
	\end{lem}
	\begin{proof}
		Let $(\gamma_n)_{n\in \N}\subseteq \overline{K}$ be a sequence.
		Since $\mathbb{G}$ is metrizable, we fix a metric $d$ that induces the topology on $\mathbb{G}$.
		Let $(\gamma'_{n})_{n\in \N}\subseteq K$ be a sequence such that $d(\gamma_n,\gamma_n')<\frac{1}{n}$.
		We write $\gamma_n'=\alpha_{\lambda_n}(\eta_n)\alpha_{\mu_n}(\xi_n)$. By passing to a subsequence, we can suppose that $\lambda_n\to \lambda\in [0,C_1]^{\mu}$, $\mu_n\to \mu\in [0,C_1]^\nu$, $\eta_n\to \eta\in K_1$, $\xi_n\to \xi\in K_2$ and
		$\lambda+\mu\in [C_2,2C_1]^\nu$.
		We write $\eta_n=(z_n,y_n,t_n)$ and $\eta=(z,y,t)$.
		For the pair $\alpha_{\lambda_n}(\eta_n)=(z_n,y_n,\lambda_n^{-1}t_n)$ and $\alpha_{\mu_n}(\xi_n)$ to be composable, necessarily $\xi_n=(y_n,x_n,\mu_n\lambda_n^{-1}t_n)$ for some $(x_n)_{n\in N}\subseteq M$.
		Since $\xi_n$ converges in $\mathbb{G}$, $x_n\to x$ and $\mu_n\lambda_n^{-1}t_n\to s$ for some $x\in M$ and $s\in \R_+^\nu$. 
		Hence, 
			\begin{equation*}\begin{aligned}
				\lambda_{n}^{-1}t_n=(\mu_n+\lambda_n)^{-1}( \mu_n\lambda_n^{-1}t_n+t_n)\to (\lambda+\mu)^{-1}(s+t)\in \R^\nu_+\backslash 0
			\end{aligned}\end{equation*}
		Therefore, $\gamma_n'=(z_n,x_n,\lambda_{n}^{-1}t_n)\to (z,x,(\lambda+\mu)^{-1}(s+t))\in M\times M\times \R_+^\nu\backslash 0$.
		So, $\gamma_n$ also converges to $(z,x,(\lambda+\mu)^{-1}(s+t))$.
	\end{proof}
	By decomposing $f$ and $g$ as in \eqref{eqn:sum_smooth_decompo_invariatn_function}, the proof is divided into $4$-cases: 
	
	\begin{enumerate}
		\item     If $f$ and $g$ are in $C^\infty_c(M\times M\times \R_+^\nu\backslash 0,\omegahalf)$, then by \Cref{lem:suppott_convolution_alpha_lambda_mu}, there exists $K\subseteq M\times M\times \R_+^\nu\backslash 0$ such that
		 $\supp(\alpha_\lambda(f)\ast \alpha_\mu(g))\subseteq K$ for all $\lambda,\mu\in ]0,C_1]^{\nu}$ such that $\lambda+\mu\in [C_2,2C_1]^\nu$. We have
		 \begin{equation*}\begin{aligned}
			\alpha_\lambda(f)\ast \alpha_\mu(g)(z,x,t)=\int_M f(z,y,\lambda t)g(y,x,\mu t)
		\end{aligned}\end{equation*}
		This extends to $\{(\lambda,\mu)\in [0,C_1]^{2\nu}:\lambda+\mu\in [C_2,2C_1]^\nu\}$. The result follows in this case.
		\item If $f\in C^\infty_c(M\times M\times \R_+^\nu\backslash 0,\omegahalf)$ and $g=\cQ_{V*}(\tilde{g}\circ \pi_{V\times \mathbb{G}^{(0)}})$ for some $\tilde{g}\in C^\infty_c(\tdom(\cQ_V),\Omega^1(V))$, then again by \Cref{lem:suppott_convolution_alpha_lambda_mu}, there exists $K\subseteq M\times M\times \R_+^\nu\backslash 0$ such that
		$\supp(\alpha_\lambda(f)\ast \alpha_\mu(g))\subseteq K$ for all $(\lambda,\mu)\in ]0,C_1]^{2\nu}$ such that $\lambda+\mu\in [C_2,2C_1]^\nu$. We have 
		\begin{equation*}\begin{aligned}
			\alpha_\lambda(f)\ast \alpha_\mu(g)(z,x,t)&=\int_V  \alpha_\lambda(f)(z,e^{\alpha_{t}(v)}\cdot x,t) \alpha_\mu(\tilde{g})(v,x,t)h(v,x,t)\\&=\int_V f(z,e^{\alpha_{\mu t}(v)}\cdot x,\lambda t)\tilde{g}(v,x,\mu t)h(\alpha_\mu(v),x, t) ,
		\end{aligned}\end{equation*}
		where $h(v,x,t)$ is the isomorphism $\Omega^{\frac{1}{2}}(T_{e^{\alpha_t(v)}\cdot x}M)\to \Omega^\frac{1}{2}( T_{x}M)$ induced from the differential of the map $M\to M$, $x\mapsto e^{\alpha_t(v)}\cdot x$.
		This extends to $\{(\lambda,\mu)\in [0,C_1]^{2\nu}:\lambda+\mu\in [C_2,2C_1]^\nu\}$. The result follows in this case.
	\item If $f=\cQ_{V*}(\tilde{f}\circ \pi_{V\times \mathbb{G}^{(0)}})$ for some $\tilde{f}\in C^\infty_c(\tdom(\cQ_V),\Omega^1(V))$ and $g\in C^\infty_c(M\times M\times \R_+^\nu\backslash 0,\omegahalf)$, then by taking the adjoint, this case follows from the previous case.
	\item It remains the case where $f=\cQ_{V*}(\tilde{f}\circ \pi_{V\times \mathbb{G}^{(0)}})$ and $g=\cQ_{V*}(\tilde{g}\circ \pi_{V\times \mathbb{G}^{(0)}})$ for some $\tilde{f},\tilde{g}\in C^\infty_c(\tdom(\cQ_V),\Omega^1(V))$.
	Let $\psi$ and $\tilde{\psi}$ be maps as in \Cref{thm:composition_bisub}.
	We define the map 
	\begin{equation*}\begin{aligned}
		\tilde{\psi}^{\lambda,\mu}(v_1,v_2,x,t)=\tilde{\psi}(\alpha_{\lambda}(v_1),\alpha_{\mu}(v_2),x,t).
	\end{aligned}\end{equation*}

	\begin{lem}\label{lem:skodqjfmkojsdlkmjfmqjsdf}
		We can find an equivariant open neighborhood  $L\subseteq \dom(\tilde{\phi})$  of $V\times V\times M\times \{0\}$ such that for any $\lambda,\mu\in [0,C_1]^{\nu}$, if $\lambda+\mu\in [C_2,2C_1]^\nu$, 
		then $\tilde{\psi}^{\lambda,\mu}$ is a submersion on $L^{\lambda,\mu}=\{(v_1,v_2,x,t):(\alpha_\lambda(v_1),\alpha_{\mu}(v_2),x,t)\in L\}$.
	\end{lem}
	\begin{proof}
		For any $S\subseteq \Z_+^\nu$, let $T_S:V\to V\times V$ be the linear map which sends $v\in V^{k}$ to $(v,0)$ if $k\in S$ and to $(0,v)$ if $k\notin S$.
		The set $\{T_S:S\subseteq \Z_+^\nu\}$ is finite because $\{k\in \Z_+^\nu:V^k\neq 0\}$ is finite which is implied by finite dimensionality of $V$.
		Without loss of generality, we suppose $C_2\leq 1\leq C_1$.
			Let $L$ be the set of points $(v_1,v_2,x,t)$ such that for any $(\lambda,\mu)\in [0,C_1]^{2\nu}$ such that $\lambda+\mu\in [C_2,2C_1]^\nu$, one has $(\alpha_\lambda(v_1),\alpha_\mu(v_2),x,t)\in \dom(\tilde{\psi})$ and 
			the linear map $\odif{\tilde{\psi}}_{(\alpha_\lambda(v_1),\alpha_{\mu}(v_2),x,t)}\circ (T_S,\mathrm{Id_{T_xM},\mathrm{Id}_{\R}}):V\times T_xM\times \R\to V\times T_{x}M\times \R$ is surjective for any $S\subseteq \Z_+^\nu$.
			The set $L\subseteq \dom(\tilde{\psi})$ because $C_2\leq 1\leq C_1$, so one can take $\lambda=\mu=(1,\cdots,1)$.
			It is equivariant by equivariance of $\tilde{\psi}$.
			It is open by a compactness argument (here we use finiteness of $\{T_S:S\subseteq \Z_+^\nu\}$).
			By \Crefitem{thm:composition_bisub}{3} together with the fact that $V=\oplus_{i\in \bb{1,\nu}}V^{\weak{i}}$, we deduce that $L$ contains $V\times V\times M\times \{0\}$.
	\end{proof}
	Let $U$ be as in \Cref{lem:support_convolution_invariant_functions} applied to $L$.
	By \Cref{lem:change_open_set_invariant_function} and the cases we previously treated, we can suppose that $\tilde{f}$ and $\tilde{g}$ are supported in $U$.
	We deduce that 
		\begin{equation*}\begin{aligned}
			h^{\lambda,\mu}(v_1,v_2,x,t)&=\alpha_\lambda(\tilde{f})(v_1,e^{\alpha_{t}(v_2)}\cdot x,t)\alpha_\mu(\tilde{g})(v_2,x,t)\in C^\infty_c(L,\Omega^2(V))&& \lambda,\mu\in ]0,C_1]^\nu\\
			\kappa^{\lambda,\mu}(v_1,v_2,x,t)&=\tilde{f}(v_1,e^{\alpha_{\mu t}(v_2)}\cdot x, \lambda t)\tilde{g}(v_2,x,\mu t)\in C^\infty_c(L^{\lambda,\mu},\Omega^2(V))&& \lambda,\mu\in [0,C_1]^\nu
		\end{aligned}\end{equation*}
		and such that $\lambda+\mu\in [C_2,2C_1]^\nu$. Here, we use \eqref{eqn:stron_equivariance_U} to show that $	\kappa^{\lambda,\mu}$ is supported in $L^{\lambda,\mu}$.
	By \eqref{eqn:diag_composition_bisub}, $\alpha_\lambda(f)\ast \alpha_\mu(g)=\cQ_{V*}(\tilde{\psi}_* (h^{\lambda,\mu})\circ \pi_{V\times \mathbb{G}^{(0)}})$.
	It suffices to show that $(\lambda,\mu) \mapsto \tilde{\psi}_* (h_{\lambda,\mu})$ extends to a smooth map $\{(\lambda,\mu)\in [0,C_1]^{2\nu}:\lambda+\mu\in [C_2,2C_1]^\nu\}\to C^\infty_c(\tdom(\cQ_V^2),\Omega^2(V))$.
	We have
	\begin{equation}\label{eqn:psi_h_alpha_lambda_mu}\begin{aligned}
		\tilde{\psi}_* (h^{\lambda,\mu})(v,x,t)
		&=\int_{\tilde{\psi}(v_1,v_2,x,t)=(v,x,t)}	\alpha_\lambda(\tilde{f})(v_1,e^{\alpha_{t}(v_2)}\cdot x,t)\alpha_\mu(\tilde{g})(v_2,x,t)\\
		&=\int_{\tilde{\psi}(\alpha_\lambda(v_1),\alpha_\mu(v_2),x,t)=(v,x,t)}	\tilde{f}(v_1,e^{\alpha_{\mu t}(v_2)}\cdot x, \lambda t)\tilde{g}(v_2,x,\mu t)
		\\&=\tilde{\psi}^{\lambda,\mu}_*(\kappa^{\lambda,\mu})(v,x,t)
	\end{aligned}\end{equation}
	Since $\kappa^{\lambda,\mu}$ are uniformly compactly supported, the result follows.
	\end{enumerate}
	This finishes the proof \Cref{thm:estimate_weakly_commut_convolution}.\qedhere
	\end{linkproof}

	\begin{linkproof}{extension_parameters}
			By \Cref{lem:weak_commutative_graded_basis}, let $(V,\natural)$ be $\nu$-graded Lie basis $\cF$ such that $V=\bigoplus_{i=1}^{\nu} V^{\weak{i}}$.
			Let $\tilde{V}=\bigoplus_{i=1}^{\tilde{\nu}} V^{\weak{i}}$ and $\dbtilde{V}=\bigoplus_{i=\tilde{\nu}+1}^{\nu} V^{\weak{i}}$.
			The pairs $(\tilde{V},\natural_{|\tilde{V}})$ and $(\dbtilde{V},\natural_{|\dbtilde{V}})$
			are $\tilde{\nu}$ and $\dbtilde{\nu}$ graded basis of $\tilde{\cF}$ and $\dbtilde{\cF}$ respectively.
 			To prove \Cref{thm:extension_parameters}, we need to check that \eqref{eqn:Convolution_extension_parameters} extends to a smooth invariant function on $\mathbb{G}$.
			By \eqref{eqn:sum_smooth_decompo_invariatn_function}, there are $4$-cases to consider:
			\begin{enumerate}
				\item     If $f\in C^\infty_c(M\times M\times \R^{\tilde{\nu}}_+\backslash 0,\omegahalf)$ and $g\in C^\infty_c(M\times M\times \R^{\dbtilde{\nu}}_+\backslash 0,\omegahalf)$, then $\Psi(f,g)$ is easily seen to define an element in $C^\infty_c(M\times M\times \R^{\nu}_+\backslash 0)$
				\item  If $f\in C^\infty_c(M\times M\times \R^{\tilde{\nu}}_+\backslash 0,\omegahalf)$ and $g=\cQ_{\dbtilde{V}*}(\tilde{g}\circ \pi_{\dbtilde{V}\times \dbtilde{\mathbb{G}}^{(0)}})$
				where $\tilde{g}\in C^{\infty}_{c}(\tdom(\cQ_{\dbtilde{V}}),\Omega^1\dbtilde{V})$, then we have the following lemma:
						\begin{lem}\label{lem:qksdjfmjsqmdjfljsqldmjfkmljqsdlmjflmqjsdlmjfkmlqs}
							If $K_1\subseteq M\times M\times \R^{\tilde{\nu}}_+\backslash 0$ and $K_2\subseteq \dbtilde{\mathbb{G}}$ are compact subsets, and 
								\begin{equation*}\begin{aligned}
									K=\{(z,x,t,s)\in M\times M\times \R^\nu_+\backslash 0:\exists y\in M, (z,y,t)\in K_1,(y,x,s)\in K_2\}		
								\end{aligned}\end{equation*}
							then the closure of $K$ in $\mathbb{G}$ is a compact subset of $M\times M\times (\R^{\tilde{\nu}}_+\backslash 0)\times \R^{\dbtilde{\nu}}_+$.
						\end{lem}
						\begin{proof}
							Let $(\gamma_n)_{n\in \N}\subseteq \overline{K}$ be a sequence. Since $\mathbb{G}$ is a metrizable, we fix a distance $d$.
							Let $(\tilde{\gamma}_n)_{n\in \N}\subseteq K$ be a sequence such that $d(\gamma_n,,\tilde{\gamma}_n)<\frac{1}{n}$.
							We can find sequences $(z_n)_{n\in \N},(y_n)_{n\in \N},(x_n)_{n\in \N}\subseteq M$, $(t_n)_{n\in \N}\subseteq \R^{\tilde{\nu}}_+\backslash 0$, $(s_n)_{n\in \N}\subseteq \R^{\dbtilde{\nu}}_+\backslash 0$
							such that $(z_n,y_n,t_n)\in K_1$, $(y_n,x_n,s_n)\in K_2$ and $\tilde{\gamma}_n=(z_n,x_n,t_n,s_n)$.
							By passing to a subsequence, we can suppose that $(z_n,y_n,t_n)$ converges to $(z,y,t)\in M\times M\times \R^{\tilde{\nu}}_+\backslash 0$, and $(y_n,x_n,s_n)$ converges to $\eta\in K_2$.
							Convergence of $(y_n,x_n,s_n)$ implies that $y_n\to  y\in M$, $x_n\to x\in M$, $s_n\to s\in \R^{\dbtilde{\nu}}_+$.
							So, $\tilde{\gamma}_n\to (z,x,t,s)\in M\times M\times (\R^{\tilde{\nu}}_+\backslash 0)\times \R^{\dbtilde{\nu}}_+$. Hence, $\gamma_n$ also converges to $(z,x,t,s)$.
						\end{proof}
						By \Cref{lem:qksdjfmjsqmdjfljsqldmjfkmljqsdlmjflmqjsdlmjfkmlqs}, it suffices to check that $\Psi(f,g)$ extends to a smooth function on $M\times M\times (\R^{\tilde{\nu}}_+\backslash 0)\times \R^{\dbtilde{\nu}}_+$.
						We have 
							\begin{equation*}\begin{aligned}
								\Psi(f,g)(z,x,t,s)=\int_V h(v,x,s) f(z,e^{\alpha_s(v)}\cdot x,t) \tilde{g}(v,x,s),	
							\end{aligned}\end{equation*}
						where $h(x,v,s):\Omega^{\frac{1}{2}}(T_{e^\alpha_s(v)\cdot x}M)\to \Omega^{\frac{1}{2}}(T_xM)$ is the canonical isomorphism induced from the differential of the diffeomorphism $e^{\alpha_s(v)}$ at $x$.

						\item The case where  $f=\cQ_{\tilde{V}*}(\tilde{f}\circ \pi_{\tilde{V}\times \tilde{\mathbb{G}}^{(0)}})$ and $g\in C^\infty_c(M\times M\times \R^{\dbtilde{\nu}}_+\backslash 0,\omegahalf)$
							where $\tilde{f}\in C^{\infty}_{c}(\tdom(\cQ_{\tilde{V}}),\Omega^1\tilde{V})$  is similar to the second case and left to the reader.
						\item It remains the case where $f=\cQ_{\tilde{V}*}(\tilde{f}\circ \pi_{\tilde{V}\times \tilde{\mathbb{G}}^{(0)}})$ and $g=\cQ_{\dbtilde{V}*}(\tilde{g}\circ \pi_{\dbtilde{V}\times \dbtilde{\mathbb{G}}^{(0)}})$
							where $\tilde{f}\in C^{\infty}_{c}(\tdom(\cQ_{\tilde{V}}),\Omega^1(\tilde{V}))$ and $\tilde{g}\in C^{\infty}_{c}(\tdom(\cQ_{\dbtilde{V}}),\Omega^1(\dbtilde{V}))$.
						Let $\tilde{\psi}$ be as in \Cref{thm:composition_bisub}. Consider t he locally-defined map
							\begin{equation*}\begin{aligned}
								\kappa:V\times M\times \R_+^{\nu}\to V\times M\times \R_+^\nu,\quad (\tilde{v}+\dbtilde{v},x,t,s)\mapsto \tilde{\psi}(\tilde{v},\dbtilde{v},x,t,s)	
							\end{aligned}\end{equation*}
						By \Crefitem{thm:composition_bisub}{3}, $\kappa(\tilde{v}+\dbtilde{v},x,0,0)=(\tilde{v}+\dbtilde{v},x,0,0)$.
						So, the differential of $\kappa$ is bijective at $t=s=0$.
						Therefore, by restricting its domain, we can suppose that $\kappa$ is an open embedding whose image contains an open neighborhood of $V\times M\times \{0\}\times \{0\}$.
						We define 
							\begin{equation*}\begin{aligned}
								h\in C^\infty_c(V\times M\times \R_+^{\nu},\Omega^1(V)),\quad h\circ \kappa(\tilde{v}+\dbtilde{v},x,t,s)=\tilde{f}(\tilde{v},e^{\alpha_s(\dbtilde{v})}\cdot x,t)\tilde{g}(\dbtilde{v},x,s),\\
							\end{aligned}\end{equation*}
						Here, we use \Cref{lem:change_open_set_invariant_function} to reduce the support of $\tilde{f}$ and $\tilde{g}$ so that $h$ is well-defined and $\supp(h)\subseteq \tdom(\cQ_V)$.
						Let $(z,x,t,s)\in M\times M\times (\R_+^{\tilde{\nu}}\backslash 0 )\times (\R_+^{\dbtilde{\nu}}\backslash 0)$. We have 
							\begin{equation*}\begin{aligned}
								\cQ_{V*}(h\circ \pi_{V\times \mathbb{G}^{(0)}})(z,x,t,s)&=\int_{\{v\in V:\cQ_{V}(v,x,t,s)=(z,x,t,s)\}}	h(v,x,t,s)
								\\&=\int_{\{\tilde{v}+\dbtilde{v}\in V:\cQ_{V}(\kappa(\tilde{v}+\dbtilde{v},x,t,s))=(z,x,t,s)\}}	\tilde{f}(\tilde{v},e^{\alpha_s(\dbtilde{v})}\cdot x,t)\tilde{g}(\dbtilde{v},x,s)
								\\&=\int_{\{\tilde{v}+\dbtilde{v}\in V:e^{\alpha_{t}(\tilde{v})}e^{\alpha_{s}(\dbtilde{v})}\cdot x=z\}}	\tilde{f}(\tilde{v},e^{\alpha_s(\dbtilde{v})}\cdot x,t)\tilde{g}(\dbtilde{v},x,s)
									\\&=\Phi(f,g)(z,x,t,s),
							\end{aligned}\end{equation*}
						where in the third equality we used \eqref{eqn:product_phi_compsotion_bisub}.
						
			\end{enumerate}
			This finishes the proof that $\Phi$ is well-defined.
			It straightforward to check that on $ M\times M\times (\R_+^{\tilde{\nu}}\backslash 0 )\times (\R_+^{\dbtilde{\nu}}\backslash 0)$, 
				\begin{equation*}\begin{aligned}
					\alpha_{(\lambda,\mu)}(\Phi(f,g))=\Phi(\alpha_{\lambda}(f),\alpha_{\mu}(g)),\quad \Phi(\theta_k(\tilde{D})\ast f,g\ast \theta_l(\dbtilde{D}))=\theta_{(k,0)}(\tilde{D})\ast \Phi(f,g)\ast \theta_{(0,l)}(\dbtilde{D}).
				\end{aligned}\end{equation*}
			By density of $ M\times M\times (\R_+^{\tilde{\nu}}\backslash 0 )\times (\R_+^{\dbtilde{\nu}}\backslash 0)$ in $\mathbb{G}$ and \eqref{eqn:dfn_algebra_vanishing_Fourier_invariant} and \Crefitem{thm:invariant_vanish_algebra}{alternate}, \eqref{eqn:homogenity_extension} follows.
	\end{linkproof}
	\begin{linkproof}{commutator_phi}
			First notice that if any smooth function in $C^\infty_c(M\times M\times \R^\nu\backslash 0,\omegahalf)$ can be written as a sum $\sum_{i=1}^\nu t_ih_i$ for some $h_i\in C^\infty_c(M\times M\times \R^\nu\backslash 0,\omegahalf)$.
			Now, we follow the proof of \MyCref{thm:extension_parameters}. 
			As we did, we divide into $4$ cases depending on the type of $f,g$. In the first three cases, both $\Phi(f,g),\Phi(g,f)\in C^\infty_c(M\times M\times \R^\nu\backslash 0,\omegahalf)$. So, the difference can be written as in \eqref{eqn:commutator_phi}.
			In the fourth case, $\Phi(f,g)$ and $\Phi(g,f)$ are shown to be equal to $\cQ_{V*}(h\circ \pi_{V\times \mathbb{G}^{(0)}})$ and $\cQ_{V*}(h'\circ \pi_{V\times \mathbb{G}^{(0)}})$ for some smooth functions $h,h'\in C^\infty_c(V\times M\times \R_+^\nu,\Omega^1(V))$.
			The functions $h,h'$ agree at $(t,s)=0$. Both satisfy 
				\begin{equation*}\begin{aligned}
					 h(\tilde{v}+\dbtilde{v},x,0)=h'(\tilde{v}+\dbtilde{v},x,0)=\tilde{f}(\tilde{v},x,0)\tilde{g}(\dbtilde{v},x,0).
				\end{aligned}\end{equation*}
			Here we are using the fact that in the Lie group $V$, the components $\tilde{v}$ and $\dbtilde{v}$ commute (in fact their product is their sum since the commutator vanishes).
			So, by Taylor series expansion which is possible because $h,h'$ live on a smooth manifold (and not $\mathbb{G}$), the difference can be written as in \eqref{eqn:commutator_phi}.
	\end{linkproof}

\linkchap{Multi-graded pseudo-differential calculus}{chapter_bi-garded_pseudo_diff}
	\paragraph{Conventions and notations:}
	We fix throughout this chapter a \textbf{weakly commutative} $\nu$-graded sub-Riemannian structure $\cF^\bullet$ on a smooth manifold $M$ without boundary.
	For simplicity, we will state and prove the results in this chapter without adding any vector bundles. We remark that adding vector bundles is purely cosmetic and adds no real technical difficulty.
	\begin{enumerate}
		\item We denote by 
		\begin{equation}\label{eqn:1underline}\begin{aligned}
			\underline{1}=(1,\cdots,1)\in \Z_+^\nu.		
		\end{aligned}\end{equation}
		\item  
			When we consider integrals,  we use the convention $\odif{\lambda}:=\odif{\lambda_1}\cdots \odif{\lambda_\nu}$.
			Combining \eqref{eqn:1underline} with \eqref{eqn:tk}, we obtain
			\begin{equation*}\begin{aligned}
				\frac{\odif{\lambda}}{\lambda^{k+\underline{1}}}=\frac{\odif{\lambda_1}\cdots \odif{\lambda_\nu}}{\lambda_1^{k_1+1}\cdots \lambda_\nu^{k_\nu+1}},\quad k\in \C^\nu.		  
			  \end{aligned}\end{equation*}
			\item If $k\in \C^\nu$, then 
					\begin{equation*}\begin{aligned}
						\Re(k):=(\Re(k_1),\cdots,\Re(k_\nu))\in \R^\nu,\quad \overline{k}:=(\overline{k_1},\dots,\overline{k_\nu})\in \C^\nu.		
					\end{aligned}\end{equation*}
		\item		We denote by $\C^\nu_{<0}$ (and $\C^\nu_{\leq 0}$) the set of $k\in \C^\nu$ such that $\Re(k_i)<0$ (and $\Re(k_i)\leq 0$) for all $i\in \bb{1,\nu}$.
		\item We extend \Cref{dfn:Schwartz_functions_vanish} to $k\in \C^\nu$ by 
			\begin{equation*}\begin{aligned}
					\cinvf{k}:=\cinvf{(l_1,\cdots,l_\nu)}
			\end{aligned}\end{equation*}
		where $l_i$ is the smallest non-negative integer such that $\Re(k_i)< l_i+1$. For example if $k\in \C^\nu_{<0}$, then $\cinvf{k+\underline{1}}=\cinv$.
		\item To avoid confusion, we reserve the symbols $u,u',u''$ for distributions on $\mathbb{G}$ and $v,v',v''$ for distributions on $M\times M$, and $w,w',w''$ for distributions on $M$.
			In \MyCref{sec:main_theorem_proof}, we will reserve the symbols $\chi,\chi',\chi''$ for classical pseudo-differential operators on $M$.
	\end{enumerate}

\linksec{Multi-graded pseudo-differential operators}{bi_graded_pseudo_diff}
		In this section, we define a $\nu$-graded pseudo-differential calculus of operators on $M$.
		Our operators are defined using iterated integrals of smooth invariant functions on $\mathbb{G}$.
		The inspiration for our formula comes from the work of Debord and Skandalis \cite{DebordSkandalis1}.
			
			\begin{linkprop}{integralIkl}
				Let $k\in \C^\nu$, $f\in \cinvf{k+\underline{1}}$, $g\in \cbbG$. The integrals 
					\begin{equation*}\begin{aligned}
						\int_{]0,1]^\nu} g\ast \alpha_\lambda(f)\frac{\odif{\lambda}}{\lambda^{k+\underline{1}}},\qquad \int_{]0,1]^\nu} \alpha_\lambda(f)\ast g\frac{\odif{\lambda}}{\lambda^{k+\underline{1}}}
					\end{aligned}\end{equation*}
				converge in $\cbbG$. Furthermore, if $g\in \cinv$, then the integrals converge in the topology of $\cinv$.
			\end{linkprop}
			By \Cref{thm:integralIkl}, we can define
			\begin{equation*}\begin{aligned}
				I_k(f)=\int_{]0,1]^\nu} \alpha_\lambda(f)\frac{\odif{\lambda}}{\lambda^{k+\underline{1}}}\in \cinvdis.
			\end{aligned}\end{equation*}
			
			Let $t\in \R^\nu_+\backslash 0$. Since $M\times \{t\}$ is a saturated subset of $\mathbb{G}^{(0)}$, by \eqref{eqn:restriction_map_saturated_distribution_closed}, we have a restriction map
			\begin{equation}\label{eqn:restriction_map_at_ev_11_distribution}\begin{aligned}
				C^{-\infty}_{r,s}(\mathbb{G},\omegahalf)\to  C^{-\infty}_{r,s}(M\times M,\omegahalf),\quad u\mapsto u_{|t}.
			\end{aligned}\end{equation}
			Recall that for any element $v\in C^{-\infty}_{r,s}(M\times M,\omegahalf)$, the convolution
				\begin{equation*}\begin{aligned}
					v\ast \cdot:C^{-\infty}(M,\Omega^\frac{1}{2})\to C^{-\infty}(M,\omegahalf)
				\end{aligned}\end{equation*}
				is well-defined, and it satisfies $v\ast C^\infty(M,\Omega^\frac{1}{2})\subseteq  C^\infty(M,\omegahalf) $, see \Crefitem{ex:distributions_quasi_Lie_groupods}{3}.
		Let 
		\begin{equation}\label{eqn:dfn_pseudo}\begin{aligned}
			\Psi^{k}_c(M):=\{I_k(f)_{|\underline{1}}:f\in \cinvf{k+\underline{1}}\},\quad \Psi^{\prec k}_c(M):=\sum_{l\in \C^\nu,l\prec k}\Psi^{l}_c(M).
		\end{aligned}\end{equation}
		Elements of $\Psi^{k}_c(M)$ are called \textit{compactly supported $\nu$-graded pseudo-differential operators on $M$.}
		\begin{rem}\label{rem:local_structure}
			Notice that our pseudo-differential operators have a particularly simple description.
			We will trivialize densities in this remark.
			If $k\in \C^\nu_{<0}$, then a kernel in $v\in C^{-\infty}_c(M\times M)$ belongs to $\Psi^k_c(M)$ if for any (or for some by \Crefitem{thm:algebra_structure_invariant_functions}{invariance}) $\nu$-graded basis $(V,\natural)$ such that $V$ is equipped with a Euclidean metric, there exists $f\in C^\infty_c(\tdom(\cQ_V))\subseteq C^\infty_c(V\times M\times \R_+^\nu)$ such that 
				\begin{equation}\label{eqn:description_of_pseduo_diff}\begin{aligned}
					v(y,x)-\int_{]0,1]^\nu}\int_{\{a\in V:e^{a}\cdot x=y\}} f(\alpha_{t^{-1}}(a),x,t) t^{-\dim_h(V)}\frac{\odif{t}\odif{a}}{t^{k+\underline{1}}}\in C^\infty_c(M\times M),
				\end{aligned}\end{equation}
			where $\dim_h(V)\in \Z_+^\nu$ is defined by $\dim_h(V)=\sum_{k\in \Z^\nu}\dim(V^k)k$.
			The set $\{a\in V:e^{a}\cdot x=y\}$ is a submanifold of $V$ near $a=0$. The set $\tdom(\cQ_V)$ is chosen small enough that the integral above only involves the part where $\{a\in V:e^{a}\cdot x=y\}$ is a submanifold which inherits then a Riemannian structure from the Euclidean structure of $V$.
			The measure $\odif{a}$ means we integrate against this Riemannian structure.
			This remark shows that to define our pseudo-differential operators we don't need the groupoid $\mathbb{G}\rightrightarrows \mathbb{G}^{(0)}$. 
			We need the groupoid $\mathbb{G}\rightrightarrows \mathbb{G}^{(0)}$ to ultimately prove \Cref{thm:main_theorem}
		\end{rem}
		
		\begin{linkthm}{main_prop_bi_graded_pseudo_diff_M}
			Let $k,l\in \C^\nu$.
			\begin{enumerate}
				\item\label{thm:main_prop_bi_graded_pseudo_diff_M:inclusion_smoothing} One has $C^\infty_c(M\times M,\omegahalf)\subseteq \Psi^k_c(M)$. 
				\item\label{thm:main_prop_bi_graded_pseudo_diff_M:1} If $v\in \Psi^{k}_c(M)$, then $v$ is compactly supported and $\WF(v)\subseteq \{(\xi,-\xi;x,x):(\xi,x)\in T^*M\backslash 0\}$.
				In particular, $\singsupp(v)\subseteq \{(x,x):x\in M\}$, and 
					\begin{equation}\label{eqn:WF_multi_pseudo}\begin{aligned}
						\WF(v\ast w)\subseteq \WF(w)\quad \forall w\in C^{-\infty}_c(M,\omegahalf).
					\end{aligned}\end{equation}
				\item\label{thm:main_prop_bi_graded_pseudo_diff_M:3} If $k\preceq l$, then $\Psi^{k}_c(M)\subseteq \Psi^{l}_c(M)$.
				\item\label{thm:main_prop_bi_graded_pseudo_diff_M:2} If $k\in \Z_+^\nu$, then $\DO^{k}_c(M)\subseteq \Psi^{k}_c(M)$.
				\item\label{thm:main_prop_bi_graded_pseudo_diff_M:4} We have
				\begin{equation*}\begin{aligned}
						\Psi^{k}_c(M)^*=\Psi^{\overline{k}}_c(M),\quad \Psi^{k}_c(M)\ast \Psi^{l}_c(M)\subseteq \Psi^{k+l}_c(M).
					\end{aligned}\end{equation*}
				\item\label{thm:main_prop_bi_graded_pseudo_diff_M:7} If $k\in\C^\nu_{<0}$, then $\Psi^{k}_c(M)\subseteq L^1(M\times M,\omega)$, where $\omega\in C^\infty(M\times M,\Omega^{\frac{1}{2},-\frac{1}{2}})$ is any section satisfying \eqref{eqn:omega_symmetry_condition}, 
																	and $L^1(M\times M,\omega)$ is the $L^1$-Banach algebra defined in \Cref{sec:compl}.
				\item\label{thm:main_prop_bi_graded_pseudo_diff_M:5} 
				There exists $C\in\R$, such that if $n\in \N $ and $k\in \C^\nu$ satisfy $\sum_{i=1}^\nu \Re(k_i)N_i^{\sgn(\Re(k_i))}\leq C-n$, then $\Psi^{k}_c(M)\subseteq C^n_c(M\times M,\omegahalf)$.
				\item\label{thm:main_prop_bi_graded_pseudo_diff_M:6}   If $g\in C^\infty_c(M)=\DO^0_c(M)$ and $v\in \Psi^k_{c}(M)$, then $[\delta_g,v]\in \Psi^{\prec k}_c(M)$.
			\end{enumerate}
		\end{linkthm}
		Let $\Psi^{k}(M)$ be the space of properly supported kernels
		$v\in C^{-\infty}_{r,s}(M\times M,\omegahalf)$ such that for any $f\in C^\infty_c(M)$, $\delta_f\ast v\in \Psi^{k}_c(M)$.
		By \Crefitem{thm:main_prop_bi_graded_pseudo_diff_M}{2} and \Crefitem{thm:main_prop_bi_graded_pseudo_diff_M}{4}, $\Psi^{k}_c(M)\subseteq \Psi^{k}(M)$.
		We also define $\Psi^{\prec k}(M):=\sum_{l\in \C^\nu,l\prec k}\Psi^{l}(M)$. It is straightforward to extend \Cref{thm:main_prop_bi_graded_pseudo_diff_M} to $\Psi^{k}(M)$.

		\begin{linkthm}{asymptotic_sum}[Asymptotic sums]
				Let $k\in \C^\nu$ and for each $\tau\in \Z_+^\nu$, let $v_\tau\in \Psi^{k-\tau}(M)$.
				Then there exists $v\in \Psi^{k}(M)$ such that 
				for any finite subset $S$ of $\Z_+^\nu$ and $l_1,\cdots,l_n\in \Z_+^\nu$, if
					\begin{equation*}\begin{aligned}
							\left\{\tau\in \Z_+^\nu:\tau \notin S\text{ and }v_\tau\neq 0\right\}\subseteq \bigcup_{i=1}^n\left\{\tau\in \Z_+^\nu:l_i\preceq \tau\right\},
					\end{aligned}\end{equation*}
				then $v-\sum_{\tau \in S}v_{\tau}\in \Psi^{ k-l_1}(M)+\cdots +\Psi^{ k-l_n}(M)$.
				The operator $v$ is unique up to addition by an element of $C^\infty(M\times M,\omegahalf)$.
		\end{linkthm}
		\paragraph{Classical calculus:}
		We now give the simplest example of our calculus.
		\begin{linkprop}{classical_pseudo}
				If $\nu=1$ and $\cF^1=\cX(M)$, then $\Psi^k(M)$ coincides with the space of properly supported classical pseudo-differential operators of classical order $k$.
		\end{linkprop}
		
		\paragraph{Compatibility with restriction to an open set.}
			Let $U\subseteq M$ be an open subset.
			By \Cref{prop:pullback_weighted_subRiem}, let $\cF_{|U}^\bullet$ be the pullback of the sub-Riemannian structure $\cF^\bullet$ by the inclusion map $U\hookrightarrow M$. 
			The tangent groupoid defined in \Cref{chap:bi-graded_tangent_groupoid} associated to $\cF_{|U}^\bullet$ is the restriction of 
			$\mathbb{G}\rightrightarrows \mathbb{G}^{(0)}$ to the saturated open subset $\pi_{\mathbb{G}^{(0)}}^{-1}(U\times \R_+^\nu)$ of $\mathbb{G}^{(0)}$ where $\pi_{\mathbb{G}^{(0)}}$ is the map from \eqref{eqn:proj_mathbbG0}.
			The sub-Riemannian structure $\cF_{|U}$ is weakly commutative.
			So, we can define a pseudo-differential calculus on $U$ associated to $\cF_{|U}^\bullet$ which we denote by $\Psi^\bullet(U)$.
			One can check that if $k\in \C^\nu$, then			
				\begin{equation}\label{eqn:compatability_pseudo_restriction_open}\begin{aligned}
					\Psi^{k}_c(U)=\{v\in \Psi^k_c(M):\supp(v)\subseteq U\times U\}.
				\end{aligned}\end{equation}

\linksec{Representation theory of the tangent groupoid}{Tangent_groupoid_representation}
In this section and the following sections, for any $x\in M$, it is convenient to denote by $\GrF{x}$ the space $\grF{x}$ seen as a Lie group.
The reason for this is that if $L\subseteq \mathfrak{g}_x$ is a Lie subalgebra, then $\mathfrak{g}_x/L$ denotes the vector space quotient, while $\GrF{x}/L$ denotes the homogeneous space quotient of $\mathfrak{G}_x$ by $L$ seen as a Lie subgroup.
Starting from this section, we will make use of the irreducible unitary representations of $C^*$-algebras and weak containment, as well as Kirillov's orbit method for nilpotent Lie groups, see \cite[Chapter 3 and 4]{Dixmier} and \Cref{chap:Orbit_method} for more details.

		\begin{linkthm}{amenability_tanget}
			The groupoid $\mathbb{G}\rightrightarrows \mathbb{G}^{(0)}$ has the stable weak containment property (SWCP). 
			In particular $C^*\mathbb{G}=C^*_r\mathbb{G}$ and $C^*\cG=C^*_r\cG$ and $C^*\cG_x=C^*_r\cG_x$ for any $x\in M$.
		\end{linkthm}
		Let $C^*_{\mathrm{inv}}\mathbb{G}$ be the closure of $\cinv$ in $C^*\mathbb{G}$. The goal of this section is to describe the space of irreducible representations of $C^*_{\mathrm{inv}}\mathbb{G}$ up to unitary equivalence.

\begin{dfn}\label{dfn:Hellfer_Nourrigat_set}
	Let $(\xi,x)\in T^*M\backslash 0$. The Helffer-Nourrigat cone $\HL{(\xi,x)}\subseteq \mathfrak{g}^*_x$ is the set of all $\eta\in \mathfrak{g}^*_x$ such that there exists sequences 
	$(\xi_n,x_n)_{n\in \N}\subseteq  T^*M\backslash 0$, ${(t_n)}_{n\in \N}\subseteq \R_+^\nu$
				such that $x_n\to x$, $t_n\to 0$ and, 
					\begin{equation}\label{eqn:Hellfer_Nourrigat_set_condition_conv}\begin{aligned}
						\langle \xi_n,t_n^{k} X(x_n)\rangle\to \langle \eta,[X]_{k,x}\rangle,\quad \forall k\in \Z_+^\nu,X\in \cF^k,
					\end{aligned}\end{equation}
 	and $(\frac{\xi_n}{\norm{\xi_n}},x_n)$ converges to $(\frac{\xi}{\norm{\xi}},x)$ in $T^*M$, where $\norm{\cdot}$ is any Riemannian metric on $M$.
	We also define the Helffer-Nourrigat cone at $x$ by
		\begin{equation}\label{eqn:union_hellfer_microlocal}\begin{aligned}
		\HL{x}=\bigcup_{\xi\in T^*_xM\backslash 0}\HL{(\xi,x)}\subseteq \mathfrak{g}_x^*.
	\end{aligned}\end{equation}
\end{dfn}
Clearly $\HL{(\lambda\xi,x)}= \HL{(\xi,x)}$ for any $\lambda\in \Rpt$.
\begin{rem}\label{rem:Helffer_Nourrigat_set_x}
	 By compactness of the sphere cotangent bundle, $\eta\in \HL{x}$, if and only if there exists sequences 
	$(\xi_n,x_n)_{n\in \N}\subseteq  T^*M$, ${(t_n)}_{n\in \N}\subseteq \R_+^\nu$
				such that $x_n\to x$, $t_n\to 0$ and, 
					\begin{equation}\label{eqn:Hellfer_Nourrigat_set_condition_conv_2}\begin{aligned}
						\langle \xi_n,t_n^{k} X(x_n)\rangle\to \langle \eta,[X]_{k,x}\rangle,\quad \forall k\in \Z_+^\nu,X\in \cF^k,
					\end{aligned}\end{equation}
	Notice that one can allow $\xi_n=0$ in the definition of $\HL{x}$ in contrast to the definition of $\HL{(\xi,x)}$.
	This is simply because we divide by $\norm{\xi_n}$ in the definition of $\HL{(\xi,x)}$.
\end{rem}
		\begin{linkprop}{Helffer_Nourrigat_cone}
			\begin{enumerate}
				\item The sets $\HL{(\xi,x)}$ and $\HL{x}$ are closed subsets of $\mathfrak{g}_x^*$ which are invariant under the co-adjoint action of $\GrF{x}$ and dilations (graded or ungraded), i.e., 
					\begin{equation}\label{eqn:dilations_HN_sets}\begin{aligned}
						&\eta\in \HL{(\xi,x)}&&\implies -\eta \in \HL{(-\xi,x)},\\
						&\lambda\in \R_+,\eta\in \HL{(\xi,x)}&&\implies \lambda \eta\in \HL{(\xi,x)},\\
						&\lambda\in \R, \eta\in \HL{x}&&\implies \lambda \eta\in \HL{x}\\
						&\lambda\in \R_+^\nu,\eta\in \HL{(\xi,x)}&&\implies  \alpha_\lambda(\eta)\in \HL{(\xi,x)}\\
						&\lambda\in \R_+^\nu,\eta\in \HL{x}&&\implies  \alpha_\lambda(\eta)\in \HL{x},	
					\end{aligned}\end{equation}
				where $\alpha_\lambda$ is defined as in \eqref{eqn:dilations_basic_V} using the grading on $\mathfrak{g}^*_x$.
				\item One has \begin{equation}\label{eqn:Hellfer_Nourrigat_set_condition_union_cones}\begin{aligned}
					\HL{x}=\bigcup_{L\in \cG^{(0)}_x} L^\perp,
				\end{aligned}\end{equation}
				where $\cG^{(0)}_x$ is the set of tangent cones at $x$.
			\end{enumerate}
		\end{linkprop}
			We denote by $\widehat{\mathrm{H\!N}}_{(\xi,x)}$ and $\widehat{\mathrm{H\!N}}_{x}$, the set of irreducible  unitary representations of $\GrF{x}$ associated to the coadjoint orbits inside $\HL{(\xi,x)}$ and $\HL{x}$ using \eqref{eqn:Orbit_method_map}.
			The following proposition gives an extremely convenient alternate definition of $\widehat{\mathrm{H\!N}}_{x}$.
		\begin{linkprop}{weakly_contained_Helffer_Nourrigat}
			Let $x\in M$, $\pi$ be an irreducible unitary representation of $\GrF{x}$.
			Then $\pi\in \widehat{\mathrm{H\!N}}_x$ if and only if there exists $L\in \cG^{(0)}_x$ such that $\pi$ is weakly contained in $\Xi_{\GrF{x}/L,0}$, see \eqref{eqn:dfn_representation_quasi_induite}.
		\end{linkprop}
		\paragraph{Spectrum of $C^*_{\mathrm{inv}}\mathbb{G}$:} 
		Let \begin{equation*}\begin{aligned}
				\widehat{\mathrm{H\!N}}=\{(\pi,x):x\in M,\pi\in \widehat{\mathrm{H\!N}}_x\},\quad
				\widehat{\mathbb{G}}_{\mathrm{inv}}=\R_+^{\nu}\backslash 0\sqcup \widehat{\mathrm{H\!N}}\times \{0\}.
			\end{aligned}\end{equation*}
		We will construct a natural bijection 
			$\widehat{\mathbb{G}}_{\mathrm{inv}}\to \widehat{C^*_{\mathrm{inv}}\mathbb{G}}$
		where $\widehat{C^*_{\mathrm{inv}}\mathbb{G}}$ is the space of equivalence classes of irreducible representations of $C^*_{\mathrm{inv}}\mathbb{G}$.
		\begin{itemize}
			\item   Let $t\in \R_+^\nu\backslash 0$. By \Cref{ex:Regular_representation} and \Cref{rem:equivalence}, we have a representation $\pi_{t}$ of $\cinv$ which acts on $L^2(M)$, and is given by 
			\begin{equation}\label{eqn:pi_nonzero_parameter_KL2M}\begin{aligned}
				\pi_{t}(f)g(x):=\int_{M}f(x,y,t)g(y)		,\quad f\in \cbbG,g\in L^2(M).
			\end{aligned}\end{equation}
			The representation $\pi_t$ extends to an irreducible representation of $C^*_{\mathrm{inv}}\mathbb{G}$ because it factors through $\cK(L^2M)$.
			\item Let $(\pi,x)\in \widehat{\mathrm{H\!N}}$.
			We define $\pi_{\mathrm{inv}}:\cinv\to \cL(L^2(\pi))$ as follows: If $f\in \cinv$, $(V,\natural)$ a $\nu$-graded basis and $\tilde{f}$ are as in \eqref{eqn:sum_smooth_decompo_invariatn_function}, then 
			\begin{equation}\label{eqn:pi_inv_definition}\begin{aligned}
				\pi_{\mathrm{inv}}(f):=\pi(\natural_{x,0*}(\tilde{f}(\cdot,x,0))),
			\end{aligned}\end{equation}
			where $\natural_{x,0*}(\tilde{f})\in C^\infty_c(\GrF{x},\Omega^1 (\mathfrak{g}_x))$ is integration of $\tilde{f}$ along the fibers of $\natural_{x,0}:V\to \mathfrak{g}_x$.	
			
			\begin{linkprop}{representation_invariant}
				The representation $\pi_{\mathrm{inv}}$ is well-defined, i.e., $\pi_{\mathrm{inv}}(f)$ doesn't depend on the choice of $(V,\natural)$ and $\tilde{f}$. 
				Furthermore, $\pi_{\mathrm{inv}}$ extends to an irreducible representation of $C^*_{\mathrm{inv}}\mathbb{G}$.
				Moreover, 
					\begin{equation}\label{eqn:Cinfinty_cinv}\begin{aligned}
							C^\infty(\pi)=\pi(C^\infty_c(\mathfrak{G}_x,\Omega^1(\mathfrak{g}_x)))L^2(\pi)=\pi_{\mathrm{inv}}(\cinv)L^2(\pi),		
					\end{aligned}\end{equation}
			\end{linkprop}
			\end{itemize}
			
		\paragraph{Topology on $\widehat{\mathbb{G}}_{\mathrm{inv}}$:}
Let 
	\begin{equation*}\begin{aligned}
		\mathfrak{g}^*:=\{(\eta,x):x\in M,\eta \in \mathfrak{g}_x^*\}		
	\end{aligned}\end{equation*}
	be the disjoint union of the dual of all osculating Lie algebras.
	 We equip  $\mathfrak{g}^*$ with the topology which makes the following map a topological embedding:
			\begin{equation*}\begin{aligned}
				\mathfrak{g}^*\hookrightarrow V^*\times M,\quad (\eta,x) \mapsto (\eta\circ \natural_{x,0},x),
			\end{aligned}\end{equation*}
where $(V,\natural)$ is a $\nu$-graded basis.
			The topology doesn't depend on the choice of the $\nu$-graded basis. To see this, we write elements of one basis using another, see the proof of \Crefitem{thm:G0_locally_compact}{indep}. 
			The space $\widehat{\mathrm{H\!N}}$ is a quotient of $\mathfrak{g}^*$. We equip it with the quotient topology.
		\begin{dfn}\label{dfn:toplogy_mathbbG_hat}
		We define a topology $\widehat{\mathbb{G}}_{\mathrm{inv}}$ as follows: A subset $A\sqcup B\times \{0\}\subseteq \widehat{\mathbb{G}}_{\mathrm{inv}}$ where $A\subseteq \R_+^\nu\backslash 0$ and $B\subseteq \widehat{\mathrm{H\!N}}$ is closed if and only if
		\begin{enumerate}
						\item  The set $A$ is closed in $\R_+^\nu\backslash 0$.
			\item     The set $B$ is closed in $\widehat{\mathrm{H\!N}}$.
			\item If $(L,x,0)\in \overline{M\times A}$, where $\overline{M\times A}$ is the closure of $M\times A$ in $\mathbb{G}^{(0)}$, then $(\pi,x)\in B$ for any $\pi$ an irreducible unitary representation $\pi$ of $\mathfrak{G}_x
			$ weakly contained in $\Xi_{\mathfrak{G}_x/L,0}$.
		\end{enumerate}
		\end{dfn}
 		\begin{linkthm}{Type1_tangent_groupoid}
			 The $C^*$-algebra $C^*_{\mathrm{inv}}\mathbb{G}$ is separable of Type \RNum{1}. Furthermore, the map
					\begin{equation}\label{eqn:representations_of_tangent_groupoid}\begin{aligned}
						\widehat{\mathbb{G}}_{\mathrm{inv}}\to \widehat{C^*_{\mathrm{inv}}\mathbb{G}},\quad t \in \R_+^\nu\backslash 0\mapsto [\pi_t],\quad (\pi,x,0)\in \widehat{\mathrm{H\!N}}\times \{0\}\mapsto [\pi_{\mathrm{inv}}]		
					\end{aligned}\end{equation}
				is a homeomorphism, where $\widehat{C^*_{\mathrm{inv}}\mathbb{G}}$ is the spectrum of $C^*_{\mathrm{inv}}\mathbb{G}$, i.e., the space of equivalence classes of irreducible representations of $C^*_{\mathrm{inv}}\mathbb{G}$ equipped with the Fell topology.
		\end{linkthm}

		\paragraph{Non-singular representations:}
		The classical principal symbol of a classical pseudo-differential operator isn't well-defined at the cotangent vector $0$ which corresponds to the trivial representation of the tangent space.
		In our setting, we call the representations at which the principal symbol is defined, \textit{the non-singular Helffer-Nourrigat cone}. 
		Recall the decomposition $\mathfrak{g}_x=\oplus_{i=1}^\nu\mathfrak{g}_x^{\weak{i}} $ from \eqref{eqn:weakly_commuting_lie_algebra}. We define 
			\begin{equation*}\begin{aligned}
				\HLnon_{(\xi,x)}&=\{\eta\in \HL{(\xi,x)}:\eta_{|\mathfrak{g}_x^{\weak{i}}}\neq 0\ \forall i\in \bb{1,\nu}\}.\\
				\HLnon_{x}&=\{\eta\in \HL{x}:\eta_{|\mathfrak{g}_x^{\weak{i}}}\neq 0\ \forall i\in \bb{1,\nu}\}=\bigcup_{\xi\in T^*_xM\backslash 0}\HLnon_{(\xi,x)}.
			\end{aligned}\end{equation*}
		The set	$\{\eta\in \mathfrak{g}_x^*:\eta_{|\mathfrak{g}_x^{\weak{i}}}\neq 0\ \forall i\in \bb{1,\nu}\}$
		is closed under the co-adjoint action. Hence, the same holds for $\HLnon_{(\xi,x)}$ and $\HLnon_{x}$.
		So, by \eqref{eqn:Orbit_method_map}, we can define $\HatHLnon_{(\xi,x)}$ and $\HatHLnon_{x}$ as the representations corresponding to $\HLnon_{(\xi,x)}$ and $\HLnon_{x}$ respectively.
		We have, 
			\begin{equation}\label{eqn:HatHL__Nonsing_is_Union}\begin{aligned}
								\HatHLnon_{x}=\bigcup_{\xi\in T^*_xM\backslash 0}\HatHLnon_{(\xi,x)}.
			\end{aligned}\end{equation}		
		Finally, let
			\begin{equation*}\begin{aligned}
				\HatHLnon=\{(\pi,x):x\in M,\pi \in \HatHLnon_{x}\}
			\end{aligned}\end{equation*}
			equipped with the subspace topology from $\widehat{\mathrm{H\!N}}$.

\linksec{Principal symbol}{principal_symbol}
	Let $(\pi,x)\in \widehat{\mathrm{H\!N}}$, $u\in \cinvdis$. 
	By \eqref{eqn:Cinfinty_cinv}, we can define $\pi_{\mathrm{inv}}(u):C^{\infty}(\pi)\to C^{\infty}(\pi)$ by $\pi(u)\pi(f)\xi=\pi(u\ast f)\xi$ for all $f\in \cinv$ and $\xi\in L^2(\pi)$.
	This is well-defined because $\pi$ is non-degenerate.
 	By duality, $\pi_{\mathrm{inv}}(u)$ extends to a continuous linear map $C^{-\infty}(\pi)\to C^{-\infty}(\pi)$.
	\begin{dfn}\label{dfn:principal_symbol}
		Let $k\in \C^\nu,v\in \Psi^{k}_c(M)$. The principal symbol of $v$ at $(\pi,x)\in \HatHLnon$ is defined by
		\begin{equation}\label{eqn:dfn_principal_symbol}\begin{aligned}
			\sigma^{k}(v,\pi,x)\xi:=\lim_{\lambda_1\to +\infty,\cdots,\lambda_\nu\to +\infty}\pi_{\mathrm{inv}}\left(\lambda^{-k}\alpha_{\lambda}(I_k(f))\right)\xi,
		\end{aligned}\end{equation}
		 where $f\in \cinvf{k+\underline{1}}$ is any smooth invariant function such that $I_k(f)_{|\underline{1}}= v$.
		More generally, if $v\in \Psi^k(M)$, then the symbol is defined by $\sigma^k(v,\pi,x):=\sigma^k(\delta_g\ast v,\pi,x)$ where $g\in C^\infty_c(M)$ is any compactly supported smooth function such that $g(x)=1$. 
	\end{dfn}
	\begin{linkthm}{principal_symbol_pseudo_diff_well_defined}
		\begin{enumerate}
			\item\label{thm:principal_symbol_pseudo_diff_well_defined:negative_infinity} If $\xi\in C^{-\infty}(\pi)$, then \eqref{eqn:dfn_principal_symbol} converges in the topology of $C^{-\infty}(\pi)$.
			Furthermore, the linear map $\sigma^k(v,\pi,x):C^{-\infty}(\pi)\to C^{-\infty}(\pi)$ is continuous.
			\item\label{thm:principal_symbol_pseudo_diff_well_defined:2} The linear map $\sigma^k(v,\pi,x):C^{-\infty}(\pi)\to C^{-\infty}(\pi)$ doesn't depend on the choice of $g\in C^\infty_c(M)$ and $f\in \cinvf{k+\underline{1}}$.
			\item\label{thm:principal_symbol_pseudo_diff_well_defined:Infinity}
			If $\xi\in C^{\infty}(\pi)$, then \eqref{eqn:dfn_principal_symbol} converges in the topology of $C^{\infty}(\pi)$ and the induced linear map $\sigma^k(v,\pi,x):C^{\infty}(\pi)\to C^{\infty}(\pi)$ is continuous. 
			\item\label{thm:principal_symbol_pseudo_diff_well_defined:3} For all $k,k'\in \C^\nu,v\in \Psi^{k}(M),v'\in \Psi^{k'}(M)$, we have
					\begin{equation*}\begin{aligned}
						&\sigma^{k+k'}(vv',\pi,x)=\sigma^{k}(v,\pi,x)\sigma^{k'}(v',\pi,x):C^{-\infty}(\pi)\to C^{-\infty}(\pi)\\ &\langle \sigma^{k}(v,\pi,x)\xi,\eta\rangle_{L^2(\pi)}=\langle \xi,\sigma^{\overline{k}}(v^*,\pi,x)\eta \rangle_{L^2(\pi)},\quad \forall \xi\in C^\infty(\pi),\eta\in C^{-\infty}(\pi) \\
					\end{aligned}\end{equation*}
			\item\label{thm:principal_symbol_pseudo_diff_well_defined:5} If $v\in \Psi^{\prec k}(M)$, then $\sigma^k(v,\pi,x):C^{-\infty}(\pi)\to C^{-\infty}(\pi)$  vanishes.
			\item\label{thm:principal_symbol_pseudo_diff_well_defined:4} 	For any $\lambda\in (\Rpt)^\nu$, 
				$\sigma^{k}(v,\pi\circ \alpha_\lambda,x)=\lambda^k\sigma^{k}(v,\pi,x)$.
		\item\label{thm:principal_symbol_pseudo_diff_well_defined:concrete_computation}If $f\in C^\infty(M)= \DO^0(M)$, then  $\sigma^0(\delta_f,\pi,x)=f(x)\mathrm{Id}_{C^{-\infty}(\pi)}$
		\item\label{thm:principal_symbol_pseudo_diff_well_defined:concrete_computation2}If $X\in \cF^{k}$, then
		$\sigma^k(X,\pi,x)=\pi([X]_{k,x})$,
		where $[X]_{k,x}\in \mathfrak{g}_x$ defined in the paragraph preceding \eqref{eqn:Lie_bracket_gx} is seen as an element of $C^{-\infty}_{r,s}(\mathfrak{G}_x,\omegahalf)$, see 
		\Crefitem{ex:distributions_quasi_Lie_groupods}{inclusion_Lie_algebroid}.
		\end{enumerate}
	\end{linkthm}

	\Crefitem{thm:principal_symbol_pseudo_diff_well_defined}{3}, \Crefitem{thm:principal_symbol_pseudo_diff_well_defined}{5}, \Crefitem{thm:principal_symbol_pseudo_diff_well_defined}{concrete_computation} and \Crefitem{thm:principal_symbol_pseudo_diff_well_defined}{concrete_computation2} give a concrete way of
	computing the principal symbol of differential operators by writing the operator as sum of monomials.
\linksec{Sobolev spaces}{sobolev_spaces}
			
			\begin{dfn}[Sobolev spaces on $M$]
				\label{dfn:Sobolev_spaces_on_M}
					Let  $k\in \R^\nu$. 
			We denote by $H^k_{\mathrm{loc}}(M)$ the space of $w\in C^{-\infty}(M,\omegahalf)$ such that for all $v\in \Psi^{k}_c(M)$, $v\ast w\in L^2(M)$.
			We equip $H^k_{\mathrm{loc}}(M)$ with the weakest topology that makes the maps 
				\begin{equation*}\begin{aligned}
					H^k_{\mathrm{loc}}(M)\to L^2(M),\quad w\mapsto v\ast w
				\end{aligned}\end{equation*}
			continuous for every $v\in\Psi^{k}_c(M)$. 
			If $M$ is compact, then we will use the notation $H^k(M)$ instead.
			\end{dfn}

			\begin{linkthm}{bounded_M_Sobolev}
				\begin{enumerate}
					\item\label{thm:bounded_M_Sobolev:1} 		$H^{0}_{\mathrm{loc}}(M)=L^2_{\mathrm{loc}}(M)$, $C^{-\infty}(M,\omegahalf)=\bigcup_{k\in \R^\nu}H^k_{\mathrm{loc}}(M)$, and $C^{\infty}(M,\omegahalf)=\bigcap_{k\in \R^\nu}H^k_{\mathrm{loc}}(M)$.						
					\item\label{thm:bounded_M_Sobolev:2}     	For any $k\in \R^\nu$, $H^k_{\mathrm{loc}}(M)$ is a Fréchet space. It is a Hilbert space if $M$ is compact.
					\item\label{thm:bounded_M_Sobolev:3}  		If $k_i\leq l_i$ for all $i\in \bb{1,\nu}$, then $H^{l}_{\mathrm{loc}}(M)\subseteq H^k_{\mathrm{loc}}(M)$ and the inclusion is continuous. 
					\item\label{thm:bounded_M_Sobolev:4}  		If $k_i\leq l_i$ for all $i\in \bb{1,\nu}$ and at least one of the inequalities is strict and $M$ is compact, then the inclusion $H^{l}(M)\subseteq H^k(M)$  is compact.
					\item\label{thm:bounded_M_Sobolev:5}  		There exists $C\in \R$ such that for any $n\in\N$, $k\in \R^\nu$, if $\sum_{i=1}^{\nu}k_iN_i^{-\sgn(k_i)}\geq C+n$, then $H^k_{\mathrm{loc}}(M)\subseteq C^n(M)$, and the inclusion is continuous.
					\item\label{thm:bounded_M_Sobolev:6}  		Let $k\in\C^\nu$, $v\in \Psi^{k}(M)$. 
																For any $l\in \R^\nu$, if $w\in H^{\Re(k)+l}_{\mathrm{loc}}(M)$, then $v\ast w\in H^{l}_{\mathrm{loc}}(M)$. 
																Furthermore, the map $v\ast \cdot: H^{\Re(k)+l}_{\mathrm{loc}}(M)\to H^{l}_{\mathrm{loc}}(M)$ is continuous.
				\end{enumerate}
			\end{linkthm}
	\begin{dfn}[Sobolev spaces on representations]
	Let $(\pi,x)\in \HatHLnon$, $k\in \R^\nu$. 
	We denote by $H^k(\pi)$ the space of all $\xi\in C^{-\infty}(\pi)$ such that $\sigma^k(v,\pi,x)\xi\in L^2(\pi)$ for all $v\in \Psi^k_c(M)$.
	We equip $H^k(\pi)$ with the weakest topology that makes the maps 
		\begin{equation*}\begin{aligned}
				H^k(\pi)\to L^2(\pi),\quad \xi\mapsto \sigma^k(v,\pi,x)\xi
		\end{aligned}\end{equation*}
		continuous for every $v\in \Psi^k_c(M)$.
	\end{dfn}

	\begin{linkthm}{bounded_pi_Sobolev}
		\begin{enumerate}
			\item\label{thm:bounded_pi_Sobolev:1} 		$H^{0}(\pi)=L^2(\pi)$, $C^{-\infty}(\pi)=\bigcup_{k\in \R^\nu}H^k(\pi)$, $C^\infty(\pi)=\bigcap_{k\in \R^\nu}H^k(\pi)$.
			\item\label{thm:bounded_pi_Sobolev:2}     	The space $H^k(\pi)$ is a Hilbertian space.
			\item\label{thm:bounded_pi_Sobolev:3}  		If $k_i\leq l_i$ for all $i\in \bb{1,\nu}$, then $H^{l}(\pi)\subseteq H^k(\pi)$ and the inclusion is continuous.
			\item\label{thm:bounded_pi_Sobolev:4}  		If $k_i<l_i$ for all $i\in \bb{1,\nu}$, then the inclusion $H^{l}(\pi)\subseteq H^k(\pi)$ is compact. 
			\item\label{thm:bounded_pi_Sobolev:5}  		Let $k\in\C^\nu$, $v\in \Psi^{k}(M)$. 
						For any $l\in \R^\nu$, $\xi\in H^{\Re(k)+l}(\pi)$, one has $\sigma^k(v,\pi,x)\xi\in H^{l}(\pi)$. Furthermore, the map $\sigma^k(v,\pi,x):H^{\Re(k)+l}(\pi)\to H^{l}(\pi)$ is continuous.
		\end{enumerate}
	\end{linkthm}
	\paragraph{Vanishing of the principal symbol:}
	 	For classical pseudo-differential calculus, vanishing of the symbol implies that the operator has order $1$ less.
		In our calculus, we have the following result:
		\begin{linkthm}{vanishing_of_principal_symbol_Total}
				If $M$ is compact, $k\in \C^\nu$, $v\in \Psi^k(M)$, then the following are equivalent:
				\begin{enumerate}
				\item 	The principal symbol $\sigma^k(v,\pi,x):C^\infty(\pi)\to C^\infty(\pi)$ vanishes for all $(\pi,x)\in \HatHLnon$.
				\item 	For any (or for some) $l\in \R^\nu$,  the map $v\ast \cdot: H^{\Re(k)+l}(M)\to H^{l}(M)$ is compact.
				\end{enumerate}
		\end{linkthm}
		\paragraph{Compactness and Fredholmness of the symbol:}
		The following theorem has no analogue in the classical pseudo-differential calculus.
		This is because $H^{k}(\pi)=\C$ in the classical pseudo-differential calculus, and so, the symbol as a map is always Fredholm and compact.
		Furthermore, $\{(t\xi,x):t\in \Rpt\}$ is a closed subset of $T^*M\backslash \{0\}$ for any $(\xi,x)\in T^*M$, and so, the set $S$ defined below is always empty in the classical case.
		\begin{linkthm}{principal_symbol_characterizations}
		Let $k \in \C^\nu$, $v \in \Psi^k(M)$, and $(\pi,x) \in \HatHLnon$,
		\[
		S = \overline{\{(\pi \circ \alpha_\lambda, x) : \lambda \in (\Rpt)^\nu\}}
		\setminus \{(\pi \circ \alpha_\lambda, x) : \lambda \in (\Rpt)^\nu\},
		\]
		where the closure is taken in $\HatHLnon$.  
		The following hold:
		\begin{enumerate}
		\item\label{thm:principal_symbol_characterizations:V} For every $(\pi',x) \in S$, 
		$ C^\infty(\pi') \xrightarrow{\sigma^k(v,\pi',x)} C^\infty(\pi')$ vanishes  
		$\Longleftrightarrow$  
		for any (or for some) $l \in \R^\nu$, the map
		$
		H^{\Re(k)+l}(\pi) \xrightarrow{\sigma^k(v,\pi,x)} H^l(\pi)
		$
		is compact.

		\item\label{thm:principal_symbol_characterizations:I} 
		The following are equivalent:
		\begin{enumerate}
			\item\label{thm:principal_symbol_characterizations:I:a}  For every $(\pi',x) \in S$, 
		$ C^\infty(\pi') \xrightarrow{\sigma^k(v,\pi',x)} C^\infty(\pi')$ and $ C^\infty(\pi') \xrightarrow{\sigma^{\bar{k}}(v^*,\pi',x)} C^\infty(\pi')$ are injective
		\item\label{thm:principal_symbol_characterizations:I:b}   For every $(\pi',x) \in S$, the maps 
			\begin{equation*}\begin{aligned}
				\sigma^k(v,\pi',x)&:C^{-\infty}(\pi')\xrightarrow{\phantom{aaaaaaaaa}} C^{-\infty}(\pi')&&\\
				\sigma^k(v,\pi',x)&:H^l(\pi')\xrightarrow{\phantom{aaaaaaaaaaa}} H^{l-\Re(k)}(\pi')&&\forall l\in \R^\nu\\
				\sigma^k(v,\pi',x)&:C^\infty(\pi')\xrightarrow{\phantom{aaaaaaaaaa}} C^\infty(\pi')&&		
			\end{aligned}\end{equation*}
		are topological isomorphisms.
		\item\label{thm:principal_symbol_characterizations:I:c}  
		For any (or for some) $l\in \R^\nu$, the map
		$
		H^{\Re(k)+l}(\pi) \xrightarrow{  \sigma^k(v,\pi,x)} H^l(\pi)
		$
		is Fredholm.
		\end{enumerate}
		
		If the above holds, then the kernel of $H^{\Re(k)+l}(\pi) \xrightarrow{  \sigma^k(v,\pi,x)} H^l(\pi)$ consists of only smooth vectors (and so it doesn't depend on $l$).
		\end{enumerate}
		\end{linkthm}

\linksec{Maximal hypoellipticity and the main theorem}{main_theorem}
		
	If $k\in \C$, then we denote by $\Psi^k_{\mathrm{classical}}(M)$ ($\Psi^k_{c,\mathrm{classical}}(M)$) the subspace of $C^{-\infty}_{r,s}(M\times M,\omegahalf)$ of properly supported (compactly supported) classical pseudo-differential operators of order $k$.
	\begin{dfn}\label{dfn:Microlocality}
		Let $l\in \R^\nu$ and $w\in C^{-\infty}(M,\omegahalf)$.
		We define $\WF_{l}(w)\subseteq T^*M\backslash 0$ by the formula 
		\begin{equation*}\begin{aligned}
				(\xi,x)\notin \WF_l(w)\Leftrightarrow \exists w'\in H^l_{\mathrm{loc}}(M), \text{ such that } (\xi,x)\notin \WF(w-w').
			\end{aligned}\end{equation*}
	\end{dfn}
	It is clear that $\WF_l(w)$ is a closed cone. 
	We now give its main properties.
	\begin{linkthm}{multi_parameter_wave_front_set}
		\begin{enumerate}
			\item\label{thm:multi_parameter_wave_front_set:7} If $l_i\leq k_i$ for all $i\in \bb{1,\nu}$, then $\WF_l(w)\subseteq \WF_k(w)$.
			\item\label{thm:multi_parameter_wave_front_set:1} $\WF_l(w)=\emptyset$ if and only if $w\in H^l_{\mathrm{loc}}(M)$.
			\item\label{thm:multi_parameter_wave_front_set:2} $\WF(w)=\overline{\bigcup_{l\in \R^\nu}\WF_l(w)}$, and $\WF_l(w+w')\subseteq \WF_l(w)\cup\WF_l(w')$.
			\item\label{thm:multi_parameter_wave_front_set:3} If $k\in \C^\nu$ and $v\in \Psi^k(M)$, then $\WF_{l}(v\ast w)\subseteq \WF_{\Re(k)+l}(w)	$.
			\item\label{thm:multi_parameter_wave_front_set:5} One has $(\xi,x)\notin \WF_l(w)$ if and only if  there exists $\chi\in \Psi^0_{\mathrm{classical}}(M)$,  such that  $\chi$ is elliptic at $(\xi,x)$ and $\chi\ast w\in H^l_{\mathrm{loc}}(M)$.
		\item\label{thm:multi_parameter_wave_front_set:4} There exists
		$\chi\in \Psi^0_{\mathrm{classical}}(M)$ such that $\chi$ is elliptic on the complement of $\WF_l(w)$ and $\chi\ast w\in H^l_{\mathrm{loc}}(M)$.
		\end{enumerate}
	\end{linkthm}
		\begin{linkthm}{main_theorem}[Main theorem]
			Let $k\in\C^\nu$, $v\in \Psi^{k}(M)$,
				\begin{equation*}\begin{aligned}
					\Gamma=\left\{(\xi,x)\in T^*M\backslash 0:\forall (\pi,x)\in \HatHLnon_{(\xi,x)},\ \sigma^{k}(v,\pi,x):C^\infty(\pi)\to C^\infty(\pi) \text{ is injective}\right\} 
				\end{aligned}\end{equation*}
			The following hold: 
			\begin{enumerate}
				 \item\label{thm:main_theorem:left-inv} 	If $(\xi,x)\in \Gamma$, $\pi\in \HatHLnon_{(\xi,x)}$, then the operators 
					 \begin{equation*}\begin{aligned}
						& \sigma^{k}(v,\pi,x):C^{-\infty}(\pi)\to C^{-\infty}(\pi)&&
						\\ &\sigma^{k}(v,\pi,x):H^{l+k}(\pi)\to H^l(\pi)&&\text{ for all }l\in \R^\nu
						\\ &\sigma^{k}(v,\pi,x):C^\infty(\pi)\to C^\infty(\pi)&&
					 \end{aligned}\end{equation*}
				 are left-invertible.
				\item\label{thm:main_theorem:inj} The cone $\Gamma$ is open, and $\Gamma\subseteq \{(\xi,x)\in T^*M\backslash 0:(-\xi,\xi;x,x)\in \WF(v)\}$. 
				\item\label{thm:main_theorem:parametrix}
				For any closed cone $\Gamma'\subseteq \Gamma$. 
				There exists $v'\in C^{-\infty}_{r,s}(M\times M,\omegahalf)$ such that for any $l\in \R^\nu$ and $w\in C^{-\infty}(M,\omegahalf)$,
					\begin{equation}\label{eqn:parametrix_WF}\begin{aligned}
							\WF_l(v'\ast w)\subseteq \WF_{l-\Re(k)}(w), 	\ \WF_l(v^{\prime *}\ast w)\subseteq \WF_{l-\Re(k)}(w), \ \WF(v'\ast v\ast w-w)\cap \Gamma'=\emptyset.
					\end{aligned}\end{equation}
					In particular, by \MyCref{thm:multi_parameter_wave_front_set}[2], $\WF(v'\ast w)\subseteq \WF(w)$ and $\WF(v^{\prime *}\ast w)\subseteq \WF(w)$, and so $v$ is smooth off the diagonal.
				 \item\label{thm:main_theorem:micro_local_max_hypo}    For any $l\in \R^\nu$, and $w\in C^{-\infty}(M,\omegahalf)$,
					 \begin{equation}\label{eqn:sdqjofjlkqsdjf}\begin{aligned}
						 					 \WF_{l+\Re(k)}(w)\cap \Gamma=\WF_{l}(v\ast w)\cap \Gamma.
					 \end{aligned}\end{equation}
				 Furthermore, if $\Gamma'$ is any cone such that \eqref{eqn:sdqjofjlkqsdjf} holds (with $\Gamma$ replaced by $\Gamma'$) for some $l\in \R^\nu$, then $\Gamma'\subseteq \Gamma$.
				\item\label{thm:main_theorem:inequa}For any $v''\in \Psi^{k}_c(M)$, there exists $\chi\in \Psi^0_{\mathrm{classical}}(M)$, $\chi'\in C^\infty_c(M\times M,\omegahalf)$ and $C>0$ such that $\chi$ is elliptic on $\Gamma$ and 
						 \begin{equation}\label{eqn:inequ_microlocal_main_thm}\begin{aligned}
								 \norm{\chi\ast v''\ast w}_{L^2(M)}\leq C(\norm{v\ast w}_{L^2(M)}+\norm{\chi'\ast w}_{L^2(M)}),\quad \forall w\in C^\infty(M,\omegahalf).
						 \end{aligned}\end{equation}
					 If $\Gamma=T^*M\backslash 0$, then one can suppose that $\chi$ is the identity operator.
				
					\item\label{thm:main_theorem:Fredholm} If $M$ is compact and $\Gamma=T^*M\backslash 0$, then
						$v\ast \cdot:H^{l+\Re(k)}(M)\to H^{l}(M)$ is left-invertible modulo compact operators for any $l\in \R^\nu$.
			\end{enumerate}
			We say that $v$ is microlocally maximally hypoelliptic of order $k$ on a cone $\Gamma'$ if $\Gamma'\subseteq \Gamma$. 
			If $\Gamma=T^*M\backslash 0$, then we say that $v$ is maximally hypoelliptic of order $k$.
		\end{linkthm}
		\begin{rem}\label{rem:on_main_thm}
		\begin{enumerate}
			\item  		Inequalities \eqref{eqn:inequ_microlocal_main_thm} are called the microlocal maximal hypoellipticity inequalities.
			\item   We defer the question of whether the parametrix $v'$ belongs to $\Psi^{-k}(M)$ to a forthcoming article.
		\end{enumerate}
		\end{rem}
		We also have the two-sided version of \MyCref{thm:main_theorem}
		\begin{linkthm}{main_theorem_double_sided}
			Let $k\in\C^\nu$, $v\in \Psi^{k}(M)$,
				\begin{equation*}\begin{aligned}
					\Gamma=\{(\xi,x)\in T^*M\backslash 0:\forall (\pi,x)\in \HatHLnon_{(\xi,x)},\ &\sigma^{k}(v,\pi,x):C^\infty(\pi)\to C^\infty(\pi) \text{ and } \\&\sigma^{\bar{k}}(v^*,\pi,x):C^\infty(\pi)\to C^\infty(\pi)\text{ are injective}\} 
				\end{aligned}\end{equation*}
			The following hold: 
			\begin{enumerate}
				 \item 	If $(\xi,x)\in \Gamma$, $\pi\in \HatHLnon_{(\xi,x)}$, then the operators 
					 \begin{equation*}\begin{aligned}
						& \sigma^{k}(v,\pi,x):C^{-\infty}(\pi)\to C^{-\infty}(\pi)&&
						\\ &\sigma^{k}(v,\pi,x):H^{l+k}(\pi)\to H^l(\pi)&&\text{ for all }l\in \R^\nu
						\\ &\sigma^{k}(v,\pi,x):C^\infty(\pi)\to C^\infty(\pi)&&
					 \end{aligned}\end{equation*}
				 are topological isomorphisms.
				\item The cone $\Gamma$ is open, and $\Gamma\subseteq \{(\xi,x)\in T^*M\backslash 0:(-\xi,\xi;x,x)\in \WF(v)\cap\WF(v^*)\}$. 
				\item
				For any closed cone $\Gamma'\subseteq \Gamma$, 
				there exists $v'\in C^{-\infty}_{r,s}(M\times M,\omegahalf)$ such that for any $l\in \R^\nu$ and $w\in C^{-\infty}(M,\omegahalf)$,
					\begin{equation}\label{eqn:parametrix_WF2}\begin{aligned}
							\WF_l(v'\ast w)\subseteq \WF_{l-\Re(k)}(w), 	\quad \WF_l(v^{\prime *}\ast w)\subseteq \WF_{l-\Re(k)}(w),\\
							 \ \WF(v'\ast v\ast w-w)\cap \Gamma'=\WF( v\ast v'\ast w-w)\cap \Gamma'=\emptyset.
					\end{aligned}\end{equation}
		
				 \item    For any $l\in \R^\nu$, and $w\in C^{-\infty}(M,\omegahalf)$,
					 \begin{equation}\label{eqn:sdqjofjlkqsdjf2}\begin{aligned}
						 					 \WF_{l+\Re(k)}(w)\cap \Gamma=\WF_{l}(v\ast w)\cap \Gamma=\WF_{l}(v^*\ast w)\cap \Gamma.
					 \end{aligned}\end{equation}
				 Furthermore, if $\Gamma'$ is any cone such that \eqref{eqn:sdqjofjlkqsdjf} holds (with $\Gamma$ replaced by $\Gamma'$) for some $l\in \R^\nu$, then $\Gamma'\subseteq \Gamma$.
					\item If $M$ is compact and $\Gamma=T^*M\backslash 0$, then
						$v\ast \cdot:H^{l+\Re(k)}(M)\to H^{l}(M)$ is Fredholm for any $l\in \R^\nu$.
			\end{enumerate}
			We say that $v$ is microlocally maximally hypoelliptic of order $k$ on a cone $\Gamma'$ if $\Gamma'\subseteq \Gamma$. 
			If $\Gamma=T^*M\backslash 0$, then we say that $v$ is maximally hypoelliptic of order $k$.
			If $v$ satisfies the above conditions, then $v$ is called bi-maximally hypoelliptic of order $k$.
		\end{linkthm}

\linksec{Extension of parameters}{extension_of_parameters}
		Let $\tilde{\nu},\dbtilde{\nu}\in \bb{1,\nu}$ such that $\nu=\tilde{\nu}+\dbtilde{\nu}$.
		Let $\tilde{\cF}^{\bullet}$ and $\dbtilde{\cF}^\bullet$ be as in \eqref{eqn:extension_bigradedStructure} and \eqref{eqn:cF_tilde_dbtilde}. 
		Since $\tilde{\cF}$ and $\dbtilde{\cF}$ are weakly commutative, we can define all the objects we defined so far for $\tilde{\cF}$ and $\dbtilde{\cF}$. The corresponding objects are denoted by $\tilde{\Psi}^\bullet(M)$, $\tilde{\mathfrak{g}}_x$, $\tilde{\mathfrak{G}}_x$ $\tilde{\HL{}}$, $\tilde{H}^\bullet$, $\tilde{\sigma}$ and $\dbtilde{\Psi}^\bullet(M)$, $\dbtilde{\mathfrak{g}}_x$, $\dbtilde{\mathfrak{G}}_x$ $\dbtilde{\HL{}}$, $\dbtilde{H}^\bullet$, $\dbtilde{\sigma}$. 
		\begin{linkthm}{extending_parameters_pseudo_diff}
			\begin{enumerate}
				\item\label{thm:extending_parameters_pseudo_diff:inclusion}   If $k\in \C^{\tilde{\nu}},l\in \C^{\dbtilde{\nu}}$, then 
					\begin{equation*}\begin{aligned}
						\tilde{\Psi}^{k}(M)\ast \dbtilde{\Psi}^{l}(M)\subseteq \Psi^{(k,l)}(M),\quad \forall k\in \C^{\tilde{\nu}},l\in \C^{\dbtilde{\nu}}.\\
					\end{aligned}\end{equation*}
				In particular, $\tilde{\Psi}^{k}(M)\subseteq \Psi^{(k,0,\cdots,0)}(M)$, and $\dbtilde{\Psi}^{l}(M)\subseteq \Psi^{(0,\cdots,0,l)}(M)$.
				Furthermore, 
					\begin{equation}\label{eqn:commutator_lower_order_extension}\begin{aligned}
						[\tilde{\Psi}^{k}(M), \dbtilde{\Psi}^{l}(M)]\subseteq \Psi^{\prec(k,l)}(M)
					\end{aligned}\end{equation}
				
				\item\label{thm:extending_parameters_pseudo_diff:Sobolev}
				 If $k\in \R^{\tilde{\nu}},l\in \R^{\dbtilde{\nu}}$, then 
					\begin{equation*}\begin{aligned}
										\tilde{H}^k_{\mathrm{loc}}(M)=H^{(k,0,\cdots,0)}_{\mathrm{loc}}(M),\quad  \dbtilde{H}^l_{\mathrm{loc}}(M)=H^{(0,\cdots,0,l)}_{\mathrm{loc}}(M).		
					\end{aligned}\end{equation*}
				\item\label{thm:extending_parameters_pseudo_diff:osculating} For any $x\in M$, $\tilde{\mathfrak{g}}_x=\oplus_{i\in \bb{1,\tilde{\nu}}}\mathfrak{g}^{\widehat{i}}_x$ and $\dbtilde{\mathfrak{g}}_x=\oplus_{i\in \bb{\tilde{\nu}+1,\nu}}\mathfrak{g}^{\widehat{i}}_x$, see \eqref{eqn:weakly_commuting_lie_algebra}.
					In particular, 
						\begin{equation}\label{eqn:osculating_Lie_algebra_sum_extension}\begin{aligned}
							\mathfrak{g}_x=\tilde{\mathfrak{g}}_x\oplus \dbtilde{\mathfrak{g}}_x, \text{ and } \mathfrak{G}_x= \tilde{\mathfrak{G}}_x\times \dbtilde{\mathfrak{G}}_x.
						\end{aligned}\end{equation}
				\item\label{thm:extending_parameters_pseudo_diff:Helffer_Nourrigat} For any $x\in M$, 
					\begin{equation*}\begin{aligned}
						\HN_x\subseteq \{\xi+\eta:\xi\in \tilde{\HN}_x,\eta\in\dbtilde{\HN}_x\},
					\end{aligned}\end{equation*}
				where we are using the decomposition $\mathfrak{g}_x^*=\tilde{\mathfrak{g}}^*_x\oplus \dbtilde{\mathfrak{g}}_x^*$.
			\end{enumerate}
		\end{linkthm}
		Let $\xi\in \tilde{\HN}^{\mathrm{Nonsing}}_x$, $\eta\in\dbtilde{\HN}^{\mathrm{Nonsing}}_x$ such that $\xi+\eta\in \HN_x$. 
		So, automatically $\xi+\eta\in \HLnon$.
		Let $\tilde{\pi}$ and $\dbtilde{\pi}$ be the representations of $\tilde{\mathfrak{G}}_x$, $\dbtilde{\mathfrak{G}}_x$ associated by Kirillov's orbit method to $\xi$ and $\eta$.  
		By Kirillov's orbit method, the representation $\pi=\tilde{\pi}\otimes \dbtilde{\pi}$ acting on $L^2(\tilde{\pi})\otimes L^2(\dbtilde{\pi})$ is the representation associated by Kirillov's orbit method to $\xi+\eta$, see \Cref{chap:Orbit_method}.
		\begin{linkthm}{principal_symbol_extension}
			We have
			\begin{equation}\label{eqn:qkosdjoqskjdkoqsjdkqsjdkqs}\begin{aligned}
				\sigma^{(k,0,\cdots,0)}(v,\pi,x)=\tilde{\sigma}^{k}(v,\tilde{\pi},x)\otimes \mathrm{Id}_{L^2(\dbtilde{\pi})},\quad \forall k\in \C^{\tilde{\nu}},v\in \tilde{\Psi}^{k}(M)\\
				\sigma^{(0,\cdots,0,l)}(v',\pi,x)=\mathrm{Id}_{L^2(\tilde{\pi})}\otimes \dbtilde{\sigma}^{l}(v',\dbtilde{\pi},x),\quad \forall l\in \C^{\dbtilde{\nu}},v'\in \dbtilde{\Psi}^{l}(M).
			\end{aligned}\end{equation}
			Furthermore, 
				\begin{equation}\label{eqn:Sobolev_spaces_extension_representation}\begin{aligned}
					H^{(k,l)}(\pi)=\tilde{H}^k(\tilde{\pi})\otimes \dbtilde{H}^l(\dbtilde{\pi}),\quad\forall (k,l)\in \R^\nu.
				\end{aligned}\end{equation}
		\end{linkthm}
	\section{Example}\label{sec:example_pseudo-diff}
			In this section, we will give some computations of the Helffer-Nourrigat cone and the principal symbol in some concrete example.
			We fix $N\geq 2$. 
			Consider the structure defined in \Crefitem{exs:exs_weighted}{Weighted_bigraded} associated to the families $\{\partial_x,\partial_y\}$ and $\{\partial_x,x^{N-1}\partial_y\}$.
			We will use the notation of \MyCref{sec:tangent_groupoid_exs} where this structure was studied.
			A star is added to denote the dual elements of \eqref{eqn:kqsdkfjqposjfkpoqsfdf} which form a basis of the osculating group, see \eqref{eqn:qksodfjopkqsdjkofjqskopdfjqspd}.
			We will only compute the non-singular Helffer-Nourrigat cone and the corresponding representations.

			If $x\neq 0$, then 
			\begin{equation*}\begin{aligned}
				\HLnon_{(x,y)}= \left\{\xi X_1^*+\eta X_2^*+\mu \xi Y^*+x^{N-1}\mu \eta Z_1^*:(\xi,\eta)\in T^*_{(x,y)}\R^2\backslash 0,\mu\in \Rpt\right\}\subseteq  \mathfrak{g}_{(x,y)}^*.
			\end{aligned}\end{equation*}
			If $x=0$, then
			\begin{equation*}\begin{aligned}
				\HLnon_{(0,y)}=&\left\{\xi X_1^*+\eta X_2^*+\mu Y^*:\xi,\eta,\mu\in \R,\mu\xi>0\right\}\\
				\sqcup&\left\{ \eta X_2^*+\mu Y^*+\sum_{k=1}^N a^{k-1} b Z_k^*:\eta,b\in\R^\times,\mu,a\in \R,a^{N-1}b\eta\geq 0\right\}\\
									\sqcup&\left\{ \eta X_2^*+\mu Y^*+  b Z_N^*:\eta \in \R^\times, (\mu,b)\in \R^2\backslash (0,0),b\eta\geq 0\right\}
			\end{aligned}\end{equation*}
			Furthermore, one can check that for any $(\xi,\eta;x,y)\in T^*\R^2\backslash 0$, 
				\begin{equation*}\begin{aligned}
					a X_1^*+bX_2^*+cY+\sum_{i=1}^Nd_i Z_i\in \HLnon_{(\xi,\eta;x,y)}		\iff \exists\lambda\in \Rpt,(a,b)=\lambda(\xi,\eta).
				\end{aligned}\end{equation*}
			
		We compute the corresponding representations using the Kirillov's orbit method.	
			If $x\neq 0$, then each element in $\HL{(x,y)}$ gives a one-dimensional representation of $\mathfrak{G}_{(x,y)}$ whose derivative is:
				\begin{equation*}\begin{aligned}
					\pi_{\xi,\eta,\mu}(X_1)=i\xi,\quad \pi_{\xi,\eta,\mu}(X_2)=i \eta,\quad \pi_{\xi,\eta,\mu}(Y)=i\mu \xi,\quad \pi_{\xi,\eta,\mu}(Z_1)=ix^{N-1}\mu \eta,				
				\end{aligned}\end{equation*}
				and $\pi_{\xi,\eta,\mu}(Z_j)=0$ for all $j\in \bb{2,N}$.
			
			If $x=0$, the decomposition of $\HL{(0,y)}$ into co-Adjoint orbits is given by 
				\begin{equation*}\begin{aligned}
					\HLnon{(0,y)}=&\bigsqcup_{\xi,\eta,\mu\in \R,\mu\xi>0}\left\{\xi X_1^*+\eta X_2^*+\mu Y^*\right\}\\					
	&\bigsqcup_{\eta\in \Rpt,(\mu,b)\in \R^2\backslash 0}\{\eta X_2^*+\mu Y^*+bZ_1^*\} \\
	&\bigsqcup_{\eta,b\in \R,b\eta>0}	\Ad^*(\mathfrak{G}_{(0,y)})\left(\eta X_2^*+bZ_N^*\right).
				\end{aligned}\end{equation*}
				In the following computation, we added a superscript to distinguish the families of representations.
			The first family gives the one-dimensional representation of $\mathfrak{G}_{(x,y)}$ whose derivative is:
				\begin{equation*}\begin{aligned}
					\pi^{(1)}_{\xi,\eta,\mu}(X_1)=i \xi,\quad \pi^{(1)}_{\xi,\eta,\mu}(X_2)=i \eta,\quad \pi^{(1)}_{\xi,\eta,\mu}(Y)=i\mu,\quad \pi^{(1)}_{\xi,\eta,\mu}(Z_j)=0,\quad \forall j\in \bb{1,N}.		
				\end{aligned}\end{equation*}
			The second family gives one-dimensional representations whose derivative is:
				\begin{equation*}\begin{aligned}
					\pi^{(2)}_{\eta,\mu,b}(X_1)=0,\ \pi^{(2)}_{\eta,\mu,b}(X_2)=i \eta,\ \pi^{(2)}_{\eta,\mu,b}(Y)=i\mu ,\ \pi^{(2)}_{\eta,\mu,b}(Z_1)=i b,\ \pi^{(2)}_{\eta,\mu,b}(Z_j)=0 \quad\forall j\in \bb{2,N},		
				\end{aligned}\end{equation*}
			The third family gives infinite-dimensional representations acting on $L^2(\R)$ whose derivative is:
				\begin{equation*}\begin{aligned}
					\pi^{(3)}_{\eta,b}(X_1)=0,\quad \pi^{(3)}_{\eta,b}(X_2)=i \eta ,\quad \pi^{(3)}_{\eta,b}(Y)=\odv{}{t},\quad \pi^{(3)}_{\eta,b}(Z_j)=i t^{N-j} b,\quad \forall j\in \bb{1,N},		
				\end{aligned}\end{equation*}
			
			Fix $c\in \C$. Consider the differential operator 
				\begin{equation*}\begin{aligned}
					D_c=\left(\frac{\partial^2}{\partial x^2}+\frac{\partial^2}{\partial y^2}\right)	\left(\frac{\partial^2}{\partial x^2}+x^{2(N-1)}\frac{\partial^2}{\partial y^2}\right)^{N}+c\frac{\partial^{4}}{\partial y^{4}}\\
				\end{aligned}\end{equation*}
				It belongs to $\DO^{(2,2N)}(\R^2)$ because it can be written as 
					\begin{equation*}\begin{aligned}
								D_c=\left(X_1^2+X_2^2\right)	\left(Y^2+Z_1^2\right)^N+cZ_N^2X_2^{2},
					\end{aligned}\end{equation*}
			The principal symbol is given by 
				\begin{equation*}\begin{aligned}
					&\sigma^{(2,2N)}(D_c,\pi_{\xi,\eta,\mu},(x,y))&&=		\mu^2(\xi^2+\eta^2)(\xi^2+x^{2(N-1)}\eta^2)&&& x\neq 0\\
					&\sigma^{(2,2N)}(D_c,\pi^{(1)}_{\xi,\eta,\mu},(0,y))&&=		\mu^2(\xi^2+\eta^2)&&&\\
					&\sigma^{(2,2N)}(D_c,\pi^{(2)}_{\eta,\mu,b},(0,y))&&=		\eta^2(\mu^2+b^2)&&& \\
					&\sigma^{(2,2N)}(D_c,\pi^{(3)}_{\eta,b},(0,y))&&=	-\eta^2\left(\frac{\mathrm{d}^2}{\mathrm{d}t^2}-b^2t^{2(N-1)}\right)^N+cb^2\eta^2&&& \\
				\end{aligned}\end{equation*}
			The first three are always non-zero.
			It follows from \MyCref{thm:principal_symbol_characterizations}[I] that $\sigma^{(2,2N)}(D_c,\pi_{\eta,b}^{(3)},(0,y)):H^{(2,2N)}(\pi_{\eta,b}^{(3)})\to L^2(\R)$ is a Fredholm operator.
			So, by taking $c=0$, and using \MyCref{thm:bounded_M_Sobolev}[4], it follows that $\left(\frac{\mathrm{d}^2}{\mathrm{d}t^2}-b^2t^{2(N-1)}\right)^N$ as an operator $L^2(\R)\to L^2(\R)$ has a compact resolvent.
			Since it is  self-adjoint, it is diagonalizable. Its eigenvectors are elements of the kernel of $\sigma^{(2,2N)}(D_c,\pi_{\eta,b}^{(3)},(0,y))$ as $c$ varies.
			Hence, by \MyCref{thm:principal_symbol_characterizations}[I], the eigenvectors are smooth vectors of the representation $\pi_{\eta,b}^{(3)}$.
			By \MyCref{thm:smooth_vectors_Kirillov}, smooth vectors of $\pi_{\eta,b}^{(3)}$ are Schwartz functions on $\R$.
			Let $(\lambda_n)\subseteq \R$ be the eigenvalues of $\frac{1}{b^2}\left(\frac{\mathrm{d}^2}{\mathrm{d}t^2}-b^2t^{2(N-1)}\right)^N$.
			The sequence of eigenvalues doesn't depend on $b$ because all these operators are unitarily equivalent to each other the dilation operators.
			To summarize, the following hold:
			\begin{prop}\label{prop:sdkfjmdsqfqsd}
				For any $c\in \C$, the operator $D_c$ is microlocally maximally hypoelliptic of order $(2,2N)$ on the cone $\{(\xi,\eta;x,y)\in T^*\R^2\backslash 0:(\xi,x)\neq (0,0)\}$.
				Furthermore, it is maximally hypoelliptic of order $(2,2N)$ if and only if $c\notin\{\lambda_n:n\in \N\}$.
			\end{prop}
			We remark that the operator $D_c$ also belongs to $\DO^{(2+2N,0)}(\R^2)$ (that is it has classical order is $2+2N$) and it belongs to $\DO^{(0,4N)}(\R^2)$ because it can be written as 
				\begin{equation*}\begin{aligned}
				D_c=\left(Y^2+Z_N^2\right)	\left(Y^2+Z_1^2\right)^N+cZ_N^4.		
				\end{aligned}\end{equation*}
			The order $4N$ is optimal because of \MyCref{thm:principal_symbol_pseudo_diff_well_defined}[5] and
				\begin{equation*}\begin{aligned}
					\sigma^{(0,4N)}(D_c,	\pi_{\xi,\eta,\mu}^{(1)},(0,y))=(-1)^{N+1}\mu^{2+2N}\neq 0		
				\end{aligned}\end{equation*}
			Nevertheless, we have 
				\begin{equation*}\begin{aligned}
				\sigma^{(0,4N)}(D_c,\pi_{\eta,\mu,b}^{(2)},(0,y))=(-1)^{N+1}\mu^2(\mu^2+b^2)^N
				\end{aligned}\end{equation*}
			which vanishes when $\mu=0$. So, $D_c$ is not maximally hypoelliptic of order $(0,4N)$ (for any $c\in \C$).
			It is also not elliptic which means it is not maximally hypoelliptic of order $(2+2N,0)$.
			\chapter{Proofs of the results in Chapter \ref{chap:chapter_bi-garded_pseudo_diff}}\label{chap:chapter_bi-garded_pseudo_diff_proof}
\section{Proofs of the results in Section \ref{sec:bi_graded_pseudo_diff}}\label{sec:bi_graded_pseudo_diff_proof}	
	\begin{linkproof}{integralIkl}
		By symmetry, we only consider the first integral.
		By \eqref{eqn:invariant_vanish_algebra:alternate}, we can suppose that $f=h\ast \theta_l(D)$ for some $h\in \cinv$, $l\in \Z_+^\nu$, $D\in \DO^l(M)$
		such that $k-l\in \C^\nu_{<0}$.
		By \eqref{eqn:homogenity_theta}, 
			\begin{equation*}\begin{aligned}
				\int_{]0,1]^\nu} g\ast \alpha_{\lambda}(f)\frac{\odif{\lambda}}{\lambda^{k+\underline{1}}}=\left(\int_{]0,1]^\nu}g\ast \alpha_{\lambda}(h) \frac{\odif{\lambda}}{\lambda^{k-l+\underline{1}}} \right)\ast \theta_l(D)
			\end{aligned}\end{equation*}
		The integral $\int_{]0,1]^\nu}g\ast \alpha_{\lambda}(h) \frac{\odif{\lambda}}{\lambda^{k-l+\underline{1}}}$
		converges using \Cref{thm:equivariance_Cn_norm}.
	\end{linkproof}

		The following identities which follow from \eqref{eqn:homogenity_theta} will be used repeatedly
			\begin{equation}\label{eqn:identities_I_Use_a_alot_vector_field_out_2}\begin{aligned}
				I_k(\alpha_\lambda(f))&=\alpha_\lambda(I_k(f)),\\   I_{k+l}(\theta_{l}(D)\ast f)&=\theta_{l}(D)\ast I_{k}(f),\\ I_{k+l}(f\ast \theta_{l}(D))&= I_{k}(f)\ast \theta_{l}(D),
			\end{aligned}\end{equation}
			where $\lambda\in (\Rpt)^\nu$ and $f\in \cinvf{k}$. It is useful in this chapter to introduce the spaces
				\begin{equation*}\begin{aligned}
					\Psi^{k}(\mathbb{G})=I_k(\cinvf{k+\underline{1}} ),\quad
						\Psi^{\preceq k}(\mathbb{G}):=\sum_{l\in \C^\nu,l\preceq k}\Psi^{l}(\mathbb{G}), \quad\Psi^{\prec k}(\mathbb{G}):=\sum_{l\in \C^\nu,l\prec k}\Psi^{l}(\mathbb{G}).		
					\end{aligned}\end{equation*}
	By \eqref{eqn:identities_I_Use_a_alot_vector_field_out_2}, we obtain 
	\begin{equation}\label{eqn:main_prop_bi_graded_pseudo_diff:diff}\begin{aligned}
		\theta_l(\DO^{l}(M))\ast \Psi^{k}(\mathbb{G})\subseteq \Psi^{k+l}(\mathbb{G}),\quad \Psi^{k}(\mathbb{G})\ast \theta_l(\DO^{l}(M))\subseteq \Psi^{k+l}(\mathbb{G}).
	\end{aligned}\end{equation}

		\begin{lem}\label{thm:main_prop_bi_graded_pseudo_diff}
			\begin{enumerate}
				\item\label{thm:main_prop_bi_graded_pseudo_diff:3} For any $k,l\in \C^\nu$,
					 $\Psi^{k}(\mathbb{G})^*=\Psi^{\overline{k}}(\mathbb{G})$ and $\Psi^{k}(\mathbb{G})\ast \Psi^{l}(\mathbb{G})\subseteq \Psi^{k+l}(\mathbb{G})$.
				\item\label{thm:main_prop_bi_graded_pseudo_diff:L1} If $k\in \C^\nu_{<0}$, then $\Psi^{k}(\mathbb{G})\subseteq  L^1(\mathbb{G},\omega)$ for any $\omega\in C^\infty(\mathbb{G},\Omega^{\frac{1}{2},-\frac{1}{2}})$ obtained from \Cref{thm:L1normRiemMetric}.
				\item\label{thm:main_prop_bi_graded_pseudo_diff:incl}If $l\in \Z_+^\nu$, then $\theta_l(\DO^{l}_c(M))\subseteq \Psi^{\preceq l}(\mathbb{G})$.

			\end{enumerate}
		\end{lem}
		
		\begin{proof}

			The identity $\Psi^{k}(\mathbb{G})^*=\Psi^{\overline{k}}(\mathbb{G})$ follows from \Crefitem{thm:invariant_vanish_algebra}{algebra}.
			To prove $\Psi^{k}(\mathbb{G})\ast \Psi^{l}(\mathbb{G})\subseteq \Psi^{k+l}(\mathbb{G})$, by using \eqref{eqn:dfn_algebra_vanishing_Fourier_invariant}, \eqref{eqn:invariant_vanish_algebra:alternate}, \eqref{eqn:identities_I_Use_a_alot_vector_field_out_2} and \eqref{eqn:main_prop_bi_graded_pseudo_diff:diff},
			it suffices to show that 
			if $k,l\in \C^\nu_{<0}$ and $f,g\in \cinv$, then $I_k(f)\ast I_{l}(g)\in \Psi^{k+l}(\mathbb{G})$.
			For each $S\subseteq \bb{1,\nu}$, let 
				\begin{equation*}\begin{aligned}
					\Lambda_S=\{(\lambda,\mu)\in ]0,1]^{2\nu}:\forall i\in S, \lambda_i\leq \mu_i,\ \forall i\in S^c, \lambda_i>\mu_i \}.
				\end{aligned}\end{equation*}
			By a simple change of variables, we have
			\begin{equation}\label{eqn:qimskjfiqjsdmlfjlkmqsdmfjqsdjflmqjsmdjf2}\begin{aligned}
					I_k(f)\ast I_{l}(g)
					= \int_{]0,1]^{2\nu}}\alpha_\lambda(f)\ast  \alpha_{\mu}(g)\frac{\odif{\lambda} \odif{\mu} }{\lambda^{k+\underline{1}}\mu^{l+\underline{1}}} 
					&=\sum_{S\subseteq \bb{1,\nu}}\int_{\Lambda_S}\alpha_\lambda(f)\ast  \alpha_{\mu}(g)\frac{\odif{\lambda} \odif{\mu} }{\lambda^{k+\underline{1}}\mu^{l+\underline{1}}}
					\\&=\sum_{S\subseteq \bb{1,\nu}}I_{k+l}\left( \int_{]0,1]^{\nu}}\alpha_{\lambda_{S}}(f) \ast \alpha_{\lambda_{S^c}}(g)\frac{\odif{\lambda}}{\lambda^{\tau_S}} \right),
					\end{aligned}\end{equation}
			where 
				\begin{equation}\label{eqn:qimskjfiqjsdmlfjlkmqsdmfjqsdjflmqjsmdjf}\begin{aligned}
					\lambda_S=\begin{cases}
						\lambda_i &\text{if }i\in S\\
						1 &\text{if }i\notin S
					\end{cases},\quad \lambda_{S^c}=\begin{cases}
						\lambda_i &\text{if }i\notin S\\
						1 &\text{if }i\in S
					\end{cases},\quad \tau_{Si}=\begin{cases}
						k_i+1 &\text{if }i\in S\\
						l_i+1 &\text{if }i\notin S
					\end{cases}	.
				\end{aligned}\end{equation}
			By \Cref{thm:estimate_weakly_commut_convolution}, the integral $\int_{]0,1]^{\nu}}\alpha_{\lambda_{S}}(f) \ast \alpha_{\lambda_{S^c}}(g)\frac{\odif{\lambda}}{\lambda^{\tau_S}}$ defines an element of $\cinv$.
			This finishes the proof of \Crefitem{thm:main_prop_bi_graded_pseudo_diff}{3}.
			\Crefitem{thm:main_prop_bi_graded_pseudo_diff}{L1} follows immediately from \eqref{eqn:L1norm_invariance}.

			To prove \Crefitem{thm:main_prop_bi_graded_pseudo_diff}{incl}, by \eqref{eqn:main_prop_bi_graded_pseudo_diff:diff}, it suffices to show that $\theta_0(C^\infty_c(M))\subseteq \Psi^{\preceq 0}(\mathbb{G})$.
			Let $f\in C^\infty_c(M)$. 
			We fix $g\in \cinv$ any invariant function such that $\alpha_{0}(g)=\theta_{0}(\delta_f)$, see \eqref{eqn:value_of_alpha_00}. 
				For example, after fixing a $\nu$-graded basis $(V,\natural)$, one takes $g=\cQ_{V*}(\tilde{g}\circ \pi_{V\times \mathbb{G}^{(0)}})$, where $\tilde{g}\in C^\infty_c(\tdom(\cQ_V),\Omega^1(V))$ satisfies	
				 $\int_{v\in V} \tilde{g}(v,x,0)=f(x)$.
				Clearly,
					\begin{equation}\label{eqn:tempo_eqn_jksdmfkjqsdkljfklmqjsd}\begin{aligned}
						\theta_{0}(\delta_f)=\alpha_{0}(g)=\int_{]0,1]^\nu}\pdv*[mixed-order={\nu}]{\Big((\lambda_1-1)\cdots(\lambda_\nu-1)\alpha_\lambda(g)\Big)}{\lambda_1\cdots,\lambda_\nu}\odif{\lambda}.
					\end{aligned}\end{equation}
				Let $S=\{s_1,\cdots,s_n\}\subseteq \bb{1,\nu}$. We define 
					\begin{equation}\label{eqn:g_S}\begin{aligned}
						g_S=\pdv*[mixed-order={|S|}]{\alpha_{\lambda}(g)}{\lambda_{s_1}\cdots,\lambda_{s_n}}_{\lambda_{s_1}=1,\cdots,\lambda_{s_n}=1}.
					\end{aligned}\end{equation}
				 By the chain rule,
					\begin{equation}\label{eqn:diff_tempo_formula_jkqhjsk}\begin{aligned}
						\pdv*[mixed-order={|S|}]{\alpha_{\lambda}(g)}{\lambda_{s_1}\cdots,\lambda_{s_n}}=\frac{1}{\lambda_{s_1}\cdots\lambda_{s_n}}\alpha_\lambda(g_S),\quad \forall \lambda\in (\Rpt)^\nu.
					\end{aligned}\end{equation}
				So, putting \eqref{eqn:tempo_eqn_jksdmfkjqsdkljfklmqjsd} and \eqref{eqn:diff_tempo_formula_jkqhjsk} together, we obtain   
					\begin{equation}\label{eqn:theta_0_sum_lower_order}\begin{aligned}
						\theta_{0}(\delta_f)=\sum_{S\subseteq \bb{1,\nu}}\int_{]0,1]^\nu} \alpha_{\lambda}(g_S)\prod_{s\in S}\left(1-\lambda_s^{-1}\right)  \odif{\lambda}=\sum_{S\subseteq \bb{1,\nu}}\sum_{S'\subseteq S}(-1)^{|S'|}I_{\kappa_{S'}}(g_S),
					\end{aligned}\end{equation}
				where $\kappa_{S'i}=-1$ if $i\notin S'$, $\kappa_{S'i}=0$ if $i\in S'$.
				By \Cref{thm:derivative_vanish_at_0}, $g_S\in \cinvf{\kappa_{S}+\underline{1}}$. 
				Since $\kappa_{S'}+\underline{1}\preceq \kappa_{S}+\underline{1}$, by \eqref{eqn:theta_product_cinfk}, $g_S\in \cinvf{\kappa_{S'}+\underline{1}}$.
				So, $I_{\kappa_{S'}}(g_S)\in \Psi^{\kappa_{S'}}(\mathbb{G})\subseteq \Psi^{\preceq 0}(\mathbb{G})$.
			\end{proof}
			\begin{rem}\label{rem:decomposing_pseudo_M}
				By \eqref{eqn:identities_I_Use_a_alot_vector_field_out_2} and \eqref{eqn:dfn_algebra_vanishing_Fourier_invariant}, 
				any element $\Psi^{k}_c(M)$ can be written as sum of elements of the form $D\ast v$ where $D\in \DO^l(M)$ and $v\in \Psi^{k-l}_c(M)$ with $k-l\in \C^\nu_{<0}$.
				This remark will be used repeatedly to reduce theorems about positive order operators to negative order operators.
			\end{rem}
			\begin{linkproof}{main_prop_bi_graded_pseudo_diff_M}
				Let us prove \Crefitem{thm:main_prop_bi_graded_pseudo_diff_M}{inclusion_smoothing}.
				Let $f\in C^\infty_c(M\times M,\omegahalf)$. We fix any $g\in C^\infty_c((\Rpt)^\nu)$ such that $\int_{]0,1]^\nu}g(t)\frac{dt}{t^{k+\underline{1}}}=1$.
				So, $f(y,x)g(t)\in C^\infty_c(M\times M\times (\Rpt)^\nu,\omegahalf)$. By \Crefitem{thm:invariant_vanish_algebra}{inclusion}, it follows that $f(y,x)g(t)\in \cinvf{k+\underline{1}}$.
				Clearly, $I_k(fg)_{|\underline{1}}=f$.

				Let us prove \Crefitem{thm:main_prop_bi_graded_pseudo_diff_M}{1}. Let $f\in \cinvf{k+\underline{1}}$, $\pi:\mathbb{G}\to M\times M$ be the natural projection 
					\begin{equation}\label{eqn:natural_proj_pi}\begin{aligned}
						\pi(y,x,t)=(y,x),\quad \pi(A,x,0)=(x,x).		
					\end{aligned}\end{equation}
				The function $\pi$ is smooth because $r,s,\pi_{\mathbb{G}^{(0)}}$ are smooth.
				By \eqref{eqn:support_restriction_dist} and \eqref{eqn:support_automorphism},
					\begin{equation*}\begin{aligned}
						\supp\left(I_k(f)_{|\underline{1}}\right)=\supp(I_k(f))\cap (M\times M\times \{\underline{1}\})\subseteq \pi\left(\supp(I_k(f))\right)&\subseteq \pi\left(\overline{\bigcup_{\lambda\in ]0,1]^\nu}\supp(\alpha_\lambda(f))}\right)\\
						&\subseteq	\overline{\bigcup_{\lambda\in ]0,1]^\nu}\pi\left(\alpha_\lambda(\supp(f))\right)}\\&=\pi(\supp(f)).
					\end{aligned}\end{equation*}
				So, $\supp(I_k(f)_{|\underline{1}})$ is compact.
				To compute the wave-front set, by \Cref{rem:decomposing_pseudo_M}, we can suppose that $k\in \C^\nu_{<0}$.
				We have that
					\begin{equation}\label{eqn:qsjkdjfkjsdlfq}\begin{aligned}
						I_{k}(C^\infty_c(M\times M\times \R_+^\nu\backslash 0,\omegahalf))_{|\underline{1}}\subseteq C^\infty_c(M\times M,\omegahalf).	
					\end{aligned}\end{equation}
				Let $(V,\natural)$ be a $\nu$-graded basis as in \Cref{lem:weak_commutative_graded_basis}, $\tilde{f}$ as in \eqref{eqn:sum_smooth_decompo_invariatn_function}.
				By \eqref{eqn:qsjkdjfkjsdlfq}, 
					\begin{equation*}\begin{aligned}
						I_k(f)_{|\underline{1}}-\phi_*\left(\int_{]0,1]^\nu} \tilde{f}(\alpha_{t^{-1}}(a),x,t)\frac{\odif{t}}{t^{k+\underline{1}}}\right)\in C^\infty_c(M\times M,\omegahalf),
					\end{aligned}\end{equation*}
				where $\phi:V\times M\to M\times M$ is the map $(a,x)\mapsto (e^a\cdot x,x)$.
				For each $S\subseteq \bb{1,\nu}$, let $V_S=\{\sum_{i\in S}v_i:v_i\in V^{\weak{i}}\}\subseteq V$.
				Since $\tilde{f}$ is compactly supported, it is straightforward to see that  
					\begin{equation*}\begin{aligned}
						\WF\left(\int_{]0,1]^\nu} \tilde{f}(\alpha_{t^{-1}}(a),x,t)\frac{\odif{t}}{t^{k+\underline{1}}}\right)\subseteq \bigcup_{S\subseteq \bb{1,\nu}}N^{V\times M}_{V_S\times M},
					\end{aligned}\end{equation*}
				where $N^{V\times M}_{V_S\times M}$ is the conormal bundle of $V_S\times M$ in $V\times M$.
				For each $S\neq \emptyset$, the map $\phi$ seen as a map $V_S\times M\to M\times M$ is a submersion because of \eqref{eqn:finit_N_condition}, and the map $\phi$ seen as a map $V_\emptyset\times M\to \{(x,x):x\in M\}$ is a submersion.
				\Crefitem{thm:main_prop_bi_graded_pseudo_diff_M}{1} now follows from the following: 
				\begin{lem}[{\cite[Eq. (3.6) on p. 332]{GuilleminSternbergBook}}]\label{lem:pushforward_distributions}
        		Let $M,M'$ be smooth manifolds, $\phi:M\to M'$ a submersion, $Z_1,\cdots,Z_n\subseteq M$ and $Z'_1,\cdots,Z_n'\subseteq M'$ submanifolds such that $\phi(Z_i) \subseteq Z'_i$ and $\phi_{|Z_i}:Z_i\to Z'_i$ is a submersion,
				$w\in C^{-\infty}_c(M,\Omega^1(\ker(\odif{\phi})))$ a distribution.
				If $\WF(w)\subseteq \bigcup_{i=1}^nN^{M}_{Z_i}$, then $\WF(\phi_*(u))\subseteq \bigcup_{i=1}^{n}N^{M'}_{Z_i'}$.
				\end{lem}

				Let us prove \Crefitem{thm:main_prop_bi_graded_pseudo_diff_M}{3}. Let $f\in \cinvf{k+\underline{1}}$.
				By \eqref{eqn:identities_I_Use_a_alot_vector_field_out_2}, $I_{l}(\theta_{l-k}(\delta_1)\ast f)=\theta_{l-k}(\delta_1)\ast I_k(f)$.
				Since the restriction of $\theta_{l-k}(\delta_1)\ast I_k(f)$ at $t=\underline{1}$ is equal to $I_{k}(f)_{|\underline{1}}$, \Crefitem{thm:main_prop_bi_graded_pseudo_diff_M}{3} follows.
				\Crefitem{thm:main_prop_bi_graded_pseudo_diff_M}{2} follows from taking the restriction at $t=\underline{1}$ in \Crefitem{thm:main_prop_bi_graded_pseudo_diff}{incl} and using \Crefitem{thm:main_prop_bi_graded_pseudo_diff_M}{3}.
				\Crefitem{thm:main_prop_bi_graded_pseudo_diff_M}{4} follows from \Crefitem{thm:main_prop_bi_graded_pseudo_diff}{3}.
	
				To prove \Crefitem{thm:main_prop_bi_graded_pseudo_diff_M}{7}, notice that any section of $C^\infty(M\times M,\Omega^{\frac{1}{2},-\frac{1}{2}})$ which satisfies \eqref{eqn:omega_symmetry_condition} is of the form $\omega(y,x)=\vol_x^{\frac{1}{2}}\otimes \vol_{x}^{-\frac{1}{2}}$ for some Riemannian metric on $M$.
				So, the result follows from  \Crefitem{thm:main_prop_bi_graded_pseudo_diff}{L1} by restricting at $t=\underline{1}$.

				\Crefitem{thm:main_prop_bi_graded_pseudo_diff_M}{5} is a corollary of the classical Sobolev embedding theorem together with \Crefitem{thm:main_prop_bi_graded_pseudo_diff_M}{4}.
				Its proof is divided into several lemmas:
				\begin{lem}\label{lem:Sobolev_Psi_M}
					Let $n\in \N$, $k\in \C^\nu_{<0}$, $S\subseteq \bb{1,\nu}$. 
					If $\sum_{i\in S}\Re(k_i)N_i^{-1}<-\dim(M)-1-\nu-n$, then 
						\begin{equation*}\begin{aligned}
							\Psi^{k}_c(M)\subseteq C^{n}_c(M\times M,\omegahalf).
						\end{aligned}\end{equation*}
				\end{lem}
				\begin{proof}
					The Sobolev inclusion, $W^{1,\dim(M)+1}(\R^{\dim(M)})\subseteq C^0(\R^{\dim(M)})$ (the $L^1$ Sobolev space of order $\dim(M)+1$ is contained in $C^0$), implies that if $T\in L^1_s(M\times M,\omega)$ satisfies $X_1\ast \cdots \ast X_{\dim(M)+1}\ast T\in L^1_s(M\times M,\omega)$ for every $X_1,\cdots, X_{\dim(M)+1}\in \cX(M)$, then $T\in C^0(M\times M,\omegahalf)$.
				Let $X\in \cX(M)$. Since $X\in \cF^{(0,\cdots,0,N_i,0,\cdots,0)}$.
				By \Crefitem{thm:main_prop_bi_graded_pseudo_diff_M}{2} and \Crefitem{thm:main_prop_bi_graded_pseudo_diff_M}{4},
					\begin{equation*}\begin{aligned}
						X\ast \Psi^{k}_c(M)\subseteq \Psi^{(k_1,\cdots,k_{i-1},k_i+N_i,k_{i+1},\cdots,k_\nu)}_c(M),\quad	\Psi^{k}_c(M)\ast X\subseteq \Psi^{(k_1,\cdots,k_{i-1},k_i+N_i,k_{i+1},\cdots,k_\nu)}_c(M).
					\end{aligned}\end{equation*}
				Now, the lemma follows from the Sobolev inequalities and \Crefitem{thm:main_prop_bi_graded_pseudo_diff_M}{7}.
				\end{proof}
				\begin{lem}\label{lem:decomposition_D}
					For any $l\in \Z_+^\nu$, $S\subseteq \bb{1,\nu}$, let $l_S\in \Z_+^\nu$ be defined by $l_{Si}=l_i$ if $i\in S$ and $0$ otherwise.
					So, $l=l_S+l_{S^c}$. Then, any $D\in \DO^l(M)$ can be written as a sum of differential operators of the form $D_1D_2$ with $D_1\in \DO^{l_S}(M)$ and $D_2\in \DO^{l_{S^c}}(M)$.
				\end{lem}
				\begin{proof}
				We can suppose that $D$ is a monomial. Furthermore, by \eqref{eqn:weakly_commuting}, we can further suppose that $D$ is a monomial of the form $X_1\cdots X_n$ with $X_i\in \cF^{k_i}$ and each $k_i$ is the form 
				$(\overbrace{0,\cdots,0}^a,b,\overbrace{0\cdots,0}^{\nu-a-1})$ for some $a\in \bb{0,\nu-1}$, and $b\in\Z_+$.
				Now, to conclude we need to be able to reorder the vector fields. This follows by induction on $n$ from the fact that if $X\in \cF^{(0,\cdots,0,a,0,\cdots,0)}$ and $Y\in \cF^{(0,\cdots,0,b,0,\cdots,0)}$ with $a$ and $b$ not in the same position, then
				\begin{equation*}
					XY-YX=[X,Y]\in \cF^{(0,\cdots,0,a,0,\cdots,0)+(0,\cdots,0,b,0,\cdots,0)}=	\cF^{(0,\cdots,0,a,0,\cdots,0)}+\cF^{(0,\cdots,0,b,0,\cdots,0)}\qedhere
				\end{equation*}
				\end{proof}
				\begin{lem}\label{lem:inclusion_unusual_pseudo_diff}
					For any $k\in \C^\nu$, any element $v\in \Psi^k_c(M)$ can be written as sum of elements of the form $D\ast v'$ with $D\in \DO^l(M)$ and $v'\in \Psi^{k-l}_c(M)$ such that $k-l\in \C^\nu_{<0}$ and for any $i\in \bb{1,\nu}$, either $l_i=0$ (which happens when $\Re(k_i)<0$) or $-N_i\leq \Re(k_i)-l_i$ (which happens when $\Re(k_i)\geq 0$)
				\end{lem}
				\begin{proof}
					By \Cref{rem:decomposing_pseudo_M}, we can suppose that $v=D\ast v'$ with $D\in \DO^l(M)$ and $v'\in \Psi^{k-l}_c(M)$ and $k-l\in \C^\nu_{<0}$.
					We now optimize this decomposition on each parameter. More precisely, staring from the last parameter, using \Cref{lem:decomposition_D}, we can suppose that $D=D_1 D_2$ with $D_1\in \DO^{(l_1,\cdots,l_{\nu-1},0)}$ and $D_2\in \DO^{(0,\cdots,0,l_\nu)}$.
					By writing $D_2$ as sum of monomials, we can suppose that either $l_\nu=0$ or $-N_\nu\leq \Re(k_\nu)-l_\nu$. We then do the same for the other parameters.
				\end{proof}
				Now, let $k\in \C^\nu$, $n\in \N$ such that $\sum_{i=1}^\nu \Re(k_i)N_i^{\sgn(\Re(k_i))}\leq -\dim(M)-1-\nu-\sum_{i=1}^{\nu}N_i^2 -n $ and  $v\in \Psi^k(M)$. 
				We write $v$ as in \Cref{lem:inclusion_unusual_pseudo_diff}.
				Let $S=\{i:\Re(k_i)< 0\}$. We have
					\begin{equation*}\begin{aligned}
					\sum_{i\in S}\Re(k_i-l_i)N_i^{-1}=	\sum_{i\in S}\Re(k_i)N_i^{-1} &\leq -\dim(M)-1-\nu-n -\sum_{i=1}^{\nu}N_i^2-\sum_{i\notin S}\Re(k_i)N_i\\&\leq -\dim(M)-1-\nu-n - \sum_{i\notin S}l_iN_i
					\end{aligned}\end{equation*}
				So, by \Cref{lem:Sobolev_Psi_M}, $v'\in C^{n+\sum_{i\notin S}l_iN_i}_c(M\times M,\omegahalf)$.
				Since $D\in \DO^l(M)$ is a differential operator of classical order $\leq \sum_{i\notin S}l_iN_i$, \Crefitem{thm:main_prop_bi_graded_pseudo_diff_M}{5} follows.

				Let us prove \Crefitem{thm:main_prop_bi_graded_pseudo_diff_M}{6}. By \Cref{rem:decomposing_pseudo_M} and \Crefitem{thm:properties_bi_grading_diff_op}{3}, we can suppose that $k\in \C^\nu_{<0}$.
				Let $f\in \cinv$, $h\in C^\infty(M\times M)$ defined by $h(x,y)=g(x)-g(y)$, $\tilde{h}=h\circ \pi\in C^\infty(\mathbb{G})$ where $\pi$ is defined in \eqref{eqn:natural_proj_pi}.
				Since $\pi\circ\alpha_\lambda=\pi$, $\alpha_{\lambda}(\tilde{h})=\tilde{h}$ for all $ \lambda\in (\Rpt)^\nu$.
				Furthermore, 
					\begin{equation*}\begin{aligned}
						I_k(\tilde{h}f)_{|\underline{1}}=\delta_g\ast I_k(f)_{|\underline{1}}-I_k(f)_{|\underline{1}}\ast \delta_g.
					\end{aligned}\end{equation*}
				Let $(V,\natural)$ be a $\nu$-graded basis, $\tilde{f}$ as in \Cref{dfn:space_of_invariant_functions}.
				We define $r\in  C^\infty_c( \tdom(\cQ_V),\Omega^1(V))$ as follows:
					\begin{equation*}\begin{aligned}
						r(v,x,t)=h(e^{\alpha_{t}(v)}\cdot x,x)\tilde{f}(v,x,t).		
					\end{aligned}\end{equation*}
				Notice that $\cQ_{V_*}(r\circ \pi_{V\times \mathbb{G}^{(0)}})=\tilde{h} \cQ_{V^*}(\tilde{f}\circ \pi_{V\times \mathbb{G}^{(0)}})$.
				So, 
					\begin{equation*}\begin{aligned}
						\tilde{h}f-\cQ_{V_*}(r\circ \pi_{V\times \mathbb{G}^{(0)}})=\tilde{h}(f-\cQ_{V^*}(\tilde{f}\circ \pi_{V\times \mathbb{G}^{(0)}})) \in C^\infty_c(M\times M\times \R_+^\nu\backslash 0,\omegahalf).
					\end{aligned}\end{equation*}
				By \eqref{eqn:qsjkdjfkjsdlfq}, it suffices to show that $I_{k}(\cQ_{V_*}(r\circ \pi_{V\times \mathbb{G}^{(0)}}))_{|\underline{1}}\in \Psi^{\prec k}(M)$.
				The function $r$ vanishes at $t=0$ because $h$ vanishes on the diagonal.
				So, $r=\sum_{i=1}^\nu t_ir_i$ for some functions $r_i\in C^\infty_c(\tdom(\cQ_V),\Omega^1(V))$.
				Hence, 
					\begin{equation*}\begin{aligned}
						I_{k}(\cQ_{V_*}(r\circ \pi_{V\times \mathbb{G}^{(0)}}))_{|\underline{1}}=\sum_{i=1}^\nu I_{k-(\underbrace{0,\cdots,0}_{i-1},1,\underbrace{0,\cdots,0}_{\nu-i})}(\cQ_{V*}(r_i\circ \pi_{V\times \mathbb{G}^{(0)}}))_{|\underline{1}}.
					\end{aligned}\end{equation*}
				This finishes the proof of \Cref{thm:main_prop_bi_graded_pseudo_diff_M}.
			\end{linkproof}
			\begin{linkproof}{asymptotic_sum}
				If $v$ and $v'$ both satisfy the condition of theorem, then $v-v'\in \Psi^{k-(n,0,\cdots,0)}(M)+\cdots+\Psi^{k-(0,\cdots,0,n)}(M)$ for all $n\in \N$.
				It follows from \Crefitem{thm:main_prop_bi_graded_pseudo_diff_M}{5} that $v-v'\in C^\infty(M\times M,\omegahalf)$. This proves uniqueness.

				We now prove existence.
				The set $S=\{l\in \Z_+^\nu:l_i\leq \max(\Re(k_i)+N_i,0),\ \forall i \in \bb{1,\nu}\}$ is finite.
				By \Crefitem{thm:properties_bi_grading_diff_op}{5}, we fix a family of generators of the modules $\DO^l(M)$ for each $l\in S$.
				Let $D_1\in \DO^{l_1}(M),\cdots,D_n\in \DO^{l_n}(M)$ be the generators.
				By \Cref{lem:inclusion_unusual_pseudo_diff}, each $v_\tau$ can be written as a sum $v_\tau=\sum_{i=1}^n D_i \ast v_{\tau,i}$ with $v_{\tau,i}\in \Psi^{k-l_i-\tau}(M)$ and $k-l_i-\tau\in \C^\nu_{<0}$.
				By convention, $v_{\tau,i}=0$ if $v_\tau=0$.
				It is enough to prove the existence of the asymptotic sum $\sum_{\tau}v_{\tau,i}$ for each $i$.
				Here, we make use of the fact that for each $i\in \bb{1,n}$, $\{\tau\in \Z_+^\nu:v_{\tau,i}\neq 0\}\subseteq \{\tau\in \Z_+^\nu:v_\tau\neq 0\}$.
				
				Therefore, without loss of generality we can suppose that $k-\tau\in \C_{-}^\nu$ for every $v_{\tau}\neq 0$.
				In the following argument, we only run over $\tau$'s such that $v_\tau\neq 0$.
				By using a partition of unity argument, we can suppose that for some compact $K\subseteq M\times M$, $\supp(v_{\tau})\subseteq K$ for all $\tau\in \Z_+^\nu$.
				There exists functions $f_{\tau}\in \cinv$ such that $v_\tau(y,x)=I_{k-\tau}(f_{\tau})_{|\underline{1}}$.
				We fix a $\nu$-graded basis $(V,\natural)$.
				Let $\tilde{f}_{\tau}$ be as in \eqref{eqn:sum_smooth_decompo_invariatn_function}.
				By \eqref{eqn:qsjkdjfkjsdlfq}, we can suppose that $v_\tau(y,x)=I_{k-\tau}(\cQ_{V*}(\tilde{f}_{\tau}\circ \pi_{V\times \mathbb{G}^{(0)}}))_{|\underline{1}}$.
				Furthermore, by \Cref{lem:change_open_set_invariant_function}, we can suppose that the supports of the functions $\tilde{f}_{\tau}(v,x,t)$ are uniformly compact.
				Let $\chi\in C^\infty_c(\R^\nu)$ equal to $1$ in a neighborhood of $0$.
				Consider 
					\begin{equation*}\begin{aligned}
						\tilde{f}(v,x,t)=\sum_{\tau\in \Z_+^\nu} 	\chi(\epsilon_\tau t )t^{\tau}\tilde{f}_{\tau}(v,x,t),	
					\end{aligned}\end{equation*}
				where $\epsilon_\tau>0$ are constants small enough so that the sum converges in $C^\infty_c(\tdom(\cQ_V),\Omega^1(V))$.
				This can be done using an elementary diagonalization argument.
				Now, let 
					\begin{equation*}\begin{aligned}
						v=I_{k}\Big(\cQ_{V*}(\tilde{f}\circ \pi_{V\times \mathbb{G}^{(0)}})\Big)_{|\underline{1}}\in \Psi^k_c(M).	
					\end{aligned}\end{equation*}
				It is straightforward to check that $v$ satisfies the required properties, see the proof of \Crefitem{thm:main_prop_bi_graded_pseudo_diff_M}{3}.
			\end{linkproof}
			\begin{linkproof}{classical_pseudo}
					By \eqref{eqn:compatability_pseudo_restriction_open}, we can suppose that $M=\R^n$.
					Clearly in the case $\nu=1$ and $\cF^{1}=\cX(M)$, the order on differential operators defined in \Cref{sec:bi_graded_pseudo_diff} coincides with the classical order.
					So, by \Cref{rem:decomposing_pseudo_M}, we can further suppose that $k<0$.
					We also choose the obvious $1$-graded basis $V=\R^n$, and $\natural:\R^n\to \cX(\R^n)$ the obvious map whose image is the set of constant vector fields.
					Notice that in this case $\tdom(\cQ)=V\times M\times \R_+$.
					So, if $f\in C^\infty_c(V\times M\times \R_+)$ (we are trivializing all densities in this proof), and $\tilde{f}=\cQ_{V*}(f\circ \pi_{V\times M})\in \cinv$, then by a direct computation or using \eqref{eqn:description_of_pseduo_diff}, we have.
						\begin{equation*}\begin{aligned}
							I_k(\tilde{f})_{|\underline{1}}(y,x)=\int_0^1 f(\frac{y-x}{t},x,t) t^{-n}\frac{\odif{t}}{t^{k+1}}		
						\end{aligned}\end{equation*}
					By Fourier inversion formula, we have 
						\begin{equation*}\begin{aligned}
							f(v,x,t)=\int_{\R^n} g(\xi,x,t)e^{i \langle \xi,v\rangle}\odif{\xi},	
						\end{aligned}\end{equation*}
					where $g$ is the Fourier transform of $f$ in the first variable. Hence, 
					\begin{equation*}\begin{aligned}
							I_k(\tilde{f})_{|\underline{1}}(y,x)=\int_{\R^n}e^{i \langle \xi,y-x\rangle}\int_0^1 g(t\xi ,x,t) \frac{\odif{t}}{t^{k+1}}	\odif{\xi}	
						\end{aligned}\end{equation*}
						The integral $\int_0^1 g(t\xi ,x,t) \frac{\odif{t}}{t^{k+1}}$ defines a classical symbol of order $k$. Furthermore, any classical symbol of order $k$ can be written in this form.
				\end{linkproof}

\section{Proofs of the results in Section \ref{sec:Tangent_groupoid_representation}}\label{sec:Tangent_groupoid_representation_proof}
	\begin{linkproof}{amenability_tanget}
			The set $A=M\times \R_+^\nu\backslash 0$ is a saturated subset of $\mathbb{G}^{(0)}$.
			Furthermore, $\mathbb{G}_{A^c}=\cG$.
			By \Cref{prop:stable_weak_amenable_saturated}, it suffices to show that $\mathbb{G}_{A}$ and $\mathbb{G}_{A^c}$ have the SWCP.
			It is clear that the groupoid $\mathbb{G}_{A}=M\times M\times \R_+^\nu\backslash 0\rightrightarrows M\times \R_+^\nu\backslash 0$ has the SWCP.
			The groupoid $\mathbb{G}_{A^c}$ is equal to the groupoid $\cG\rightrightarrows \cG^{(0)}$.
			Let $(V,\natural)$ be a $\nu$-graded Lie basis.
			The group $V$ acts on the space $\cG^{(0)}$ by the map $v\cdot (L,x)=(\natural_{x,0}(v)L\natural_{x,0}(-v),x)$.
			So, we can form the crossed product groupoid $V\ltimes \cG^{(0)}\rightrightarrows \cG^{(0)}$, see \eqref{eqn:crossed_product_groupoid}.
			We have an obvious surjective submersion map 
				\begin{equation*}\begin{aligned}
					V\rtimes \cG^{(0)}		\to \cG,\quad ((L,x),v,(L',x))\mapsto (\natural_{x,0}(v)L',x)
				\end{aligned}\end{equation*}
			The result now follows from \Cref{prop:Crossed_product_amenable}.
		\end{linkproof}
	\begin{linkproof}{Helffer_Nourrigat_cone}
		The identities \eqref{eqn:dilations_HN_sets} are rather straightforward. They come from replacing $(\xi_n,x)$ by $(\lambda \xi_n,x_n)$ for $\lambda\in \R$, and $t_n$ by $\lambda t_n$ for $\lambda\in \R_+^\nu$.
		This proves \eqref{eqn:dilations_HN_sets}.

			We now prove \eqref{eqn:Hellfer_Nourrigat_set_condition_union_cones}. 
		It is rather easy to see that one can replace ${(t_n)}_{n\in \N}\subseteq \R_+^\nu$ by ${(t_n)}_{n\in \N}\subseteq \R_+^\nu\backslash 0$	in the definition of the Helffer-Nourrigat cone.	
		Let $(V,\natural)$ be a $\nu$-graded basis.
		From the definition of a $\nu$-graded basis, it follows that $\eta\in \HL{x}$ if and only if there exists sequences $(\xi_n,x_n)_{n\in \N}\subseteq  T^*M$, ${(t_n)}_{n\in \N}\subseteq \R_+^\nu\backslash 0$
			such that $x_n\to x$, $t_n\to 0$, and $\xi_n\circ \natural_{x_n,t_n}\in V^*$ converges to $ \eta\circ \natural_{x,0}\in V^*$.

			Let $\eta\in \HL{x}$. 
			By passing to a subsequence, we can suppose that $\ker(\natural_{x_n,t_n})$ converges. By \Cref{thm:conv_kernel_mathbbG0}, the limit is of the form $\natural_{x,0}^{-1}(L)$ for some $L\in \cG^0_x$.
			Clearly, $\xi_n\circ \natural_{x_n,t_n}\in \ker(\natural_{x_n,t_n})^\perp$. So, by \Cref{prop:grass_prop}, $ \eta\circ \natural_{x,0}\in \natural_{x,0}^{-1}(L)^\perp$ which implies that $\eta \in L^\perp$.

			Conversely, if $\eta\in L^{\perp}$, then there exists sequences $(x_n)_{n\in \N}\subseteq M$, $(t_n)_{n\in \N}\subseteq \R_+^\nu\setminus \{0\}$ such that
		$x_n\to x$, $t_n\to 0$ and $\ker(\natural_{x_n,t_n})\to \natural_{x,0}^{-1}(L)$.
		Therefore, $\ker(\natural_{x_n,t_n})^\perp\to \natural_{x,0}^{-1}(L)^\perp$.
		Since $\eta\in L^\perp$, $\eta\circ \natural_{x,0}\in \natural_{x,0}^{-1}(L)^\perp$.
		By \Cref{prop:grass_prop}, there exists $\eta_n\in \ker(\natural_{x_n,t_n})^\perp$ such that $\eta_n\to \eta\circ \natural_{x,0}$. 
		Since $\eta_n\in \ker(\natural_{x_n,t_n})^\perp$,
		there exists unique $\xi_n\in T_{x_n}^*M$
		such that $\eta_n=\xi_n\circ \natural_{x_n,t_n}$.
		So, $\eta \in \HL{x}$.
	
		The set $\HL{x}$ is closed under the co-adjoint action of $\mathfrak{G}_x$ because of \Crefitem{thm:Lie_algebra}{adjoint_action}. 

		The set $\HL{x}$ is closed because if $\eta_n\in \HL{x}$ converges to $\eta\in \mathfrak{g}_x^*$, then for each $n\in \N$, there exists $L_n\in \cG^0_x$ such that $\eta_n\in L_n^\perp$.
		By passing to a subsequence, we can suppose that $L_n$ converges to some $L\in \cG^0_x$. By \Cref{prop:grass_prop}, $\eta\in L^{\perp}$.
		
		It remains to show that $\HL{(\xi,x)}$ is closed, and closed under the co-adjoint action of $\mathfrak{G}_x$.
		Consider the $(\nu+1)$-sub Riemannian structure of depth $(N,1)$ defined by
			\begin{equation}\label{eqn:extension_structure}\begin{aligned}
				\tilde{\cF}^{(k,0)}=\cF^{k},\quad \tilde{\cF}^{(k,1)}=\cX(M).	
			\end{aligned}\end{equation}
		This structure is weakly-commutative.
		Its osculating groups at $x$ is equal to $\mathfrak{g}_x\oplus T_xM$, see \eqref{eqn:weakly_commuting_lie_algebra}.
		It has its own Helffer-Nourrigat cone $\widetilde{\mathrm{H\! N}}_x\subseteq \mathfrak{g}_x^*\oplus T^*_xM$.
		It is immediate to see that
			\begin{equation}\label{eqn:Helffer-Nourrigat_extension_eqn}\begin{aligned}
						\widetilde{\mathrm{H\! N}}_x= (\HL{x}\times\{0\})\sqcup\bigsqcup_{\xi\in T^*_xM\backslash 0}\HL{(\xi,x)}\times \{\xi\},\ 	\widetilde{\mathrm{H\! N}}_{(\xi,x)}=(\HL{(\xi,x)}\times\{0\})\sqcup \bigsqcup_{t\in \Rpt} \HL{(\xi,x)}\times \{t\xi\}
			\end{aligned}\end{equation}
		We have proved above that $\widetilde{\mathrm{H\! N}}_x$ is a closed subset of $\mathfrak{g}_x^*\oplus T_x^*M$ which closed under the co-adjoint action of the group $\mathfrak{G}_x\times T_xM$.
		So, $\HL{(\xi,x)}$ is a closed subset of $\mathfrak{g}_x^*$ which is closed under the co-adjoint action of the group $\mathfrak{G}_x$.
	\end{linkproof}
	\begin{linkproof}{weakly_contained_Helffer_Nourrigat}
		Let $L\in \cG^0_x$. By \Cref{prop:support_Kirillovs_Orbit_Method_Induced}, the set of irreducible unitary representations which are weakly contained in $\Xi_{\GrF{x}/L,0}$ correspond by the orbit method \eqref{eqn:Orbit_method_map} to 
			$\overline{\bigcup_{g\in \GrF{x}}(Ad(g)L)^\perp}/\mathfrak{G}_x$.
		By \Crefitem{thm:Lie_algebra}{adjoint_action} and \eqref{eqn:Hellfer_Nourrigat_set_condition_union_cones}, and that $\HN_x$ is a closed subset of $\mathfrak{g}_x^*$, we deduce that $\HN_x= \bigcup_{L\in \cG^0_x}\overline{\bigcup_{g\in \GrF{x}}(Ad(g)L)^\perp}$. The result follows.
	\end{linkproof}

		\begin{linkproof}{representation_invariant}
				By \Cref{thm:weakly_contained_Helffer_Nourrigat}, $\pi$ is weakly contained in $\Xi_{\GrF{x}/L,0}$ for some tangent cone $L$ at $x$.
				Let $\Xi_{(L,x,0)}$ be the regular representation of the groupoid $\mathbb{G}\rightrightarrows \mathbb{G}^{(0)}$, see \Cref{ex:Regular_representation}.
				Let $f\in\cinv$, $(V,\natural)$ be a $\nu$-graded basis, $\tilde{f}$ as in \eqref{eqn:sum_smooth_decompo_invariatn_function}.
				The space $\mathbb{G}_{(L,x,0)}=\{(gL,x,0):g\in \GrF{x}\}$ identifies with $\GrF{x}/L$ and $\Xi_{(L,x,0)}(f)=\Xi_{\GrF{x}/L,0}(\natural_{x,0*}(\tilde{f}(\cdot,x,0)))$.
				So, 
					\begin{equation*}\begin{aligned}
						\norm{\pi\left(\natural_{x,0*}(\tilde{f}(\cdot,x,0))\right)}\leq 		\norm{\Xi_{\GrF{x}/L,0}(\natural_{x,0*}(\tilde{f}(\cdot,x,0)))}=\norm{\Xi_{(L,x,0)}(f)}\leq \norm{f}_{C^*\mathbb{G}},
					\end{aligned}\end{equation*}
				where in the first identity we used the fact that $\pi$ is weakly contained in $\Xi_{\GrF{x}/L,0}$.
				It follows that $\pi_{\mathrm{inv}}$ is well-defined and extends to $C^*_{\mathrm{inv}}\mathbb{G}$.
				The equality 
					\begin{equation}\label{eqn:pismooth_pi_inv_smooth}\begin{aligned}
						\pi(C^\infty_c(\GrF{x},\Omega^1 \mathfrak{g}_x))=\pi_{\mathrm{inv}}(\cinv)
					\end{aligned}\end{equation}
				 follows from the fact that given any $h\in C^\infty_c(\GrF{x},\Omega^1 \mathfrak{g}_x)$, 
				there exists $h'\in C^\infty_c(\tdom(\cQ_V),\Omega^1(V))$ such that $\natural_{x,0*}(h',x,0)=h$ which implies that $\pi_{\mathrm{inv}}(\cQ_{V*}(h'\circ \pi_{V\times \mathbb{G}^{(0)}}))=\pi(h)$.
				By \Cref{thm:Dixmier-Malliavin_representation} applied to $\pi:C^*\mathfrak{G}_x\to \cL(L^2(\pi))$, $C^\infty(\pi)=\pi(C^\infty_c(\mathfrak{G}_x,\Omega^1(\mathfrak{g}_x)))L^2(\pi)$. So \eqref{eqn:Cinfinty_cinv} follows.
				Finally, \eqref{eqn:pismooth_pi_inv_smooth} and irreducibility of $\pi$ imply that $\pi_{\mathrm{inv}}$ is irreducible.
		\end{linkproof}		
		
		\begin{linkproof}{Type1_tangent_groupoid}
				The $C^*$-algebra $C^*\mathbb{G}$ is separable because $\mathbb{G}$ is second countable.
				So, $C^*_{\mathrm{inv}}\mathbb{G}$ is also separable.
				By \eqref{eq:exact_saturated_closed}, we have a short exact sequence 
					\begin{equation}\label{eqn:exact_seqeence_mathbbG_C*}\begin{aligned}
						0\to \mathcal{K}(L^2(M))\otimes C_0(\R_+^\nu\backslash 0) \to C^*\mathbb{G}\to  C^*\cG\to 0		
					\end{aligned}\end{equation}
				The $C^*$-subalgebra $C^*_{\mathrm{inv}}\mathbb{G}$ contains $\mathcal{K}(L^2(M))\otimes C_0(\R_+^\nu\backslash 0)$ because $C^\infty_c(M\times M\times \R_+^\nu,\omegahalf)\subseteq \cinv$.
				Let 
					\begin{equation*}\begin{aligned}
						C^*_{\mathrm{inv}}\cG:=\frac{C^*_{\mathrm{inv}}\mathbb{G}}{\mathcal{K}(L^2(M))\otimes C_0(\R_+^\nu\backslash 0)}
					\end{aligned}\end{equation*}
				which is a $C^*$-subalgebra of $C^*\cG$ by \eqref{eqn:exact_seqeence_mathbbG_C*}.
				We have the short exact sequence
					\begin{equation}\label{eqn:exact_sequence_3_TypeI}\begin{aligned}
						0\to \mathcal{K}(L^2(M))\otimes C_0(\R_+^\nu\backslash 0) \to C^*_{\mathrm{inv}}\mathbb{G}\to  C^*_{\mathrm{inv}}\cG\to 0.		
					\end{aligned}\end{equation}
				Let $(V,\natural)$ be a $\nu$-graded Lie basis.
				By \eqref{eqn:ksqjfmksqd} (the $L^1$-estimate), if $g\in C^\infty_c(\tdom(\cQ_V),\Omega^1(V))$ vanishes on $V\times M\times\{0\}$, then $\cQ_{V*}(g\circ \pi_{V\times \mathbb{G}^{(0)}})\in \mathcal{K}(L^2(M))\otimes C_0(\R_+^\nu\backslash 0)$.
				From this observation, we can define the map 
					\begin{equation*}\begin{aligned}
						\phi: C^\infty_c(V\times M,\Omega^1(V))\to 	C^*_{\mathrm{inv}}\cG,\quad f\mapsto \cQ_{V*}(f'\circ \pi_{V\times \mathbb{G}^{(0)}}) \bmod \mathcal{K}(L^2(M))\otimes C_0(\R_+^\nu\backslash 0)
					\end{aligned}\end{equation*}
				where $f'\in C^\infty_c(\tdom(\cQ_V),\Omega^1(V))$ is any smooth function such that $f'_{|V\times M\times \{0\}}=f$.
				The map $\phi$ is a $*$-algebra homomorphism by \Cref{rem:product_and _adjoint_invariant_functions}.
				Here, the $*$-algebra structure on $V\times M$ comes from the groupoid structure $V\times M\rightrightarrows M$ which is the product of the groupoid $V\rightrightarrows \{\mathrm{pt}\}$ and the trivial groupoid $M\rightrightarrows M$.
				By the definition of the maximum $C^*$-algebra of $V\times M\rightrightarrows M$ , it follows that $\phi$ extends to a $*$-homomorphism 
					\begin{equation*}\begin{aligned}
						\phi:C^*(V)	\otimes C_0(M)\to C^*_{\mathrm{inv}}\cG.	
					\end{aligned}\end{equation*}
				Since any $*$-homomorphism between $C^*$-algebras has closed image, and
				the image of $\phi$ contains $\cinv$, we deduce that $\phi$ is surjective.
				The $C^*$-algebra $C^*(V)	\otimes C_0(M)$ is of Type \RNum{1}. In fact, it is liminal, see \cite{Dixmier}.
				By \Cref{prop:exat_sequence_3_TypeI}, $C^*_{\mathrm{inv}}\cG$ is of Type \RNum{1}. So, by \Cref{prop:exat_sequence_3_TypeI} once more, $C^*_{\mathrm{inv}}\mathbb{G}$ is of Type \RNum{1}.
				
				We now show that \eqref{eqn:representations_of_tangent_groupoid} is a bijection.
				By \eqref{eqn:exact_sequence_3_TypeI}, 
					\begin{equation}\label{eqn:sqjdkfsqdhflksqdhjlfkhqsldfhjsqkdhfl}\begin{aligned}
							\widehat{C^*_{\mathrm{inv}}\mathbb{G}}= (\R_+^\nu\backslash 0)\sqcup \widehat{C^*_{\mathrm{inv}}\cG}
					\end{aligned}\end{equation}
				and $\R_+^\nu\backslash 0$ with its usual topology is an open subset of $\widehat{C^*_{\mathrm{inv}}\mathbb{G}}$.
				So, we need to show that $\widehat{C^*_{\mathrm{inv}}\cG}=\widehat{\mathrm{H\!N}}$.

				The groupoids  are fibered over $M$.
				Furthermore, the map $\phi$ is a $C_0(M)$-$*$-homomorphism.
				It follows that for any $x\in M$, the map $\phi$ descend to a $*$-homomorphism between the fibers of $C^*(V\times M)$ and $C^*_{\mathrm{inv}}\cG$ at $x$.
				Since the groupoids $V\times M\rightrightarrows M$ and $\cG\rightrightarrows \cG^{(0)}$ satisfy the SWCP, we can apply \Cref{prop:fiber_C*_algebra}.
				The fiber of $C^*(V\times M)$ is $C^*V$, and the fiber of $C^*\cG$ at $x$ is $C^*\cG_x$.
				Let $C^*_{\mathrm{inv}}\cG_x$ be the fiber of $C^*_{\mathrm{inv}}\cG$ at $x$.
				Since $C^*_{\mathrm{inv}}\cG$ is a $C^*$-subalgebra of $C^*\cG$, $C^*_{\mathrm{inv}}\cG_x$ is also a $C^*$-subalgebra of $C^*\cG_x$.
				This fact is folklore which follows easily from  \cite[Proposition C.10.a and Proposition C.10.c]{WilliamsCrossedBook}.
				So, we get a map $\phi_x:C^*(V)\to C^*_{\mathrm{inv}}\cG_x$.

				Since $\phi$ is surjective, irreducible representations of $C^*_{\mathrm{inv}}\cG$ are the irreducible representations of $C^*(V)\otimes C_0(M)$ which factor through $\ker(\phi)$.
				An irreducible representation of $C^*(V)\otimes C_0(M)$ is a pair $(\pi,x)$ where $\pi:C^\infty_c(V,\Omega^1(V))\to \cL(L^2(\pi))$ is an irreducible unitary representation of $V$ as a Lie group, and $x\in M$.
				The pair $(\pi,x)$ factors through $\ker(\phi)$ if and only if
					\begin{equation*}\begin{aligned}
						\norm{\pi(a)}\leq \norm{\phi_x(a)}_{C^*_{\mathrm{inv}}\cG_x},\quad \forall a\in 		C^*(V).
					\end{aligned}\end{equation*}
				Since $C^*_{\mathrm{inv}}\cG_x$ is a $C^*$-subalgebra of $C^*\cG_x$ and the groupoid $\cG_x$ satisfies WCP, it follows that 
					\begin{equation*}\begin{aligned}
						\norm{b}_{C^*_{\mathrm{inv}}\cG_x}=\norm{b}_{C^*\cG_x}=\norm{b}_{C^*_r\cG_x}=\sup\{\norm{\Xi_{\mathfrak{G}_x/L}(b)}:L\in \cG^0_x\},\quad \forall b\in C^*_{\mathrm{inv}}\cG_x.		
					\end{aligned}\end{equation*}
				So the pair $(\pi,x)$ factors through $\ker(\phi)$ if and only if  
					\begin{equation*}\begin{aligned}
						\norm{\pi(a)}\leq 		\sup\{\norm{\Xi_{\mathfrak{G}_x/L}(\phi_x(a))}:L\in \cG^0_x\},\quad \forall a\in C^*(V)
					\end{aligned}\end{equation*}
				By density of $C^\infty_c(V,\Omega^1(V))$, it is enough to verify the above for $a\in C^\infty_c(V,\Omega^1(V))$.
				If $f\in C^\infty_c(V,\Omega^1(V))$, then $\Xi_{\mathfrak{G}_x/L}(\phi_x(f))=\Xi_{\mathfrak{G}_x/L,0}(\natural_{x,0*}(f))$.
				So, the pair $(\pi,x)$ factors through $\phi$ if and only if $\pi$ is weakly contained in the representations $\{\Xi_{\mathfrak{G}_x/L}\circ \natural_{x,0*}:L\in \cG^0_x\}$.
				By \Cref{thm:weakly_contained_Helffer_Nourrigat}, this set is the set of representations of $V$ of the form $\pi\circ \natural_{x,0*}$ where $\pi\in \widehat{\mathrm{H\!N}}_x$.
				This finishes the proof that $\widehat{C^*_{\mathrm{inv}}\cG}=\widehat{\mathrm{H\!N}}$.
				A careful look at the proof shows that we also showed that the map \eqref{eqn:representations_of_tangent_groupoid} is a homeomorphism between  $\widehat{C^*_{\mathrm{inv}}\cG}$ and $\widehat{\mathrm{H\!N}}$.

				It remains to show that \eqref{eqn:representations_of_tangent_groupoid} is a homeomorphism.
				We just need to show that if $A\subseteq \R_+^\nu\backslash 0$ is a closed subset (with its usual Euclidean topology), then the closure of $A$ in $\widehat{C^*_{\mathrm{inv}}\mathbb{G}}$ coincides with the closure of $A$ in the topology in \Cref{dfn:toplogy_mathbbG_hat}.
				To this end, consider the set $M\times A$ as a subset of $\mathbb{G}^{(0)}$. It is a saturated subset. So, its closure $\overline{M\times A}$ in $\mathbb{G}^{(0)}$ is also saturated.
				Therefore, $C^*\mathbb{G}_{\overline{M\times A}}$ is a quotient of $C^*\mathbb{G}$.
				Let $C^*_{\mathrm{inv}}\mathbb{G}_{\overline{M\times A}}$ be the image of $C^*_\mathrm{inv}\mathbb{G}$ in $C^*\mathbb{G}_{\overline{M\times A}}$.
				Since $C^*_{\mathrm{inv}}\mathbb{G}_{\overline{M\times A}}$ is a quotient of $C^*_\mathrm{inv}\mathbb{G}$, its spectrum is a closed subset of $\widehat{C^*_{\mathrm{inv}}\mathbb{G}}$.
				We will show that its spectrum is in fact the closure of $A$ in $\widehat{C^*_{\mathrm{inv}}\mathbb{G}}$.
				To this end, it suffices to show that $A$ is a dense subset of the spectrum of  $C^*_\mathrm{inv}\mathbb{G}_{\overline{M\times A}}$.
				Notice that by \Cref{prop:stable_weak_amenable_saturated}, $\mathbb{G}_{\overline{M\times A}}$ satisfies the SWCP.
				Furthermore, by its construction $M\times A$ is dense in its space of objects. 
				So by \Cref{prop:fiber_dense_reduced}, 
					\begin{equation}\label{eqn:qksjdkfjkldsjkfjksdqljjqsdfjsdqf}\begin{aligned}
						\norm{a}_{C^*\mathbb{G}_{\overline{M\times A}}}=\norm{a}_{C^*_r\mathbb{G}_{\overline{M\times A}}}=\sup\{\pi_{t}(a):t\in A\},\quad \forall a\in C^*\mathbb{G}_{\overline{M\times A}}	.
					\end{aligned}\end{equation}
				Since, $C^*_{\mathrm{inv}}\mathbb{G}_{\overline{M\times A}}$ is a sub $C^*$-algebra of $ C^*\mathbb{G}_{\overline{M\times A}}	$, \eqref{eqn:qksjdkfjkldsjkfjksdqljjqsdfjsdqf} remains true for $a\in C^*_{\mathrm{inv}}\mathbb{G}_{\overline{M\times A}}$.
				This implies that in the spectrum of $C^*_{\mathrm{inv}}\mathbb{G}_{\overline{M\times A}}$, $A$ is dense.

				We now compute the spectrum of $C^*_{\mathrm{inv}}\mathbb{G}_{\overline{M\times A}}$.
				In fact, the groupoid $\mathbb{G}_{\overline{M\times A}}$ is very similar to the groupoid $\mathbb{G}$. It is equal to 
					\begin{equation*}\begin{aligned}
								\mathbb{G}_{\overline{M\times A}}=M\times M\times A\sqcup \{(gL,x):(L,x,0)\in \overline{M\times A},g\in \mathfrak{g}_x\}\times \{0\},
					\end{aligned}\end{equation*}
				The proof we gave above to compute the spectrum of $C^*_{\mathrm{inv}}\mathbb{G}$ works equally well for $C^*_{\mathrm{inv}}\mathbb{G}_{\overline{M\times A}}$, to deduce that the spectrum of $C^*_{\mathrm{inv}}\mathbb{G}_{\overline{M\times A}}$ is precisely the closure of $A$ in the topology in \Cref{dfn:toplogy_mathbbG_hat}.
		\end{linkproof}		

		In the sequel of this chapter, is useful to remark that the set of non-singular representations have the following quite natural description.
		The closure $\cinvf{\underline{1}}$ in $C^*_{\mathrm{inv}}\mathbb{G}$ will be denoted by $C^*_{\mathrm{inv},\underline{1}}(\mathbb{G})$.
		By \Crefitem{thm:invariant_vanish_algebra}{algebra}, $C^*_{\mathrm{inv},\underline{1}}(\mathbb{G})$ is a two-sided ideal of $C^*_{\mathrm{inv}}\mathbb{G}$.
		So, the spectrum of $C^*_{\mathrm{inv},\underline{1}}\mathbb{G}$ is equal to the open subset of the spectrum of $C^*_{\mathrm{inv}}\mathbb{G}$ of irreducible unitary representations of $C^*_{\mathrm{inv}}\mathbb{G}$ which don't vanish on 
		$C^*_{\mathrm{inv},\underline{1}}(\mathbb{G})$.
		\begin{lem}\label{lem:support_cinv_1}
			The $C^*$-algebra $C^*_{\mathrm{inv},\underline{1}}(\mathbb{G})$ is separable of Type \RNum{1} and its spectrum is $(\Rpt)^\nu \sqcup \HatHLnon\times \{0\}$.
		\end{lem}
		\begin{proof}
		The $C^*$-algebra $C^*_{\mathrm{inv},\underline{1}}(\mathbb{G})$ is separable because it is a two-sided ideal of a separable $C^*$-algebra. It is of Type \RNum{1} by \Cref{prop:exat_sequence_3_TypeI}.
		We now compute the spectrum.
		If $f\in \cinvf{\underline{1}}$, then $f$ can be written as sum of elements of the form $\theta_{k}(D)\ast g$, where $\underline{1}\preceq k$ and $g\in \cinv$.
		On $M\times M\times \R^\nu_+\backslash 0$,
			\begin{equation*}\begin{aligned}
						\theta_{k}(D)\ast g(y,x,t)=t^kD_yg(y,x,t)
			\end{aligned}\end{equation*}
		Since $\underline{1}\preceq k$, $1\leq k_i$ for all $i\in \bb{1,\nu}$. So, if $t\in \R^\nu_+\backslash 0$ but not in $(\Rpt)^\nu$, then $\theta_{k}(D)\ast g(y,x,t)$ vanishes.
		Hence, $\pi_t(f)$ vanishes.
		On the other hand, by \Crefitem{thm:invariant_vanish_algebra}{inclusion}, if $t\in (\Rpt)^\nu$, then $\pi_t$ doesn't vanish on $C^*_{\mathrm{inv},\underline{1}}\mathbb{G}$.
		We have thus proved that $\widehat{C^*_{\mathrm{inv},\underline{1}}(\mathbb{G})}\cap \R^\nu_+\backslash 0=(\Rpt)^\nu$.

		By the description of the topology on $\widehat{C^*_{\mathrm{inv}}(\mathbb{G})}$ in \Cref{thm:Type1_tangent_groupoid} and the fact that $\widehat{C^*_{\mathrm{inv},\underline{1}}(\mathbb{G})}$ is an open subset of $\widehat{C^*_{\mathrm{inv}}\mathbb{G}}$, it follows that $\widehat{C^*_{\mathrm{inv},\underline{1}}(\mathbb{G})}\subseteq (\Rpt)^\nu \sqcup \HatHLnon\times \{0\}$.
		It remains to show that if $(\pi,x)\in \HatHLnon$, then $\pi_{\mathrm{inv}}$ doesn't vanish on $\cinvf{\underline{1}}$.
		Notice that if $f\in C^\infty_c(\mathfrak{G}_x,\Omega^1)$ and $D$ is a right-invariant differential operator on $\mathfrak{G}_x$ which satisfies $\alpha_\lambda(D)=\lambda^n D$ for some $n\in \N^\nu$, then there exists $D'\in \DO^n(M)$ and $f'\in \cinv$ such that $\pi(D(f))=\pi_{\mathrm{inv}}(\theta_n(D)\ast f')$.
		This clearly implies that $\pi_{\mathrm{inv}}$ doesn't vanish on $\cinvf{\underline{1}}$.
		\end{proof}

\section{Proofs of the results in Section \ref{sec:principal_symbol}}\label{sec:principal_symbol_proof}
\hypertarget{InternalLink:principal_symbol_pseudo_diff_well_defined}{}
	Before, we proceed with the proof of \Cref{thm:principal_symbol_pseudo_diff_well_defined}, we need to prove a few lemmas.
	
	\begin{lem}\label{lem:principal_symbol}
		Let $k\in \C^\nu$, $u\in \Psi^k(\mathbb{G})$, $g\in \cinvf{-k+\underline{1}}$. The limit 
			\begin{equation}\label{eqn:dfn_sigma_k}\begin{aligned}
				\sigma^k(u)\ast g:=\lim_{\lambda_1\to +\infty,\cdots,\lambda_\nu\to +\infty}\lambda^{-k}\alpha_{\lambda}(u)\ast g		
			\end{aligned}\end{equation}
		exists in $C^*_{\mathrm{inv}}\mathbb{G}$, and thus defines a linear map $\sigma^k(u)\ast \cdot:\cinvf{-k+\underline{1}}\to C^*_{\mathrm{inv}}\mathbb{G}$.
	\end{lem}
	\begin{proof}
		Let $f\in\cinvf{k+\underline{1}}$ such that $u=I_k(f)$. 
		By writing $f$ as in \eqref{eqn:invariant_vanish_algebra:alternate}, and $g$ as in \eqref{eqn:dfn_algebra_vanishing_Fourier_invariant} and using \eqref{eqn:identities_I_Use_a_alot_vector_field_out_2}, it suffices to show that if $f,g\in \cinv$, $k\in \C^\nu_{<0}$, $l\in \Z_+^\nu$, $D\in \DO^l(M)$ with $\Re(k_i)+l_i>0$ for all $i\in \bb{1,\nu}$, then
			\begin{equation*}\begin{aligned}
				\lim_{\lambda_1\to +\infty,\cdots,\lambda_\nu\to +\infty}\lambda^{-k}\alpha_{\lambda}(I_k(f))\ast  \theta_l(D) \ast g		
			\end{aligned}\end{equation*}
		exists in $C^*_{\mathrm{inv}}\mathbb{G}$.
		For any $S\subseteq \bb{1,\nu}$ and $\lambda\in ]1,+\infty[^{\nu}$, let 
			\begin{equation*}\begin{aligned}
				\Lambda_S^{\lambda}=\{\mu\in (\Rpt)^\nu:\forall i \in S,\mu_i\in [1,\lambda_i[\text{ and }\forall i \notin S,\mu_i\in ]0,1]\}\\
				\Lambda_S=\{\mu\in (\Rpt)^\nu:\forall i \in S,\mu_i\in [1,+\infty[\text{ and }\forall i \notin S,\mu_i\in ]0,1]\}	
			\end{aligned}\end{equation*}
		By a change of variables,
			\begin{equation*}\begin{aligned}
				\lambda^{-k}\alpha_{\lambda}(I_k(f))\ast  \theta_l(D) \ast g&=\int_{]0,\lambda_1]\times\cdots \times ]0,\lambda_\nu]}\alpha_{\mu}(f)\ast  \theta_l(D) \ast g \frac{\odif{\mu}}{\mu^{k+\underline{1}}}\\&=\sum_{S\subseteq \bb{1,\nu}}\int_{\Lambda_S^{\lambda}}\alpha_{\mu}(f)\ast  \theta_l(D) \ast g \frac{\odif{\mu}}{\mu^{k+\underline{1}}}
			\end{aligned}\end{equation*}
		By writing $D$ as in \Cref{lem:decomposition_D}, we only need to deal with integrals of the form 
			\begin{equation*}\begin{aligned}
				\int_{\Lambda_S^{\lambda}}\alpha_{\mu}(f)\ast  \theta_l(D_1D_2) \ast g \frac{\odif{\mu}}{\mu^{k+\underline{1}}}=\int_{\Lambda_S^{\lambda}}\alpha_{\mu}(f\ast \theta_{l_{S}}(D_1)) \ast  \theta_{l_{S^c}}(D_2) \ast g \frac{\odif{\mu}}{\mu^{k+l_S+\underline{1}}}
			\end{aligned}\end{equation*}
		By \eqref{eqn:equiv_norm_Cstar_Debord_Skand},
			\begin{equation*}\begin{aligned}
				\lim_{\lambda_1\to +\infty,\cdots,\lambda_\nu\to +\infty}	&\int_{\Lambda_S^{\lambda}}\norm{\alpha_{\mu}(f\ast \theta_{l_{S}}(D_1)) \ast  \theta_{l_{S^c}}(D_2) \ast g}_{C^*\mathbb{G}} \frac{\odif{\mu}}{\mu^{\Re(k)+l_S+\underline{1}}}
				\\=&\int_{\Lambda_S}\norm{\alpha_{\mu}(f\ast \theta_{l_{S}}(D_1)) \ast  \theta_{l_{S^c}}(D_2) \ast g}_{C^*\mathbb{G}} \frac{\odif{\mu}}{\mu^{\Re(k)+l_S+\underline{1}}} 
				\\\leq &\int_{\Lambda_S}\norm{f\ast \theta_{l_{S}}(D_1)}_{C^*\mathbb{G}} \norm{ \theta_{l_{S^c}}(D_2) \ast g}_{C^*\mathbb{G}} \frac{\odif{\mu}}{\mu^{\Re(k)+l_S+\underline{1}}}<+\infty
			\end{aligned}\end{equation*}
		The result follows.
	\end{proof}
	
	\begin{lem}\label{lem:equivariance_smooth_non_zero_fibers}
	 Let $k\in \C^\nu$ and $u\in \Psi^k(\mathbb{G})$. If $\lambda\in (\Rpt)^\nu$, then 
	 the restriction of $\lambda^{-k}\alpha_\lambda(u)-u$ to $M\times M\times (\Rpt)^\nu$ belongs to $C^\infty(M\times M\times (\Rpt)^\nu,\omegahalf)$.
	 Furthermore, there exists $C>0$ (independent of $\lambda$) such that 
		 \begin{equation*}\begin{aligned}
			\supp(\lambda^{-k}\alpha_\lambda(u)-u)\subseteq 	M\times M\times [0,C]^\nu\backslash 0\cup \cG\times \{0\},\quad \forall \lambda\in [1,+\infty[^\nu.	 
		 \end{aligned}\end{equation*}
	\end{lem}
	\begin{proof}
		Let $f\in \cinvf{k+\underline{1}}$ such that $u=I_k(f)$.
		By writing $f$ as in \eqref{eqn:dfn_algebra_vanishing_Fourier_invariant}, we can without loss of generality suppose that $k\in \C^\nu_{<0}$.
		By a change of variables,
			\begin{equation*}\begin{aligned}
			\lambda^{-k}\alpha_\lambda(u)-u=\int_{]0,\lambda_1]\times \cdots ]0,\lambda_\nu]}\alpha_\mu(f)\frac{\odif{\mu}}{\mu^{k+\underline{1}}}-	\int_{]0,1]^\nu}\alpha_\mu(f)\frac{\odif{\mu}}{\mu^{k+\underline{1}}}
			\end{aligned}\end{equation*}
		So, \begin{equation*}\begin{aligned}
			\lambda^{-k}\alpha_\lambda(u)(y,x,t)-u(y,x,t)=\int_{]0,\lambda_1]\times \cdots ]0,\lambda_\nu]}f(y,x,\mu t)\frac{\odif{\mu}}{\mu^{k+\underline{1}}}-	\int_{]0,1]^\nu}f(y,x,\mu t)\frac{\odif{\mu}}{\mu^{k+\underline{1}}}
			\end{aligned}\end{equation*}
		Since this integral avoids $\mu=0$ (which introduces singularities from $f$ being smooth on $\mathbb{G}$ and not $M\times M\times \R_+^\nu$) and $k\in \C^\nu_{<0}$, it follows that the difference defines a smooth function.
		The difference vanishes if $t$ is large enough because $f$ is compactly supported. So, $C$ exists.
	\end{proof}
		Notice that \Cref{thm:amenability_tanget} and \Cref{prop:fiber_dense_reduced}, and density of $M\times (\Rpt)^\nu$ in $\mathbb{G}^{(0)}$ imply that 
	\begin{equation}\label{eqn:norm_identity_tangent_2}\begin{aligned}
				\norm{f}_{C^*\mathbb{G}}=\sup\{\norm{\pi_t(f)}:t\in (\Rpt)^\nu\},\quad \forall f\in \cinv.
			\end{aligned}\end{equation}
	The following identity is obvious
		\begin{equation}\label{eqn:fiber_at_dilations}\begin{aligned}
			\alpha_{\lambda}(u)_{|t}=u_{|\lambda t},\quad \forall \lambda\in (\Rpt)^\nu,t\in \R^\nu_+\backslash 0,u\in C^{-\infty}_{r,s}(\mathbb{G},\omegahalf).
		\end{aligned}\end{equation}
	We also recall that if $u\in C^{-\infty}_{r,s}(\mathbb{G},\omegahalf)$ and $t\in \R^\nu\backslash 0$, then $u_{|t}$ is the Schwartz kernel of $\pi_{t}(u)$. We will basically treat both objects as the same thing. 

	\begin{lem}\label{lem:vanishing_symbol_well_defined}
		If $k\in \C^\nu$, $u\in \Psi^k(\mathbb{G})$ satisfies $u_{|\underline{1}}\in C^\infty_c(M\times M,\omegahalf)$, then for all $g\in \cinvf{-k+\underline{1}}$ and $(\pi,x)\in \widehat{\mathrm{H\!N}}$, $\pi_{\mathrm{inv}}(\sigma^k(u)\ast g)=0$.
		Equivalently, 
			$\sigma^k(u)\ast g\in \cK(L^2(M))\otimes C_0(\R_+^\nu\backslash 0)$,
		see \eqref{eqn:exact_sequence_3_TypeI} and \Cref{thm:Type1_tangent_groupoid}.
	\end{lem}
	This lemma implies that $\pi_{\mathrm{inv}}(\sigma^k(u)\ast g)$ only depends on $u_{|\underline{1}}$ modulo $C^\infty_c(M\times M,\omegahalf)$ and $g$.
	\begin{proof}
		By \eqref{eqn:fiber_at_dilations} and \Cref{lem:equivariance_smooth_non_zero_fibers}, for all $t\in (\Rpt)^\nu$, we have
			\begin{equation*}\begin{aligned}
				u_{|t}=\alpha_t(u)_{|\underline{1}}= t^{k}(t^{-k}\alpha_t(u)_{|\underline{1}}-u_{|\underline{1}}+u_{|\underline{1}})\in C^\infty_c(M\times M,\omegahalf)
			\end{aligned}\end{equation*}
		Let $C$ as in \Cref{lem:equivariance_smooth_non_zero_fibers}, $\underbar{C}=(C,\cdots,C)\in (\Rpt)^\nu$. 
		There exists $h\in C^\infty_{c}(M\times M\times [C/2,C]^\nu,\omegahalf)$ such that $I_k(h)_{|\underline{C}}=u_{|\underline{C}}$, see the proof of \Crefitem{thm:main_prop_bi_graded_pseudo_diff_M}{inclusion_smoothing}.
		By the hypothesis on the support of $h$, $I_k(h)_{|\lambda}=\lambda^ku_{|\underline{C}}$ for all $\lambda\in [C,+\infty[^\nu$.
		So, by \eqref{eqn:fiber_at_dilations} and \Cref{lem:equivariance_smooth_non_zero_fibers}, 
			\begin{equation}\label{eqn:qsdilfljqisdjfoqsjmdf}\begin{aligned}
				\alpha_\lambda(u-I_k(h))_{|\underline{C}}=(\alpha_\lambda(u)-\lambda^k u)_{|\underline{C}}=0 ,\quad \forall\lambda\in [1,+\infty[^\nu.
			\end{aligned}\end{equation}
		Let $g\in \cinvf{-k+\underline{1}}$.
		By \eqref{eqn:qsdilfljqisdjfoqsjmdf}, $\pi_{\underline{C}}(\sigma^k(u-I_k(h))\ast g)=0$.
		For any $\lambda,\mu\in (\Rpt)^\nu$ and $u'\in \Psi^k(\mathbb{G})$, we have
			\begin{equation*}\begin{aligned}
				\pi_{\lambda\mu}(\sigma^k(u')\ast g)=\pi_{\lambda} (\alpha_\mu(\sigma^k(u')\ast g))=\mu^k\pi_\lambda(\sigma^k(u')\ast \alpha_\mu(g))		
			\end{aligned}\end{equation*}
		So, $\pi_{\lambda}(\sigma^k(u-I_k(h))\ast g)=0$ for all $\lambda\in (\Rpt)^\nu$.
		By \eqref{eqn:norm_identity_tangent_2}, $\sigma^k(u-I_k(h))\ast g=0$. Therefore, $\sigma^k(u)\ast g=\sigma^k(I_k(h))\ast g$.
		On the other hand, since $h$ is supported away from the fiber at $0$, 
		using \eqref{eqn:support_automorphism},
		for any $\lambda\in (\Rpt)^\nu$, $\alpha_\lambda(I_k(h))\ast g\in C^\infty_c(M\times M\times \R_+^\nu\backslash 0,\omegahalf) $.
		So, $\sigma^k(I_k(h))\ast g\in \cK(L^2(M))\otimes C_0(\R_+^\nu\backslash 0)$.
	\end{proof}

	\begin{lem}\label{lem:inv_D_is_pi_D}
		If $l\in \Z_+^\nu$ and $X\in \cF^l$, then $\pi_{\mathrm{inv}}(\theta_l(X))=\pi([X]_{l,x})$ as maps $C^{-\infty}(\pi)\to C^{-\infty}(\pi)$.
	\end{lem}
	\begin{proof}
		By density of $C^\infty(\pi)$ in $C^{-\infty}(\pi)$, it is enough to check equality on $C^\infty(\pi)$.
		By \eqref{eqn:Cinfinty_cinv}, we only need to check that both agree when composed by $\pi_{\mathrm{inv}}(f)$ for $f\in \cinv$.
		Let $(V,\natural)$ be a $\nu$-graded Lie basis.
		We choose $V$ such that there exists $v_0\in V^l$ such that $\natural(v_0)=X$ (which exists by \Cref{lem:bigger_bi_graded_Lie}).
		So, $\natural_{x,0}(v_0)=[X]_{l,x}$, and the lemma follows from \Cref{rem:more_details_on_proof} and the definition of $\pi_{\mathrm{inv}}$ in \eqref{eqn:pi_inv_definition}
	\end{proof}
	
	\begin{lem}\label{lem:smooth_vectors_as_sums_of_differentials}
		Let $(\pi,x)\in \HatHLnon$, $k\in \C^\nu$. One has 
			\begin{equation*}\begin{aligned}
				C^\infty(\pi)= \left\{\sum_{i=1}^n\pi_{\mathrm{inv}}(g_i)\xi_i:g_i\in \cinvf{k},\xi_i\in C^\infty(\pi)\right\}.		
			\end{aligned}\end{equation*}
	\end{lem}
	\begin{proof}
		By an iterated application of \Cref{thm:pi_diff_appendix_decompo_vector_fields} together with \Cref{lem:inv_D_is_pi_D}, we have 
			\begin{equation*}\begin{aligned}
						C^\infty(\pi)= \left\{\sum_{i=1}^n\pi_{\mathrm{inv}}(\theta_{l_i}(D_i))\xi_i:l_i\in \Z_+^\nu,k\preceq l_i,D_i\in \DO^{l_i}(M),\xi_i\in C^\infty(\pi)\right\}.
			\end{aligned}\end{equation*}
		By applying \eqref{eqn:Cinfinty_cinv} to each $\xi_i$, the result follows from \eqref{eqn:dfn_algebra_vanishing_Fourier_invariant}.
	\end{proof}

	\begin{lem}\label{lem:differential_symbol_same}
		Let $u\in \Psi^0(\mathbb{G})$ and $f\in C^\infty_c(M)$ such that $u_{|\underline{1}}=\delta_f$, then for any $h\in \cinvf{\underline{1}}$ and $(\pi,x)\in \widehat{\HN}$, $\pi_{\mathrm{inv}}(\sigma^0(u)\ast h)=f(x)\pi_{\mathrm{inv}}(h)$.
	\end{lem}
	\begin{proof}
		By \Cref{lem:vanishing_symbol_well_defined}, it is enough to prove the lemma for a single $u$.
		We follow the notation in the proof of \Crefitem{thm:main_prop_bi_graded_pseudo_diff}{incl}.
		So, let $g\in \cinv$ such that $\alpha_0(g)=\theta_0(\delta_f)$, and for any $S\subseteq \bb{1,\nu}$, let $g_S$ be as in \eqref{eqn:g_S}.
		By \eqref{eqn:theta_0_sum_lower_order} and \eqref{eqn:identities_I_Use_a_alot_vector_field_out_2},
			\begin{equation*}\begin{aligned}
				\delta_f=		\sum_{S\subseteq \bb{1,\nu}}\sum_{S'\subseteq S}(-1)^{|S'|}I_{\kappa_{S'}}(g_S)_{|\underline{1}}=\sum_{S\subseteq \bb{1,\nu}}\sum_{S'\subseteq S}(-1)^{|S'|}I_{0}(\theta_{-\kappa_{S'}}(\delta_1) \ast g_S)_{|\underline{1}}.
			\end{aligned}\end{equation*}
		So we take $u=\sum_{S\subseteq \bb{1,\nu}}\sum_{S'\subseteq S}(-1)^{|S'|}I_0(\theta_{-\kappa_{S'}}(\delta_1) \ast g_S)$.
		Let $h\in \cinvf{\underline{1}}$.
				We will now compute $\sigma^0(u)\ast h$.
		If $S'\neq \bb{1,\nu}$, then
			\begin{equation*}\begin{aligned}
				\alpha_\lambda(I_{0}(\theta_{-\kappa_{S'}}(\delta_1) \ast g_S))\ast h=\lambda^{-\kappa_{S'}}\theta_{-\kappa_{S'}}(\delta_1)\ast \alpha_\lambda(I_{\kappa_{S'}}(g_S))\ast h		
			\end{aligned}\end{equation*}
		is an invariant function which vanishes on $\cG\times \{0\}$ because $\kappa_{S'}\neq 0$.
		So, by \eqref{eqn:ksqjfmksqd} (the $L^1$-estimate), 
			\begin{equation*}\begin{aligned}
				\alpha_\lambda(I_{0}(\theta_{-\kappa_{S'}}(\delta_1) \ast g_S))\ast h\in \cK(L^2(M))\otimes C_0(\R_+^\nu\backslash 0).
			\end{aligned}\end{equation*}
		Hence, the limit is also in $\cK(L^2(M))\otimes C_0(\R_+^\nu\backslash 0)$.
		For the last term (the term with $S'=S=\bb{1,\nu}$), we have by \eqref{eqn:diff_tempo_formula_jkqhjsk},
			\begin{equation*}\begin{aligned}
				\sigma^0(I_0(g_{\bb{1,\nu}}))\ast h=\int_{(\Rpt)^\nu} \alpha_\lambda(g_{\bb{1,\nu}})\ast h \frac{\odif{\lambda}}{\lambda^{\underline{1}}}=\int_{(\Rpt)^\nu} \pdv*[mixed-order={\nu}]{\alpha_{\lambda}(g)\ast h}{\lambda_{1}\cdots,\lambda_{\nu}}\odif{\lambda}.
			\end{aligned}\end{equation*}
			We claim that for any $i\in \bb{1,\nu}$, there exists $C>0$ such that
				\begin{equation}\label{eqn:kqlsdjfkjsqdmljfqsdjfklqmsd}\begin{aligned}
					\norm{\alpha_\lambda(g)\ast h}_{C^*\mathbb{G}}\leq C \lambda_i^{-1}	,\quad \forall \lambda\in (\Rpt)^\nu,
				\end{aligned}\end{equation}
			for some constant $C>0$.
			To see this, if $\lambda_i\leq 1$, then $\norm{\alpha_\lambda(g)\ast h}$ is bounded by $\norm{g}_{C^*\mathbb{G}}\norm{h}_{C^*\mathbb{G}}$, see \eqref{eqn:equiv_norm_Cstar_Debord_Skand}.
			If $\lambda_i\geq 1$ by \eqref{eqn:dfn_algebra_vanishing_Fourier_invariant}, we can suppose that $h=\theta_l(D)\ast h'$ where $l\in \Z_+^\nu$, $h'\in \cinv$, $D\in \DO^{l}(M)$, $\underline{1}\preceq l$.
			By applying \Cref{lem:decomposition_D} with $S=\{i\}$, we can suppose that $D=D_1D_2$ with $D_1\in \DO^{l_S}(M)$ and $D_2\in \DO^{l_{S^c}}(M)$.
			By \eqref{eqn:equiv_norm_Cstar_Debord_Skand}, we have 
		\begin{equation*}\begin{aligned}
			\norm{\alpha_\lambda(g)\ast h}_{C^*\mathbb{G}}&=\norm{\alpha_\lambda(g\ast \theta_{l_S}(D_1))\ast \theta_{l_{S^c}}(D_2)\ast h'}_{C^*\mathbb{G}}\lambda_i^{-l_i}\\&\leq \norm{g\ast \theta_{l_S}(D_1)}_{C^*\mathbb{G}}\norm{\theta_{l_{S^c}}(D_2)\ast h'}_{C^*\mathbb{G}}\lambda_i^{-l_i} .		
		\end{aligned}\end{equation*}
		Finally, since $1\preceq l$ (and so $1\leq l_i$) and $1\leq \lambda_i$, \eqref{eqn:kqlsdjfkjsqdmljfqsdjfklqmsd} follows.
		By \eqref{eqn:kqlsdjfkjsqdmljfqsdjfklqmsd}, 
			\begin{equation*}\begin{aligned}
				\int_{(\Rpt)^\nu} \pdv*[mixed-order={\nu}]{\alpha_{\lambda}(g)\ast h}{\lambda_{1}\cdots,\lambda_{\nu}}\odif{\lambda}=(-1)^\nu\alpha_{0}(g)\ast h=	(-1)^\nu\theta_{0}(\delta_f)\ast h	
			\end{aligned}\end{equation*}
		The result follows.
	\end{proof}
	\paragraph{Proof of \Cref{thm:principal_symbol_pseudo_diff_well_defined}.}
		By \Cref{lem:smooth_vectors_as_sums_of_differentials} and \Cref{lem:principal_symbol}, it follows that for any $\xi\in C^\infty(\pi)$, Limit \eqref{eqn:dfn_principal_symbol} converges in the topology of $L^2(\pi)$ and thus defines a continuous linear map $C^\infty(\pi)\to L^2(\pi)$.
		To check that the image is in $C^\infty(\pi)$ and that the convergence actually happens in the topology of $C^\infty(\pi)$, it suffices to use the identity $\sigma^{k+l}(\theta_l(D)\ast u)=\theta_l(D)\ast \sigma^k(u)$. 
		So, \Crefitem{thm:principal_symbol_pseudo_diff_well_defined}{Infinity} follows.
		By duality, \Crefitem{thm:principal_symbol_pseudo_diff_well_defined}{negative_infinity} follows.
		By \Cref{lem:vanishing_symbol_well_defined}, it follows that $\sigma^k(v,\pi,x)$ is independent of the choice of $f$ such that $I_k(f)=v$.
		This proves that the symbol is well-defined in the compactly supported case.
		\Crefitem{thm:principal_symbol_pseudo_diff_well_defined}{3} in the compactly supported case is straightforward to prove.
		\Crefitem{thm:principal_symbol_pseudo_diff_well_defined}{3} in the compactly supported case together with \Cref{lem:differential_symbol_same}, one deduces that the symbol is well-defined in the non-compactly supported case.
		We now prove \Crefitem{thm:principal_symbol_pseudo_diff_well_defined}{5}.
		If $v\in \Psi^k_c(M)$ and $k\prec l$ and $\xi\in C^\infty(\pi)$, then 
			\begin{equation*}\begin{aligned}
				\sigma^l(v,\pi,x)&=\lim_{\lambda_1\to +\infty,\cdots,\lambda_\nu\to +\infty}\pi_{\mathrm{inv}}(\lambda^{-l}\alpha_{\lambda}(I_k(f)))\xi\\&=\lim_{\lambda_1\to +\infty,\cdots,\lambda_\nu\to +\infty}\lambda^{k-l}\pi_{\mathrm{inv}}(\lambda^{-k}\alpha_{\lambda}(I_k(f)))\xi=0\times \sigma^{k}(v,\pi,x)=0.	
			\end{aligned}\end{equation*}
		The rest of \Cref{thm:principal_symbol_pseudo_diff_well_defined} is straightforward to prove.\hfill$\qed$

\section{The \texorpdfstring{$C^*$}{C-Star}-algebras of operators of order \texorpdfstring{$0$}{0}}\label{sec:cstar_algebras_of_pseudo_of_order0}

The goal of this section is to study the $C^*$-algebra of pseudo-differential operators. Ultimately, we want to show the following: \begin{enumerate}
	\item   Compactly supported operators (and their symbols) of order $k\in \C^\nu_{\leq 0}$ are bounded.
	\item 	The $C^*$-algebra of operators of order $0$ is of Type I, and its irreducible representations all come from either the principal symbol defined in \Cref{sec:principal_symbol} or from the natural representation on $L^2(M)$.
	\item The third hypothesis of \Cref{thm:simple_arg_sobolev} is satisfied, see \Cref{lem:vk}
\end{enumerate}
Once we do these three steps, we can apply \Cref{thm:simple_arg_sobolev}. All the results of \Cref{sec:sobolev_spaces}, \Cref{sec:main_theorem}, \Cref{sec:extension_of_parameters} will then follow easily.

Our first goal is to prove that operators of order $0$ are bounded.
	\begin{lem}\label{lem:Cotlar_Stein_lemma}
			For any $f,g\in \cinvf{\underline{1}}$, $p\in ]0,+\infty[$,
			\begin{equation*}\begin{aligned}
				\int_{(\Rpt)^\nu}\norm{f\ast \alpha_{\lambda}(g)}_{C^*\mathbb{G}}^{p}\frac{\odif{\lambda}}{\lambda^{\underline{1}}}<+\infty
			\end{aligned}\end{equation*}
		\end{lem}	
	\begin{proof}
			For any $S\subseteq \bb{1,\nu}$, let $\Lambda_S=\{\lambda\in (\Rpt)^\nu:\forall i\in S,\lambda_i< 1,\ \forall i\in S^c,\lambda_i\geq 1\}$.
			We only need to show that the integral over each $\Lambda_S$ is finite.
			Let us fix $S$. 
			We write $f$ as in \eqref{eqn:invariant_vanish_algebra:alternate}, and $g$ as in \eqref{eqn:dfn_algebra_vanishing_Fourier_invariant}.
			Since the sum of $L^p$-functions is again $L^p$ (even for $p<1$), we can suppose that $f$ and $g$ are of the form $f'\ast \theta_k(D)$ and $\theta_l(D')\ast g'$ with $f',g'\in \cinv$, $D\in \DO^k(M)$, $D'\in \DO^l(M)$, $\underline{1}\preceq k$, $\underline{1}\preceq l$.
			Hence, by \eqref{eqn:homogenity_theta} and \eqref{eqn:theta_product}, 
				\begin{equation*}\begin{aligned}
					f\ast \alpha_\lambda(g)=\lambda^kf'\ast \theta_{l+k}(DD')\ast \alpha_{\lambda}(g').
				\end{aligned}\end{equation*}
			We now apply \Cref{lem:decomposition_D} to $DD'$.
			So, we can suppose that $DD'$ is of the form $D_1D_2$ with $D_1\in \DO^{k_S+l_{S}}(M)$, $D_2\in \DO^{k_{S^c}+l_{S^c}}(M)$.
			Therefore, 
			\begin{equation*}\begin{aligned}
					f\ast \alpha_\lambda(g)=\lambda^{k_{S}-l_{S^c}}f'\ast \theta_{l_S+k_S}(D_1)\ast \alpha_{\lambda}(\theta_{k_{S^c}+l_{S^c}}(D_2)\ast g'). 
				\end{aligned}\end{equation*}
			Hence, by \eqref{eqn:equiv_norm_Cstar_Debord_Skand},
				\begin{equation*}\begin{aligned}
							\norm{f\ast \alpha_\lambda(g)}^p_{C^*\mathbb{G}}	\leq \norm{f'\ast \theta_{l_S+k_S}(D_1)}_{C^*\mathbb{G}}^p\norm{\theta_{k_{S^c}+l_{S^c}}(D_2)\ast g'}^p_{C^*\mathbb{G}}\lambda^{pk_{S}-pl_{S^c}}
				\end{aligned}\end{equation*}
			Since $1\preceq k$ and $1\preceq l$, it follows that $1\leq k_i$ and $1\leq l_i$ for all $i\in \bb{1,\nu}$.
			If $i\in S$, then the $i$'th coordinate of $pk_{S}-pl_{S^c}$ is equal to $pk_{i}>0$.
			If $i\notin S$, then the $i$'th coordinate of $pk_{S}-pl_{S^c}$ is equal to $-pl_{i}<0$.
			So, the integral $\int_{\Lambda_S}\lambda^{pk_S-p{l}_{S^c}-\underline{1}}\odif{\lambda}$ converges.
			The result follows.
	\end{proof}
	\begin{prop}\label{lem:operators_of_order_0_on_M_are_bounded}
		\begin{enumerate}
			\item\label{lem:operators_of_order_0_on_M_are_bounded:1} Any $u\in \Psi^0(\mathbb{G})$ extends to a bounded multiplier of $C^*_{\mathrm{inv}}\mathbb{G}$.
			In fact if $f\in \cinvf{\underline{1}}$, then 
				\begin{equation}\label{eqn:bound_on_u}\begin{aligned}
					\norm{I_0(f)}_{\cM(C^*_{\mathrm{inv}}\mathbb{G})}\leq \sqrt{\int_{(\Rpt)^\nu}\norm{f^*\ast \alpha_{\lambda}(f)}_{C^*\mathbb{G}}^{\frac{1}{2}}\frac{\odif{\lambda}}{\lambda^{\underline{1}}}}\sqrt{\int_{(\Rpt)^\nu}\norm{f\ast \alpha_{\lambda}(f)^*}_{C^*\mathbb{G}}^{\frac{1}{2}}\frac{\odif{\lambda}}{\lambda^{\underline{1}}}}		
				\end{aligned}\end{equation}
			
			\item\label{lem:operators_of_order_0_on_M_are_bounded:2} For any $u\in \Psi^0(\mathbb{G})$, $\sigma^0(u)$ defined in \Cref{lem:principal_symbol} extends to a bounded multiplier of $C^*_{\mathrm{inv},\underline{1}}(\mathbb{G})$.
			Furthermore, 
				\begin{equation}\label{eqn:sigma_0_bounded_by_norm}\begin{aligned}
					\norm{\sigma^0(u)}_{\cM(C^*_{\mathrm{inv},\underline{1}}\mathbb{G})}\leq \norm{u}_{\cM(C^*_{\mathrm{inv}}\mathbb{G})}.
				\end{aligned}\end{equation}
			
			\item\label{lem:operators_of_order_0_on_M_are_bounded:4}For any $k\in \C^\nu_{\leq 0}$, $v\in \Psi^k_c(M)$, $v\ast L^2(M)\subseteq L^2(M)$ and $v:L^2(M)\to L^2(M)$ is bounded.
			Furthermore, if $\Re(k_i)<0$ for some $i$, then $v\ast L^2(M)\subseteq L^2(M)$ and $v:L^2(M)\to L^2(M)$ is compact.
			\item\label{lem:operators_of_order_0_on_M_are_bounded:5}For any $k\in \C^\nu_{\leq 0}$, $v\in \Psi^k_c(M)$ and $(\pi,x)\in \HatHLnon$, $\sigma^k(v,\pi,x)\ast L^2(\pi)\subseteq L^2(\pi)$ and $\sigma^k(v,\pi,x):L^2(\pi)\to L^2(\pi)$ is bounded. 
	Furthermore, if $\Re(k_i)<0$ for all $i$, then $\sigma^k(v,\pi,x):L^2(\pi)\to L^2(\pi)$ is compact. 
		\end{enumerate}
	\end{prop}
	\begin{proof}
		Let us prove \MyCref{lem:operators_of_order_0_on_M_are_bounded}[1].
		Let $f\in \cinvf{\underline{1}}$ such that $u=I_0(f)$.
		We need to show that $u$ extends to a bounded multiplier of $C^*_{\mathrm{inv}}\mathbb{G}$.
		It is already well-defined on the dense set of elements $\cinv$. So, let $g\in\cinv$.
		We need to bound $\norm{u\ast g}_{C^*\mathbb{G}}$.
		By \eqref{eqn:norm_identity_tangent_2}, we only need to bound $\pi_t(u\ast g)$ uniformly in $t\in (\Rpt)^\nu$.
		Formally, we have  
			\begin{equation*}\begin{aligned}
				\pi_t(u\ast g)=\left(\int_{]0,1]^\nu}\pi_t(\alpha_{\lambda}(f))\frac{\odif{\lambda}}{\lambda^{\underline{1}}}\right)\ast \pi_t(g).		
			\end{aligned}\end{equation*}
		Therefore, it suffices to show that the integrals $\int_{]0,1]^\nu}\pi_t(\alpha_{\lambda}(f))\frac{\odif{\lambda}}{\lambda^{\underline{1}}}$ define bounded operators on $L^2(M)$ whose operator norm is uniformly bounded in $t$.
		By the Cotlar Stein-lemma, see \cite[Lemma 18.6.5]{HormanderIII},
		it suffices to show that
			\begin{equation*}\begin{aligned}
						\sup_{t\in (\Rpt)^\nu}\sup_{\mu\in ]0,1]^\nu}\int_{]0,1]^\nu}\norm{\pi_t(\alpha_{\mu}(f)^*\ast\alpha_{\lambda}(f))}^\frac{1}{2}_{\cL(L^2(M))}\frac{\odif{\lambda}}{\lambda^{\underline{1}}}<+\infty,\\
						\sup_{t\in (\Rpt)^\nu}\sup_{\mu\in ]0,1]^\nu}\int_{]0,1]^\nu}\norm{\pi_t(\alpha_{\lambda}(f)\ast\alpha_{\mu}(f)^*)}^\frac{1}{2}_{\cL(L^2(M))}\frac{\odif{\lambda}}{\lambda^{\underline{1}}}<+\infty.
			\end{aligned}\end{equation*}
		We will only consider the first. The second is dealt with similarly.
		By \eqref{eqn:equiv_norm_Cstar_Debord_Skand}, 
			\begin{equation*}\begin{aligned}
				\norm{\pi_t(\alpha_{\mu}(f)^*\alpha_{\lambda}(f))}_{\cL(L^2(M))}\leq \norm{\alpha_{\mu}(f)^*\alpha_{\lambda}(f)}_{C^*\mathbb{G}}=\norm{f^*\ast \alpha_{\lambda\mu^{-1}}(f)}_{C^*\mathbb{G}}.
			\end{aligned}\end{equation*}
		Therefore, it suffices to show that 
			\begin{equation*}\begin{aligned}
				\int_{(\Rpt)^\nu}\norm{f^*\ast \alpha_{\lambda}(f)}_{C^*\mathbb{G}}^{\frac{1}{2}}\frac{\odif{\lambda}}{\lambda^{\underline{1}}}<+\infty.
			\end{aligned}\end{equation*}
		This follows from \Cref{lem:Cotlar_Stein_lemma}.
		This finishes the proof of \Crefitem{lem:operators_of_order_0_on_M_are_bounded}{1}.

		We now prove \MyCref{lem:operators_of_order_0_on_M_are_bounded}[2].
		By \eqref{eqn:invariant_vanish_algebra:alternate}, if $g\in \cinvf{\underline{1}}$, then $\alpha_\lambda(u)\ast g\in \cinvf{1}$.
		Hence, $\sigma^0(u)\ast g$ is in $C^*_{\mathrm{inv},\underline{1}}(\mathbb{G})$.
		By \eqref{eqn:equiv_norm_Cstar_Debord_Skand}, one deduces that $\norm{\alpha_\lambda(u)}_{\cM(C^*_{\mathrm{inv}}\mathbb{G})}=\norm{u}_{\cM(C^*_{\mathrm{inv}}\mathbb{G})}$.
		Hence, $\sigma^0(u)$ extends to a bounded multiplier of $C^*_{\mathrm{inv},\underline{1}}(\mathbb{G})$ whose adjoint is $\sigma^0(u^*)$, and \eqref{eqn:sigma_0_bounded_by_norm} holds.

		We now prove \MyCref{lem:operators_of_order_0_on_M_are_bounded}[4].
		By replacing $v$ by $v^*\ast v$, we can suppose that $k\in \R^\nu$ with $k_i\leq  0$ for all $i$.
		Any non-degenerate representation of a $C^*$-algebra induces a representation of its multiplier algebra, see \cite[Chapter 4]{LanceBook}.
		The case $k=0$ follows from \MyCref{lem:operators_of_order_0_on_M_are_bounded}[1] using the representation $\pi_{\underline{1}}$ of $C^*_{\mathrm{inv}}(\mathbb{G})$.
		So, we can suppose that $k_i<0$ for some $i$.
		By \Crefitem{thm:main_prop_bi_graded_pseudo_diff_M}{5}, for $n$ big enough,	$(v^*\ast v)^n$ maps continuously $L^2(M)$ to $H^{2nk}(M)\subseteq C^1(M)$. Since $v$ is compactly supported, the image lies in $C^1_c(K)$ for some $K\subseteq M$ compact.
		The inclusion $C^1(K)\hookrightarrow L^2(M)$ is compact (Rellich's theorem).
		Hence, $(v^*\ast v)^n\ast\cdot:L^2(M)\to L^2(M)$ is bounded and compact. 
		It follows that $v\ast \cdot:L^2(M)\to L^2(M)$ is bounded and compact.

		Let us prove \MyCref{lem:operators_of_order_0_on_M_are_bounded}[5].
		Again by replacing $v$ with $v^*\ast v$, we can suppose that $k\in \R^\nu$ with $k_i\leq 0$ for all $i$.
		Formally, we have 
			\begin{equation*}\begin{aligned}
					\sigma^k(v,\pi,x)=\int_{(\Rpt)^\nu} \pi_{\mathrm{inv}}(\alpha_\lambda(f))\frac{\odif{\lambda}}{\lambda^{k+\underline{1}}}
			\end{aligned}\end{equation*}
		We now to use Cotlar-Stein lemma on the coordinates where $k_i=0$ and Riemann-Lebesgue lemma \MyCref{lem:Riemann-Lebesgue} on those where $k_i<0$.
		More precisely, first recall that since the $C^*$-algebra of nilpotent Lie groups is liminal, see \cite[13.11.12]{Dixmier}, $\pi_{\mathrm{inv}}(\alpha_\lambda(f))$ is a compact operator.
		We divide the integral into $\int_{(\Rpt)^\nu\backslash]0,1]^\nu }+\int_{]0,1]^\nu}$. By \MyCref{lem:Riemann-Lebesgue}, the first integral converges in norm. So, it defines a compact bounded operator.
		For  the second integral, let $S=\{i\in\bb{1,\nu}:k_i=0\}$.
		If $S=\emptyset$, then the second integral also converges in norm (because $k_i<0$ for all $i$), so it also defines a compact operator.
		If $S\neq \emptyset$, we proceed as follows:
		We divide the integral into
			\begin{equation*}\begin{aligned}
				\int_{]0,1]^\nu} \pi_{\mathrm{inv}}(\alpha_\lambda(f))\frac{\odif{\lambda}}{\lambda^{k+\underline{1}}}=\int_{]0,1]^S} \int_{]0,1]^{S^c}}\pi_{\mathrm{inv}}(\alpha_\lambda(f))\frac{\odif{\lambda}}{\lambda^{k+\underline{1}}}.
			\end{aligned}\end{equation*}
		The argument we used above with Cotlar Stein works equally well to give boundedness of the integral $\int_{]0,1]^{S^c}}$. Furthermore, the bound of the operator norm of $\int_{]0,1]^{S^c}}$ obtained from the Cotlar Stein lemma is uniform for in $\lambda\in ]0,1]^S$. From this, we get that the integral converges in the strong operator topology to a bounded operator.
	\end{proof}	
	We now pass to the second goal which is analyzing the $C^*$-algebra generated by compactly supported operators of order $0$.
	\begin{prop}\label{thm:main_representation_theorem_morita_2}
		The map 
			\begin{equation}\label{eqn:sigma_0_map_Psi0M}\begin{aligned}
				\sigma^0(\Psi^0(\mathbb{G}))\to 		\overline{\Psi^0_c(M)},\quad \sigma^0(u)\mapsto \pi_{\underline{1}}(\sigma^0(u)),\quad \forall u\in \Psi^0(\mathbb{G}),
			\end{aligned}\end{equation}
		is well-defined, and extends to a $*$-isomorphism between the closures $\overline{\sigma^0(\Psi^0(\mathbb{G}))}\subseteq \cL(C^*_{\mathrm{inv},\underline{1}}(\mathbb{G}))$ and $\overline{\Psi^0_c(M)}\subseteq \cL(L^2(M))$.
		In particular, the spectrums of $\overline{\sigma^0(\Psi^0(\mathbb{G}))}$ and  $\overline{\Psi^0_c(M)}$ are homeomorphic.
	\end{prop}
	\begin{proof}
			Let $u\in \Psi^0(\mathbb{G})$. We claim that 
				\begin{equation}\label{eqn:kqsdjfsqjdmjflmjqsldmjfklmqjsdklfj}\begin{aligned}
					\pi_{\underline{1}}(\sigma^0(u))-\pi_{\underline{1}}(u)\in \cK(L^2(M)).		
				\end{aligned}\end{equation}
			To see this, if $C$ is as in \Cref{lem:equivariance_smooth_non_zero_fibers}, then for any $t\in ]C,+\infty[^\nu$, $\pi_t(\sigma^0(u))=\pi_t(u)$.
			By equivariance of $\sigma^0(u)$, $\pi_t(\sigma^0(u))=\pi_{\underline{1}}(\sigma^0(u))$ 
			By \Cref{lem:equivariance_smooth_non_zero_fibers} and \eqref{eqn:fiber_at_dilations}, $u_{|t}-u_{|\underline{1}}\in C^\infty_c(M\times M,\omegahalf)$.
			Our claim follows.

			By \eqref{eqn:dfn_pseudo}, $\pi_{\underline{1}}(u)\in \Psi^0_c(M)$.
			By \Crefitem{thm:main_prop_bi_graded_pseudo_diff_M}{inclusion_smoothing}, $\cK(L^2(M))\subseteq \overline{\Psi^0_c(M)}$.
			So, $\pi_{\underline{1}}(\sigma^0(u))\in \overline{\Psi^0_c(M)}$, and the map \eqref{eqn:sigma_0_map_Psi0M} is well-defined.
			It is an isometry by \eqref{eqn:norm_identity_tangent_2} and equivariance of $\sigma^0(u)$.
			It is clearly surjective.
	\end{proof}
	We will now analyze $\overline{\sigma^0(\Psi^0(\mathbb{G}))}\simeq \overline{\Psi^0_c(M)}$.
	By \Cref{lem:support_cinv_1} and \Crefitem{lem:operators_of_order_0_on_M_are_bounded}{2}, we have a map 
		\begin{equation}\label{eqn:spectrum_map_tempojkqsdlhf}\begin{aligned}
			(\Rpt)^\nu\sqcup \HatHLnon\to 		\widehat{\overline{\sigma^0(\Psi^0(\mathbb{G}))}}.
		\end{aligned}\end{equation}
	Since $\alpha_\lambda(\sigma^0(u))=\sigma^0(u)$ for any $u\in \Psi^0(\mathbb{G})$ and $\lambda\in (\Rpt)^\nu$, \eqref{eqn:spectrum_map_tempojkqsdlhf} descends to a map 
		\begin{equation}\label{eqn:spectrum_map_tempojkqsdlhf_2}\begin{aligned}
			\{\underline{1}\}\sqcup \HatHLnon/(\Rpt)^\nu=\big((\Rpt)^\nu\sqcup \HatHLnon\big)/(\Rpt)^\nu\to 		\widehat{\sigma^0(\overline{\Psi^0(\mathbb{G})})}.
		\end{aligned}\end{equation}

	\begin{theorem}\label{thm:main_representation_theorem_morita_1}
		The $C^*$-algebra $\overline{\sigma^0(\Psi^0(\mathbb{G}))}$ is of Type I. Furthermore, \eqref{eqn:spectrum_map_tempojkqsdlhf_2} is a homeomorphism where the left-hand-side is equipped with the quotient topology.
		In particular $\{\underline{1}\}$ is an open dense subset of $\widehat{\overline{\sigma^0(\Psi^0(\mathbb{G})})}$.
	\end{theorem}
	\begin{proof}
		To prove this theorem, we will show that $\overline{\sigma^0(\Psi^0(\mathbb{G}))}$ is Morita equivalence to the cross-product $C^*_{\mathrm{inv},\underline{1}}\mathbb{G}\rtimes_r (\Rpt)^\nu$. A general reference for Morita equivalence and cross products are the books by Raeburn and Williams \cite[Chapter 3]{WilliansMoritaBook} and Williams \cite[Chapter 2]{WilliamsCrossedBook}.
		We need the following results:
		\begin{enumerate}
			\item    If $A$ and $B$ are Morita equivalent $C^*$-algebras, then the spectrum of $A$ and the spectrum of $B$ are homeomorphic, see \cite[Corollary 3.33]{WilliansMoritaBook}.
					In addition, if $A$ and $B$ are separable, then by \cite[Chapter 9]{Dixmier}, $A$ is of Type I if and only if $B$ is of Type I.
 			\item 	We will also need a result on the cross product of $C^*$-algebras due to Rieffel \cite{RieffelProperActions}. 
					The result of Rieffel has been improved by Buss and Echterhoff \cite{BussEchterhoffRieffel}. 
					We will state the version by Buss and Echterhoff. 
					Since, we only need their theorem for the group $(\Rpt)^\nu$, we will state the theorem only for unimodular groups
					\begin{theorem}\label{thm:SiegrfiedBuss}
						Let $G$ be a unimodular locally compact second countable group that acts on a $C^*$-algebra $A$, $d\mu$ a Haar measure on $G$.
						Suppose that $\cA\subseteq A$ is a dense linear subspace such that for any $a,b\in \cA$, $\int_G \norm{b^*\alpha_t(a)}\,d\mu(t)<+\infty$.
						Then, the following holds:\begin{enumerate}
							\item   If $a,b\in \cA$, then $\int_{G}\alpha_t(ab^*)\,d\mu(t)$ converges and defines an element of $\cL(A)$.\footnote{See \cite[Deinition 2,2]{BussEchterhoffRieffel} for the precise meaning of the convergence of the integral.}
							\item 	The $C^*$-algebra generated by $\{\int_{G}\alpha_t(ab^*)\,d\mu(t): a,b\in \cA\}$ inside $\cL(A)$ is Morita equivalent to a two-sided ideal of the reduced crossed product $A\rtimes_r G$.
						\end{enumerate}
					\end{theorem}
		\end{enumerate}
		By \Cref{lem:Cotlar_Stein_lemma}, we can apply \Cref{thm:SiegrfiedBuss} with $A=C^*_{\mathrm{inv},\underline{1}}\mathbb{G}$, $\cA=\cinvf{\underline{1}}$, $G=(\Rpt)^\nu$.
		If $f,g\in \cinvf{\underline{1}}$, then 
			\begin{equation*}\begin{aligned}
				\int_{(\Rpt)^\nu}\alpha_{\lambda}(f \ast g^*)\frac{\odif{\lambda}}{\lambda^{\underline{1}}}=\sigma^0(I_0(f\ast  g)).		
			\end{aligned}\end{equation*}
			\begin{lem}\label{lem:main_representation_theorem_morita_1_lemma_1}
							 The $C^*$-algebra $C^*_{\mathrm{inv},\underline{1}}\mathbb{G}\rtimes_r (\Rpt)^\nu$ is separable of Type I. Furthermore, its spectrum is equal to $\{\underline{1}\}\sqcup \HatHLnon/(\Rpt)^\nu$.

			\end{lem}
			\begin{lem}\label{lem:main_representation_theorem_morita_1_lemma_2}
				One has 
				\begin{equation*}\begin{aligned}
					\overline{\mathrm{span}(\{\sigma^0(I_0(f\ast g)):f,g\in \cinvf{\underline{1}}\})}=\overline{\sigma^0(\Psi^0(\mathbb{G}))}.		
				\end{aligned}\end{equation*}
			\end{lem}
			\Cref{lem:main_representation_theorem_morita_1_lemma_2} and	\Cref{thm:SiegrfiedBuss} imply that $\overline{\sigma^0(\Psi^0(\mathbb{G}))}$ is Morita equivalent to a two-sided ideal of the crossed product $C^*_{\mathrm{inv},\underline{1}}\mathbb{G}\rtimes_r (\Rpt)^\nu$.
			So, by \Cref{prop:exat_sequence_3_TypeI} and \Cref{lem:main_representation_theorem_morita_1_lemma_1}, $\overline{\sigma^0(\Psi^0(\mathbb{G}))}$ is separable of Type I, and its spectrum is an open subset of the spectrum $C^*_{\mathrm{inv},\underline{1}}\mathbb{G}\rtimes_r (\Rpt)^\nu$.
			Since the spectrum  of  $\overline{\sigma^0(\Psi^0(\mathbb{G}))}$ already contains  $\{\underline{1}\}\sqcup \HatHLnon/(\Rpt)^\nu$, it follows that $\overline{\sigma^0(\Psi^0(\mathbb{G}))}$ is Morita equivalent to  $C^*_{\mathrm{inv},\underline{1}}\mathbb{G}\rtimes_r (\Rpt)^\nu$. This finishes the proof of \Cref{thm:main_representation_theorem_morita_1} once we prove \Cref{lem:main_representation_theorem_morita_1_lemma_1} and \Cref{lem:main_representation_theorem_morita_1_lemma_2}.
	\end{proof}
	\begin{proof}[Proof of \Cref{lem:main_representation_theorem_morita_1_lemma_1}.]
		Let $(V,\natural)$ be a $\nu$-graded basis like in \Cref{lem:weak_commutative_graded_basis}.
		Recall from the proof of \Cref{thm:Type1_tangent_groupoid} that we have a short exact sequence 
			\begin{equation*}\begin{aligned}
										0\to \mathcal{K}(L^2(M))\otimes C_0(\R_+^\nu\backslash 0) \to C^*_{\mathrm{inv}}\mathbb{G}\to  C^*_{\mathrm{inv}}\cG\to 0,				
			\end{aligned}\end{equation*}
		and that $C^*_{\mathrm{inv}}\cG$ is a quotient of the $C^*$-algebra $C^*(V)\otimes C_0(M)$.
		The $C^*$-algebra $C^*_{\mathrm{inv},\underline{1}}\mathbb{G}$ is a two-sided ideal of $C^*_{\mathrm{inv}}\mathbb{G}$ whose spectrum by \Cref{lem:support_cinv_1} is equal to $(\Rpt)^\nu \sqcup \HatHLnon\times \{0\}$.
		It follows that $C^*_{\mathrm{inv},\underline{1}}\mathbb{G}$ lies in a short exact sequence 
			\begin{equation*}\begin{aligned}
																0\to \mathcal{K}(L^2(M))\otimes C_0((\Rpt)^\nu) \to C^*_{\mathrm{inv},\underline{1}}\mathbb{G}\to  C^*_{\mathrm{inv},\underline{1}}\cG\to 0,				
			\end{aligned}\end{equation*}
		where $C^*_{\mathrm{inv},\underline{1}}\cG$ is the two-sided ideal of $C^*_{\mathrm{inv}}\cG$ whose spectrum is $\HatHLnon$.
		Taking the crossed product by $(\Rpt)^\nu$, we have 
		\begin{equation*}\begin{aligned}
									0\to \mathcal{K}(L^2(M))\otimes \Big(C_0((\Rpt)^\nu)\rtimes (\Rpt)^\nu\Big) \to C^*_{\mathrm{inv},\underline{1}}\mathbb{G}\rtimes (\Rpt)^\nu\to  C^*_{\mathrm{inv},\underline{1}}\cG\rtimes (\Rpt)^\nu\to 0.		
			\end{aligned}\end{equation*}
		By the definition of crossed product, $C_0((\Rpt)^\nu)\rtimes (\Rpt)^\nu=\cK(L^2(\Rpt)^\nu)$.
		So, it remains to show that $C^*_{\mathrm{inv},\underline{1}}\cG\rtimes (\Rpt)^\nu$ is of Type I and that its spectrum is $\HatHLnon/(\Rpt)^\nu$.
		The $C^*$-algebra $C^*_{\mathrm{inv},\underline{1}}\cG$ is the quotient of $C^*_{\mathrm{nonsing}}(V)\otimes C_0(M)$ whose spectrum is $\HatHLnon$, see the paragraph preceding \MyCref{thm:Corssed-Product_by_action} for the definition of $C^*_{\mathrm{nonsing}}(V)$.
		So, $C^*_{\mathrm{inv},\underline{1}}\cG\rtimes (\Rpt)^\nu$ is also a quotient of $(C^*_{\mathrm{nonsing}}(V)\otimes C_0(M))\rtimes (\Rpt)^\nu$.
		The result follows from \MyCref{thm:Corssed-Product_by_action} and \MyCref{prop:exat_sequence_3_TypeI}.
	\end{proof}
	\begin{proof}[Proof of \Cref{lem:main_representation_theorem_morita_1_lemma_2}.]
		By \eqref{eqn:sigma_0_bounded_by_norm}, it suffices to show that 
			\begin{equation*}\begin{aligned}
				\overline{\{I_0(f\ast g):f,g\in \cinvf{\underline{1}}\}}=\overline{\{I_0(f):f\in \cinvf{\underline{1}}\}},		
			\end{aligned}\end{equation*}
		where the closure is taken in $\cM(C^*_{\mathrm{inv}}\mathbb{G})$.
		It suffices to show that if $f,g\in \cinvf{\underline{1}}$, then
			\begin{equation*}\begin{aligned}
				I_0(f)\ast I_0(g)\in 		\overline{\{I_0(f'\ast g'):f',g'\in \cinvf{\underline{1}}\}}	.		
			\end{aligned}\end{equation*}
		By writing $f$ as in \eqref{eqn:dfn_algebra_vanishing_Fourier_invariant} and $g$ as in \eqref{eqn:invariant_vanish_algebra:alternate}, and using \eqref{eqn:identities_I_Use_a_alot_vector_field_out_2}, it suffices to show the following:
		If $f,g\in \cinv$ and $D\in \DO^k(M)$ and $D'\in \DO^l{M}$ for some $k,l\in \N^\nu$, then 
		\begin{equation*}\begin{aligned}
				\theta_{k}(D)\ast I_{-k}(f)\ast  	I_{-l}(g)\ast \theta_{l}(D')\in\overline{\mathrm{span}(\{I_0(f'\ast g'):f',g'\in \cinvf{\underline{1}}\})}\subseteq \cM(C^*_{\mathrm{inv}}\mathbb{G}).
			\end{aligned}\end{equation*}
		Following the proof of \Crefitem{thm:main_prop_bi_graded_pseudo_diff}{3}, see \eqref{eqn:qimskjfiqjsdmlfjlkmqsdmfjqsdjflmqjsmdjf2}, it suffices to show that for any $S\subseteq \bb{1,\nu}$,
			\begin{equation}\label{eqn:sqdklmjfjqsdljfmkqlsjdmfjqlmsjdfq}\begin{aligned}
				\theta_k(D)\ast I_{-k-l}\left(\int_{]0,1]^\nu} \alpha_{\lambda_S}(f)\ast \alpha_{\lambda_{S^c}}(g) \frac{\odif{\lambda}}{\lambda^{\tau_S}}\right)\ast \theta_{l}(D')  \in \overline{\{I_0(f'\ast g'):f',g'\in \cinvf{\underline{1}}\}},
			\end{aligned}\end{equation}
		where $\lambda_S$ and $\lambda_{S^c}$ and $\tau_S$ are defined in \eqref{eqn:qimskjfiqjsdmlfjlkmqsdmfjqsdjflmqjsmdjf}.
		Furthermore, the integral converges in $\cinv$ by \Cref{thm:estimate_weakly_commut_convolution}.
		Consider the map 
			\begin{equation}\label{eqn:qsjkdofjklqsjdkfjqsmkjdfqsdfq}\begin{aligned}
				\cinv\to  		\cM(C^*_{\mathrm{inv}}\mathbb{G}),\ h\mapsto \theta_k(D)\ast I_{-k-l}(h)\ast  \theta_{l}(D')=I_0(\theta_k(D)\ast h\ast \theta_l(D')).
			\end{aligned}\end{equation}
		By \eqref{eqn:bound_on_u} and the proof of \Cref{lem:Cotlar_Stein_lemma}, one deduces that \eqref{eqn:qsjkdofjklqsjdkfjqsmkjdfqsdfq} is continuous, where $\cinv$ is equipped with its LF topology, see \Cref{dfn:space_of_invariant_functions}. 
		(Here, one needs to use \Cref{rem:continuity_dis} to the maps in \Crefitem{thm:compatability_cinvdiv_cinv}{inclusion_vectorfield_cinvdiv}).
		By writing the integral in \eqref{eqn:sqdklmjfjqsdljfmkqlsjdmfjqlmsjdfq} as the limit in $\cinv$ of a Riemannian sum, \Cref{lem:main_representation_theorem_morita_1_lemma_2} follows.
		\end{proof}

		An immediate corollary of \Cref{thm:main_representation_theorem_morita_2} and \Cref{thm:main_representation_theorem_morita_1} is the following:
		\begin{cor}\label{thm:short_exact_seq_psi0}
			Let $C^*\HL{}$ be the quotient of $\overline{\Psi^0_c(M)}$ whose spectrum is $\HatHLnon/(\Rpt)^\nu$. Then, one has a short exact sequence 
				\begin{equation}\label{eqn:short_exact_seq_psi0}\begin{aligned}
					0\to \cK(L^2(M))\to \overline{\Psi^0_c(M)}\xrightarrow{\sigma^0} C^*\HL{}\to 0
				\end{aligned}\end{equation}
			In particular, for any $v\in \Psi^0_c(M)$, if $\sigma^0(v,\pi,x):L^2(\pi)\to L^2(\pi)$ vanishes for all $(\pi,x)\in \HatHLnon$, then $v\ast \cdot:L^2(M)\to L^2(M)$ is a compact operator.
		\end{cor}

	We finally prove that there exists operators as in the third hypothesis of \Cref{thm:simple_arg_sobolev}. Let $K$ be as in \Cref{thm:simple_arg_sobolev}.
		\begin{prop}\label{lem:vk}
						Suppose $M$ is compact. There exists a family $v_k\in \Psi^k(M)$ 
						for $k\in K$ such that
						\begin{enumerate}
							\item The operator $v_0$ is invertible in $\overline{\Psi^0(M)}$,
							\item For any $v\in \Psi^0(M)$, the map 
                    \begin{equation}\label{eqn:continuity_hypothesis_conc}\begin{aligned}
                        k\in K \mapsto v_{k}\ast v\ast v_{-k}\in \overline{\Psi^0(M)}
                    \end{aligned}\end{equation}
				is continuous.
						\end{enumerate}
		\end{prop}
		\begin{proof}
			We start with the following lemma.
			\begin{lem}\label{lem:v_injective_on_distributions}
				There exists $g\in C^\infty(M\times M,\omegahalf)$ which is positive (as an operator $L^2(M)\to L^2(M)$) such that the operator $g\ast \cdot:C^{-\infty}(M,\omegahalf)\to C^{\infty}(M,\omegahalf)$ is injective.		
			\end{lem}
			\begin{proof}
					Let $(f_n)_{n\in \mathbb{N}}$ a dense sequence in $C^\infty(M,\omegahalf)$. We take
						\begin{equation*}\begin{aligned}
							g(y,x)=\sum_{n=0}^{+\infty}a_nf_n(y)\overline{f_n(x)},		
						\end{aligned}\end{equation*}
					where $a_n$ are strictly positive constants which converge to $0$ sufficiently fast so that the series converges in $C^\infty(M\times M,\omegahalf)$.
			\end{proof}
			 There exists $f\in \cinvf{\underline{1}}$ which is positive as an element of $C^*\mathbb{G}$ such that $\pi_{\mathrm{inv}}(f):L^2(\pi)\to L^2(\pi)$ is injective for any $(\pi,x)\in \HatHLnon$.
				 The proof of this is a straightforward adaptation of the proof of \cite[Lemma 2.17.2]{MohsenAbstractMaxHypo}, and is left to the reader.

			We define $v_k$ by 
				\begin{equation*}\begin{aligned}
					v_k=I_k(f)_{|\underline{1}}+g		
				\end{aligned}\end{equation*}
			Let $k=0$ and $h\in L^2(M)$, $\xi\in L^2(\pi)$. The inner products $\langle v_0\ast h,h\rangle_{L^2(M)} $ and  $\langle \sigma^0(v_0,\pi,x)\xi,\xi\rangle_{L^2(\pi)} $ are well-defined.
			So we get injectivity of $v_0\ast \cdot:L^2(M)\to L^2(M)$ from injectivity of $g$ on $L^2(M)$, and injectivity of  $\sigma^0(v_0,\pi,x)=\int_{(\Rpt)^\nu}\alpha_\lambda(f)\frac{\odif{\lambda}}{\lambda^{\underline{1}}}$ on $L^2(\pi)$ from injectivity of $\pi_{\mathrm{inv}}(f)$.
			We deduce that $v_0$ is left-invertible \Cref{prop:simple_arg}.
			Since $v_0^*=v_0$ it is invertible in $ \overline{\Psi^0(M)}$.
				
		Continuity of \eqref{eqn:continuity_hypothesis_conc} follows from \eqref{eqn:bound_on_u}, and the proof of \Cref{lem:Cotlar_Stein_lemma}, and \Crefitem{thm:main_prop_bi_graded_pseudo_diff}{3}.
		\end{proof}
		By all the above, we can apply \Cref{thm:simple_arg_sobolev}.
		We suppose that $M$ is compact so that the identity operator is in $\Psi^0(M)$ (which follows from \Crefitem{thm:main_prop_bi_graded_pseudo_diff_M}{2}).
		The identity operator is needed so that $\overline{\Psi^0(M)}$ is a unital $C^*$-algebra which is a hypothesis of \Cref{thm:simple_arg_sobolev}, see \MyCref{dfn:CstarCalculus}.
		\begin{theorem}\label{thm:conclusion}
			Suppose $M$ is compact. Then, 
			\begin{enumerate}
				\item  If $k\in \R^\nu$, and $(\pi,x)\in \HatHLnon$, the spaces $H^k(M)$ and $H^k(\pi)$ are Hilbertian spaces, and 
					\begin{equation*}\begin{aligned}
						&C^{-\infty}(M,\omegahalf)=\bigcup_{k\in \R^\nu}H^k(M), &&C^{-\infty}(\pi)=\bigcup_{k\in \R^\nu}H^k(\pi)&&&\\
						&C^{\infty}(M,\omegahalf)=\bigcap_{k\in \R^\nu}H^k(M),&&C^{\infty}(\pi)=\bigcap_{k\in \R^\nu}H^k(\pi)&&&\\
												&H^0(M)=L^2(M)&& H^0(\pi)=L^2(\pi)&&&\\
						&H^{l}(M)\subseteq H^k(M),&&H^{l}(\pi)\subseteq H^k(\pi)&&&\text{ if }k\preceq l
					\end{aligned}\end{equation*}
					
				\item\label{thm:conclusion:2} For any $k\in \C^\nu$, $v\in \Psi^k(M)$, if the maps
							\begin{equation*}\begin{aligned}
							v\ast \cdot&:C^{\infty}(M,\omegahalf)\;\xrightarrow{\phantom{aaaaaaaa}}\;C^{\infty}(M,\omegahalf),\\
							v^*\ast \cdot&:C^{\infty}(M,\omegahalf)\;\xrightarrow{\phantom{aaaaaaaa}}\;C^{\infty}(M,\omegahalf),\\
							\sigma^k(v,\pi,x)&:C^{\infty}(\pi)\;\xrightarrow{\phantom{aaaaaaaaaaaaa}}\;C^{\infty}(\pi),\quad\forall (\pi,x)\in \HatHLnon\\
							\sigma^{\bar{k}}(v^*,\pi,x)&:C^{\infty}(\pi)\;\xrightarrow{\phantom{aaaaaaaaaaaaa}}\;C^{\infty}(\pi),\quad\forall (\pi,x)\in \HatHLnon
\end{aligned}\end{equation*}
							are injective, then for all $(\pi,x)\in \HatHLnon$,
							\begin{equation}\label{eqn:skqdjfmsqmjdfjsmjmfqsfd2}\begin{aligned}
								v\ast \cdot &:C^{-\infty}(M,\omegahalf)\xrightarrow{\phantom{aaaa}} C^{-\infty}(M,\omegahalf)&&\\
								v\ast \cdot &:H^l(M)\xrightarrow{\phantom{aaaaaaaaaa}} H^{l-\Re(k)}(M)&&\forall l\in \R^\nu\\
								v\ast \cdot &:C^{\infty}(M,\omegahalf)\xrightarrow{\phantom{aaaaa}} C^{\infty}(M,\omegahalf)&&\\
								\sigma^k(v,\pi,x)&:C^{-\infty}(\pi)\xrightarrow{\phantom{aaaaaaaaa}} C^{-\infty}(\pi)&&\\
								\sigma^k(v,\pi,x)&:H^l(\pi)\xrightarrow{\phantom{aaaaaaaaaaa}} H^{l-\Re(k)}(\pi)&&\forall l\in \R^\nu\\
								\sigma^k(v,\pi,x)&:C^\infty(\pi)\xrightarrow{\phantom{aaaaaaaaaa}} C^\infty(\pi)&&
                                \end{aligned}\end{equation}
					are topological isomorphism. 
					\item\label{thm:conclusion:3} For any $k\in \C^\nu$, there exists $v_k\in \Psi^k(M)$ such that \eqref{eqn:skqdjfmsqmjdfjsmjmfqsfd2} are topological isomorphisms
			\end{enumerate}
		\end{theorem}
		\begin{proof}
			We apply \Cref{thm:simple_arg_sobolev} by completing $\Psi^0(M)$ using the $C^*$-norm $\norm{\cdot}_{\cL(L^2(M))}$ to obtain $\overline{\Psi^0(M)}$.
		\end{proof}
	\section{Proofs of the results in Section \ref{sec:sobolev_spaces}}\label{sec:sobolev_spaces_proof}

\begin{linkproof}{bounded_M_Sobolev} 
				We will first prove \Cref{thm:bounded_M_Sobolev} under the assumption that $M$ is compact.
				So, for the moment, we suppose that $M$ is compact.

				\Crefitem{thm:bounded_M_Sobolev}{1}, \Crefitem{thm:bounded_M_Sobolev}{2}, \Crefitem{thm:bounded_M_Sobolev}{3} and \Crefitem{thm:bounded_M_Sobolev}{6} follow from \Cref{thm:conclusion}. 
				\Crefitem{thm:bounded_M_Sobolev}{5} follows immediately from \MyCref{thm:conclusion}[3] and \Crefitem{thm:main_prop_bi_graded_pseudo_diff_M}{5}.
				\Crefitem{thm:bounded_M_Sobolev}{4}  follows immediately from \MyCref{thm:conclusion}[3] and \MyCref{lem:operators_of_order_0_on_M_are_bounded}[4].

				We now no longer suppose that $M$ is compact.
					The Sobolev spaces are local.
				More precisely, if $w\in C^{-\infty}(M,\omegahalf)$, then $w\in H^k_{\mathrm{loc}}(M)$ if and only if for any $x\in M$, there exists an open neighbourhood $U$ of $x$ such that $w_{|U}\in H^k_{\mathrm{loc}}(U)$, where $H^k_{\mathrm{loc}}(U)$ is defined using
				$\cF_{|U}^\bullet$, see \eqref{eqn:compatability_pseudo_restriction_open}.
				\begin{lem}\label{lem:compactification}
			For every $x\in M$, there exists a weakly commutative weighted sub-Riemannian structure $\cGF^\bullet$ of depth $N$ on the $\dim(M)$-dimensional sphere $S^{\dim(M)}$, and
			there exists $U\subseteq M$ an open neighborhood of $x$,
			$U'\subseteq S^{\dim(M)}$ an open subset, and a diffeomorphism $\phi:U\to U'$ such that $\cF^k_{|U}=\phi^*(\cGF_{|U'})$.
		\end{lem}
		\begin{proof}
			Let $\phi:U'\to V'$ be any diffeomorphism from an open neighborhood $U'$ of $x$ to an open subset $V'$ of $S^{\dim(M)}$.
			Let $\chi\in C^\infty_c(V')$ be a smooth function which is equal to $1$ in a neighborhood $V$ of $\phi(x)$.
			We take $U=\phi^{-1}(V)$.
			For each $k\in \Z_+^\nu$, if $k_i<N_i$ for all $i$, then we define $\cGF^{k}$ to be the submodule of $\cX(S^{\dim(M)})$ generated by $\chi \phi_*(X)$ for $X\in \cF^k$.
			Otherwise, we take $\cGF^N=\cX(S^{\dim(M)})$.
			It is clear that the family $\cGF^\bullet$ consists of finitely generated modules.
			Furthermore, by the identity,
			\begin{equation}\label{eqn:}\begin{aligned}
					[\chi\phi_*(X),\chi\phi_*(Y)]=\chi^2\phi_*([X;Y])+ \phi_*(X)(\chi)\chi\phi_*(Y)-\phi_*(Y)(\chi)\chi\phi_*(X)
				\end{aligned}\end{equation}
			it follows that $\cGF^\bullet$ satisfies \eqref{eqn:Liebracket_cFi} for $i,j\in \N$ such that $i+j<N$.
			Since, \eqref{eqn:Liebracket_cFi} for $i,j\in \N$ such that $i+j=N$ is trivially satisfied, we deduce that $\cGF^\bullet$ is a $\nu$-graded sub-Riemannian structure on $S^{\dim(M)}$.
			Finally, since $\chi=1$ on $V$, $\phi^*(\cGF^{\bullet}_{|V})=\cF^\bullet_{|U}$.
		\end{proof}
				Using \Cref{lem:compactification}, to prove \Cref{thm:bounded_M_Sobolev}, we can suppose that $M$ is an open subset of $S^{\dim(M)}$.
				So, the theorem follows from the compact case.
				This finishes the proof of \Cref{thm:bounded_M_Sobolev}.
		\end{linkproof}
	
		\begin{linkproof}{bounded_pi_Sobolev}
					Since the statement is local at $x$, by \Cref{lem:compactification}, we can suppose that $M$ is compact.
					The argument is identical to \MyCref{thm:bounded_M_Sobolev}.
		\end{linkproof}
		\begin{linkproof}{vanishing_of_principal_symbol_Total}
			By \Crefitem{thm:conclusion}{3}, we can suppose that $k=l=0$.
			The result now follows from \Cref{thm:short_exact_seq_psi0}, and density of $C^\infty(\pi)$ in $L^2(\pi)$ for the $L^2(\pi)$-topology.	
		\end{linkproof}
		\begin{linkproof}{principal_symbol_characterizations}
				Since the statement is local at $x$, by \Cref{lem:compactification}, we can suppose that $M$ is compact.
				By \Crefitem{thm:conclusion}{3}, we can suppose that $k=l=0$.
				
				The set $S/(\Rpt)^\nu$ seen as a subset of $\HatHLnon/(\Rpt)^\nu$ is equal to the set $\overline{\{(\pi,x)\}}\backslash \{(\pi,x)\}$ where the closure of $\{(\pi,x)\}$ is taken in $\HatHLnon/(\Rpt)^\nu$.
				The set $\overline{\{(\pi,x)\}}$ is a closed subset of $\HatHLnon/(\Rpt)^\nu$.
				By the computation of the spectrum of $\overline{\Psi^0(M)}$ in \Cref{thm:main_representation_theorem_morita_1} and \Cref{thm:main_representation_theorem_morita_2}, we can define a quotient $\overline{\Psi^0(M)}$ whose spectrum is $\overline{\{(\pi,x)\}}$.
				Let us denote this quotient by $A$ which is of Type I by \Cref{prop:exat_sequence_3_TypeI}.
				By \cite[Proposition 3.3.2]{Dixmier}, $A$ is the Hausdorff completion of $\Psi^0(M)$ by the semi-norm $v\mapsto \norm{\sigma^0(v,\pi,x)}_{\cL(L^2(\pi))}$.
				Now, let $B$ be the Hausdorff completion $\Psi^0(M)$ by the semi-norm  $v\mapsto \norm{\sigma^0(v,\pi,x)}_{\cL(L^2(\pi))/\cK(L^2(M))}$.
				Clearly, $B$ is a quotient of $A$. 
				In fact, since in the spectrum of $A$, the point $(\pi,x)$ is dense, \MyCref{prop:liminal_mapped_to_compacts} (rather its proof) implies that the spectrum of $B$ is precisely the set $S$.
				
				By the computation of the spectrum of $B$, we deduce that if $v\in \Psi^0(M)$, then $v$ vanishes in $B$, i.e., $L^2(\pi)\xrightarrow{\sigma^0(v,\pi,x)} L^2(\pi)$ is compact if and only if for all $(\pi',x)\in S$, $L^2(\pi')\xrightarrow{\sigma^0(v,\pi',x)} L^2(\pi')$ vanishes, which by density of $C^\infty(\pi)$ is equivalent to vanishing of
				$C^\infty(\pi')\xrightarrow{\sigma^0(v,\pi',x)} C^\infty(\pi')$.
				This finishes the proof of \Crefitem{thm:principal_symbol_characterizations}{V}.
					
				To obtain \MyCref{thm:conclusion}, we used \MyCref{thm:simple_arg_sobolev}, where we completed $\Psi^0(M)$ by the norm $v\mapsto \norm{v}_{\cL(L^2(M))}$.
				Here, instead, we will apply \MyCref{thm:simple_arg_sobolev} but completing $\Psi^0(M)$ by the semi-norm $v\mapsto \norm{\sigma^0(v,\pi,x)}_{\cL(L^2(\pi))/\cK(L^2(M))}$.
				By \MyCref{thm:simple_arg_sobolev}[inv] and the computation of the spectrum of $B$ above, we obtain the following: 
				If $v\in \Psi^0(M)$, then $v$ is left-invertible in $B$ if and only if $C^\infty(\pi')\xrightarrow{\sigma^0(v,\pi',x)} C^\infty(\pi')$ is injective for all $(\pi',x)\in S$.
				From this, we obtain the equivalence between \MyCref{thm:principal_symbol_characterizations}[I:a] and \MyCref{thm:principal_symbol_characterizations}[I:c].
				The equivalence between \MyCref{thm:principal_symbol_characterizations}[I:a] and \MyCref{thm:principal_symbol_characterizations}[I:b] is the statement of \MyCref{thm:simple_arg_sobolev}[3].
				
				Finally, we need to show that if $H^{\Re(k)+l}(\pi) \xrightarrow{  \sigma^k(v,\pi,x)} H^l(\pi)$ is Fredholm, then its kernel consists of only smooth vectors (and so it doesn't depend on $l$).
				This will again be an application of \MyCref{thm:simple_arg_sobolev}. This time, we complete $\Psi^0(M)$ by the semi-norm $v\mapsto \norm{\sigma^0(v,\pi,x)}_{\cL(L^2(\pi))}$.
				By the computation of the spectrum of $A$ above,
				\MyCref{thm:simple_arg_sobolev}[inv] implies that $v$ is left-invertible in $A$ if and only if $C^\infty(\pi)\xrightarrow{\sigma^0(v,\pi',x)} C^\infty(\pi)$ is injective for all $(\pi',x)\in S$, and $C^\infty(\pi')\xrightarrow{\sigma^0(v,\pi,x)} C^\infty(\pi')$ is injective.
				In other words, if $H^{\Re(k)+l}(\pi) \xrightarrow{  \sigma^k(v,\pi,x)} H^l(\pi)$ is Fredholm, then it is left-invertible if and only if it is injective on $C^\infty(\pi)$.
				This finishes the proof of \MyCref{thm:principal_symbol_characterizations}[I].
		\end{linkproof}
\section{Proofs of the results in Section \ref{sec:extension_of_parameters}}\label{sec:extension_of_parameters_proof}
		
	\begin{linkproof}{extending_parameters_pseudo_diff}[principal_symbol_extension]

				For technical reasons, we will first prove the results in \Cref{sec:extension_of_parameters} and then prove the results in \Cref{sec:main_theorem}.
				Let $\tilde{\mathbb{G}}\rightrightarrows \tilde{\mathbb{G}}^{(0)}$ and $\dbtilde{\mathbb{G}}\rightrightarrows \dbtilde{\mathbb{G}}^{(0)}$ be the tangent groupoids associated to $\tilde{\cF}$ and $\dbtilde{\cF}$, $\Phi$ the map as in \Cref{thm:extension_parameters}.
				Let $k\in \C^{\tilde{\nu}}$ and $l\in \C^{\dbtilde{\nu}}$.
				If $f\in C^\infty_{c,\mathrm{inv},k+\underline{1}}(\tilde{\mathbb{G}},\omegahalf)$ and $g\in C^\infty_{c,\mathrm{inv},l+\underline{1}}(\dbtilde{\mathbb{G}},\omegahalf)$, then $\Phi(f,g)\in \cinvf{(k,l)+\underline{1}}$, and we have 
					\begin{equation*}\begin{aligned}
						I_{(k,l)}(\Phi(f,g))=\int_{]0,1]^{\tilde{\nu}}\times ]0,1]^{\dbtilde{\nu}}}\alpha_{(\lambda,\mu)}(\Phi(f,g))\frac{\odif{\lambda} \odif{\mu}}{\lambda^{k+\underline{1}}\mu^{l+\underline{1}}}=\int_{]0,1]^{\tilde{\nu}}\times ]0,1]^{\dbtilde{\nu}}}\Phi(\alpha_{\lambda}(f),\alpha_\mu(g))\frac{\odif{\lambda} \odif{\mu}}{\lambda^{k+\underline{1}}\mu^{l+\underline{1}}}				
					\end{aligned}\end{equation*}
				So, 
					\begin{equation}\label{eqn:qksjdkofjqkosdjfpoqfsd}\begin{aligned}
						I_{(k,l)}(\Phi(f,g))_{|\underline{1}}&=\int_{]0,1]^{\tilde{\nu}}\times ]0,1]^{\dbtilde{\nu}}}\Phi(\alpha_{\lambda}(f),\alpha_\mu(g))_{|\underline{1}}\frac{\odif{\lambda} \odif{\mu}}{\lambda^{k+\underline{1}}\mu^{l+\underline{1}}}\\
					&=\int_{]0,1]^{\tilde{\nu}}\times ]0,1]^{\dbtilde{\nu}}}\alpha_{\lambda}(f)_{|\underline{1}}\ast \alpha_\mu(g)_{|\underline{1}}\frac{\odif{\lambda} \odif{\mu}}{\lambda^{k+\underline{1}}\mu^{l+\underline{1}}}=I_k(f)_{|\underline{1}}\ast I_l(g)_{|\underline{1}}.
					\end{aligned}\end{equation}
				We have thus proved that $\tilde{\Psi}^k_c(M)\ast \dbtilde{\Psi}^l_c(M)\subseteq \Psi^{(k,l)}_c(M)$.
				
				We now prove \eqref{eqn:commutator_lower_order_extension}.
				In fact this immediately from \eqref{eqn:commutator_phi}, see the proof of \MyCref{thm:main_prop_bi_graded_pseudo_diff_M}[3]. 
				The proof of \Crefitem{thm:extending_parameters_pseudo_diff}{inclusion} is thus complete.

				The proof of \Crefitem{thm:extending_parameters_pseudo_diff}{osculating} and \Crefitem{thm:extending_parameters_pseudo_diff}{Helffer_Nourrigat} is straightforward and left to the reader.
		
				We now prove \Cref{thm:principal_symbol_extension}.
			Let $k,l,f,g$ as in \eqref{eqn:qksjdkofjqkosdjfpoqfsd}
			By \eqref{eqn:qksjdkofjqkosdjfpoqfsd}, We have 
				\begin{equation*}\begin{aligned}
				\sigma^{(k,l)}(I_k(f)_{|\underline{1}}\ast I_l(g)_{|\underline{1}},\pi,x)=&\lim_{\lambda_1\to +\infty,\cdots,\lambda_{\tilde{\nu}}\to +\infty,\mu_1\to +\infty,\cdots,\mu_{\dbtilde{\nu}}\to +\infty}\lambda^{-k}\mu^{-l}\pi_{\mathrm{inv}}(\alpha_{(\lambda,\mu)}(\Phi(f,g)))\\
					&=\lim_{\lambda_1\to +\infty,\cdots,\lambda_{\tilde{\nu}}\to +\infty,\mu_1\to +\infty,\cdots,\mu_{\dbtilde{\nu}}\to +\infty}\lambda^{-k}\mu^{-l}\tilde{\pi}_{\mathrm{inv}}(\alpha_{\lambda}(f))\dbtilde{\pi}_{\mathrm{inv}}(\alpha_{\mu}(g))\\
					&=\sigma^k(I_k(f),\tilde{\pi},x)\otimes \sigma^{l}(I_l(f),\dbtilde{\pi},x).
				\end{aligned}\end{equation*}
			From this, \eqref{eqn:qkosdjoqskjdkoqsjdkqsjdkqs} easily follows. 
			By \Cref{thm:conclusion}, \eqref{eqn:Sobolev_spaces_extension_representation} follows as well. This finishes the proof of \Cref{thm:principal_symbol_extension}.
			Finally, \Crefitem{thm:extending_parameters_pseudo_diff}{Sobolev} follows from \Cref{thm:principal_symbol_extension} and \Cref{thm:conclusion}.
				\end{linkproof}

\section{Proofs of the results in Section \ref{sec:main_theorem}}\label{sec:main_theorem_proof}
We will first need to modify the sub-Riemannian structure.
	We will replace $\cF$ by $\tilde{\cF}$ defined in \eqref{eqn:extension_structure}.
	\Cref{thm:extending_parameters_pseudo_diff} ensures that the calculus associated to $\cF$ is contained in the calculus associated to $\tilde{\cF}$, and that the Sobolev spaces don't change (for the appropriate indices).
	\Cref{thm:principal_symbol_extension} together with \eqref{eqn:Helffer-Nourrigat_extension_eqn} ensure that we don't change the principal symbol.
	So, it suffices to prove the results of this section for the calculus $\tilde{\cF}$.
	To summarize, we will suppose that $\nu \geq 2$, and that $N_{(0,\cdots,0,1)}=\cX(M)$.
	This implies the following:
	\begin{enumerate}
		\item For any $x\in M$, the union \eqref{eqn:union_hellfer_microlocal} is a disjoint union once one writes the union as a union of the co-sphere bundle. So, there is a projection map
			\begin{equation}\label{eqn:jqksldjfojqsdmfjmlqsjdlfjqsdf}\begin{aligned}
				\HatHLnon\to (T^*M\backslash 0)/\Rpt,	
			\end{aligned}\end{equation}
			which is continuous.
		\item     If $\chi\in \Psi^k_{\mathrm{classical}}(M)$ with $k\in \C$, then, by \Cref{thm:classical_pseudo}, $\chi\in \Psi^{(0,\cdots,0,k)}(M)$ and
		\begin{equation}\label{eqn:symbol_classical_extension}\begin{aligned}
			\sigma^{(0,\cdots,0,k)}(\chi,\pi,x)=\sigma^k_{\mathrm{classical}}(\chi,\xi,x),\quad \forall (\xi,x)\in T^*M\backslash 0,\pi\in \HatHLnon_{(\xi,x)}.
		\end{aligned}\end{equation}
		In particular, by \Cref{lem:commutativity} and \Crefitem{thm:bounded_M_Sobolev}{6}, $\chi\ast H^l_{\mathrm{loc}}(M)\subseteq H^{l-(0,\cdots,0,k)}_{\mathrm{loc}}(M)$ for all $l\in \R^\nu$.
		This will be used repeatedly in this section.
	\end{enumerate}
	We also recall the standard terminology, we say that $\chi\in \Psi^k_{\mathrm{classical}}(M)$ is smoothing on a cone $\Gamma\subseteq T^*M\backslash 0$, if $\Gamma\subseteq \{(\xi,x):(\xi,-\xi;x,x)\notin \WF(\chi)\}$.
	This implies that $\WF(\chi\ast w)\cap \Gamma=\emptyset$ for all $w\in C^\infty(M,\omegahalf)$.
	
	It is straightforward to show that given any closed cones $\Gamma_1,\Gamma_2$ such that $\Gamma_1\cap\Gamma_2=\emptyset$, there exists $\chi\in \Psi^0_{\mathrm{classical}}(M)$ such that $\chi$ is smoothing on $\Gamma_1$, and $\chi$ is elliptic on $\Gamma_2$.

	We now generalise \Crefitem{thm:main_prop_bi_graded_pseudo_diff_M}{6} and \eqref{eqn:commutator_lower_order_extension}.
	\begin{lem}\label{lem:commutativity}
		 If $\chi\in \Psi^k_{c,\mathrm{classical}}(M)$ and $v\in \Psi^{l}_c(M)$ where $k\in \C$ and $l\in \C^\nu$, then $[\chi,v]\in \Psi^{\prec l+(0,\cdots,0,k)}_c(M)$.
	\end{lem}
	\begin{proof}
		The proof of this is very similar to the proof of \eqref{eqn:commutator_lower_order_extension}, and will be left to the reader.
	\end{proof}

\begin{linkproof}{multi_parameter_wave_front_set}
		\MyCref{thm:multi_parameter_wave_front_set}[7] follows immediately from \MyCref{thm:bounded_M_Sobolev}[1]

		Recall the following classical characterization of the wave-front set.
		\begin{theorem}[{\cite[Theorem 18.1.27]{HormanderIII}}]\label{thm:Hormander_WF}
			Let $w\in C^{-\infty}(M,\omegahalf)$. For any $(\xi,x)\in T^*M\backslash 0$, $(\xi,x)\notin \WF(w)$ if and only if there exists $\chi\in \Psi^0_{\mathrm{classical}}(M)$ such that $\chi$ is elliptic at $(\xi,x)$ and $\chi\ast w\in C^\infty(M,\omegahalf)$.
		\end{theorem}
			The implication $\implies$ in \MyCref{thm:multi_parameter_wave_front_set}[5] follows immediately from \MyCref{thm:Hormander_WF} and the fact that classical pseudo-differential operators preserve the Sobolev spaces $H^l_{\mathrm{loc}}(M)$ for all $l\in \R^\nu$.
			For, the implication $\impliedby$ in \MyCref{thm:multi_parameter_wave_front_set}[5], we can find $\chi'\in \Psi^0_{\mathrm{classical}}(M)$ such that $\chi'\ast \chi-\delta_1$ is smoothing in a conic neighborhood $(\xi,x)$.
			So, 
				\begin{equation*}\begin{aligned}
					w=\chi'\ast\chi\ast w+(\delta_1-\chi'\ast \chi)\ast w		
				\end{aligned}\end{equation*}
			One has $\chi'\ast\chi\ast w\in H^{l}_{\mathrm{loc}}(M)$, and $(\xi,x)\notin \WF((\delta_1-\chi'\ast \chi)\ast w)$.

			Let us prove \MyCref{thm:multi_parameter_wave_front_set}[4]. 
			If $\sigma^0_{\mathrm{classical}}(\chi,\xi,x)\neq 0$, then $\sigma^0_{\mathrm{classical}}(\chi,\eta,y)\neq 0$ on an open neighborhood of $(\xi,x)$.
			So, since the complement of the space $\WF_l(w)$ is Lindelöf, there exists a sequence $\chi_n\in \Psi^0_{\mathrm{classical}}(M)$ such that for any $n$, $\chi_n\ast w\in H^l_{\mathrm{loc}}(M)$, and for any $(\xi,x)\notin \WF_l(w)$, there exists $n$ such that $\sigma^0_{\mathrm{classical}}(\chi_n,\xi,x)\neq 0$.
			Since classical pseudo-differential operators preserve Sobolev spaces, $\chi^*_n\ast \chi_n\ast w\in H^l_{\mathrm{loc}}(M)$.
			Let $c_n>0$ be a sequence that converges fast enough to $0$ that the sum $\chi=\sum_n c_n\chi_n^*\ast\chi_n$ converges in the space $\Psi^0_{\mathrm{classical}}(M)$, and that $\sum_{n}c_n \chi_n^*\ast\chi_n\ast w$ converges in the Fréchet space $H^l_{\mathrm{loc}}(M)$. 
			(Here, one might have to modify the supports of $\chi_n$ by adding a smoothing operator so that $\chi_n$ are uniformly properly supported, which implies that $\chi$ is properly supported).
			One has $\chi\ast w\in H^l_{\mathrm{loc}}(M)$ and $\sigma^0_{\mathrm{classical}}(\chi,\xi,x)\neq 0$ for all $(\xi,x)\notin \WF_l(w)$.

			\MyCref{thm:multi_parameter_wave_front_set}[1] follows immediately from \MyCref{thm:multi_parameter_wave_front_set}[4].

			\MyCref{thm:multi_parameter_wave_front_set}[3] follows immediately from \eqref{eqn:WF_multi_pseudo} and \MyCref{thm:bounded_M_Sobolev}[6].
			
			Let us prove that  $\WF(w)=\overline{\bigcup_{l\in \R^\nu}\WF_l(w)}$. The inclusion $\supseteq$ is obvious from \MyCref{thm:Hormander_WF} and \MyCref{thm:multi_parameter_wave_front_set}[5].
			For inclusion $\subseteq$, let $(\xi,x)\notin \overline{\bigcup_{l\in \R^\nu}\WF_l(w)}$.
			We can find a $\chi\in \Psi^0_{\mathrm{classical}}(M)$ such that $\chi$ is elliptic at $(\xi,x)$, and $\chi$ is smoothing on a conic open neighborhood of $\overline{\bigcup_{l\in \R^\nu}\WF_l(w)}$.
			
			By  \MyCref{thm:multi_parameter_wave_front_set}[4], 
			for each $l\in \R^\nu$, we can find $\chi_l\in \Psi^0_{\mathrm{classical}}(M)$ such that for any $l$, $\chi_l\ast w\in H^l_{\mathrm{loc}}(M)$, and $\chi_l$ is elliptic on the complement of $\overline{\bigcup_{l\in \R^\nu}\WF_l(w)}$.
			We can divide $\chi$ by $\chi_l$, i.e., there exists $\chi_l'\in \Psi^0_{\mathrm{classical}}(M)$ such that $\chi=\chi_l'\ast \chi_l$ modulo smoothing operators.
			Hence, $\chi\ast w\in \bigcap_l H^l_{\mathrm{loc}}(M)=C^\infty(M,\omegahalf)$. So, by \MyCref{thm:Hormander_WF}, $(\xi,x)\notin \WF(w)$.
			
			The identity $\WF_l(w+w')\subseteq \WF_l(w)\cup\WF_l(w')$ is straightforward to prove.
			\end{linkproof}

\begin{linkproof}{main_theorem}
	Let us prove \MyCref{thm:main_theorem}[left-inv].	
	Since this is local in $x$, we can suppose that $M$ is compact by \Cref{lem:compactification}.
	By \Crefitem{thm:conclusion}{3}, we can suppose that $v\in \Psi^0(M)$.
		Let $(\xi,x)\in \Gamma$ be fixed.
		We fix $\chi\in \Psi^0_{\mathrm{classical}}(M)$ such that $\sigma^0_{\mathrm{classical}}(\chi,\xi,x)=0$, and $\sigma^0_{\mathrm{classical}}(\chi,\eta,y)\neq 0$ for all $(\eta,y)\in T^*M\backslash \R_+ \cdot (\xi,x)$.
		Let $g$ as in \Cref{lem:v_injective_on_distributions}, $\tilde{v}=v^*\ast v+ \chi^*\ast\chi+g\in \Psi^0(M)$. By \eqref{eqn:symbol_classical_extension},
			\begin{equation*}\begin{aligned}
				\sigma^0(\tilde{v},\pi,y)=&\sigma^0(v,\pi,y)^*\sigma^0(v,\pi,y)	+|\sigma^0_{\mathrm{classical}}(\chi,\eta,y)|^2,\quad \forall\pi\in \HatHLnon_{(\eta,y)}.
			\end{aligned}\end{equation*}
		So, $\tilde{v}$ satisfies the hypothesis of \Crefitem{thm:conclusion}{2}.
		By \Crefitem{thm:conclusion}{2}, \MyCref{thm:main_theorem}[left-inv] follows.

		 We now show that $\Gamma$ is open.
				Since this is local, we can suppose $M$ is compact.
				By \Crefitem{thm:conclusion}{3}, we can suppose that $v\in \Psi^0(M)$.
				By \MyCref{thm:main_theorem}[left-inv],
					\begin{equation*}\begin{aligned}
											\Gamma=\{(\xi,x)\in T^*M\backslash 0:\forall \pi\in \HatHLnon_{(\xi,x)},\ \sigma^{0}(v,\pi,x):L^2(\pi)\to L^2(\pi) \text{ is injective}\}
					\end{aligned}\end{equation*}
				Consider the $C^*$-algebra $\overline{\Psi^0(M)}/\cK(L^2(M))$.
				There is an obvious map $C((T^*M\backslash 0)/\Rpt)\to \overline{\Psi^0(M)}/\cK(L^2(M))$ which maps a smooth function on the cotangent sphere bundle to the corresponding classical pseudo-differential operator modulo compact operators.
				\MyCref{lem:commutativity} implies that the image of $C((T^*M\backslash 0)/\Rpt)$ lies in the center of the $C^*$-algebra $\overline{\Psi^0(M)}/\cK(L^2(M))$.
				So, the $C^*$-algebra $\overline{\Psi^0(M)}/\cK(L^2(M))$ is fibered over $(T^*M\backslash 0)/\Rpt$.
				This fibration induces a fibration on the spectrum which is precisely the fibration given by \eqref{eqn:jqksldjfojqsdmfjmlqsjdlfjqsdf}.
				In a fibered $C^*$-algebra, the set of points on which the fiber is left-invertible is open, see \cite[Proposition 2.4]{BlanchardHopf}.
				By \MyCref{prop:simple_arg}, it follows that $\Gamma$ is open.

	We now prove \MyCref{thm:main_theorem}[parametrix].
	We will first suppose that $M$ is compact.
	By \Crefitem{thm:conclusion}{3}, we can suppose that $v\in \Psi^0(M)$.
	We fix $\chi\in \Psi^0_{\mathrm{classical}}(M)$ such that $\chi$ is elliptic on the complement of $\Gamma$, and $\chi$ is smoothing on $\Gamma'$.
	Let $g\in C^\infty(M\times M,\omegahalf)$ be as in \Cref{lem:v_injective_on_distributions}.
	By replacing $v$ by $v^*\ast v+\chi^*\ast \chi+g$, we can suppose that $v$ is formally self-adjoint, and that $v\ast \cdot:C^\infty(M,\omegahalf)\to C^\infty(M,\omegahalf)$ is injective, and that 
	for all $(\pi,x)\in \HatHLnon$, $\sigma^0(v,\pi,x):C^\infty(\pi)\to C^\infty(\pi)$ is injective.
	By \MyCref{thm:conclusion}[2], $v\ast \cdot :C^{-\infty}(M,\omegahalf)\to C^{-\infty}(M,\omegahalf)$ is a topological isomorphism. 
	Let $v'$ be its inverse. Since $v^*=v$, $v^{\prime*}=v'$. By \MyCref{thm:conclusion}[2], $v'$ maps $C^\infty(M,\omegahalf)$ to $C^\infty(M,\omegahalf)$, and $H^{l}(M)$ to $H^{l}(M)$.
	So, by Schwartz kernel theorem $v'\in C^{-\infty}_{r,s}(M\times M,\omegahalf)$.
	To finish the proof of \MyCref{thm:main_theorem}[parametrix] in the compact case, we need to show that 
		\begin{equation*}\begin{aligned}
			\WF_l(v'\ast w)=\WF_l(w),\quad \forall w\in C^{-\infty}(M,\omegahalf),\ l\in \R^\nu.		
		\end{aligned}\end{equation*}
	Since $v'$ is the inverse of $v$, equivalently, we need to show that $\WF_l(v\ast w)=\WF_l(w)$.
	By \MyCref{thm:multi_parameter_wave_front_set}[3], we need to show that $\WF_l(w)\subseteq \WF_l(v\ast w)$. 
	Since $v\ast \cdot$ is a topological isomorphism on Sobolev spaces, it suffices to show that $\WF(w)\subseteq \WF(v\ast w)$.
	Let 
	$(\xi,x)\notin \WF(v\ast w)$ be fixed.
			\begin{lem}\label{lem:chi_n}
				There exists a sequence $(\chi_n)_{n\in \N}\subseteq \Psi^0_{\mathrm{classical}}(M)$ and $\chi_{\infty}\in \Psi^0_{\mathrm{classical}}(M)$ such that for all $n\in \N$ \begin{enumerate}
					\item    $\chi_n\ast v\ast w\in C^\infty(M,\omegahalf)$
					\item If $n,m\in \N\cup \{\infty\}$ with $m<n$, then $\chi_{n}\ast \chi_m-\chi_{n}$ is a smoothing operator, i.e., they map $C^{-\infty}(M,\omegahalf)$ to $C^\infty(M,\omegahalf)$.
					\item $\sigma^0_{\mathrm{classical}}(\chi_{\infty},\xi,x)\neq 0$.
				\end{enumerate}
			\end{lem}
			\begin{proof}
				By \Cref{thm:Hormander_WF}, there exists $\chi\in \Psi^0_{\mathrm{classical}}(M)$ such that $\chi$ is elliptic at $(\xi,x)$ and $\chi\ast v\ast w\in C^\infty(M,\omegahalf)$.
				Let $U\subseteq M$ be an open chart around $x$.
				We take $K_n\subseteq T^*M\backslash 0$ a sequence of closed conic neighborhoods of $(\xi,x)$ such that $K_n\subseteq \{(\eta,y)\in T^*U\backslash 0:\sigma^0(\chi,\eta,y)\neq 0\}$, $K_{n+1}\subseteq \mathrm{Int}(K_n)$, and $\bigcap_{n\in \N}K_n$ is a conic neighbourhood of $(\xi,x)$.
				Let $\kappa_n\in C^\infty(T^*M\backslash 0)$ homogeneous of degree $0$ such that $\kappa_n=1$ on $K_{n+1}$, $\supp(\kappa_n)\subseteq \mathrm{int}(K_n)$.
				Let $\kappa_\infty\in C^\infty(T^*M\backslash 0)$ such that $\kappa_{\infty}(\xi,x)=1$ and $\supp(\kappa_{\infty})\subseteq \bigcap_{n\in \N}K_n$.
				So, $\kappa_n\kappa_{m}=\kappa_n$ if $m<n$ for any $m,n\in \N\cup \{\infty\}$.
				Let $\chi_n,\chi_{\infty}\in \Psi^0_{\mathrm{classical}}(M)$ be classical pseudo-differential operators whose total symbols are $\kappa_n$ and $\kappa_{\infty}$.
				The first assertion follows from the fact that since $K_n\subseteq \{(\eta,y)\in T^*U\backslash 0:\sigma^0(\chi,\eta,y)\neq 0\}$, there exists $\chi'\in \Psi^0_{\mathrm{classical}}(M)$ such that $\chi_n=\chi'\ast \chi$ modulo smoothing operators.
				The second assertion now follows from \cite[Theorem 18.1.8]{HormanderIII}.
				The third assertion is obvious.
			\end{proof}
			We define the Sobolev spaces 
				\begin{equation*}\begin{aligned}
					H^{t}_n(M)=\sum_{l\in \Z_+^\nu,l_1+\cdots+l_\nu\geq n}H^{t+l}(M)	,\quad n\in \N,t\in \R^\nu.	
				\end{aligned}\end{equation*}
			Equivalently, $H^t_n(M)$ is equal to the sum of $H^{t'}(M)$ such that there exists a sequence $t=t_0\prec t_1\prec \cdots\prec t_n=t'$.
			By \Crefitem{thm:bounded_M_Sobolev}{5}, $\bigcap_{n\in \N} H^{t}_n(M)=C^\infty(M)$.
			Also, operators in $\Psi^0(M)$ (and thus also classical pseudo-differential operators of order $0$) preserve $H^{t}_n(M)$.

			By \Crefitem{thm:bounded_M_Sobolev}{1}, there exists some $t\in \R^\nu$ such that $w\in H^t(M)$. We fix such $t$.
			\begin{lem}\label{lem:chi_n_2}
				For all $n\in\N$, $\chi_n\ast w\in H^{t}_n(M)$.
			\end{lem}
			\begin{proof}
				We will prove this by induction.
				We have
				\begin{equation*}\begin{aligned}
						v\ast \chi_1\ast w&=[v,\chi_1]\ast w + \chi_1\ast v\ast w&&\mod C^\infty(M,\omegahalf)\\
						&=[v,\chi_1]\ast w &&\mod C^\infty(M,\omegahalf)
				\end{aligned}\end{equation*}
				So, by \Cref{lem:commutativity} and \Crefitem{thm:bounded_M_Sobolev}{6}, $v\ast \chi_1\ast w\in H^{t+k}_1(M)$.
			Hence, by \MyCref{thm:conclusion}[2], $\chi_1\ast w\in H^{t}_1(M)$. Since $\chi_{n}\ast \chi_{1}-\chi_n$ is smoothing, $\chi_n\ast w\in H^{t}_1(M)$ for all $n\in \N$.
			Now, 
				\begin{equation*}\begin{aligned}
			 v\ast \chi_2\ast w&=v\ast \chi_2\ast \chi_1\ast w\\&=[v,\chi_2]\ast \chi_1\ast w+\chi_2 \ast v\ast \chi_1 \ast w&&\mod C^\infty(M,\omegahalf)
			  \\& =    [v,\chi_2]\ast \chi_1\ast w+\chi_2 \ast [v,\chi_1]\ast w &&\mod C^\infty(M,\omegahalf)\\
			&=[v,\chi_2]\ast \chi_1\ast w+[\chi_2 ,[v,\chi_1]]\ast w+[v,\chi_1]\ast \chi_2\ast w    &&\mod C^\infty(M,\omegahalf)		
				\end{aligned}\end{equation*}
			The three terms belong to $H^t_2(M)$ by \Cref{lem:commutativity} and $\chi_2\ast w\in H^t_1(M)$, and $\chi_1\ast w \in H^t_1(M)$.
			Hence, $ v\ast \chi_2\ast w\in H^t_2(M)$. By \Cref{thm:conclusion}, $\chi_2\ast w\in H^{t}_2(M)$.
			More generally, 
			\begin{equation*}\begin{aligned}
			 v\ast \chi_{n}\ast w&=\sum_{S\subseteq \bb{1,n},S\neq \emptyset} [\chi_{s_1},[\chi_{s_2},\cdots,[v,\chi_{s_{|S|}}]]\cdots]\ast \chi_{\max{S^c}}\ast w &&\mod C^\infty(M,\omegahalf)	,	
				\end{aligned}\end{equation*}
				where $s_1,\cdots,s_{|S|}$ are the elements of $S$ ordered in decreasing order, and with the convention $\chi_{\max{S^c}}\ast w=w$ if $S=\bb{1,n}$.
			The induction is now clear.
			\end{proof}
		Since for all $n\in\N$,
				\begin{equation*}\begin{aligned}
					\chi_{\infty}\ast w=\chi_{\infty}\ast\chi_n\ast w &\mod C^\infty(M,\omegahalf),
				\end{aligned}\end{equation*}
			it follows that $\chi_{\infty}\ast w\in \bigcap_{n\in \N} H^{t}_n(M)=C^\infty(M)$.
			Hence, by \Cref{thm:Hormander_WF}, $(\xi,x)\notin \WF(w)$. This finishes the proof of \MyCref{thm:main_theorem}[parametrix] in the case where $M$ is compact.

			We now prove \MyCref{thm:main_theorem}[parametrix] in the general case where $M$ isn't necessarily compact.
			\begin{lem}\label{lem:existence_all_orders}
				For any $k\in \R^\nu$, there exists $v_k\in \Psi^k(M)$ such that $v_k^*=v_k$ and for any $x\in M$, $\pi\in \HatHLnon_x$, $\sigma^k(v_k,\pi,x):C^\infty(\pi)\to C^\infty(\pi)$ is injective.
			\end{lem}
			\begin{proof}
				This follows from \MyCref{lem:compactification} and \MyCref{thm:conclusion}[3] using a partition of unity argument. The assumption of $k\in \R^\nu$ is necessary because to make the partition of unity argument work, one takes a sum of positive operators.
				We leave the details to the reader.
			\end{proof}
			Let $\Gamma'\subseteq \Gamma$ be a closed cone, $\chi\in \Psi^0_{\mathrm{classical}}(M)$ such that $\chi$ is elliptic on the complement of $\Gamma$, and $\chi$ is smoothing on $\Gamma'$.
			By replacing $v$ by $v^*\ast v_{-\Re(k)}^*\ast v_{-\Re(k)}\ast v+\chi^*\ast\chi$, we can suppose that $v\in \Psi^0(M)$ satisfies $v=v^*$, and 
			for all $(\pi,x)\in \HatHLnon$, $\sigma^0(v,\pi,x):C^\infty(\pi)\to C^\infty(\pi)$ is injective.
			\begin{lem}\label{lem:skqdmjfkojqkosmdjfkmjqsdkmjfsqdf}
				The distribution $v$ has a parametrix, i.e., there exists $v'\in C^{-\infty}_{r,s}(M\times M,\omegahalf)$ such that 
					\begin{equation*}\begin{aligned}
						\delta_1-v'\ast v,\delta_1-v\ast v'	\in C^\infty(M\times M,\omegahalf).
					\end{aligned}\end{equation*}
			\end{lem}
			\begin{proof}
				By \MyCref{lem:compactification} and \MyCref{thm:main_theorem}[parametrix] in the compact case that we just proved, 
			we can create parametrices for $v$ locally.
			Furthermore, these parametrices are smooth off the diagonal.
			So, by using a partition of unity argument and \MyCref{thm:main_prop_bi_graded_pseudo_diff_M}[6], we can find $v'\in C^{-\infty}_{r,s}(M\times M,\omegahalf)$ such that $v'$ is smooth off the diagonal, and $v''=\delta_1-v'\ast v\in C^{-\infty}_{r,s}(M\times M,\omegahalf)$ satisfies
				\begin{equation}\label{eqn:hjjqsdhjmfsdqmfmsqdjfmqsd}\begin{aligned}
					v''\ast H^k_{\mathrm{loc}}(M)\subseteq H^{\prec k}_{\mathrm{loc}}(M),\quad \forall k\in \R^\nu
				\end{aligned}\end{equation}
			Let $n\in \N$. By replacing $v'$ by $(1+v''+(v'')^2+\cdots (v'')^n)\ast v'$, we can suppose that
			$v''=\delta_1-v'\ast v$ decreases the order of Sobolev spaces as much as we want.
			So, by taking $n$ big enough, and using \MyCref{thm:bounded_M_Sobolev}[5] (and the classical Sobolev embedding theorem) we can suppose that $v''\in C(M\times M,\omegahalf)$ which is smooth off the diagonal because $v$ and $v'$ are smooth off the diagonal.
			We can thus write $v''$ as $v''_1+v''_2$, where $v''_1\in C^\infty(M\times M,\omegahalf)$ and $v''_2\in C(M\times M,\omegahalf)$ extends to a bounded operator $L^2(M)\to L^2(M)$ of norm is $<1$. Here, we can use Schur test to bound $v''_2$.
			Since $v''$ is properly supported, we can suppose that $v''_1$ and $v''_2$ are properly supported.
			The sum $\sum_{n=0}^\infty (v_2'')^n$ converges to a bounded operator $L^2(M)\to L^2(M)$.
			By Schwartz kernel theorem, this operator comes from convolution by a distribution $v'''\in C^{-\infty}(M\times M,\omegahalf)$.
			By writing $v'''$ as 
				\begin{equation*}\begin{aligned}
						v'''=\delta_1+v''+ v''\ast v'''\ast v'',		
				\end{aligned}\end{equation*}
			 we deduce that $v'''$ and $(v'''')^*$ map $C^\infty_c(M,\omegahalf)$ to $C^\infty(M,\omegahalf)$ (in fact inside $C^\infty_c(M,\omegahalf)$ because the support is proper).
			 So, $v'''\in C^{-\infty}_{r,s}(M\times M,\omegahalf)$.
			 Finally, 
				 \begin{equation*}\begin{aligned}
						v'''\ast v'\ast v=\delta_1-v'''\ast v_1'', 
				 \end{aligned}\end{equation*}
			 and $v''''\ast v_1''$ is smoothing operator because $v_1''$ is smoothing and $v'''\in C^{-\infty}_{r,s}(M\times M,\omegahalf)$.
			 We have thus obtained a left parametrix. By taking the adjoint, we obtain a right parametrix.
			\end{proof}
			So, we proved that $v$ has a parametrix $v'$.
			By \MyCref{lem:compactification} and \MyCref{thm:main_theorem}[parametrix] in the compact case, the parametrix $v'$ has to satisfy $\WF_l(v'\ast w)\subseteq \WF_{l}(w)$ and $\WF_l(v^{\prime *}\ast w)\subseteq \WF_l(w)$ for all $w\in C^{-\infty}(M,\omegahalf)$ and $l\in \R^\nu$.
			This finishes the proof of \MyCref{thm:main_theorem}[parametrix].

			The inclusion $\Gamma\subseteq \{(\xi,x)\in T^*M\backslash 0:(-\xi,\xi;x,x)\in \WF(v)\}$ follows immediately from \MyCref{thm:main_theorem}[parametrix].
			This finishes the proof of \MyCref{thm:main_theorem}[inj].

			We now prove \MyCref{thm:main_theorem}[micro_local_max_hypo].	
			From \MyCref{thm:main_theorem}[parametrix], \eqref{eqn:sdqjofjlkqsdjf} follows easily.
			Now, suppose that $(\xi,x)\in T^*M\backslash 0$ and $l\in \R^\nu$ satisfy
				\begin{equation*}\begin{aligned}
					(\xi,x)\in \WF_{l+\Re(k)}(w)	\iff (\xi,x)\in \WF_{l}(v\ast w)
				\end{aligned}\end{equation*}
			 for all $w\in C^{-\infty}(M,\omegahalf)$.
			 We will prove that $(\xi,x)\in \Gamma$.
			Since this is a local statement, by \MyCref{lem:compactification}, we can suppose that $M$ is compact.
			The operators $v_{l}\in \Psi^l(M)$ obtained in \MyCref{thm:conclusion}[3] satisfy \eqref{eqn:sdqjofjlkqsdjf} with $\Gamma=T^*M\backslash 0$ by what we proved above. 
			So, by replacing $v$ by $v_{l}\ast v\ast v_{-l-k}$, we can suppose that $v\in \Psi^0(M)$, and 
			\begin{equation}\label{eqn:qskomdjfkmoqsjdmofjqds}\begin{aligned}
					(\xi,x)\in \WF_{0}(w)	\iff (\xi,x)\in \WF_{0}(v\ast w).
				\end{aligned}\end{equation}
			Let $\chi\in \Psi^0_{\mathrm{classical}}(M)$ whose classical principal symbol vanishes only at $\Rpt\cdot (\xi,x)$.
			We also fix an operator $v'\in \Psi^{(-1,\cdots,-1)}(M)$ as in \MyCref{thm:conclusion}[3].
			Consider the vector space 
				\begin{equation*}\begin{aligned}
					H=\{w\in C^{-\infty}(M,\omegahalf):v\ast w,\ \chi\ast w,\ v'\ast w\in L^2(M)\}		
				\end{aligned}\end{equation*}
			equipped with the norm 
				\begin{equation*}\begin{aligned}
					\norm{w}_H:=\max\left(\norm{v\ast w}_{L^2(M)},\norm{\chi\ast w}_{L^2(M)},\norm{v'\ast w}_{L^2(M)}\right).
				\end{aligned}\end{equation*}
			The space $H$ is complete because $v'\ast \cdot:H^{(-1,\cdots,-1)}(M)\to L^2(M)$ is a topological isomorphism.
			We claim that $H=L^2(M)$ as a set.
			The inclusion $L^2(M)\subseteq H$ is clear. For the other inclusion, let $w\in H$.
			Since $\chi\ast w\in L^2(M)$, $\WF_0(w)\subseteq \Rpt (\xi,x)$.
			Since $v\ast w\in L^2(M)$, $\WF_0(v\ast w)=\emptyset$. So, by \eqref{eqn:qskomdjfkmoqsjdmofjqds}, $(\xi,x)\notin \WF_0(w)$.
			Hence, $\WF_0(w)=\emptyset$. So, by \MyCref{thm:multi_parameter_wave_front_set}[1], $w\in L^2(M)$.
			The inclusion $L^2(M)\to H$ is clearly continuous. By the open mapping theorem, it is a topological isomorphism.
			Hence, there exists $C>0$ such that 
				\begin{equation*}\begin{aligned}
					\norm{w}_{L^2(M)}\leq C 	\norm{w}_H.	
				\end{aligned}\end{equation*}
			This means that the operator $v^*\ast v+\chi^*\ast \chi+v^{\prime *}\ast v':L^2(M)\to L^2(M)$ is invertible.
			The operator $v^{\prime *}\ast v'$ maps $L^2(M)$ into $H^{(2,\cdots,2)}(M)$. So, as an operator $L^2(M)\to L^2(M)$, it is compact by \MyCref{thm:bounded_M_Sobolev}[4].
			So, $v^*\ast v+\chi^*\ast \chi$ is invertible in $\overline{\Psi^0(M)}/\cK(L^2(M))$.
			By \MyCref{thm:short_exact_seq_psi0}, $\sigma^0(v^*\ast v+\chi^*\ast \chi,\pi,x):L^2(\pi)\to L^2(\pi)$ is injective for all $\pi\in\HatHLnon_x$.
			Hence, by \eqref{eqn:symbol_classical_extension} and since the classical principal symbol of $\chi$ vanishes at $(\xi,x)$, $\sigma^0(v,\pi,x):L^2(\pi)\to L^2(\pi)$ is injective for all $\pi\in \HatHLnon_{(\xi,x)}$. So, $(\xi,x)\in \Gamma$.
			This finishes the proof of \MyCref{thm:main_theorem}[micro_local_max_hypo].

			We now prove \MyCref{thm:main_theorem}[inequa].
			We fix $v''\in \Psi^{k}_c(M)$.
			Let $\Gamma_n\subseteq \Gamma$ be a sequence of open cones such that $\bigcup_n \Gamma_n=\Gamma$, and $\overline{\Gamma_{n}}\subseteq \Gamma_{n+1}$.
			For each $n$, we fix $v_n'$ as in \MyCref{thm:main_theorem}[parametrix] applied to $\overline{\Gamma_n}$, and $\chi_n\in \Psi^0_{\mathrm{classical}}(M)$ such that $\chi_n$ is elliptic on $\overline{\Gamma_{n-1}}$, and smoothing on the complement of $\Gamma_n$, and $\chi_n$ are (formally) positive.
			Since $v_n'$  is smooth off the diagonal, by modifying $v_n'$ and $\chi_n$ outside the diagonal, we can suppose that the operators $(v_n')_{n\in \N}$ and $(\chi_n)_{n\in \N}$ are uniformly properly supported.
			
			For each $n$, let $r_n=\delta_1-v'_n\ast v$. So, we have 
				\begin{equation}\label{eqn:qksdjfkjqsdmojf}\begin{aligned}
					\chi_n\ast v''=\chi_n\ast v''\ast r_n+\chi_n\ast v''\ast v_n'\ast v		
				\end{aligned}\end{equation}
			By \eqref{eqn:parametrix_WF} and \eqref{eqn:WF_multi_pseudo} (applied to $v''$), $\chi_n\ast v'' \ast r_n$ is a smoothing operator, which is uniformly compactly supported because $v''$ is compactly supported.
			So, $\chi_n\ast v''\ast r_n\in C^\infty_c(M\times M,\omegahalf)$.
			For each $n$, the operator $\chi_n\ast v''\ast v_n'$ maps $L^2_{\mathrm{loc}}(M)$ to $L^2_{\mathrm{loc}}(M)$. To see this, let $w\in L^2_{\mathrm{loc}}(M)$. 
			Then, 
				\begin{equation*}\begin{aligned}
					\WF_0(\chi_n\ast v''\ast v_n'\ast w)&\subseteq \WF_0(v''\ast v_n'\ast w)&&\text{ by \MyCref{thm:multi_parameter_wave_front_set}[3]}\\&\subseteq \WF_{\Re(k)}(v_n'\ast w)&&\text{ by \MyCref{thm:multi_parameter_wave_front_set}[3]}\\&=\WF_0(w)&&\text{ by \eqref{eqn:parametrix_WF}}\\&=\emptyset&&\text{ by \MyCref{thm:multi_parameter_wave_front_set}[1]}.		
				\end{aligned}\end{equation*}
			Furthermore, since $v''$ is compactly supported, the image is actually in $L^2(M)$. So, the operator $\chi_n\ast v''\ast v_n'$ maps $L^2(M)$ to $L^2(M)$. 
			Hence, it is continuous by the closed graph theorem.
			Now, we take a sequence $c_n>0$ which converges to $0$ fast enough so that the sums
				\begin{equation*}\begin{aligned}
					\chi=&\sum_n c_n\chi_n &&\text{ converges in }\Psi^0_{\mathrm{classical}}(M)\\
					r=&\sum_n c_n\chi_n\ast v''\ast r_n&&\text{ converges in }C^\infty_c(M\times M,\omegahalf)\\
					v'=&\sum_n c_n\chi_n\ast v''\ast v_n'&&\text{ converges in }\cL(L^2(M))
				\end{aligned}\end{equation*}
			 By \eqref{eqn:qksdjfkjqsdmojf}, we have 
				 \begin{equation}\label{eqn:sqdkjfkmsqdjlmkjflqsd}\begin{aligned}
					 \chi\ast v''=r+v'\ast v.
				 \end{aligned}\end{equation}
			  From this, \eqref{eqn:inequ_microlocal_main_thm} follows.
			If $\Gamma=T^*M\backslash 0$, then there is no need to do the above. One simply applies \MyCref{thm:main_theorem}[parametrix] with $\Gamma'=\Gamma$.
			The details are left to the reader.
			This finishes the proof of \MyCref{thm:main_theorem}[inequa].

			We finally prove \MyCref{thm:main_theorem}[Fredholm].
			By \MyCref{thm:conclusion}[3], we can suppose that $k=0$.
			By \MyCref{thm:main_theorem}[left-inv], $\sigma^0(v,\pi,x)$ is injective on $L^2(\pi)$.
			So, by \eqref{eqn:short_exact_seq_psi0} and \MyCref{prop:simple_arg}, $v$ is left invertible in $\overline{\Psi^0(M)}/\cK(L^2(M))$, i.e., it is left-invertible modulo compact operators.
		\end{linkproof}
		
		\begin{linkproof}{main_theorem_double_sided}
				It suffices to apply \Cref{thm:main_theorem} to $v^*\ast v$ and $v\ast v^*$.
		\end{linkproof}

\begin{appendices}
					
\chapter{Kirillov's orbit method}\label{chap:Orbit_method}
			Let $G$ be a simply connected nilpotent Lie group, $\mathfrak{g}$ its Lie algebra.
			So, the exponential map $\exp:\mathfrak{g}\to G$ is a diffeomorphism.
			For any $\xi\in \mathfrak{g}^*$, let $B_\xi:\mathfrak{g}\times \mathfrak{g}\to \R$ be the antisymmetric bilinear form $B_\xi(v,w)=\xi([v,w])$, and $\ker(B_\xi)=\{v\in \mathfrak{g}:B_\xi(v,w)=0,\ \forall w\in \mathfrak{g}\}$.
			Since $B_{\xi}$ is antisymmetric, $\codim(\ker(B_\xi))$ is even.
			Let $\mathfrak{h}\subseteq \mathfrak{g}$ be a Lie subalgebra such that $B_\xi(\mathfrak{h},\mathfrak{h})=0$, $H\subseteq G$ the Lie subgroup integrating $\mathfrak{h}$. 
			The map $\exp$ restricts to a diffeomorphism $\mathfrak{h}\to H$, so in particular $H$ is closed.
			We define the Hilbert space 
			$L^2(G/H,\xi)$ to be the space of measurable functions $f:G\to \Omega^{\frac{1}{2}}(\frac{\mathfrak{g}}{\mathfrak{h}})$ such that $f(x\exp(y))=f(x)e^{i\xi (y)}$ for all $x\in G$ and $y\in \mathfrak{h}$, and $\int_{G/H}|f|^2<+\infty$.			
			We have a unitary representation of $G$ on $L^2(G/H,\xi)$ given by 
				\begin{equation}\label{eqn:dfn_representation_quasi_induite}\begin{aligned}
					\Xi_{G/H,\xi}(f)\ast g(x)=\int_{G} f(y)g(y^{-1}x),\quad f\in C^\infty_c(G,\Omega^1(\mathfrak{g})),g\in L^2(G/H,\xi) .
				\end{aligned}\end{equation}
			The main theorem on Kirillov's orbit method can be summarized as follows:
			\begin{theorem}[{\cite[Section 2.2]{RepNilpotentLieGroupsCorwinGreenleaf}}]\label{thm:Kirillov_orbit_method}
				\begin{enumerate}
					\item\label{thm:Kirillov_orbit_method:exist} 	For any $\xi\in \mathfrak{g}^*$, there exists a Lie subalgebra $\mathfrak{h}$ which satisfies $B_\xi(\mathfrak{h},\mathfrak{h})=0$, and $\codim(\mathfrak{h})=\codim(\ker(B_\xi))/2$.
							We call $\mathfrak{h}$ a polarizing Lie subalgebra for $\xi$.
					\item\label{thm:Kirillov_orbit_method:surj} 	If $\xi\in \mathfrak{g}^*$ and $\mathfrak{h}$ a polarizing Lie subalgebra, then $\Xi_{G/H,\xi}$ is an irreducible unitary representation.
							Furthermore, any irreducible unitary representation of $G$ is unitarily equivalent to $\Xi_{G/H,\xi}$ for some $ \xi\in \mathfrak{g}^*$ and $\mathfrak{h}$ a polarizing Lie subalgebra.
					\item\label{thm:Kirillov_orbit_method:equiv} 	If $\xi,\xi'\in \mathfrak{g}^*$, and $\mathfrak{h}$ and $\mathfrak{h}'$ are polarizing Lie subalgebras for $\xi$ and $\xi'$, then $\Xi_{G/H,\xi}$ is unitarily equivalent to $\Xi_{G/H',\xi'}$ if and only if 
							there exists $g\in G$ such that $\xi=\Ad^*(g)\xi'$.
				\end{enumerate}
			\end{theorem}
			By taking $\xi=\xi'$ in \Crefitem{thm:Kirillov_orbit_method}{equiv}, it follows that the equivalence class of $\Xi_{G/H,\xi}$ doesn't depend on $\mathfrak{h}$.
			Thus, we have a bijective map 
				\begin{equation}\label{eqn:Orbit_method_map}\begin{aligned}
					\mathfrak{g}^*/G\to \hat{G},\quad [\xi]\mapsto [\Xi_{G/H,\xi}],\quad \xi\in \mathfrak{g}^*
				\end{aligned}\end{equation}
			where $\mathfrak{g}^*/G$ is the quotient of $\mathfrak{g}^*$ by the co-adjoint action.
			\begin{theorem}[{Brown's theorem \cite{BrownKirrilovOrbitMethod}}]\label{thm:Brown}
				The map \eqref{eqn:Orbit_method_map} is a homeomorphism when $\mathfrak{g}/G$ is equipped with the quotient topology and $\hat{G}$ is equipped with the Fell topology.
			\end{theorem}

			\Cref{thm:Kirillov_orbit_method} gives irreducibility of the representations $\Xi_{G/H,\xi}$ when $\mathfrak{h}$ is polarizing. The following theorem due to Ludwig \cite{LudwigGoodIdeal} analyzes the representations $\Xi_{G/H,\xi}$ in the general case.
			\begin{theorem}\label{prop:support_Kirillovs_Orbit_Method_Induced}
					Let $\xi\in \mathfrak{g}^*$ and $\mathfrak{h}$ a Lie subalgebra which satisfies $B_\xi(\mathfrak{h},\mathfrak{h})=0$. Then, the set of irreducible unitary representations of $G$ which are weakly contained in $\Xi_{G/H,\xi}$
					correspond by the bijection \eqref{eqn:Orbit_method_map} to the set 
						\begin{equation*}\begin{aligned}
							\overline{\bigcup_{g\in G}\{\Ad^*(g)\eta\in \mathfrak{g}^*:\eta_{|\mathfrak{h}}=\xi\}}/G
						\end{aligned}\end{equation*}
			\end{theorem}
			\paragraph{Smooth vectors}
			Smooth vectors of the representation $\Xi_{G/H,\xi}$ in the irreducible case have a very simple characterization:
			\begin{theorem}[{\cite[Corollary 4.1.2]{RepNilpotentLieGroupsCorwinGreenleaf}}]\label{thm:smooth_vectors_Kirillov}
				Let $\xi\in \mathfrak{g}^*$, $\mathfrak{h}\subseteq \mathfrak{g}$ a polarizing Lie subalgebra, then if one identifies $L^2(G/H,\xi)$ with $L^2(\R^k)$
				for some $k\in \Z_+$ using any weak Maclev basis (see \cite{RepNilpotentLieGroupsCorwinGreenleaf}), then the space of smooth vectors of $\Xi_{G/H,\xi}$ is equal to the space of Schwartz functions on $\R^k$.
			\end{theorem}

			\paragraph{Graded Lie-algebra}
			We say that $\mathfrak{g}$ is $\nu$-graded if $\mathfrak{g}=\bigoplus_{k\in \Z_+^\nu}\mathfrak{g}_{k}$ with $\mathfrak{g}_0=0$, and 
			$[\mathfrak{g}_{k},\mathfrak{g}_{l}]\subseteq \mathfrak{g}_{k+l}$ for all $k,l\in \Z_+^\nu$.
			Such an algebra admits an $\R_+^\nu$-dilations $\alpha_\lambda:\mathfrak{g}\to \mathfrak{g}$, see \eqref{eqn:dilations_basic_V}.

			\paragraph{Decomposition of the spectrum:}
				We will need the following result by I. Beltita, D. Beltita, and J. Ludwig \cite[Theorem 4.11 and Proposition 4.5, and definition 2.9]{FourierTransformNilpotentGroups} on the structure of the space $\mathfrak{g}^*/G$.
				\begin{theorem}\label{thm:decomposition_structure_Calgebra}
					There exists a finite sequence of open subsets $\emptyset=V_0\subseteq V_1\subseteq \cdots\subseteq V_n=\mathfrak{g}^*$ for some $n$ such that
					\begin{enumerate}
						\item For any $i\in \bb{0,n}$, $V_i$ is closed under the co-adjoint action, and closed under the canonical vector space $\Rpt$-dilation.
						\item For all $i\in \bb{0,n-1}$, the quotient $(V_{i+1}\backslash V_{i})/G$ is a Hausdorff locally compact space (in fact a semi-algebraic space).
							Furthermore, $V_{n}\backslash V_{n-1}=[\mathfrak{g},\mathfrak{g}]^{\perp}$. Note that the co-adjoint action on $[\mathfrak{g},\mathfrak{g}]^{\perp}$ is trivial.
						\item If $I_i$ denotes the closed two-sided ideal of $C^*G$ whose spectrum is $V_i$, then one has 
							\begin{equation*}\begin{aligned}
								0=I_0\subseteq I_1\cdots\subseteq I_{n-1}\subseteq I_n=C^*G,
							\end{aligned}\end{equation*}
						and 
							\begin{equation*}\begin{aligned}
									&I_{i+1}/I_i\simeq C_0((V_{i+1}\backslash V_i)/G)\otimes \cK(l^2(\mathbb{N})),&& \forall i\in \bb{0,n-2}\\
									&I_{n}/I_{n-1}\simeq C_0(V_n\backslash V_{n-1})		&&
							\end{aligned}\end{equation*}
						\item\label{thm:decomposition_structure_Calgebra:4} If $\mathfrak{g}$ is $\nu$-graded by $\Z_+^\nu$, then one can suppose that each $V_i$ is invariant by the dilations $\alpha_\lambda$ for $\lambda\in (\Rpt)^\nu$.
					\end{enumerate}
				\end{theorem}
				\MyCref{thm:decomposition_structure_Calgebra}[4] isn't explicitly stated in \cite{FourierTransformNilpotentGroups}, but it follows easily from the construction of the sets $V_i$, see \cite[Section 4.1 and Section 4.2]{FourierTransformNilpotentGroups} (one has to choose a Jordan-Holder basis which consist of pure elements).		
			\paragraph{Direct sum:}
				For $i\in \{1,2\}$, if $G_i$ are simply connected nilpotent Lie groups, $\mathfrak{g}_{i}$ their Lie algebras, $\xi_i\in \mathfrak{g}^*_i$, and $\mathfrak{h}_i\subseteq \mathfrak{g}_i$ are polarizing Lie subalgebras, then 
				$\mathfrak{h}_1\oplus \mathfrak{h_2}\subseteq \mathfrak{g}_1\oplus \mathfrak{g}_2$ is a polarizing Lie subalgebra for $\xi_1\oplus \xi_2$.
				Furthermore, 
					\begin{equation}\label{eqn:direct_sum}\begin{aligned}
						\Xi_{(G_1\times G_2)/(H_1\times H_2),\xi_1\oplus \xi_2}\simeq \Xi_{G_1/H_1,\xi_1}\otimes \Xi_{G_2/H_2,\xi_2}.
					\end{aligned}\end{equation}

			\paragraph{Some auxiliary results in graded Lie algebras:}	
			Motivated by \eqref{eqn:weakly_commuting_lie_algebra}, we say that a $\nu$-graded Lie algebra $\mathfrak{g}$ is weakly-commutative if $\mathfrak{g}_k=0$ whenever $|\{i\in \bb{1,\nu}:k_i\neq 0\}|>1$.
			We define $\mathfrak{g}^{\weak{i}}=\bigoplus_{j\in \Z_+} \mathfrak{g}^{(\overbrace{0,\cdots,0}^{i-1},j,\overbrace{0\cdots,0}^{\nu-i})}$.
			So, we have
			\begin{equation}\label{eqn:weakly_commuting_lie_algebra_2}\begin{aligned}
				\mathfrak{g}=\bigoplus_{i=1}^{\nu}\mathfrak{g}^{\weak{i}},\quad \forall x\in M.
			\end{aligned}\end{equation}
			Furthermore, $[\mathfrak{g}^{\weak{i}},\mathfrak{g}^{\weak{j}}]=0$ for all $i,j\in \bb{1,\nu}$ with $i\neq j$.
			
			The set $\{\xi\in \mathfrak{g}^*:\exists i\in \bb{1,\nu},\ \xi_{|\mathfrak{g}^{\weak{i}}}=0\}$ is closed under the co-adjoint action.
			We say that an irreducible unitary representation is singular if it corresponds by \eqref{eqn:Orbit_method_map} to a co-adjoint orbit which is contained in the above set.
			Otherwise, we say that it is non-singular.

			We denote by $C^*_{\mathrm{nonsing}}G$ the closed ideal of $C^*G$ whose spectrum is the set of non-singular irreducible unitary representations.
			We denote by $\Hat{G}_{\mathrm{nonsing}}$ the equivalence classes of non-singular irreducible unitary representations of $G$, i.e., the spectrum of $C^*_{\mathrm{nonsing}}G$.
			\begin{theorem}\label{thm:Corssed-Product_by_action}
				The crossed product $C^*_{\mathrm{nonsing}}G\rtimes (\Rpt)^\nu$ is a $C^*$-algebra of Type I whose spectrum is homeomorphic $\Hat{G}_{\mathrm{nonsing}}/(\Rpt)^\nu$.
			\end{theorem}
			\begin{proof}
				This follows from \MyCref{thm:decomposition_structure_Calgebra}. 
				One has to check that the action of $(\Rpt)^\nu$ on the open subset $\Hat{G}_{\mathrm{nonsing}}\cap \left((V_{i+1}\backslash V_{i})/G\right)$ of the locally compact space $ (V_{i+1}\backslash V_{i})/G$ is free and proper.
				The details are left to the reader.
			\end{proof}
			\begin{theorem}\label{thm:pi_diff_appendix_decompo_vector_fields}
				If $\pi$ is a non-singular representation of $G$, and $\xi\in C^\infty(\pi)$ is a smooth vector, then for any $i\in \bb{1,\nu}$, there exists some $X_1,\cdots,X_n\in \mathfrak{g}^{\weak{i}}$ and $\xi_1,\cdots,\xi_n\in C^\infty(\pi)$ 
				such that
				$\xi=\pi(X_1)\xi_1+\cdots +\pi(X_n)\xi_n$, where $\pi(X_i)$ is defined using the differential of $\pi$.
			\end{theorem}
			\begin{proof}
				By writing $G$ as a product of the subgroups that integrate $\mathfrak{g}^{\weak{i}}$, the statement becomes the following: 
				If $G$ is a simply connected nilpotent Lie group, $\pi$ is a non-trivial irreducible unitary representation, then for some $X_1,\cdots,X_n\in \mathfrak{g}$ and $\xi_1,\cdots,\xi_n\in C^\infty(\pi)$, one has 
				$\xi=\pi(X_1)\xi_1+\cdots +\pi(X_n)\xi_n$.
				This is elementary. By Dixmier-Malliavin theorem, we can suppose that $\xi=\pi(f)\eta$ for some $f\in C^\infty_c(G)$, $\eta\in C^\infty(\pi)$.
				The fact that $\pi$ is non-trivial implies that the trivial representation isn't weakly contained in $\pi$.
				So, there exists $g\in C^\infty_c(G)$ such that $\int g=1$ but $\norm{\pi(g)}_{\cL(L^2(\pi))}<1$.
				Hence, $\Id-\pi(g):L^2(\pi)\to L^2(\pi)$ is invertible. So, $\eta=\eta'-\pi(g)\eta'$ for some $\eta'\in L^2(\pi)$.
				Hence, $\xi=\pi(f-f\ast g)\eta'$. Since $\int f-f\ast g=0$, we can write it as a sum of functions of the form $X_i\ast f_i$, where $X_i\ast f_i$ is the action of the right-invariant vector field associated to $X_i$ on $f_i$, see \cite[Proposition 2.1]{MohsenMaxHypo}.
				So, $\xi$ is a sum of elements of the form $\pi(X_i)\pi(f_i)\eta'$.
			\end{proof}
			
			\begin{theorem}[Riemann-Lebesgue lemma]\label{lem:Riemann-Lebesgue}
				If $f\in C^\infty_c(G)$ and $\pi$ is a non-singular irreducible unitary representation of $G$, then  for all $i\in \bb{1,\nu}$ and $n\in \N$, there exists $C>0$ such that 
					\begin{equation*}\begin{aligned}
						\norm{\pi(\alpha_\lambda(f))}_{\cL(L^2(\pi))}\leq \frac{C}{\lambda_i^n}		,\quad\forall \lambda\in (\Rpt)^\nu.
					\end{aligned}\end{equation*}
			\end{theorem}
			\begin{proof}
				The proof is a straightforward adaptation of \cite[Lemma 2.11]{MohsenAbstractMaxHypo}.
			\end{proof}

\chapter{Abstract theorem of maximal hypoellipticity}\label{thm:asbtract}
	\begin{prop}[{\cite{MohsenAbstractMaxHypo}}]\label{prop:simple_arg} Let $A$ be a unital $C^*$-algebra. 
			  An element $a\in A$ is left invertible if and only if $\pi(a):L^2(\pi)\to L^2(\pi)$ is injective for every irreducible representation $\pi$ of $A$.
	\end{prop}

	\begin{dfn}[Kaplansky and Glimm] A $C^*$-algebra $A$ is called \begin{itemize}
			\item Liminal if for every irreducible representation  $\pi$ of $A$, $\cK(L^2(\pi))=\pi(A)$.
			\item Type I if for every irreducible representation $\pi$ of $A$, $\cK(L^2(\pi))\subseteq \pi(A)$.
		\end{itemize}
	\end{dfn}
	We refer the reader to \cite[Chapters 4 and 9]{Dixmier} for a detailed discussion of liminal and Type I $C^*$-algebras.
	The examples of liminal $C^*$-algebras that will be of interest to us are the $C^*$-algebras of simply connected nilpotent Lie groups, and of Type \RNum{1} $C^*$-algebras are the closure pseudo-differential operators of order $0$ in our calculus.
	\begin{prop}[{\cite[Proposition 4.3.5, Proposition 4.3.4, and the main theorem of Chapter 9]{Dixmier}}]\label{prop:exat_sequence_3_TypeI}
		If $I$ is a two-sided ideal of a $C^*$-algebra $A$, then $A$ is of Type \RNum{1} if and only if $I$ and $A/I$ are of Type \RNum{1}.
	\end{prop}
	\begin{prop}\label{prop:liminal_mapped_to_compacts}
		Let $A$ be a $C^*$-algebra of Type I, $\pi$ an irreducible representation of $A$ such that the singleton $\{\pi\}$ is an open dense subset of $\hat{A}$.
		Then, for any $a\in A$, $\pi(a)\in \cK(L^2(\pi))$ if and only if $\pi'(a)=0$ for all $\pi'\in \hat{A}\backslash \{\pi\}$.
	\end{prop}
	\begin{proof}
		Since $\{\pi\}$ is dense in $\widehat{A}$, $\pi:A\to \cL(L^2(\pi))$ is injective, see  \cite[Proposition 3.3.2]{Dixmier}.
		So, if we consider $A$ as a $C^*$-subalgebra of $\cL(L^2(\pi))$, then $J:=\{a\in A:\pi(a)\in \cK(L^2(\pi))\}$ is equal to $\cK(L^2(\pi))$.
		So the spectrum of $J$ is a singleton. At the same time since $J$ is a closed two-sided ideal of $A$, its spectrum is an open subset of $\hat{A}$ consisting of irreducible representations which don't vanish on $J$.
		Clearly $\pi$ doesn't vanish on $J$. So $\hat{J}=\{\pi\}$.
	\end{proof}
    \begin{dfn}\label{dfn:CstarCalculus}
        A $C^*$-calculus (with $\nu\geq 1$-parameters) is a family of $\C$-vector spaces $(\cA_{k})_{k\in \C^\nu}$ such that for each $k,l\in \C^\nu$, we have a product map $\cA_k\times \cA_l\to \cA_{k+l}$ which is bilinear and associative, i.e., $(ab)c=a(bc)$ if $a\in \cA_{k_1}$, $b\in \cA_{k_2}$, $c\in \cA_{k_3}$.
				There is a unit element $0\neq 1\in \cA_0$, i.e., 1a$=a1=a$ for all $a\in \cA_k$.
                We also have an involution map $a\in \cA_k\mapsto a^*\in \cA_{\bar{k}}$ which is anti-linear and satisfies $1^*=1$ and $(ab)^*=b^*a^*$ for all $a\in \cA_{k_1}$, $b\in \cA_{k_2}$.
                We also have a $C^*$-semi-norm $\norm{\cdot}_{\overline{\cA_0}}$ on $\cA_0$, i.e., a semi-norm such that $\norm{ab}_{\overline{\cA_0}}\leq \norm{a}_{\overline{\cA_0}}\norm{b}_{\overline{\cA_0}}$ and $\norm{a^*a}_{\overline{\cA_0}}=\norm{a}^2_{\overline{\cA_0}}$ for all $a,b\in \cA_0$.
                We denote by $\overline{\cA_0}$ the Hausdorff completion of $\cA_0$.  
    \end{dfn}
    We don't suppose that we can add elements of $\cA_k$ and $\cA_l$ if $k\neq l$.
    For our applications, there is no need to suppose that. 
    Also, the sum of classical pseudo-differential operators on a smooth manifold of order $k\in \C$ and $l\in \C$ is a classical pseudo-differential operator if and only if $k-l\in \Z$.
    \begin{dfn}\label{dfn:unitary_representation_abstract}
        A representation of a $C^*$-calculus $(\cA_k)_{k\in \C^\nu}$ consists of a non-zero $\C$-vector space $C^\infty(\pi)$ equipped with a positive definite sesquilinear form $\langle \cdot,\cdot\rangle:C^\infty(\pi)\times C^\infty(\pi)\to \C$, 
        and a family of linear maps $\pi_k:\cA_k\to \mathrm{End}(C^\infty(\pi))$ for every $k\in \C^\nu$ such that the following hold:
		\begin{enumerate}
			\item   If $k,l\in \C^\nu$, $a\in\cA_k$, $b\in \cA_l$, $\xi,\eta\in C^\infty(\pi)$, then
                    \begin{equation*}\begin{aligned}
                     \pi_0(1)=\mathrm{Id}_{\C^\infty(\pi)},\quad  \pi_{k}(a)\pi_{l}(b)=\pi_{k+l}(ab),\quad   \langle  \pi_k(a)\xi,\eta\rangle=\langle \xi,\pi_{\bar{k}}(a^*)\eta\rangle.
                    \end{aligned}\end{equation*}
                
            \item If $k\in \C^\nu_{\leq 0}$ and $a\in \cA_{k}$, then $\pi_{k}(a)$ extends to a bounded operator $L^2(\pi)\to L^2(\pi)$, where $L^2(\pi)$ is the Hilbert space completion of $C^\infty(\pi)$ with respect to the norm $\norm{\xi}_{L^2(\pi)}^2:=\langle \xi,\xi\rangle$.
                    We also suppose that if $a\in \cA_0$, then 
                        \begin{equation}\label{eqn:bound_norm_rep}\begin{aligned}
                            \norm{\pi_0(a)}\leq \norm{a}_{\overline{\cA_0}}.
                        \end{aligned}\end{equation}
                    So, $\pi_0$ extends to a representation of $\overline{\cA_0}$.
			\item The space $C^\infty(\pi)$ equipped with the topology generated by the family of semi-norms $\norm{\pi_k(a)\xi}_{L^2(\pi)}$ for $k\in \C^\nu$ and $a\in \cA_k$ is complete.
                    \end{enumerate}
            The representation is called irreducible if it is irreducible as a representation of the $C^*$-algebra $\overline{\cA_0}$, i.e., if $L\subseteq L^2(\pi)$ is a closed subspace which is invariant by $\pi_0(\cA_0)$, then $L=0$ or $L=L^2(\pi)$.
    \end{dfn}
            We denote by $C^{-\infty}(\pi)$ the space of continuous anti-linear functionals on $C^\infty(\pi)$. 
        The action of $\xi\in C^{-\infty}(\pi)$ on $\eta\in C^\infty(\pi)$ is denoted by $\langle \eta,\xi\rangle$.
        The space $C^{-\infty}(\pi)$ is equipped with the topology generated by the semi-norms $\xi\mapsto |\langle \eta,\xi\rangle|$ for every $\eta\in C^\infty(\pi)$.
		We have obvious continuous linear inclusions $C^\infty(\pi)\subseteq L^2(\pi)\subseteq C^{-\infty}(\pi)$.
		For each $k\in \C^\nu$, we extend $\pi_k$ to a linear map $\pi_k:\cA_k\to \cL(C^{-\infty}(\pi))$ by the formula 
            \begin{equation*}\begin{aligned}
                \langle \eta,\pi_k(a)\xi\rangle:=\langle \pi_{\bar{k}}(a^*)\eta,\xi \rangle,\quad \forall a\in \cA_k,\xi\in C^{-\infty}(\pi),\eta\in C^{\infty}(\pi).        
            \end{aligned}\end{equation*}

            There are various ways to define Sobolev spaces associated to the representation $\pi$.
            We will use here the most restrictive one. Our main theorem implies that they coincide with the least restrictive one, see \eqref{eqn:HkPi_alternate}.
		Let $H^k(\pi)$ be the set of $\xi\in C^{-\infty}(\pi)$ such that 
			there exists a sequence $(\xi_n)_{n\in \N}\subseteq C^\infty(\pi)$ such that $\xi_n\to \xi$ in $C^{-\infty}(\pi)$, and for any $l\in \C^\nu$ and $a \in\cA_l$ such that $\Re(l)\leq k$, one has $\pi_l(a)\xi\in L^2(\pi)$, and $\pi_l(a)\xi_n\to \pi_l(a)\xi$ in $L^2(\pi)$.
			We equip $H^k(\pi)$ with the topology generated by the family of semi-norms $\xi\mapsto \norm{\pi_l(a)\xi}_{L^2(\pi)}$ for all $l\in \C^\nu$ and $a\in \cA_l$ such that $\Re(l)\leq k$.
            Clearly, 
                \begin{equation*}\begin{aligned}
                    H^0(M)=L^2(\pi),\quad H^l(\pi)\subseteq H^{k}(\pi),\quad \text{if } k\leq l,
                \end{aligned}\end{equation*}
            and if $k\in \C^\nu$, $l\in \R^\nu$ and $a\in \cA_k$, then $\pi_k(a)(H^l(\pi))\subseteq H^{l-\Re(k)}(\pi)$, and the map $H^l(\pi)\xrightarrow{\pi_k(a)}  H^{l-\Re(k)}(\pi)$ is continuous.
             The spaces $C^\infty(\pi)$, $C^{-\infty}(\pi)$, $H^k(\pi)$ are called the space of smooth vectors, distribution vectors and Sobolev spaces associated to the representation $\pi$ respectively.
            \begin{prop}[{\cite{MohsenAbstractMaxHypo}}]One has
                \begin{equation}\label{eqn:Cinfinty_intersection}\begin{aligned}
                    C^\infty(\pi)=\bigcap_{k\in \R^\nu_+}H^k(\pi).
                \end{aligned}\end{equation}
            \end{prop}
        In the following theorem, $K$ is the set of $k\in \C^\nu$ such that $\Re(k_i),\Im(k_i)\in [-\frac{1}{2},\frac{1}{2}]$ for all $i\in \bb{1,\nu}$, and $k_i\neq 0$ for at most one $i$.
    \begin{theorem}[{\cite{MohsenAbstractMaxHypo}}]\label{thm:simple_arg_sobolev}
        Let $(\cA)_{k\in \C^\nu}$ be a $C^*$-calculus, $\Sigma(\cA)$ a set of representations of $(\cA)_{k\in \C^\nu}$ such that the following hold:
        \begin{enumerate}
            \item The $C^*$-algebra $\overline{\cA_0}$ is of Type I.
            \item Any irreducible representation of $\overline{\cA_0}$ is unitarily equivalent to $\pi_0$ for some $\pi\in \Sigma(\cA)$.
            \item We can find a family of elements $v_k\in \cA_k$ for $k\in  K$ such that: \begin{enumerate}
				\item The element $v_0$ is invertible in $\overline{\cA_0}$.
				\item For any $a\in \cA_0$, the map 
                    \begin{equation}\label{eqn:continuity_hypothesis}\begin{aligned}
                        k\in K \mapsto v_{k}av_{-k}\in \overline{\cA_0}
                    \end{aligned}\end{equation}
				is continuous, where $\overline{\cA_0}$ is equipped with the norm topology from $\norm{\cdot}_{\overline{\cA_0}}$.
        \end{enumerate}
        		\end{enumerate}

		Then following hold:
		\begin{enumerate}
            \item\label{thm:simple_arg_sobolev:3} Let $k\in\C^\nu$ and $a\in \cA_k$. If for any $\pi\in \Sigma(\cA)$, $C^\infty(\pi)\xrightarrow{\pi_k(a)} C^\infty(\pi)$ and $C^\infty(\pi)\xrightarrow{\pi_k(a^*)} C^\infty(\pi)$ are injective, then for every representation $\pi$ of $(\cA)_{k\in \C^\nu}$ (not necessarily irreducible and not necessarily in $\Sigma(\cA)$), the maps
                                \begin{equation}\label{eqn:skqdjfmsqmjdfjsmjmfqsfd}\begin{aligned}
                                                                 &\pi_k(a):C^{-\infty}(\pi)\to C^{-\infty}(\pi)\\
&\pi_k(a):H^l(\pi)\to H^{l-\Re(k)}(\pi)&&\forall l\in \R^\nu\\
&\pi_k(a):C^\infty(\pi)\to C^\infty(\pi)
                                \end{aligned}\end{equation}
					are topological isomorphisms.
					Furthermore, for any $k\in \C^\nu$, there exists $v\in \cA_k$ which satisfies the above.
 \item For any representation $\pi$ of $(\cA)_{k\in \C^\nu}$, $H^k(\pi)$ is a Hilbertian space, and \begin{equation}\label{eqn:HkPi_alternate}\begin{aligned}
                        		H^k(\pi)=\{\xi\in C^{-\infty}(\pi):\pi_k(a)\xi\in L^2(\pi),\ \forall a\in \cA_k\}		
                    \end{aligned}\end{equation} and $C^{-\infty}(\pi)=\bigcup_{k\in \R^\nu}H^k(\pi)$.
		
			\item\label{thm:simple_arg_sobolev:inv} If $a\in \cA_0$, then $a$ is left-invertible in $\overline{\cA_0}$ if and only if for all $\pi\in \Sigma(\cA)$, $C^\infty(\pi)\xrightarrow{\pi_0(a)} C^\infty(\pi)$ is injective
		\end{enumerate}
			\end{theorem}

		\end{appendices}

	\begin{refcontext}[sorting=nyt]
		\printbibliography
	\end{refcontext}

\end{document}